\documentclass[a4paper,11pt]{article} 

\usepackage{amsmath} 
\usepackage{amssymb}
\usepackage{amsthm} 
\usepackage{graphicx} 
\usepackage{enumerate}
\usepackage{amsfonts} 
\usepackage{ulem} 
\usepackage{float} 
\usepackage{hyperref} 
\usepackage{cite} 
\usepackage[dvipsnames]{xcolor} 
\usepackage{tikz} 
\usepackage{tikz-cd} 
\usepackage{mathabx} 
\usepackage{faktor} 
\usepackage{array} 
\usepackage{transparent} 
\usepackage{mathdots} 
\usepackage{adjustbox} 
\usepackage{thmtools} 
\usepackage{thm-restate} 
\usepackage[margin=1.25in]{geometry} 
\usepackage{mathtools} 

\usetikzlibrary{decorations.markings} 

\theoremstyle{plain} 
\newtheorem{thm}{Theorem}[subsection] 
\newtheorem{lemma}[thm]{Lemma} 
\newtheorem{prop}[thm]{Proposition} 
\newtheorem{cor}[thm]{Corollary} 

\theoremstyle{definition} 
\newtheorem{defn}[thm]{Definition} 
\newtheorem{example}{Example} 
\newtheorem{notation}[thm]{Notation}
\newtheorem{con}[thm]{Convention}

\theoremstyle{remark} 
\newtheorem*{rem}{Remark} 

\newtheorem{obs}[thm]{Observation} 

\DeclareMathOperator{\aut}{Aut} 
\DeclareMathOperator{\out}{Out} 
\DeclareMathOperator{\outs}{\out_{\mathfrak{S}}} 
\DeclareMathOperator{\inn}{Inn} 
\DeclareMathOperator{\stab}{Stab} 
\DeclareMathOperator{\ad}{Ad} 
\DeclareMathSymbol{\shortminus}{\mathbin}{AMSa}{"39} 
\DeclareMathOperator{\dash }{\textbf{\textemdash}} 

\newcommand{\overbar}[1]{\mkern 1.5mu\overline{\mkern-1.5mu#1\mkern-1.5mu}\mkern 1.5mu} 
\newcommand{\Int}{\aleph_{1}\cap\aleph_{2}} 
\newcommand{\diag}{\begin{tikzpicture} \draw (0,0) -- (0.2,0.1) -- (0.4,0) -- (0.2,-0.1) -- cycle; \draw (0,-0.075) -- (0.2,-0.175); \node at (0,0) {}; \node at (0.4,0) {}; \end{tikzpicture}} 

\tikzset{->-/.style={decoration={markings, mark=at position .55 with {\arrow{>}}}, postaction={decorate}}} 
\tikzset{-<-/.style={decoration={markings, mark=at position .55 with {\arrow{<}}}, postaction={decorate}}} 

\allowdisplaybreaks

		\newcommand{\Taleph}[5]{
		\begin{tikzpicture}[scale=0.8] 
		\draw[thick] (0,0) -- (-0.5,0.866);
		\draw[thick] (0,0) -- (-0.5,-0.866);
		\draw[thick] (0,0) -- (1,0);
		\draw[thick] (2,0) -- (2.5,0.866);
		\draw[thick] (2,0) -- (2.5,-0.866);
		\draw[thick] (2,0) -- (1,0);
		\draw[yellow, fill] (0,0) circle [radius=0.09];
		\draw[red, fill] (-0.5,0.866) circle [radius=0.09];
		\draw[red, fill] (-0.5,-0.866) circle [radius=0.09];
		\filldraw[color=blue, fill=yellow, very thick] (1,0) circle [radius=0.125];
		\draw[yellow, fill] (2,0) circle [radius=0.09];
		\draw[red, fill] (2.5,0.866) circle [radius=0.09];
		\draw[red, fill] (2.5,-0.866) circle [radius=0.09];
		\node[text=blue] at (1,0.4) {#1};
		\node at (-0.8,0.9) {#2};
		\node at (-0.8,-0.9) {#3};
		\node at (2.8,0.9) {#4};
		\node at (2.8,-0.9) {#5};
		\end{tikzpicture}
		}
		
		\newcommand{\Trho}[3]{
		\begin{tikzpicture}[scale=0.8]  
		\draw[thick] (0,0) -- (-0.5,0.866);
		\draw[thick] (0,0) -- (-0.5,-0.866);
		\draw[thick] (0,0) -- (1,0);
		\draw[yellow, fill] (0,0) circle [radius=0.09];
		\draw[red, fill] (-0.5,0.866) circle [radius=0.09];
		\draw[red, fill] (-0.5,-0.866) circle [radius=0.09];
		\filldraw[color=blue, fill=yellow, very thick] (1,0) circle [radius=0.125];
		\node[text=blue] at (1,0.4) {#1};
		\node at (-0.8,0.9) {#2};
		\node at (-0.8,-0.9) {#3};
		\end{tikzpicture}
		}
		
		\newcommand{\Tsigma}[6]{
		\begin{tikzpicture}[scale=0.8]  
		\draw[thick] (0,0) -- (-0.5,0.866);
		\draw[thick] (0,0) -- (-0.5,-0.866);
		\draw[thick] (0,0) -- (1,0);
		\draw[thick] (2,0) -- (2.5,0.866);
		\draw[thick] (2,0) -- (2.5,-0.866);
		\draw[thick] (2,0) -- (1,0);
		\draw[yellow, fill] (0,0) circle [radius=0.09];
		\draw[red, fill] (-0.5,0.866) circle [radius=0.09];
		\draw[red, fill] (-0.5,-0.866) circle [radius=0.09];
		\filldraw[color=blue, fill=red, very thick] (1,0) circle [radius=0.125];
		\draw[yellow, fill] (2,0) circle [radius=0.09];
		\draw[red, fill] (2.5,0.866) circle [radius=0.09];
		\draw[red, fill] (2.5,-0.866) circle [radius=0.09];
		\node[text=blue] at (1,0.4) {#1};
		\node at (1,-0.4) {#2};
		\node at (-0.8,0.9) {#3};
		\node at (-0.8,-0.9) {#4};
		\node at (2.8,0.9) {#5};
		\node at (2.8,-0.9) {#6};
		\end{tikzpicture}
		}
		
		\newcommand{\Ttau}[5]{
		\begin{tikzpicture}[scale=0.8]  
		\draw[thick] (0,0) -- (-1,0);
		\draw[thick] (0,0) -- (1,0);
		\draw[thick] (2,0) -- (2.5,0.866);
		\draw[thick] (2,0) -- (2.5,-0.866);
		\draw[thick] (2,0) -- (1,0);
		\draw[red, fill] (0,0) circle [radius=0.09];
		\draw[red, fill] (-1,0) circle [radius=0.09];
		\filldraw[color=blue, fill=yellow, very thick] (1,0) circle [radius=0.125];
		\draw[yellow, fill] (2,0) circle [radius=0.09];
		\draw[red, fill] (2.5,0.866) circle [radius=0.09];
		\draw[red, fill] (2.5,-0.866) circle [radius=0.09];
		\node[text=blue] at (1,0.4) {#1};
		\node at (0,-0.4) {#2};
		\node at (-1,-0.4) {#3};
		\node at (2.8,0.9) {#4};
		\node at (2.8,-0.9) {#5};
		\end{tikzpicture}
		}
		
		\newcommand{\Talpha}[1]{
		\begin{tikzpicture}[scale=0.8]  
		\filldraw[color=blue, fill=yellow, very thick] (1,0) circle [radius=0.125];
		\node[text=blue] at (1,0.4) {#1};
		\end{tikzpicture}
		}
		
		\newcommand{\Tbeta}[3]{
		\begin{tikzpicture}[scale=0.8]  
		\draw[thick] (0,0) -- (-1,0);
		\draw[thick] (0,0) -- (1,0);
		\draw[red, fill] (0,0) circle [radius=0.09];
		\draw[red, fill] (-1,0) circle [radius=0.09];
		\filldraw[color=blue, fill=yellow, very thick] (1,0) circle [radius=0.125];
		\node[text=blue] at (1,0.4) {#1};
		\node at (0,-0.4) {#2};
		\node at (-1,-0.4) {#3};
		\end{tikzpicture}
		}
		
		\newcommand{\Tgamma}[4]{
		\begin{tikzpicture}[scale=0.8]  
		\draw[thick] (0,0) -- (-0.5,0.866);
		\draw[thick] (0,0) -- (-0.5,-0.866);
		\draw[thick] (0,0) -- (1,0);
		\draw[yellow, fill] (0,0) circle [radius=0.09];
		\draw[red, fill] (-0.5,0.866) circle [radius=0.09];
		\draw[red, fill] (-0.5,-0.866) circle [radius=0.09];
		\filldraw[color=blue, fill=red, very thick] (1,0) circle [radius=0.125];
		\node[text=blue] at (1,0.4) {#1};
		\node at (1,-0.4) {#2};
		\node at (-0.8,0.9) {#3};
		\node at (-0.8,-0.9) {#4};
		\end{tikzpicture}
		}
		
		\newcommand{\Tdelta}[6]{
		\begin{tikzpicture}[scale=0.8]  
		\draw[thick] (0,0) -- (-1,0);
		\draw[thick] (0,0) -- (1,0);
		\draw[thick] (2,0) -- (2.5,0.866);
		\draw[thick] (2,0) -- (2.5,-0.866);
		\draw[thick] (2,0) -- (1,0);
		\draw[red, fill] (0,0) circle [radius=0.09];
		\draw[red, fill] (-1,0) circle [radius=0.09];
		\filldraw[color=blue, fill=red, very thick] (1,0) circle [radius=0.125];
		\draw[yellow, fill] (2,0) circle [radius=0.09];
		\draw[red, fill] (2.5,0.866) circle [radius=0.09];
		\draw[red, fill] (2.5,-0.866) circle [radius=0.09];
		\node[text=blue] at (1,0.4) {#1};
		\node at (1,-0.4) {#2};
		\node at (0,-0.4) {#3};
		\node at (-1,-0.4) {#4};
		\node at (2.8,0.9) {#5};
		\node at (2.8,-0.9) {#6};
		\end{tikzpicture}
		}
		
		\newcommand{\Tepsilon}[5]{
		\begin{tikzpicture}[scale=0.8]  
		\draw[thick] (0,0) -- (-1,0);
		\draw[thick] (0,0) -- (1,0);
		\draw[thick] (2,0) -- (3,0);
		\draw[thick] (2,0) -- (1,0);
		\draw[red, fill] (0,0) circle [radius=0.09];
		\draw[red, fill] (-1,0) circle [radius=0.09];
		\filldraw[color=blue, fill=yellow, very thick] (1,0) circle [radius=0.125];
		\draw[red, fill] (2,0) circle [radius=0.09];
		\draw[red, fill] (3,0) circle [radius=0.09];
		\node[text=blue] at (1,0.4) {#1};
		\node at (0,-0.4) {#2};
		\node at (-1,-0.4) {#3};
		\node at (2,-0.4) {#4};
		\node at (3,-0.4) {#5};
		\end{tikzpicture}
		}
		
		\newcommand{\TA}[2]{
		\begin{tikzpicture}[scale=0.8]  
		\filldraw[color=blue, fill=red, very thick] (1,0) circle [radius=0.125];
		\node[text=blue] at (1,0.4) {#1};
		\node at (1,-0.4) {#2};
		\end{tikzpicture}
		}
		
		\newcommand{\TB}[4]{
		\begin{tikzpicture}[scale=0.8]  
		\draw[thick] (0,0) -- (-1,0);
		\draw[thick] (0,0) -- (1,0);
		\draw[red, fill] (0,0) circle [radius=0.09];
		\draw[red, fill] (-1,0) circle [radius=0.09];
		\filldraw[color=blue, fill=red, very thick] (1,0) circle [radius=0.125];
		\node[text=blue] at (1,0.4) {#1};
		\node at (1,-0.4) {#2};
		\node at (0,-0.4) {#3};
		\node at (-1,-0.4) {#4};
		\end{tikzpicture}
		}
		
		\newcommand{\TC}[6]{
		\begin{tikzpicture}[scale=0.8]  
		\draw[thick] (0,0) -- (-1,0);
		\draw[thick] (0,0) -- (1,0);
		\draw[thick] (2,0) -- (3,0);
		\draw[thick] (2,0) -- (1,0);
		\draw[red, fill] (0,0) circle [radius=0.09];
		\draw[red, fill] (-1,0) circle [radius=0.09];
		\filldraw[color=blue, fill=red, very thick] (1,0) circle [radius=0.125];
		\draw[red, fill] (2,0) circle [radius=0.09];
		\draw[red, fill] (3,0) circle [radius=0.09];
		\node[text=blue] at (1,0.4) {#1};
		\node at (1,-0.4) {#2};
		\node at (0,-0.4) {#3};
		\node at (-1,-0.4) {#4};
		\node at (2,-0.4) {#5};
		\node at (3,-0.4) {#6};
		\end{tikzpicture}
		}
		
		\newcommand{\Talphaexp}[5]{
		\scalebox{1}{
		\begin{tikzpicture}[scale=1]  
		\draw[thick] (0,0) -- (0,1); 
		\draw[thick] (0,0) -- (0.707,0.707); 
		\draw[thick] (0,0) -- (1,0); 
		\draw[thin] (0.177,-0.177) -- (0.530,-0.530);
		\draw[thin] (0,-0.25) -- (0,-0.75);
		\draw[thin] (-0.177,-0.177) -- (-0.530,-0.530);
		\draw[thick] (0,0) -- (-1,0); 
		\draw[thick] (0,0) -- (-0.707,0.707); 
		\filldraw (0,0) circle [radius=0.09cm];
		\filldraw (0,1) circle [radius=0.07cm];
		\filldraw (0.707,0.707) circle [radius=0.07cm];
		\filldraw (1,0) circle [radius=0.07cm];
		\filldraw (0.707,-0.707) circle [radius=0.065cm];
		\filldraw (0,-1) circle [radius=0.065cm];
		\filldraw (-0.707,-0.707) circle [radius=0.065cm];
		\filldraw (-1,0) circle [radius=0.07cm];
		\filldraw (-0.707,0.707) circle [radius=0.07cm];
		\node at (0,1.275) {#1};
		\node at (0.95,0.9) {#2};
		\node at (1.4,-0.025) {#3};
		\node at (-1.5,0) {#4};
		\node at (-0.95,0.95) {#5};
		\end{tikzpicture}
		}}
		
\begin{document}

\title{A Presentation for the Group of Pure Symmetric Outer Automorphisms of a Given Splitting of a Free Product}
\author{Harry M. J. Iveson\thanks{H.M.J.Iveson@soton.ac.uk}}
\date{\today}

\maketitle


\begin{abstract}
We give a concise presentation for the group of pure symmetric outer automorphisms of a given splitting of a free product $G_{1}\ast\dots\ast G_{n}$.
These are the (outer) automorphisms which preserve the conjugacy classes of the free factors $G_{i}$.
This is achieved by considering the action of these automorphisms on a particular subcomplex of `Outer Space', which we show to be simply connected.
We then apply a theorem of K. S. Brown to extract our presentation.
\end{abstract}

\section*{Introduction}\label{section introduction}

The study of group presentations, especially finite ones, is a core part of geometric group theory, dating back to the work of M. Dehn in the 1910's.
Providing such group presentations is not only necessary for such study, but interesting in and of itself.
Automorphism groups of free groups, and more generally, of free products, are natural objects to consider in this area.
Different presentations may display various desirable properties, such as having few generators or relations, or highlighting some structure of the group.

In 2008, H. Armstrong, B. Forrest, and K. Vogtmann \cite{Armstrong2008} gave a finite presentation for $\aut(F_{r})$, the automorphism group of a free group of rank $r$.
They achieved this by applying a theorem of Brown \cite[Theorem 1]{Brown1984} to a subcomplex of a version of M. Culler and K. Vogtmann's `Auter Space' \cite{Culler1986} on which $\aut(F_{r})$ acts `nicely'.

While finite presentations for $\aut(F_{r})$ were already known (for example, see the works of J. Nielsen \cite{Nielsen1924} from 1924, whose presentation demonstrates the surjectivity of the map to $GL_{n}(\mathbb{Z})$, and B. Neumann \cite{Neumann1933} from 1933, whose presentation had only 2 generators, but many relations),  Armstrong,  Forrest, and  Vogtmann \cite{Armstrong2008} gave a presentation whose generators are all involutions, with a relatively small number of relations, making it straightforward to comprehend and apply.

In 1986, J. McCool \cite{Mccool1986} gave a concise presentation for the subgroup of $\aut(F_{r})$ comprising automorphisms which map each generator to a conjugate of itself.
McCool's presentation comprised $r^{2}-r$ generators, but only three (families of) relations.

There is a longstanding trend of generalising results from automorphisms of free groups to automorphisms of free products.
In the 1940's, D. I. Fouxe-Rabinovitch \cite{F-R1940}, \cite{F-R1941} gave a finite presentation for the automorphism group of a free product $G=G_{1}\ast\dots\ast G_{n}\ast F_{k}$, where $F_{k}$ is the free group of rank $k$ and where each $G_{i}$ is non-trivial, freely indecomposable, and not infinite cyclic (i.e. $G_{i}\not\cong\mathbb{Z}$).
N. D. Gilbert \cite[Theorem 2.20]{Gilbert1987} gave an equivalent presentation for $\aut(G)$ in 1987 with fewer relations, using `peak reduction' methods of J. H. C. Whitehead \cite{Whitehead1936} adapted to the free product case by D. J. Collins and H. Zieschang \cite{Collins1984}.

\subsection*{Main Result}

We follow the methods of Armstrong, Forrest, and Vogtmann \cite{Armstrong2008} to give a concise presentation for the group of pure symmetric outer automorphisms of a given splitting $G_{1}\ast\dots\ast G_{n}$ of  free product $G$, denoted $\out(G;G_{1},\dots,G_{n})$.
In our case, we have a `strict fundamental domain' for the action of $\out(G;G_{1},\dots,G_{n})$, so can apply a more straightforward theorem of Brown \cite[Theorem 3]{Brown1984} to extract our presentation, given below:

\begin{restatable*}{thm}{presentation}\label{thm n>=5 presentation}
Let $G_{1}\ast\dots\ast G_{n}$ be a free splitting of a group $G$ where each $G_{i}$ is non-trivial and $n\ge5$.
For $i\in[n]:=\{1,\dots,n\}$ and $j\in[n]-\{i\}$, let $f_{i_{j}}:G_{i}\to G_{i_{j}}$ be group isomorphisms, and for $g\in G_{i}$ let $\ad_{G_{i}}(g)$ be the inner automorphism $x\mapsto g^{-1}xg$ of $G_{i}$.
Then the group $\out(G;G_{1},\dots,G_{n})$ is generated by the $n(n-1)$ groups $G_{i_{j}}\cong G_{i}$ and $\Phi=\prod_{i=1}^{n}\aut(G_{i})$,
subject to relations:
\begin{enumerate}
\item $[f_{i_{j}}(g), f_{i_{k}}(h)]=1$ $\forall g,h\in G_{i}$, for all $i\in[n]$, $j,k\in[n]-\{i\}$
\item $[f_{i_{j}}(g), f_{k_{l}}(h)]=1$ $\forall g\in G_{i}, h\in G_{k}$, for all distinct $i,j,k,l\in[n]$
\item $[f_{j_{k}}(g),f_{i_{j}}(h)f_{i_{k}}(h) ]=1$ $\forall g\in G_{j}, h\in G_{i}$, for all distinct $i,j,k\in[n]$
\item $f_{i_{v_{1}}}(g)\dots f_{i_{v_{n-1}}}(g)=\ad_{G_{i}}(g)$ $\forall g\in G_{i}$, for all $i\in[n]$ and $\{v_{1},\dots,v_{n-1}\}=[n]-\{i\}$
\item $\varphi^{-1}f_{i_{j}}(g)\varphi=f_{i_{j}}(\varphi(g))$ $\forall g\in G_{i}$, for all distinct $i,j\in[n]$ and all $\varphi\in\Phi$
\end{enumerate}
As well as all relations in $G$ and $\Phi$.
\end{restatable*}

\begin{restatable*}{cor}{presentationcor}
If a group $G$ splits as a free product where the factor groups are non-trivial, freely indecomposable, not infinite cyclic, and pairwise non-isomorphic, then Theorem \ref{thm n>=5 presentation} gives a presentation for $\out(G)$.
\end{restatable*}

\begin{obs}
In the case where some of the factor groups may be isomorphic, one may choose to study the symmetric automorphisms  of the splitting.
Then a finite direct product of symmetric groups, $\Pi$, acts on the splitting by permuting all possible isomorphic factors.
The group of symmetric outer automorphisms of the splitting is then given by $\out(G;G_{1},\dots,G_{n})\rtimes\Pi$.
While this is be hard to see geometrically using the methods of this paper, it may deduced algebraically.
\end{obs}

The cases $n=4$ and $n=3$ are similar, and are given in Theorems \ref{thm n=4 presentation} and \ref{thm n=3 presentation} in Section \ref{section presentation}.

If each of the groups $G_{i}$ and $\aut(G_{i})$ are finitely presented, then one may extract a finite presentation for $\out(G;G_{1},\dots,G_{n})$ from this theorem by replacing each group $G_{i_{j}}$ with a set of elements $\{f_{i_{j}}(g_{1}),\dots,f_{i_{j}}(g_{m_{i}})\}$ such that the $g_{k}$'s generate $G_{i}$, and replacing the group $\Phi$ with a generating set $\{\varphi_{1},\dots,\varphi_{m_{\Phi}}\}$. 
For conciseness, we do not make this more formal.
%

Our result may be considered to be a generalisation of McCool's presentation for the group of pure symmetric automorphisms of a free group.

\begin{thm}[McCool \cite{Mccool1986}]
Let $F_{r}=\langle x_{1},\dots,x_{r} \rangle$ be the free group on $r$ generators.
The group of pure symmetric automorphisms of $F_{r}$ is generated by $r(r-1)$ elements $(x_{i};x_{j})$ (for $i,j\in\{1,\dots,r\}$ and $i\ne j$), subject to commutation relations:
\begin{enumerate}
\item $(x_{i};x_{j})(x_{k};x_{j})=(x_{k};x_{j})(x_{i};x_{j})$
\item $(x_{i};x_{j})(x_{k};x_{l})=(x_{k};x_{l})(x_{i};x_{j})$
\item $(x_{i};x_{j})(x_{k};x_{j})(x_{i};x_{k})=(x_{i};x_{k})(x_{i};x_{j})(x_{k};x_{j})$
\end{enumerate}
(where $i,j,k,l$ are assumed to be distinct).
\end{thm}

Observe by comparing indices that these relations directly translate to our Relations 1--3.
Our Relation 4 is an `outer' relation so is not present in the automorphism group, and our Relation 5 describes automorphisms within a given factor, which are trivial in McCool's case.

\smallskip
In the case $n=3$, we recover
a special case of Gilbert's result \cite[Theorem 2.20]{Gilbert1987}, given by D. J. Collins and N. D. Gilbert \cite{Collins1990} in 1990, for three freely indecomposable, non-trivial, not infinite cyclic, pairwise non-isomorphic factors:

\begin{thm}[Collins--Gilbert {\cite[Proposition 4.1]{Collins1990}}]
For $G=X\ast Y\ast Z$ where each of $X, Y, Z$ is freely indecomposable, non-trivial, not infinite cyclic, and where none of $X,Y,Z$ are isomorphic to each other, we have that $\out(G)$ is generated by \[ \{ (Y,x), (Z,y), (X,z), \varphi | x\in X, y\in Y, z\in Z, \varphi\in \Phi \} \]
where $\Phi$ is the set of {factor automorphisms} (see Definition \ref{defn factor autos}),
subject to relations:
\begin{itemize}
\item $(Y,x_{1})(Y,x_{2})=(Y,x_{1}x_{2})$
\item $(Z,y_{1})(Z,y_{2})=(Z,y_{1}y_{2})$
\item $(X,z_{1})(X,z_{2})=(X,z_{1}z_{2})$
\item $\varphi^{-1}(Y,x)\varphi=(Y,x\varphi)$
\item $\varphi^{-1}(Z,y)\varphi=(Z,y\varphi)$
\item $\varphi^{-1}(X,z)\varphi=(X,z\varphi)$
\item All relations from $\Phi$
\end{itemize}
In particular,
$\out(G)\cong G\rtimes\Phi$.
\end{thm}

Collins and Gilbert's result may be thought of as a presentation for the pure symmetric outer automorphisms preserving a free splitting structure $G_{1}\ast G_{2}\ast G_{3}$ where the only condition on the $G_{i}$'s is that they are non-trivial.
Thus our result may also be seen as both a special case of Gilbert's presentation \cite[Theorem 2.20]{Gilbert1987} and a generalisation of Collins and Gilbert's presentation \cite[Proposition 4.1]{Collins1990}.

\smallskip
In future, we hope to generalise this further to free splittings of the form $G_{1}\ast\dots\ast G_{n}\ast F_{k}$ where automorphisms need not preserve the conjugacy classes of the generators for $F_{k}$.
However this greatly increases the number of cells in our chosen subcomplex of Outer Space.
Moreover the fundamental domain of the action ceases to be strict, meaning we can no longer apply the simplified version of Brown's theorem.
These complications increase the complexity of the problem, though we hope that the end result will still be a pleasing presentation.
%

\subsection*{Methods and Techniques}

To achieve our presentation, we choose a particular subcomplex of the `Outer Space' for a free product with no free rank, introduced by D. McCullough and A. Miller \cite{McCullough1996} in order to study the symmetric automorphisms of a free product.
We work with the definition of the space provided by Guirardel and Levitt \cite{Guirardel2007}, since this interpolates between the Outer Spaces of Culler and Vogtmann \cite{Culler1986} and of McCullough and Miller \cite{McCullough1996}, which lends itself well to future work in the case of a splitting $G_{1}\ast\dots\ast G_{n}\ast F_{k}$.

We call our chosen complex $\mathcal{C}_{n}$, discussed in Section \ref{section subcomplex}.
In the cases $n=3$ and $n=4$, $\mathcal{C}_{n}$ is precisely the barycentric spine of Guirardel and Levitt's Outer Space for a free product whose Grushko decomposition has four non-isomorphic free factors and no free rank (see Section \ref{subsection outer space}).
Definition \ref{def subcomplex Cn} details the construction of the complex $\mathcal{C}_{n}$ for $n\ge5$.

\smallskip
In order to apply Brown's theorem \cite[Theorem 3]{Brown1984}, we require that $\outs(G)$ acts cellularly on our complex $\mathcal{C}_{n}$ and with a strict fundamental domain, and that the complex $\mathcal{C}_{n}$ is both connected and simply connected. We will also need suitable presentations for $\outs(G)$-stabilisers of vertices (graphs of groups) in $\mathcal{C}_{n}$.

The action of $\outs(G)$ on $\mathcal{C}_{n}$ and its fundamental domain are studied in Section \ref{subsection action on Cn}, and the connectedness of the complex $\mathcal{C}_{n}$ is Corollary \ref{cor Cn pc} of Section \ref{section connected}.

Vertex stabilisers are studied in Section \ref{stabilisers}, where we combine techniques of Guirardel and Levitt \cite{Guirardel2007} with those of H. Bass and R. Jiang \cite{Bass1996} to procure presentations which are both concise and precise (Propositions \ref{prop brown stabs 1}, \ref{prop brown stabs 2}, and \ref{prop brown stabs 3}).

Showing that the complex $\mathcal{C}_{n}$ is simply connected is highly non-trivial and is delayed until the second half of the paper, comprising Sections \ref{section space of domains} and \ref{section peak reduction}.
We give a brief overview of the idea of the proof below.

\smallskip

In 1928, P. Alexandroff \cite{Alexandroff1928} introduced the notion of a `nerve complex' associated to a cover of a space. In ideal conditions, this shares many of the same topological properties as the original space, while often being a much simpler object to understand.

We apply a similar concept in Section \ref{section space of domains}, introducing the `Space of Domains' (see Definition \ref{Space of Domains defn}) as a way of recording intersection patterns of $\outs(G)$-images of the fundamental domain in $\mathcal{C}_{n}$.
Unlike Alexandroff's nerve complex, we are only interested in 2-way and 3-way intersections.

We show in Proposition \ref{prop Cn is sc if SoD is} that in order to prove simple connectivity of the complex $\mathcal{C}_{n}$, it suffices to show that the Space of Domains is simply connected (having already shown that our fundamental domain of the $\outs(G)$-action on $\mathcal{C}_{n}$ is simply connected in Theorem \ref{thm fun dom sc} of Section \ref{section fun dom sc}).

Finally in Section \ref{section peak reduction} we apply `peak reduction' techniques as used by Collins and Zieschang \cite{Collins1984} and Gilbert \cite{Gilbert1987} to deduce that the Space of Domains is simply connected (Theorem \ref{Space of Domains is simply connected}).

\subsection*{Acknowledgements}

I am grateful to my supervisor, Armando Martino, for his continued help and support. I am also grateful to Naomi Andrew for her helpful and generous feedback on this manuscript.

\setcounter{tocdepth}{1}
\tableofcontents


\section{Preliminaries}\label{section background}

\subsection{Some Useful Definitions and Notation}\label{subsection background}

We adapt the notation for automorphisms used by Gilbert \cite[Section 1]{Gilbert1987}.
Throughout, we consider a group $G$ which splits as a free product $G_{1}\ast\dots\ast G_{n}$, where each $G_{i}$ is non-trivial and $n\ge3$.
We refer to each $G_{i}$ as a \emph{factor group}.

\begin{notation}
Let $G$ be a group.
We denote by $\aut(G)$ the group of automorphisms of $G$, that is, isomorphisms from $G$ to itself.
We say $\psi\in\aut(G)$ is an \emph{inner automorphism} if there exists $x\in G$ so that for all $g\in G$, $\psi(g)=g^{x}=x^{-1}gx$.
The collection of inner automorphisms forms a normal subgroup, $\inn(G)$, of $\aut(G)$.
We then define $\out(G):=\faktor{\aut(G)}{\inn(G)}$, and call this the \emph{outer automorphism group} of $G$.
\end{notation}

\begin{defn}[Pure Symmetric Automorphism]\label{defn pure symmetric autos}
Let $G=G_{1}\ast\dots\ast G_{n}$ be a group which splits as a free product. We say $\psi\in\aut(G)$ is a \emph{pure symmetric automorphism} of the splitting $G_{1}\ast\dots\ast G_{n}$ if for each $i$ there is some $g_{i}\in G$ such that $\psi(G_{i})=G_{i}^{g_{i}}=g_{i}^{-1}G_{i}g_{i}$.
We say $\hat{\psi}\in\out(G)$ is a \emph{pure symmetric outer automorphim} of the splitting if there is some $\psi\in\hat{\psi}$ which is a pure symmetric automorphism of the splitting.
\end{defn}

\begin{rem}
It is easy to see that if $\psi$ is a pure symmetric automorphism of some free splitting, and $\iota$ is an inner automorphism of the free product, then $\iota\psi$ is also a pure symmetric automorphism of the splitting. Thus the concept of `pure symmetric outer automorphisms' is well-defined.
It is not hard to verify that the collection of pure symmetric (outer) automorphisms forms a subgroup of $\aut(G)$ (respectively, $\out(G)$).
\end{rem}

\begin{notation}
We denote by $\out(G;G_{1}\ast\dots\ast G_{n})$ the subgroup of $\out(G)$ comprising pure symmetric outer automorphisms of the splitting $G_{1}\ast\dots\ast G_{n}$ of $G$.
Given such a splitting, we may set $\mathfrak{S}$ to be the tuple $(G_{1},\dots,G_{n})$ and let $\out(G;G_{1},\dots,G_{n})=:\outs(G)$, for brevity.
We may similarly define $\aut(G;G_{1},\dots,G_{n})$ and $\aut_{\mathfrak{S}}(G)$.
We will sometimes refer to $\mathfrak{S}$ itself as the splitting, as opposed to the product $G_{1}\ast\dots\ast G_{n}$.
\end{notation}

\begin{obs}
Given a splitting $G=G_{1}\ast\dots\ast G_{n}$ with $\mathfrak{S}=(G_{1},\dots,G_{n})$, it is clear that $\inn(G)\subseteq\aut_{\mathfrak{S}}(G)$.
Since $\inn(G)\trianglelefteq\aut(G)$ then $\inn(G)\trianglelefteq\aut_{\mathfrak{S}}(G)$, and it follows that $\faktor{\aut_{\mathfrak{S}}(G)}{\inn(G)}\cong\outs(G)$, as one would expect.
\end{obs}

\begin{defn}[Factor Automorphism]\label{defn factor autos}
We say $\varphi\in\aut(G;G_{1},\dots,G_{n})$ is a \emph{factor automorphism} if for each $i\in\{1,\dots,n\}$, $\varphi|_{G_{i}}$ (that is, $\varphi$ with domain restricted to the embedding of $G_{i}$ in $G$) is an automorphism of $G_{i}$ (i.e. $\varphi|_{G_{i}}\in\aut(G_{i})$).
We will say $\hat{\varphi}\in\out(G;G_{1},\dots,G_{n})$ is a \emph{factor automorphism} if $\hat{\varphi}$ has a representative $\varphi\in\aut(G;G_{1},\dots,G_{n})$ which is a factor automorphism.
We will denote the set of factor automorphisms in $\out(G;G_{1},\dots,G_{n})$ by $\Phi$. 
\end{defn}

The set of factor automorphisms $\Phi$ forms a subgroup of $\out(G;G_{1},\dots,G_{n})$, with $\Phi\cong\displaystyle\prod_{i=1}^{n}\aut(G_{i})$.

\begin{notation}\label{notation ad(g)}
We write $\ad_{G_{i}}(g)$ for the inner automorphism of $G_{i}$ which conjugates each element of $G_{i}$ by $g$ (with $g\in G_{i}$).
Since $\ad_{G_{i}}(g)\in\inn(G_{i})\le\aut(G_{i})$, then $\ad_{G_{i}}(g)\in\Phi\le\out(G;G_{1},\dots,G_{n})$.
Note however that $\ad_{G_{i}}(g)$ is \textbf{not} in $\inn(G)$.
\end{notation}

We will often abuse notation by writing $\psi$ for both an automorphism in $\aut(G)$ (or $\aut(G;G_{1},\dots,G_{n})$), and for the class it represents in $\out(G)$ (or $\out(G;G_{1},\dots,G_{n})$).

\begin{defn}\label{defn s labelling}
Let $\mathfrak{S}=(G_{1},\dots,G_{n})$ be the tuple associated to a group $G$ which splits as a free product $G_{1}\ast\dots\ast G_{n}$, and let $T$ be a finite tree on at least $n$ vertices.
\\ \noindent A free product $H_{1}\ast\dots\ast H_{n}$ is an \emph{$\mathfrak{S}$ free factor splitting} for $G=G_{1}\ast\dots\ast G_{n}$ if for each $i$, there exists $g_{i}\in G$ so that $H_{i}=G_{i}^{g_{i}}$, and the subgroups $G_{1}^{g_{1}},\dots, G_{n}^{g_{n}}$ generate the group $G$.
\\ \noindent An \emph{$\mathfrak{S}$-labelling} of $T$ is an assignment of $n$ vertex groups $H_{v}$ to vertices $v\in V(T)$ so that $H_{1}\ast\dots\ast H_{n}$ is an $\mathfrak{S}$ free factor splitting for $G$.
\\ \noindent Given an $\mathfrak{S}$-labelling $(H_{1},\dots,H_{n})$ of $T$, we may consider the graph of groups $\mathbf{T}=(T,(H_{1},\dots,H_{n}))$ formed by associating the trivial group $\{1\}$ to any remaining vertices of $T$, and setting all edge groups to also be trivial.
\end{defn}

\begin{lemma}\label{lemma automorphisms exist between splittings}
Let $G$ be a group with splitting $\mathfrak{S}=(G_{1},\dots,G_{n})$ and let $H_{1}\ast\dots\ast H_{n}$ be an $\mathfrak{S}$ free factor splitting for $G$.
Then there exists $\psi\in\aut_{\mathfrak{S}}(G)$ with $\psi(G_{i})=H_{i}$ for each $i$.
\end{lemma}

\begin{proof}
Since $H_{1}\ast\dots\ast H_{n}$ is an $\mathfrak{S}$ free factor splitting for $G$, for each $i$ there exists $g_{i}\in G$ so that $H_{i}=G_{i}^{g_{i}}$.
Let $\psi_{i}:G_{i}\to H_{i}$ be the map $\psi(x)=g_{i}^{-1}xg_{i}$ $\forall x\in G_{i}$. Clearly, $\psi_{i}$ is an isomorphism of (sub)groups.
By the universal property of free products, these isomorphisms $\psi_{i}$ extend to an endomorphism $\psi:G\to G$.
Since $H_{1}\ast\dots\ast H_{n}$ is an $\mathfrak{S}$ free factor splitting of $G_{1}\ast\dots\ast G_{n}$, then $G$ is generated by the subgroups $H_{1},\dots,H_{n}$, and so $\psi$ is surjective.
Repeating this process on the maps $\psi_{i}^{-1}:H_{i}\to G_{i}$, we recover a surjective homomorphism $\varphi:G\to G$, which composes with $\psi$ to give the identity map.
Thus $\varphi$ is an inverse for $\psi$, and so $\psi\in\aut(G)$.
Moreover, $\psi$ restricts to $\psi_{i}$ on each $G_{i}$, that is, $\psi(G_{i})=G_{i}^{g_{i}}=H_{i}$, and so $\psi\in\aut_{\mathfrak{S}}(G)$, as required.
\end{proof}

\begin{defn}[Whitehead Automorphism]\label{defn whitehead autos}
An automorphism in $\aut(G;G_{1},\dots,G_{n})$ which, for each $j$, either pointwise fixes $G_{j}$, or pointwise conjugates $G_{j}$ by a given $x\in G$ is called a \emph{Whitehead automorphism}.
Given $x\in G_{i}$ and $A\subseteq\{G_{1},\dots,G_{n}\}-\{G_{i}\}$, we write $(A,x)$ for the Whitehead automorphism which pointwise fixes any $G_{j}\not\in A$, and pointwise conjugates by $x$ any $G_{j}\in A$.

Given finite sequences $\mathbf{x}=(x_{1},\dots,x_{k})\subset G_{i}$ and $\mathbf{A}=(A_{1},\dots,A_{k})\subset\{G_{1},\dots,G_{n}\}-\{G_{i}\}$ (with the $A_{j}$'s pairwise disjoint), we write $(\mathbf{A},\mathbf{x})$ for the composition $(A_{1},x_{1})\dots(A_{k},x_{k})$ (which should be read from left to right, since we consider the action of $\aut(G)$ on $G$ to be a right action). We call such a map a \emph{multiple Whitehead automorphism}.

An element $\hat{\psi}\in\out(G;G_{1},\dots,G_{n})$ will be called a (multiple) Whitehead automorphism if it has some representative $\psi\in\aut(G;G_{1},\dots,G_{n})$ which is a (multiple) Whitehead automorphism.
\end{defn}

\begin{rem}
Our notation differs from that of Gilbert \cite{Gilbert1987} in that we decide not to include the operating factor $G_{i}$ (see below) in the set $A$.
\end{rem}

Note that we write $g^{x}$ for the conjugation $x^{-1}gx$.
More detail on Whitehead automorphisms, including relative Whitehead automorphisms, can be found in Section \ref{whitehead autos}.

\begin{notation}\label{defn operating factor}
Given factor groups $G_{i}$ and $G_{j}$, we will write $G_{i_{j}}$ (sometimes abbreviated as $i_{j}$) for the group generated by automorphisms $(G_{j},x)$ where $x\in G_{i}$.
We call $G_{i}$ the \emph{operating factor} and $G_{j}$ the \emph{dependant factor}.
Additionally, given factor groups $G_{i}$ and $G_{v_{1}},\dots,G_{v_{k}}$, we will write $i_{v_{1}\dots v_{k}}$ (or $G_{i_{v_{1}\dots v_{k}}}$) for the subgroup of $i_{v_{1}}\times\dots\times i_{v_{k}}$ generated by the Whitehead automorphisms $(\{G_{v_{1}},\dots,G_{v_{k}}\},x)$ where $x\in G_{i}$.
We think of this as the diagonal subgroup, and denote this by $i_{v_{1}\dots v_{k}}\diag i_{v_{1}}\times\dots\times i_{v_{k}}$.
\end{notation}

\begin{obs}
Defining multiplication in $G_{i_{j}}$ by $(G_{j},x)(G_{j},y)=(G_{j},yx)$, we have a natural isomorphism $f_{i_{j}}:G_{i}\to G_{i_{j}}$ given by $f_{i_{j}}(x)=(G_{j},x^{-1})$.
\end{obs}

\subsection{Key Theorems}

We will later make repeated use of the Seifert--van Kampen Theorem.
As our simplicial complexes are all closed, and we usually only care about closed subcomplexes of these, we will use a `closed version' of the theorem.
Such a theorem can be found in some undergraduate Algebraic Topology notes, such as \cite{Wilton2019} delivered by H. Wilton at the University of Cambridge.

\begin{restatable}[Seifert--Van Kampen (Closed Version)]{thm}{svank}\label{s van k}
For closed sets $A$ and $B$ with $A$, $B$, and $A\cap B$ path-connected and such that there exist open sets $U\subset A$ and $V\subset B$ with $A\cap B$ a (strong) deformation retract of both $U$ and $V$, we have that the diagram:
\begin{center}
\begin{tikzcd}[ampersand replacement=\&]
\pi_{1}(A\cap B)
\ar[r, "i_{A_{\ast}}"]
\ar[d, "i_{B_{\ast}}"]
			\&	\pi_{1}(A)
				\ar[d, "j_{A_{\ast}}"]	\\
\pi_{1}(B)
\ar[r, "j_{B_{\ast}}"]
			\&	\pi_{1}(A\cup B)	
\end{tikzcd}
\end{center}
is a pushout, where $i_{A}:A\cap B\hookrightarrow A$, $i_{B}:A\cap B\hookrightarrow B$, $j_{A}:A\hookrightarrow A\cup B$, and $j_{B}:B\hookrightarrow A\cup B$ are inclusion maps.
We will abuse notation and abbreviate this by writing:
 \[\pi_{1}(A\cup B)\cong\pi_{1}(A)\ast_{\pi_{1}(A\cap B)}\pi_{1}(B)\]
\end{restatable}

\begin{rem}
This closed version follows by noting that the diagram:
\begin{center}
\begin{tikzcd}[ampersand replacement=\&]
\pi_{1}(U\cup V)
\ar[r]
\ar[d]
			\&	\pi_{1}(A\cup V)
				\ar[d]	\\
\pi_{1}(U\cup B)
\ar[r]
			\&	\pi_{1}(A\cup B)	
\end{tikzcd}
\end{center}
is a pushout by the standard Seifert--van Kampen Theorem, where the corresponding components in the closed version are neighbourhood deformation retracts of those here, and hence have the same fundamental group.
\end{rem}

Since our sets $A$, $B$, $A\cup B$, and $A\cap B$ will always be (finite) simplicial complexes, we will always have that $A\cap B$ is a neighbourhood deformation retract in both $A$ and $B$.
Indeed, we can take a union of open subsets of each simplex of $A$ containing $A\cap B$, and similarly for $B$, and we will have open sets $U$ and $V$  satisfying this requirement.
We illustrate this with an example:

\begin{example}
Let $A=$
\begin{tikzpicture}
\filldraw[thick, fill=gray!30] (0,0) -- (0,1) -- (-0.866,0.5) -- cycle;
\draw[fill] (0,0) circle [radius=0.06cm];
\draw[fill] (0,1) circle [radius=0.06cm];
\draw[fill] (-0.866,0.5) circle [radius=0.06cm];
\end{tikzpicture}
and $B=$
\begin{tikzpicture}
\filldraw[thick, fill=gray!30] (0,0) -- (0,1) -- (0.866,0.5) -- cycle;
\draw[fill] (0,0) circle [radius=0.06cm];
\draw[fill] (0,1) circle [radius=0.06cm];
\draw[fill] (0.866,0.5) circle [radius=0.06cm];
\end{tikzpicture}
be two (closed) simplices, and $X=A\cup B=$
\begin{tikzpicture}
\filldraw[thick, fill=gray!30] (0,0) --(0.866,0.5) -- (0,1) -- (-0.866,0.5) -- cycle;
\draw[thick] (0,0) -- (0,1);
\draw[fill] (0,0) circle [radius=0.06cm];
\draw[fill] (0,1) circle [radius=0.06cm];
\draw[fill] (-0.866,0.5) circle [radius=0.06cm];
\draw[fill] (0.866,0.5) circle [radius=0.06cm];
\end{tikzpicture}
a simplicial complex.
Then in $X$ we have that $A\cap B=$
\begin{tikzpicture}
\draw[thick] (0,0) -- (0,1);
\draw[fill] (0,0) circle [radius=0.06cm];
\draw[fill] (0,1) circle [radius=0.06cm];
\end{tikzpicture}
. We can then take our sets $U\subseteq A$ and $V\subseteq B$ to $U=$
\begin{tikzpicture}
\fill[fill=gray!30] (-0.1732,0.1) --(0,0) -- (0,1) -- (-0.1732,0.9) -- cycle;
\draw[thick] (-0.1732,0.1) --(0,0) -- (0,1) -- (-0.1732,0.9);
\draw[dotted,thick,gray] (-0.1732,0.9) -- (-0.1732,0.1);
\draw[fill] (0,0) circle [radius=0.06cm];
\draw[fill] (0,1) circle [radius=0.06cm];
\end{tikzpicture}
and $V=$
\begin{tikzpicture}
\fill[fill=gray!30] (0.1732,0.1) --(0,0) -- (0,1) -- (0.1732,0.9) -- cycle;
\draw[thick] (0.1732,0.1) --(0,0) -- (0,1) -- (0.1732,0.9);
\draw[dotted,thick,gray] (0.1732,0.9) -- (0.1732,0.1);
\draw[fill] (0,0) circle [radius=0.06cm];
\draw[fill] (0,1) circle [radius=0.06cm];
\end{tikzpicture}
. Then $A-U=$
\begin{tikzpicture}
\fill[fill=gray!30] (-0.1732,0.1) -- (-0.1732,0.9) -- (-0.866,0.5) -- cycle;
\draw[thick] (-0.1732,0.1) -- (-0.866,0.5) -- (-0.1732,0.9);
\draw[thick,gray] (-0.1732,0.1) -- (-0.1732,0.9);
\draw[fill] (-0.866,0.5) circle [radius=0.06cm];
\end{tikzpicture}
which is a closed set, hence $U$ is open in $A$.
Similarly, $V$ is open in $B$, and it is clear that $A\cap B$ is a deformation retract of both $U$ and $V$.
\end{example}

\medskip

In \cite{Brown1984}, Brown presents a method for extracting a group presentation from its action on a CW complex.
A streamlined version of this is given as Theorem 3 in \cite{Brown1984} which holds when the action of the group on the complex has a strict fundamental domain:

\begin{restatable}[Brown {\cite[Theorem 3]{Brown1984}}]{thm}{brown}\label{brown thm strict} %
Let $\mathcal{G}$ act on a simply connected $\mathcal{G}$-CW complex $X$ (without inversion on the 1-cells of $X$).
Suppose there is a subcomplex $W$ of $X$ so that every cell of $X$ is equivalent under the action of $\mathcal{G}$ to a unique cell of $W$.
Then $\mathcal{G}$ is generated by the isotropy subgroups $\mathcal{G}_{v}$ ($v\in V(W)$) subject to edge relations $\iota_{o(e)}(g)=\iota_{t(e)}(g)$ for all $g\in \mathcal{G}_{e}$ ($e\in E(W)$)
(where for any $e\in E(W)$, $\iota_{o(e)}:\mathcal{G}_{e}\to\mathcal{G}_{o(e)}$ and $\iota_{t(e)}:\mathcal{G}_{e}\to\mathcal{G}_{t(e)}$ are inclusions).
\end{restatable}

It is this theorem that forms the basis of Section \ref{section presentation} in which we give a presentation for $\outs(G)$.

\subsection{Outer Space for Free Products} \label{subsection outer space}

In \cite{Guirardel2007}, Guirardel and Levitt give a description of a deformation space for certain free products $G=G_{1}\ast\dots\ast G_{n}\ast F_{k}$ on which $\out(G)$ acts, allowing us to study properties of the outer automorphism group of a free product.
They call this space $\mathcal{O}$, the `Outer Space' (for a free product), and the projectivised space $\mathcal{PO}$.

The space $\mathcal{PO}$, while cellular, is not simplicial, due to `missing' faces (faces `at infinity').
To resolve this, we consider a construction called the \emph{barycentric spine} of $\mathcal{PO}$ (denoted `$\mathcal{S}$').
This is obtained by taking the first barycentric subdivision of $\mathcal{PO}$, and then linearly retracting off the missing faces, to give a simplicial complex.
This equates to taking the geometric realisation of the poset on the cells of $\mathcal{PO}$ given by $A\prec B$ iff $A$ is a face of $B$.

Whilst their construction is defined for a Grushko decomposition (i.e. each factor group $G_{i}$ is non-trivial, freely indecomposable, and not infinite cyclic), by considering instead the subgroup $\out(G;G_{1},\dots,G_{n},F_{k})$ of $\out(G)$ which preserves a given splitting of $G$, we can loosen these conditions.
We may refer to this as a `relative' Outer Space.

Since we are going to be interested in subcomplexes of the barycentric spine $\mathcal{S}$, we will now give an explicit description for it.
We will restrict ourselves to the case where every factor group in the splitting of $G$ acts elliptically (i.e. $k0$).

\subsubsection*{Points in the Barycentric Spine of Projectivised Outer Space} 

Let $G$ be a grop which splits as a free product $G_{1}\ast\dots\ast G_{n}$ where each $G_{i}$ is non-trivial, and let $\mathfrak{S}=(G_{1},\dots,G_{n})$ be the tuple associated to the splitting.

The barycentric spine $\mathcal{S}$ of $\mathcal{PO}$ is a simplicial complex whose 0-cells are graphs of groups $\Gamma$ (with $\pi_{1}(\Gamma)\cong G$), as follows:

\begin{itemize}
\item The underlying graph structure of $\Gamma$ is a tree
\item $\Gamma$ has one vertex with vertex group conjugate to $G_{i}$ for each $i$
\item All other vertex groups are trivial (vertices with trivial vertex group will be called `trivial vertices')
\item Any trivial vertex has valency at least 3
\item All edge groups are trivial
\item The vertex groups $G_{1}^{g_{1}},\dots,G_{n}^{g_{n}}$ generate the group $G$ (that is, $G_{1}^{g_{1}}\ast\dots\ast G_{n}^{g_{n}}$ is a free fractor splitting for $G$)
\end{itemize}

Note that two graphs of groups are equivalent if and only if they are isomorphic in the sense of Bass \cite[Definition 2.1]{Bass1993}.

Via Bass--Serre Theory, we could equally consider points of $\mathcal{S}$ to be certain actions of $G$ on trees $T$, up to equivariant isometry.

\subsubsection*{Structure of the Barycentric Spine $\mathcal{S}$ of $\mathcal{PO}$} \label{s3.2} 

Given two 0-cells $\Gamma_{1}$ and $\Gamma_{2}$ in our barycentric spine, we have a 1-cell $[\Gamma_{1},\Gamma_{2}]$ whenever $\Gamma_{2}$ can be achieved by collapsing an edge or edges of $\Gamma_{1}$.

Whenever a collection of 0-cells $\Gamma_{1},\dots,\Gamma_{m}$ form an $m$-clique in the 1-skeleton (that is, whenever the restriction of the 1-skeleton to the vertices $\Gamma_{1},\dots,\Gamma_{m}$ forms a complete graph), we have a unique $(m-1)$-cell $[\Gamma_{1},\dots,\Gamma_{m}]$.

Since the maximum number of edges such a graph of groups can have is $2n-3$ (when all non-trivial vertices have valency 1 and all trivial vertices have valency 3), and the minimum number of edges is $n-1$ (when there are no trivial vertices), then the dimension of the barycentric spine of projectivised Outer Space is $(2n-3)-(n-1)=n-2$.
Since $\mathcal{PO}$ is contractible \cite[Theorem 4.2 and Corollary 4.4]{Guirardel2007}, and $\mathcal{PO}$ deformation retracts onto $\mathcal{S}$, then so too is $\mathcal{S}$.

\subsubsection*{Action of $\outs(G)$ on $\mathcal{S}$} 

If we consider points of $\mathcal{S}$ to be actions $\psi:G\times T\to T, \psi(g,t)=g\cdot_{\psi}t$ on $G$-trees $T$,
then for $\theta\in \outs(G)$, the action on $\mathcal{S}$, $\theta\cdot(T,\psi)$, is defined by $\theta(\psi(g,t))=\theta(g\cdot_{\psi}t)=\theta(g)\cdot_{\psi}t$.
Considering length functions, this simply says $\theta\cdot l_{T}=l_{\theta T}$ where $l_{\theta T}(g)=l_{T}(\theta(g))$ for all $g\in G$.
This extends to a cellular action on $\mathcal{S}$.
In Section \ref{subsection action on Cn}, we give a description of the action of $\outs(G)$ on our chosen subcomplex of $\mathcal{S}$ in terms of graphs of groups.

We will be interested in finding $\outs(G)$-stabilisers of vertices in the barycentric spine.
Considering points as actions of $G$ on trees $T$,
the stabiliser of a point $T$ is precisely the group of automorphisms acting trivially on the quotient graph $\Gamma=\faktor{T}{G}$.
This is the subgroup denoted by Guirdardel and Levitt as $\out_{0}^{S}(G)$.
If the vertex $v_{i}$ of $\Gamma$ represents the orbit of the vertex in $T$ whose stabiliser is $G_{i}$, and $\mu_{i}$ is the valency of $v_{i}$ in $\Gamma$, then $\out_{0}^{S}(G)$ is isomorphic to the direct product $\prod_{i=1}^{n}(G_{i}^{\mu_{i}-1}\rtimes \aut(G_{i}))$
(where $\aut(G_{i})$ is identified with its projection in $\aut_{\mathfrak{S}}(G)$ (or $\outs(G)$)).
The precise details of this are found in \cite[Section 5]{Guirardel2007}.
We explore this more explicitly in Section \ref{stabilisers}.


\section{The Complex $\mathcal{C}_{n}$} \label{section subcomplex}

From now on, we fix a splitting $\mathfrak{S}=G_{1},\dots,G_{n})$ of a group $G=G_{1}\ast\dots\ast G_{n}$, where each $G_{i}$ is non-trivial.
We will consider graphs of groups of $G$ which respect the splitting $\mathfrak{S}$ --- note that these will all be trees, as each factor group acts elliptically in the relative Bass--Serre tree.

The barycentric spine of the projectivised relative Outer Space for $G$ with respect to $\mathfrak{S}$ has a `reasonably sized' quotient under the action of $\outs(G)$ when $n=3$ and $n=4$ (4 vertices contributing to a total of 7 cells, and 32 vertices contributing to a total of 159 cells, respectively).
As $n$ grows, this quotient space quickly becomes unwieldy.

\begin{defn}\label{defn c3 and c4}
For $n=3$ or $n=4$, we define $\mathcal{C}_{n}$ to be the barycentric spine of Guirardel and Levitt's projectivised relative Outer Space associated to the splitting $\mathfrak{S}$ of $G$.
\end{defn}

That is, $\mathcal{C}_{3}$ and $\mathcal{C}_{4}$ are the geometric realisations of the posets whose elements are simplices in the projectivised relative Outer Spaces for the splittings $G_{1}\ast G_{2}\ast G_{3}$ and $G_{1}\ast G_{2}\ast G_{3}\ast G_{4}$, respectively, where $A\prec B$ if the simplex $A$ is a face of the simplex $B$.

\begin{lemma}\label{lemma c3 and c4 sc}
$\mathcal{C}_{3}$ and $\mathcal{C}_{4}$ are contractible. In particular, they are simply connected.
\end{lemma}

\begin{proof}
This follows from contractibility of projectivised Outer Space, proven by Guirardel and Levitt \cite[Theorem 4.2 and Corollary 4.4]{Guirardel2007} , since projectivised Outer Space deformation retracts onto its spine.
\end{proof}
 
 Our goal now is to construct a simplicial complex $\mathcal{C}_{n}$ for each $n\ge5$ whose quotient under the action of $\outs(G)$ remains `reasonably sized'.
 
 \subsection{Restricting to a Subcomplex of Outer Space}

For the rest of this section, we assume $n\ge5$.
In general, the barycentric spine of the Outer Space for $n$ factors will be $(n-2)$-dimensional.

Since our graphs of groups are all trees, we will find that the stabilisers of higher-dimensional simplices in Outer Space are contained in the stabilisers of their faces.
Hence restricting ourselves to lower dimensional simplices will not sacrifice information gathered from vertex stabilisers in the barycentric spine.
We will thus restrict ourselves to the three lowest possible dimensions of simplex; then when we take the barycentyric spine of this restricted space, we will recover a 2-dimensional complex.

The lowest dimension of a simplex in Outer Space for $n$ factors is $n-1$ (since our trees will have the minimal possible number of vertices, $n$, leading to $n-1$ edges). Thus we are interested in graphs of groups with $n-1$, $n$, or $n+1$ edges.

For $n=5$, this means we just drop the top-dimensional simplices, which represent the graphs of groups of the form
	\begin{tikzpicture}[scale=0.8] 
	\draw[thick] (0,0) -- (-0.5,0.866);
	\draw[thick] (0,0) -- (-0.5,-0.866);
	\draw[thick] (0,0) -- (1,0);
	\draw[thick] (2,0) -- (2.5,0.866);
	\draw[thick] (2,0) -- (2.5,-0.866);
	\draw[thick] (2,0) -- (1,0);
	\draw[thick] (1,0) -- (1,1);
	\draw[yellow, fill] (0,0) circle [radius=0.09];
	\draw[red, fill] (-0.5,0.866) circle [radius=0.09]; 
	\draw[red, fill] (-0.5,-0.866) circle [radius=0.09]; 
	\draw[yellow, fill] (1,0) circle [radius=0.085];
	\draw[yellow, fill] (2,0) circle [radius=0.09];
	\draw[red, fill] (2.5,0.866) circle [radius=0.09]; 
	\draw[red, fill] (2.5,-0.866) circle [radius=0.09]; 
	\draw[red, fill] (1,1) circle [radius=0.09]; 
	\node at (-0.75,0.9) {$j$};
	\node at (-0.8,-0.9) {$k$};
	\node at (2.75,0.9) {$l$};
	\node at (2.9,-0.9) {$m$};
	\node at (1,1.3) {$i$};
	\end{tikzpicture}
, where the labelled (red) vertices have vertex groups conjugate to the free factors of $G$, and the unlabelled (yellow) vertices have trivial vertex group (and all edge groups are trivial).

We will call such a graph (i.e. associated to a top-dimensional simplex) a \emph{maximal graph}. Note that in our case, these are characterised by having precisely $n$ leaves (all with non-trivial vertex group) and with all other vertices (each having trivial vertex group) having valency exactly 3.

As $n$ increases, so too does the number of maximal graphs associated to `top' simplices.
For $n=6$, there are two maximal graph structures,
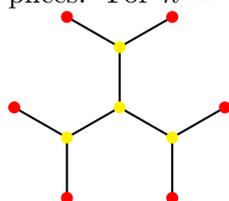
\begin{tikzpicture}[scale=0.8]
\draw[thick] (-0.5,0) -- (0.5,0);
\draw[thick] (0.5,0) -- (1,0.866);
\draw[thick] (0.5,0) -- (1,-0.866);
\draw[thick] (1,-0.866) -- (2,-0.866);
\draw[thick] (1,-0.866) -- (0.5,-1.732);
\draw[thick] (-0.5,0) -- (-1,0.866);
\draw[thick] (-0.5,0) -- (-1,-0.866);
\draw[thick] (-1,-0.866) -- (-2,-0.866);
\draw[thick] (-1,-0.866) -- (-0.5,-1.732);
\draw[yellow,fill] (-0.5,0) circle [radius=0.09];
\draw[yellow,fill] (0.5,0) circle [radius=0.09];
\draw[yellow,fill] (1,-0.866) circle [radius=0.09];
\draw[yellow,fill] (-1,-0.866) circle [radius=0.09];
\draw[red,fill] (1,0.866) circle [radius=0.09];
\draw[red,fill] (-1,0.866) circle [radius=0.09];
\draw[red,fill] (2,-0.866) circle [radius=0.09];
\draw[red,fill] (-2,-0.866) circle [radius=0.09];
\draw[red,fill] (0.5,-1.732) circle [radius=0.09];
\draw[red,fill] (-0.5,-1.732) circle [radius=0.09];
\end{tikzpicture}
and 
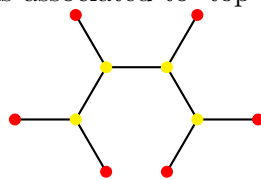
\begin{tikzpicture}[scale=0.8]
\draw[thick] (0,0) -- (0,1);
\draw[thick] (0,0) -- (0.866,-0.5);
\draw[thick] (0,0) -- (-0.866,-0.5);
\draw[thick] (0,1) -- (0.866,1.5);
\draw[thick] (0,1) -- (-0.866,1.5);
\draw[thick] (0.866,-0.5) -- (1.732,0);
\draw[thick] (0.866,-0.5) -- (0.866,-1.5);
\draw[thick] (-0.866,-0.5) -- (-1.732,0);
\draw[thick] (-0.866,-0.5) -- (-0.866,-1.5);
\draw[yellow,fill] (0,0) circle [radius=0.09];
\draw[yellow,fill] (0,1) circle [radius=0.09];
\draw[yellow,fill] (0.866,-0.5) circle [radius=0.09];
\draw[yellow,fill] (-0.866,-0.5) circle [radius=0.09];
\draw[red,fill] (0.866,1.5) circle [radius=0.09];
\draw[red,fill] (-0.866,1.5) circle [radius=0.09];
\draw[red,fill] (0.866,-1.5) circle [radius=0.09];
\draw[red,fill] (-0.866,-1.5) circle [radius=0.09];
\draw[red,fill] (1.732,0) circle [radius=0.09];
\draw[red,fill] (-1.732,0) circle [radius=0.09];
\end{tikzpicture}.
For $n=7$ there are also two types of maximal graph, for $n=8$ there are four, for $n=9$ there are six, for $n=10$ there are twelve, and for $n=11$ there are eighteen.
\footnote{This is somewhat analogous to alkane chains in organic chemistry, and the various isomers for these (if we were to pretend that carbon could make only three bonds, and not four).}
Collapsing edges (passing to faces in the associated simplex in Outer Space) in each of these leads to a variety of structures.

\begin{defn}\label{defn collapse}
In a graph of groups, we say that an edge is \emph{collapsible} if it has at least one trivial endpoint (that is, at least one endpoint whose vertex group is the trivial group).

The process of replacing a collapsible edge (including its endpoints) by a single vertex whose vertex group is the free product of the vertex groups of the endpoints of said edge is called \emph{collapsing}.

Given two graphs of groups $T_{1}$ and $T_{2}$, we will say $T_{2}$ is a \emph{collapse} of $T_{1}$ if $T_{2}$ can be achieved as the result of successively collapsing edges of $T_{1}$.
\end{defn}

\begin{rem}
Since a collapsible edge has at least one trivial endpoint, then one may think of the edge as collapsing to its other (potentially non-trivial) vertex.

That is, if 
\begin{tikzpicture}
\draw[thick] (0,0) -- (1,0);
\filldraw (0,0) circle [radius=0.06cm];
\filldraw (1,0) circle [radius=0.06cm];
\node at (0,0.2) {$u$};
\node at (1,0.2) {$v$};
\end{tikzpicture}
 is a collapsible edge, with $u$ being the trivial vertex and $v$ having vertex group $G_{v}$ (possibly also trivial), then in collapsing
 \begin{tikzpicture}
\draw[thick] (0,0) -- (1,0);
\filldraw (0,0) circle [radius=0.06cm];
\filldraw (1,0) circle [radius=0.06cm];
\node at (0,0.2) {$u$};
\node at (1,0.2) {$v$};
\end{tikzpicture}
, we replace it with a vertex whose vertex group is equal to $\{1\}\ast G_{v}=G_{v}$. Thus we may think of collapsing
\begin{tikzpicture}
\draw[thick] (0,0) -- (1,0);
\filldraw (0,0) circle [radius=0.06cm];
\filldraw (1,0) circle [radius=0.06cm];
\node at (0,0.2) {$u$};
\node at (1,0.2) {$v$};
\end{tikzpicture}
as replacing it with the vertex $v$.
Note that the new valency of $v$ is equal to the old valency of $v$ plus the valency of $u$ minus 2.

Alternatively, the collapse of such an edge \begin{tikzpicture}
\draw[thick] (0,0) -- (1,0);
\filldraw (0,0) circle [radius=0.06cm];
\filldraw (1,0) circle [radius=0.06cm];
\node at (0,0.2) {$u$};
\node at (1,0.2) {$v$};
\end{tikzpicture}
in $T_{1}$ may be thought of as a map $f:T_{1}\to T_{2}$ sending 
\begin{tikzpicture}
\draw[thick] (0,0) -- (1,0);
\filldraw (0,0) circle [radius=0.06cm];
\filldraw (1,0) circle [radius=0.06cm];
\node at (0,0.2) {$u$};
\node at (1,0.2) {$v$};
\end{tikzpicture}
to
\begin{tikzpicture}
\filldraw (1,0) circle [radius=0.06cm];
\node at (1,0.2) {$v$};
\end{tikzpicture}
and acting as the identity on the rest of $T_{1}$.
Collapses of multiple edges can be achieved by composing these maps.
\end{rem}

Recall that we have already decided to limit ourselves to graphs with $n-1$, $n$, or $n+1$ edges.
So we will restrict ourselves further to collapses of graphs of groups of the form
	\begin{tikzpicture}[scale=0.8] 
	\draw[thick] (0,0) -- (-0.5,0.866);
	\draw[thick] (0,0) -- (-0.5,-0.866);
	\draw[thick] (0,0) -- (1,0);
	\draw[thick] (2,0) -- (2.5,0.866);
	\draw[thick] (2,0) -- (2.5,-0.866);
	\draw[thick] (2,0) -- (1,0);
	\draw[thick] (1,0) -- (0.55,0.866);
	\draw[thick] (1,0) -- (0.334,0.5);
	\draw[thick] (1,0) -- (1.45,0.866);
	\draw[thick] (1,0) -- (1.666,0.5);
	\draw[yellow, fill] (0,0) circle [radius=0.09];
	\draw[red, fill] (-0.5,0.866) circle [radius=0.09]; 
	\draw[red, fill] (-0.5,-0.866) circle [radius=0.09]; 
	\draw[yellow, fill] (1,0) circle [radius=0.085];
	\draw[yellow, fill] (2,0) circle [radius=0.09];
	\draw[red, fill] (2.5,0.866) circle [radius=0.09]; 
	\draw[red, fill] (2.5,-0.866) circle [radius=0.09]; 
	\draw[red, fill] (0.55,0.866) circle [radius=0.09]; 
	\draw[red, fill] (0.334,0.5) circle [radius=0.09]; 
	\draw[red, fill] (1.45,0.866) circle [radius=0.09]; 
	\draw[red, fill] (1.666,0.5) circle [radius=0.09]; 
	\draw[red, fill] (0.8,0.98) circle [radius=0.05];
	\draw[red, fill] (1.2,0.98) circle [radius=0.05];
	\draw[red, fill] (1,1) circle [radius=0.05];
	\node at (-0.75,0.9) {$j$};
	\node at (-0.8,-0.9) {$k$};
	\node at (2.75,0.9) {$l$};
	\node at (2.9,-0.9) {$m$};
	\node at (0.4,0.7) {\footnotesize$v_{1}$};
	\node at (0.575,1.1) {\footnotesize$v_{2}$};
	\node at (1.8,1.1) {\footnotesize$v_{n-5}$};
	\node at (1.96,0.7) {\footnotesize$v_{n-4}$};
	\end{tikzpicture}
, which we will `abbreviate' as \Taleph{$n-4$}{$j$}{$k$}{$l$}{$m$} (where the `$n-4$' means we have suppressed $n-4$ leaves).
We will use this method of abbreviation on a frequent basis. We will often refer to the blue-ringed vertex (with valency dependent on $n$) as the `basepoint' of the graph.

For $n=5$ this is exactly as we have described, and results in taking the barycentric spine of the 3-skeleton of Outer Space.
This graph shape provides a natural way to generalise to $n>5$, without having to worry about the varying maximal graphs.
Note that this means that for $n>5$ there will be graphs of groups representing simplices of the `correct' dimension (i.e. $n-1$, $n$, or $n+1$) in Outer Space which we do not include in our chosen complex.

Our complex $\mathcal{C}_{n}$ will be the geometric realisation of the poset whose elements are the graphs of groups we have selected above, where the order is given by collapsing.
We formalise this in the following subsection.

\subsection{Points in the Complex}\label{subsection points in the subcomplex}

Table \ref{table n>=5 points} summarises the graph shapes we will encounter, as well as a naming convention, the number we expect to see in a fundamental domain of the subcomplex we choose, and associated colours which are useful in drawing diagrams (though can largely be ignored).

\begin{table}[h!]
\centering
\begin{tabular}{ m{4cm} | m{1.5cm} | m{1.5cm} | c }
\hspace{1cm}	Tree 					&	Name				& 	No. per Domain			&	Colour							\\ 
\hline 					\hspace{0.1cm}
\scalebox{1}{\Trho{$n-2$}{$i$}{$j$}}			&	$\rho_{ij}$			&	$\frac{n(n-1)}{2}$			&	\textcolor{Green}{\circle*{0.4cm}}		\\	\hspace{0.1cm}
\scalebox{1}{\Tsigma{$n-5$}{$i$}{$j$}{$k$}{$l$}{$m$}}	&$\sigma_{i,jk,lm}$		&	$\frac{n!}{8\times(n-5)!}$	&	\textcolor{Goldenrod}{\circle*{0.4cm}}		\\
\scalebox{1}{\Ttau{$n-4$}{$i$}{$j$}{$k$}{$l$}}	&	$\tau_{i,j,kl}$		&	$\frac{n!}{2\times(n-4)!}$	&	\textcolor{Blue}{\circle*{0.4cm}}		\\ 	\hspace{1.42cm}
\scalebox{1}{\Talpha{$n$}}				&	$\alpha$			&	$1$					&	\textcolor{Red}{\circle*{0.4cm}}			\\
\scalebox{1}{\Tbeta{$n-2$}{$i$}{$j$}}			&	$\beta_{i,j}$			&	$n(n-1)$				&	\textcolor{Cyan}{\circle*{0.4cm}}		\\ 	\hspace{0.1cm}
\scalebox{1}{\Tgamma{$n-3$}{$i$}{$j$}{$k$}}	&	$\gamma_{i,jk}$		&	$\frac{n!}{2\times(n-3)!}$	&	\textcolor{SpringGreen}{\circle*{0.4cm}}	\\
\scalebox{1}{\Tdelta{$n-5$}{$i$}{$j$}{$k$}{$l$}{$m$}} &$\delta_{i,j,k,lm}$		&	$\frac{n!}{2\times(n-5)!}$	&	\textcolor{Purple}{\circle*{0.4cm}}		\\
\scalebox{1}{\Tepsilon{$n-4$}{$i$}{$j$}{$k$}{$l$}} &	$\varepsilon_{i,j,k,l}$	&	$\frac{n!}{2\times(n-4)!}$	&	\textcolor{Lavender}{\circle*{0.4cm}}		\\	 \hspace{1.2cm}
\scalebox{1}{\TA{$n-1$}{$i$}}				&	$A_{i}$			&	$n$					&	\textcolor{Bittersweet}{\circle*{0.4cm}}	\\
\scalebox{1}{\TB{$n-3$}{$i$}{$j$}{$k$}}		&	$B_{i,j,k}$			&	$\frac{n!}{(n-3)!}$			&	\textcolor{Orange}{\circle*{0.4cm}}		\\
\scalebox{1}{\TC{$n-5$}{$i$}{$j$}{$k$}{$l$}{$m$}} &	$C_{i,j,k,l,m}$		&	$\frac{n!}{2\times(n-5)!}$	&	\textcolor{Dandelion}{\circle*{0.4cm}}		\\
\end{tabular}
\caption{Points in the Subcomplex for $n\ge5$}
\label{table n>=5 points}
\end{table}

In general, subscripts separated by a comma are ordered, whereas subscripts not separated by a comma are not ordered.
So $\rho_{ij}$ and $\rho_{ji}$ both refer to the tree \Trho{$n-2$}{$i$}{$j$} whereas $\beta_{i,j}$ and $\beta_{j,i}$ refer to distinct trees, \Tbeta{$n-2$}{$i$}{$j$} and \Tbeta{$n-2$}{$j$}{$i$}, respectively.

There is some additional symmetry from our trees, so we also have that $\sigma_{i,jk,lm}=\sigma_{i,lm,jk}$, $\varepsilon_{i,j,k,l}=\varepsilon_{k,l,i,j}$, and $C_{i,j,k,l,m}=C_{i,l,m,j,k}$.
It is always assumed that, for example, $\{i,j,k,l,v_{1},\dots,v_{n-4}\}=\{1,\dots,n\}$ as sets.


Recall from Definition \ref{defn s labelling} that an $\mathfrak{S}$-labelling is an assignement of vertex groups $H_{1},\dots,H_{n}$ to a tree $T$ so that $\pi_{1}(\mathbf{T})\cong G$ which respects the splitting $\mathfrak{S}$ of $G$. For trees in Table \ref{table n>=5 points}, vertex groups are only assigned to named (red) vertices.

\begin{defn}\label{defn equivalent labellings}
Let $H=(H_{v_{1}},\dots,H_{v_{n}})$ and $H'=(H'_{v_{1}},\dots,H'_{v_{n}})$ be two $\mathfrak{S}$-labellings of a tree $T$ from Table \ref{table n>=5 points} with $\{v_{1},\dots,v_{n}\}\subseteq V(T)$.
Then $H$ and $H'$ are \emph{equivalent} as labellings (with respect to $T$) if for each $v\in V(T)$ (including trivial vertices)
there exists $g_{v}\in G$ and (if $v$ is not a trivial vertex) $\varphi_{v}\in\aut(H_{v})$ 
so that $H'_{v}=g_{v}^{-1}\varphi_{v}(H_{v})g_{v}$,
and moreover, for any edge $e\in E(T)$ we have $g_{t(e)}g_{o(e)}^{-1}\in H_{o(e)}$ 
(where $o(e)$ is the endpoint of $e$ closest in $T$ to the `basepoint', and $t(e)$ is the further endpoint).
If $o(e)$ does not have a vertex group assigned (i.e. $o(e)$ is a trivial vertex) then $g_{t(e)}=g_{o(e)}$.

Two graphs of groups $\mathbf{T_{1}}$ and $\mathbf{T_{2}}$ are \emph{equivalent} if they each have underlying graph isomorphic to some graph $T$, and their labellings are equivalent (with respect to $T$).
We denote this equivalence by $\mathbf{T_{1}}\simeq\mathbf{T_{2}}$.
\end{defn}

Considering the fundamental group of a labelled tree $\mathbf{T}=(T,H)$ to be $\displaystyle\bigast_{i=1}^{n}H_{v_{i}}$, this equivalence induces an isomorphism $H_{v_{1}}\ast\dots\ast H_{v_{n}}\to H'_{v_{1}}\ast\dots\ast H'_{v_{n}}$.
Some basic manipulation of notation shows that this notion of equivalence corresponds to taking isomorphism classes of graphs of groups described by Bass \cite[Section 2]{Bass1993}.

When considering $\mathcal{C}_{n}$, we assume a given splitting $\mathfrak{S}$ of our group $G$, and may simply refer to `labellings' of trees.

\begin{obs}
If $\mathbf{T_{1}}\simeq\mathbf{T_{2}}$ are equivalent graphs of groups and $f$ is a collapsing map of the underlying graph $T$, then $f(\mathbf{T_{1}})\simeq f(\mathbf{T_{1}})$ are also equivalent.
\end{obs}

\begin{example}\label{eg equivalent B labellings}
Consider the labelled graph of groups $T:=$
\begin{tikzpicture} 
\draw[thick] (-1,0) -- (-2,0);
\draw[thick] (0,0) -- (-1,0);
\draw[thick] (0,0) -- (0.866,0.5);
\draw[thick] (0,0) -- (0.866,-0.5);
\draw (0.2,0) -- (0.8,0);
\draw (0.242,0.0625) -- (0.775,0.2);
\draw (0.242,-0.0625) -- (0.775,-0.2);
\draw[red, fill] (-2,0) circle [radius=0.075]; 
\draw[red, fill] (-1,0) circle [radius=0.075]; 
\draw[red, fill] (0,0) circle [radius=0.075]; 
\draw[red, fill] (0.866,0.5) circle [radius=0.075]; 
\draw[red, fill] (0.866,-0.5) circle [radius=0.075]; 
\draw[red, fill] (0.968,0.25) circle [radius=0.05];
\draw[red, fill] (1,0) circle [radius=0.05];
\draw[red, fill] (0.968,-0.25) circle [radius=0.05];
\node at (0,-0.3625) {$G_{i}$};
\node at (-1,-0.39) {$G_{j}$};
\node at (-2,-0.35) {$G_{k}$};
\node at (1.3,0.6) {$G_{v_{1}}$};
\node at (1.5,-0.65) {$G_{v_{n-3}}$};
\end{tikzpicture}
. The following labelled graphs of groups are all equivalent to $T$:
\begin{enumerate}
\item \label{inner equivalent} 					
	\begin{tikzpicture} 
	\draw[thick] (-1,0) -- (-2,0);
	\draw[thick] (0,0) -- (-1,0);
	\draw[thick] (0,0) -- (0.866,0.5);
	\draw[thick] (0,0) -- (0.866,-0.5);
	\draw (0.2,0) -- (0.8,0);
	\draw (0.242,0.0625) -- (0.775,0.2);
	\draw (0.242,-0.0625) -- (0.775,-0.2);
	\draw[red, fill] (-2,0) circle [radius=0.075]; 
	\draw[red, fill] (-1,0) circle [radius=0.075]; 
	\draw[red, fill] (0,0) circle [radius=0.075]; 
	\draw[red, fill] (0.866,0.5) circle [radius=0.075]; 
	\draw[red, fill] (0.866,-0.5) circle [radius=0.075]; 
	\draw[red, fill] (0.968,0.25) circle [radius=0.05];
	\draw[red, fill] (1,0) circle [radius=0.05];
	\draw[red, fill] (0.968,-0.25) circle [radius=0.05];
	\node at (0,-0.3675) {$G_{i}^{g}$};
	\node at (-1,-0.395) {$G_{j}^{g}$};
	\node at (-2,-0.355) {$G_{k}^{g}$};
	\node at (1.3,0.6) {$G_{v_{1}}^{g}$};
	\node at (1.5,-0.65) {$G_{v_{n-3}}^{g}$};
	\end{tikzpicture}
	where $g\in G$ --- since $gg^{-1}=1\in G_{v}$ for any $v$.
\item \label{inner factor equivalent} 				
	\begin{tikzpicture} 
	\draw[thick] (-1,0) -- (-2,0);
	\draw[thick] (0,0) -- (-1,0);
	\draw[thick] (0,0) -- (0.866,0.5);
	\draw[thick] (0,0) -- (0.866,-0.5);
	\draw (0.2,0) -- (0.8,0);
	\draw (0.242,0.0625) -- (0.775,0.2);
	\draw (0.242,-0.0625) -- (0.775,-0.2);
	\draw[red, fill] (-2,0) circle [radius=0.075]; 
	\draw[red, fill] (-1,0) circle [radius=0.075]; 
	\draw[red, fill] (0,0) circle [radius=0.075]; 
	\draw[red, fill] (0.866,0.5) circle [radius=0.075]; 
	\draw[red, fill] (0.866,-0.5) circle [radius=0.075]; 
	\draw[red, fill] (0.968,0.25) circle [radius=0.05];
	\draw[red, fill] (1,0) circle [radius=0.05];
	\draw[red, fill] (0.968,-0.25) circle [radius=0.05];
	\node at (0,-0.3675) {$G_{i}^{g_{i}}$};
	\node at (-1,-0.395) {$G_{j}^{g_{j}}$};
	\node at (-2,-0.355) {$G_{k}^{g_{k}}$};
	\node at (1.4,0.6) {$G_{v_{1}}^{g_{v_{1}}}$};
	\node at (1.5,-0.65) {$G_{v_{n-3}}^{g_{v_{n-3}}}$};
	\end{tikzpicture}
	where each $g_{v}\in G_{v}$ --- since for each $v$, $G_{v}\mapsto g_{v}^{-1}G_{v}g_{v}$ is an element of $\aut(G_{v})$.
\item \label{twists equivalent}				 	
	\begin{tikzpicture} 
	\draw[thick] (-1,0) -- (-2,0);
	\draw[thick] (0,0) -- (-1,0);
	\draw[thick] (0,0) -- (0.866,0.5);
	\draw[thick] (0,0) -- (0.866,-0.5);
	\draw (0.2,0) -- (0.8,0);
	\draw (0.242,0.0625) -- (0.775,0.2);
	\draw (0.242,-0.0625) -- (0.775,-0.2);
	\draw[red, fill] (-2,0) circle [radius=0.075]; 
	\draw[red, fill] (-1,0) circle [radius=0.075]; 
	\draw[red, fill] (0,0) circle [radius=0.075]; 
	\draw[red, fill] (0.866,0.5) circle [radius=0.075]; 
	\draw[red, fill] (0.866,-0.5) circle [radius=0.075]; 
	\draw[red, fill] (0.968,0.25) circle [radius=0.05];
	\draw[red, fill] (1,0) circle [radius=0.05];
	\draw[red, fill] (0.968,-0.25) circle [radius=0.05];
	\node at (0,-0.3625) {$G_{i}$};
	\node at (-0.9,-0.395) {$G_{j}^{i_{jk}}$};
	\node at (-1.8,-0.355) {$G_{k}^{j_{k}i_{jk}}$};
	\node at (1.375,0.6) {$G_{v_{1}}^{i_{v_{1}}}$};
	\node at (1.5,-0.65) {$G_{v_{n-3}}^{i_{v_{n-3}}}$};
	\end{tikzpicture}
	where $i_{v_{1}},\dots,i_{v_{n-3}},i_{jk}\in G_{i}$ and $j_{k}\in G_{j}$ --- since $i\in G_{i}\Rightarrow i1^{-1}\in G_{i}$ and $(j_{k} i_{jk}){i_{jk}}^{-1}\in G_{j}$.
\item \label{twisting not from basepoint} 			
	\begin{tikzpicture} 
	\draw[thick] (-1,0) -- (-2,0);
	\draw[thick] (0,0) -- (-1,0);
	\draw[thick] (0,0) -- (0.866,0.5);
	\draw[thick] (0,0) -- (0.866,-0.5);
	\draw (0.2,0) -- (0.8,0);
	\draw (0.242,0.0625) -- (0.775,0.2);
	\draw (0.242,-0.0625) -- (0.775,-0.2);
	\draw[red, fill] (-2,0) circle [radius=0.075]; 
	\draw[red, fill] (-1,0) circle [radius=0.075]; 
	\draw[red, fill] (0,0) circle [radius=0.075]; 
	\draw[red, fill] (0.866,0.5) circle [radius=0.075]; 
	\draw[red, fill] (0.866,-0.5) circle [radius=0.075]; 
	\draw[red, fill] (0.968,0.25) circle [radius=0.05];
	\draw[red, fill] (1,0) circle [radius=0.05];
	\draw[red, fill] (0.968,-0.25) circle [radius=0.05];
	\node at (0,-0.3675) {$G_{i}^{j_{i}}$};
	\node at (-1,-0.39) {$G_{j}$};
	\node at (-2,-0.355) {$G_{k}^{j_{k}}$};
	\node at (1.3,0.6) {$G_{v_{1}}^{j_{i}}$};
	\node at (1.5,-0.65) {$G_{v_{n-3}}^{j_{i}}$};
	\end{tikzpicture}
	where $j_{k},j_{i}\in G_{j}$ --- this is achieved by combining \ref{inner equivalent}. (with $g=j_{i}$), \ref{inner factor equivalent}. (conjugating $G_{j}^{j_{i}}$ by $j_{i}^{-1}\in G_{j}$), and \ref{twists equivalent}. (conjugating $G_{k}^{j_{i}}$ by $j_{i}^{-1}j_{k}\in G_{j}$) above.
\end{enumerate}
In general, elements in the equivalence class of $T$ all have the form shown in Figure \ref{fig class of B_{i,j,k}},
where $g\in G$, $g_{v}\in G_{v}$ for each $v$, $i_{v_{1}},\dots,i_{v_{n-3}},i_{jk}\in G_{i}$, and $j_{k}\in G_{j}$.
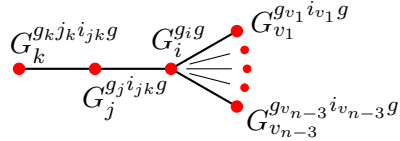
\begin{figure}[h]
\centering
	\begin{tikzpicture} 
	\draw[thick] (-1,0) -- (-2,0);
	\draw[thick] (0,0) -- (-1,0);
	\draw[thick] (0,0) -- (0.866,0.5);
	\draw[thick] (0,0) -- (0.866,-0.5);
	\draw (0.2,0) -- (0.8,0);
	\draw (0.242,0.0625) -- (0.775,0.2);
	\draw (0.242,-0.0625) -- (0.775,-0.2);
	\draw[red, fill] (-2,0) circle [radius=0.075]; 
	\draw[red, fill] (-1,0) circle [radius=0.075]; 
	\draw[red, fill] (0,0) circle [radius=0.075]; 
	\draw[red, fill] (0.866,0.5) circle [radius=0.075]; 
	\draw[red, fill] (0.866,-0.5) circle [radius=0.075]; 
	\draw[red, fill] (0.968,0.25) circle [radius=0.05];
	\draw[red, fill] (1,0) circle [radius=0.05];
	\draw[red, fill] (0.968,-0.25) circle [radius=0.05];
	\node at (0.1,0.3625) {$G_{i}^{g_{i}g}$};
	\node at (-0.6,-0.39) {$G_{j}^{g_{j}i_{jk}g}$};
	\node at (-1.4,0.35) {$G_{k}^{g_{k}j_{k}i_{jk}g}$};
	\node at (1.7,0.65) {$G_{v_{1}}^{g_{v_{1}}i_{v_{1}}g}$};
	\node at (2,-0.65) {$G_{v_{n-3}}^{g_{v_{n-3}}i_{v_{n-3}}g}$};
	\end{tikzpicture}
\caption{Equivalence Class of $\mathfrak{S}$-Labellings of $T$}
\label{fig class of B_{i,j,k}}
\end{figure}
\end{example}

We are now ready to define our complex.

\begin{defn}[The Complex $\mathcal{C}_{n}$]\label{def subcomplex Cn}
We build a (2-dimensional, oriented, simplicial) complex called $\mathcal{C}_{n}$ as follows:
\begin{itemize}
\item 	Take one 0-simplex for each equivalence class of $\mathfrak{S}$-labellings of each tree in Table \ref{table n>=5 points} (equivalently, take one 0-simplex for each equivalence class of $\mathfrak{S}$-labellings of each tree which is achieved by collapsing at least one edge of one of the trees
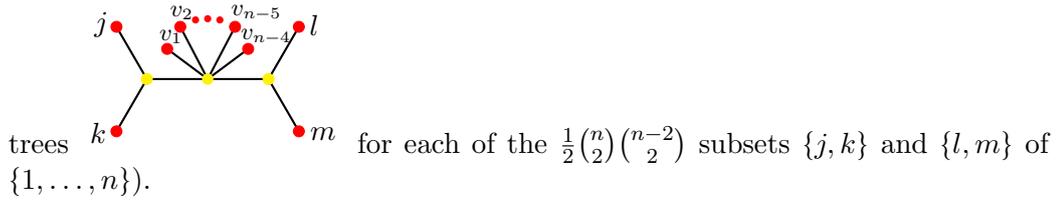
	\begin{tikzpicture}[scale=0.8] 
	\draw[thick] (0,0) -- (-0.5,0.866);
	\draw[thick] (0,0) -- (-0.5,-0.866);
	\draw[thick] (0,0) -- (1,0);
	\draw[thick] (2,0) -- (2.5,0.866);
	\draw[thick] (2,0) -- (2.5,-0.866);
	\draw[thick] (2,0) -- (1,0);
	\draw[thick] (1,0) -- (0.55,0.866);
	\draw[thick] (1,0) -- (0.334,0.5);
	\draw[thick] (1,0) -- (1.45,0.866);
	\draw[thick] (1,0) -- (1.666,0.5);
	\draw[yellow, fill] (0,0) circle [radius=0.09];
	\draw[red, fill] (-0.5,0.866) circle [radius=0.09]; 
	\draw[red, fill] (-0.5,-0.866) circle [radius=0.09]; 
	\draw[yellow, fill] (1,0) circle [radius=0.085];
	\draw[yellow, fill] (2,0) circle [radius=0.09];
	\draw[red, fill] (2.5,0.866) circle [radius=0.09]; 
	\draw[red, fill] (2.5,-0.866) circle [radius=0.09]; 
	\draw[red, fill] (0.55,0.866) circle [radius=0.09]; 
	\draw[red, fill] (0.334,0.5) circle [radius=0.09]; 
	\draw[red, fill] (1.45,0.866) circle [radius=0.09]; 
	\draw[red, fill] (1.666,0.5) circle [radius=0.09]; 
	\draw[red, fill] (0.8,0.98) circle [radius=0.05];
	\draw[red, fill] (1.2,0.98) circle [radius=0.05];
	\draw[red, fill] (1,1) circle [radius=0.05];
	\node at (-0.75,0.9) {$j$};
	\node at (-0.8,-0.9) {$k$};
	\node at (2.75,0.9) {$l$};
	\node at (2.9,-0.9) {$m$};
	\node at (0.4,0.7) {\footnotesize$v_{1}$};
	\node at (0.575,1.1) {\footnotesize$v_{2}$};
	\node at (1.8,1.1) {\footnotesize$v_{n-5}$};
	\node at (1.96,0.7) {\footnotesize$v_{n-4}$};
	\end{tikzpicture}
for each of the $\frac{1}{2} {n\choose2} {n-2\choose2}$ subsets $\{j,k\}$ and $\{l,m\}$ of $\{1,\dots,n\}$).
\item 	Given 0-simplices $[T_{1}]$ and $[T_{2}]$, insert a 1-simplex from $[T_{1}]$ to $[T_{2}]$ if and only if some represenative $T_{2}$ of $[T_{2}]$ is a collapse of some representative $T_{1}$ of $[T_{1}]$.
\item 	Insert a 2-simplex wherever there is a 3-clique $[T_{1}]\dash [T_{2}]\dash [T_{3}]\dash [T_{1}]$ in the 1-skeleton.
\end{itemize}
\end{defn}
We will often refer to simplices of $\mathcal{C}_{n}$ as cells. We will use these terms interchangeably.
Additionally, we will sometimes refer to 0-cells as `vertices', 1-cells as `edges', and 2-cells as `faces'.

Note that $\mathcal{C}_{n}$ is the barycentric spine of the subspace of Outer Space obtained by restricting to only simplices representing the above graph shapes.
As such, we will sometimes refer to it as `the/our complex', or `the/our subcomplex'.

\subsection{The Action of $\outs(G)$ on $\mathcal{C}_{n}$ and its Fundamental Domain $\mathcal{D}_{n}$}\label{subsection action on Cn}

By considering the action of $\outs(G)$ on our complex $\mathcal{C}_{n}$, there is a natural idea of a quotient of $\mathcal{C}_{n}$ (two points are equivalent if they are in the same $\outs(G)$-orbit). We can then pick a `fundamental domain' $\mathcal{D}_{n}$ for the action by choosing a lift of this quotient in $\mathcal{C}_{n}$.

We begin by defining the action of $\outs(G)$ on the 0-skeleton $\mathcal{C}_{n}^{(0)}$ of $\mathcal{C}_{n}$, and then extend this to an action on the full complex $\mathcal{C}_{n}$.

\begin{defn}[Action of $\outs(G)$ on $\mathcal{C}_{n}^{(0)}$]\label{defn action of out(G)}
Let $[\psi]\in\outs(G)$ have representative $\psi\in\aut_{\mathfrak{S}}(G)$ and let $T$ be a point in $\mathcal{C}_{n}^{(0)}$ with $\mathfrak{S}$-labelling $(H_{1},\dots,H_{n})$.
Then $T\cdot[\psi]$ is a graph of groups with the same underlying graph as $T$ and labelling $(\psi(H_{1}),\dots,\psi(H_{n}))$ (where $\psi(H_{i})$ is given by the usual action of $\aut(G)$ on $G$, noting that $H_{i}\le G$).
\end{defn}

\begin{rem}
Given $[\psi_{1}]=[\psi_{2}]\in\outs(G)$, there is some $\iota_{g}:x\mapsto g^{-1}xg \in\inn(G)$ so that $\psi_{2}=\psi_{1}\circ\iota$.
Then for any point $T\in\mathcal{C}_{n}^{(0)}$ with labelling $\left(H_{1},\dots,H_{n}\right)$, we have $\psi_{2}(H_{i})=(\psi_{1}(H_{i}))^{g}$.
As noted in \ref{inner equivalent} of Example \ref{eg equivalent B labellings}, $\left((\psi_{1}(H_{1}))^{g},\dots,(\psi_{1}(H_{n}))^{g}\right)$ and $\left(\psi_{1}(H_{1}),\dots,\psi_{1}(H_{n})\right)$ are equivalent as labellings.
So we really do have that $T\cdot[\psi_{1}]=T\cdot[\psi_{2}]$ --- that is, the action here is well-defined.
As such, we will often write $T\cdot\psi$ (or even $\psi(T)$) for $T\cdot[\psi]$.
\end{rem}

\begin{lemma}\label{lemma action commutes with collapses}
Let $S,T\in\mathcal{C}_{n}^{(0)}$ such that $S$ is a collapse of $T$, and let $f:T\to S$ be the collapsing map.
Let $\psi\in\outs(G)$.
Then $f(T\cdot\psi)=f(T)\cdot\psi$.
\end{lemma}

\begin{proof}
Suppose $T$ as a graph has labelling $(H_{1},\dots,H_{n})$.
Recall from Definition \ref{defn collapse} that permitted collapses do not alter vertex groups in any way.
Thus $(H_{1},\dots,H_{n})$ must also be a labelling for $S$.
Now $T\cdot\psi$ is a graph of groups with the same underlying graph as $T$, and labelling $(\psi(H_{1}),\dots,\psi(H_{n}))$.
Similarly, $S\cdot\psi$ has the same underlying graph as $S$, with labelling $(\psi(H_{1}),\dots,\psi(H_{n}))$.
Since $T\cdot\psi$ has the same underlying graph as $T$, applying $f$ to $T\cdot\psi$ yields a graph of groups whose underlying graph is the same as that of $S$, and has $(\psi(H_{1}),\dots,\psi(H_{n}))$ as a labelling.
But this exactly describes the graph of groups $S\cdot\psi$.
That is, $f(T\cdot\psi)=S\cdot\psi=f(T)\cdot\psi$.
\end{proof}

Since $\mathcal{C}_{n}$ is a simplicial complex, then any cell is uniquely determined by its vertices (0-cells).
We will thus denote a cell by $[T_{0},\dots,T_{k}]$ where $T_{0},\dots,T_{k}$ are its vertices.
Note that for us we will only ever have $k=1$ or $k=2$ (or $k=0$).

\begin{prop}\label{prop action preserves cells}
Let $\psi\in\outs(G)$. If $[T_{0},\dots,T_{k}]$ is a cell in $\mathcal{C}_{n}$, then so is 
\\ \noindent $[T_{0}\cdot\psi,\dots,T_{k}\cdot\psi]$.
\end{prop}

\begin{proof}
This is true by definition of the action for $k=0$.

Let $[T_{0},T_{1}]$ be an edge in $\mathcal{C}_{n}$.
Then $T_{0}$ and $T_{1}$ are graphs of groups with $T_{1}$ a collapse of $T_{0}$ --- say $f:T_{0}\to T_{1}$ is the collapsing map.
We know that $T_{0}\cdot\psi$ is a point in $\mathcal{C}_{n}^{(0)}$, and since it has the same underlying graph as $T_{0}$, then so is $f(T_{0}\cdot\psi)$.
So we have an edge $[T_{0}\cdot\psi,f(T_{0}\cdot\psi)]\in\mathcal{C}_{n}$.
But by Lemma \ref{lemma action commutes with collapses}, $f(T_{0}\cdot\psi)=f(T_{0})\cdot\psi=T_{1}\cdot\psi$.
So if $[T_{0},T_{1}]$ is a cell in $\mathcal{C}_{n}$, then so is $[T_{0}\cdot\psi,T_{1}\cdot\psi]$.

Now suppose $[T_{0},T_{1},T_{2}]$ is a 2-cell in $\mathcal{C}_{n}$. Then we must have a 3-clique 
\\ \noindent $[T_{0}]\dash [T_{1}]\dash[T_{2}]\dash [T_{0}]$, so $[T_{0},T_{1}]$, $[T_{1},T_{2}]$, and $[T_{0},T_{2}]$ are 1-cells in $\mathcal{C}_{n}$.
Then 
\\ \noindent $[T_{0}\cdot\psi,T_{1}\cdot\psi]$, $[T_{1}\cdot\psi,T_{2}\cdot\psi]$, and $[T_{0}\cdot\psi,T_{2}\cdot\psi]$ are 1-cells in $\mathcal{C}_{n}$ forming a 3-clique, hence by Definition \ref{def subcomplex Cn} we must have a 2-cell $[T_{0}\cdot\psi,T_{1}\cdot\psi,T_{2}\cdot\psi]$.
\end{proof}

\begin{defn}[Action of $\outs(G)$ on $\mathcal{C}_{n}$]
The action of an element $\psi\in\outs(G)$ on a $k$-cell $[T_{0},\dots,T_{k}]$ of $\mathcal{C}_{n}$ is defined to be:
\[[T_{0},\dots,T_{k}]\cdot\psi:=[T_{0}\cdot\psi,\dots,T_{k}\cdot\psi]\]
\end{defn}

We now construct a fundamental domain for this action.
The quotient space obtained from the action has one cell for each orbit of cells in $\mathcal{C}_{n}$.
The obvious choice to make here is to take the lift to be the subcomplex supported by vertices which are all the graphs of groups (as listed in Table \ref{table n>=5 points}) whose vertex groups are precisely the factor groups $G_{1},\dots,G_{n}$. This is formalised below:

\begin{defn}[Construction of $\mathcal{D}_{n}$]\label{defn Dn}
We take the 0-skeleton $\mathcal{D}_{n}^{(0)}$ of $\mathcal{D}_{n}$ to be the set of graphs of groups $T$ whose underlying graph  is a tree from Table \ref{table n>=5 points} so that, up to permuting the indices, $T$ has a labelling $(G_{1},\dots,G_{n})$.
We now define $\mathcal{D}_{n}$ to be the subcomplex of $\mathcal{C}_{n}$ made up of all cells whose vertices are in $\mathcal{D}_{n}^{(0)}$.
\end{defn}

\begin{example}
Note that in our selection of graphs of groups, we still allow permutation of the vertex labels, just not conjugation.
So \Trho{$n-2$}{$G_{i}$}{$G_{j}$} and \Trho{$n-2$}{$G_{i}$}{$G_{k}$} are both in $\mathcal{D}_{n}^{(0)}$, while \Trho{$n-2$}{$G_{i}^{x}$}{$G_{j}$} is not (for $x\not\in G_{i}$, i.e. $G_{i}\ne G_{i}^{x}$ as sets).
Note however that (for $y\in G_{i}$) \Tbeta{$n-2$}{$G_{i}$}{$G_{j}^{y}$} \textbf{is} in $\mathcal{D}_{n}^{(0)}$, since \Tbeta{$n-2$}{$G_{i}$}{$G_{j}^{y}$} is equivalent to \Tbeta{$n-2$}{$G_{i}$}{$G_{j}$} under Definition \ref{defn equivalent labellings}.
\end{example}

\begin{notation}\label{notation stab}
Given a vertex $T$ in $\mathcal{C}_{n}$, we denote by $\stab(T)$ the $\outs(G)$-stabiliser of $T$, that is, the set $\{\psi\in\outs(G) | T\simeq T\cdot\psi\}$,
where $\simeq$ is the equivalence described in Definition \ref{defn equivalent labellings}. We will often abuse notation and write $T=S$ for $T\simeq S$.
\end{notation}

\begin{lemma}\label{lem stabs include}
Let $T$ be a vertex in $\mathcal{C}_{n}$ (so $T$ is a graph of groups) and let $S$ be achieved by collapsing edges of $T$. Then $\stab(T)\subseteq\stab(S)$.
\end{lemma}

\begin{proof}
Let $f:T\to S$ be the collapsing map, and let $\psi\in\stab(T)\le\outs(G)$.
By Lemma \ref{lemma action commutes with collapses}, $S\cdot\psi=f(T)\cdot\psi=f(T\cdot\psi)$.
Since $\psi\in\stab(T)$ then $T\cdot\psi=T$, hence $S\cdot\psi=f(T)=S$.
That is, $\psi\in\stab(S)$.
\end{proof}

\begin{prop} \label{prop Dn is fun dom}
The subcomplex $\mathcal{D}_{n}$ of $\mathcal{C}_{n}$ described above is indeed a fundamental domain for the action of $\outs(G)$.
\end{prop}

\begin{proof}
We need to show that every orbit of cells in $\mathcal{C}_{n}$ is represented in $\mathcal{D}_{n}$.
That is, if $C\in\mathcal{C}_{n}$ is a $k$-cell of $\mathcal{C}_{n}$ (for $k\in\{0,1,2\}$), then there is some $\psi\in\outs(G)$ so that $C\cdot\psi^{-1}\in\mathcal{D}_{n}$.

Let $(T,(H_{1},\dots,H_{n}))$ be a point in $\mathcal{C}_{n}^{(0)}$. 
Since $H_{1}\ast\dots\ast H_{n}$ is an $\mathfrak{S}$ free factor splitting of $G_{1}\ast\dots\ast G_{n}$, then by Lemma \ref{lemma automorphisms exist between splittings}, there exists $\psi\in\aut_{\mathfrak{S}}(G)$ so that for each $i$, $\psi(G_{i})=H_{i}$.
Then $(T,(H_{1},\dots,H_{n}))\cdot[\psi^{-1}]=(T,(G_{1},\dots,G_{n}))\in\mathcal{D}_{n}$, with $[\psi^{-1}]\in\outs(G)$ as required.

Now let $[\mathbf{T},\mathbf{S}]$ be an edge in $\mathcal{C}_{n}$ (so $\mathbf{S}$ is a collapse of $\mathbf{T}$), and choose $\psi\in\outs(G)$ so that $\mathbf{T}\cdot\psi^{-1}\in\mathcal{D}_{n}$.
Then $(G_{1},\dots,G_{n})$ is an $\mathfrak{S}$-labelling for $\mathbf{T}\cdot\psi^{-1}$, and by Lemma \ref{lemma action commutes with collapses}, $\mathbf{S}\cdot\psi$ is a collapse of $\mathbf{T}\cdot\psi^{-1}$ and hence $(G_{1},\dots,G_{n})$ is also an $\mathfrak{S}$-labelling for $\mathbf{S}\cdot\psi^{-1}$.
That is, $[\mathbf{T}\cdot\psi^{-1},\mathbf{S}\cdot\psi^{-1}]$ is an edge in $\mathcal{D}_{n}$.

Similarly, if $[\mathbf{T_{0}},\mathbf{T_{1}},\mathbf{T_{2}}]$ is a face in $\mathcal{C}_{n}$, then $\mathbf{T_{2}}$ is a collapse of $\mathbf{T_{1}}$, which in turn is a collapse of $\mathbf{T_{0}}$. Choosing $\psi\in\outs(G)$ with $\mathbf{T_{0}}\cdot\psi^{-1}\in\mathcal{D}_{n}$, the above argument then yields that $[\mathbf{T_{0}}\cdot\psi^{-1},\mathbf{T_{1}}\cdot\psi^{-1},\mathbf{T_{2}}\cdot\psi^{-1}]$ is a face in $\mathcal{D}_{n}$.
\end{proof}

\begin{prop}\label{prop strict fun dom} 
The fundamental domain $\mathcal{D}_{n}$ described above is \emph{strict}. That is, it contains precisely one representative of each vertex, edge, and face (2-cell) orbit.
\end{prop}

\begin{proof}
First, note that for two vertices to share an $\outs(G)$-orbit, they must have the same underlying graph structure. Moreover, since our automorphisms are pure symmetric (i.e. do not permute factor groups), they must have the same indexing of vertices.
Since our fundamental domain was chosen to allow only one labelling for each distinct graph structure, this precisely means that each vertex of $\mathcal{D}_{n}$ is in a distinct orbit.

Now suppose we have two faces in the fundamental domain, $[T_{0},T_{1},T_{2}]$ and $[S_{0},S_{1},S_{2}]$, which are in the same orbit.
Then their vertices are also in the same respective orbits (i.e. $T_{i}$ and $S_{i}$ share an orbit for each $i$).
Since our fundamental domain contains only one representative of each vertex orbit, we must have that $T_{i}=S_{i}$ for each $i=1,2,3$.
But when we constructed $\mathcal{C}_{n}$, we inserted only one 2-cell for each 3-clique.
That is, a face is uniquely determined by its vertices, so $[T_{0},T_{1},T_{2}]=[S_{0},S_{1},S_{2}]$.

The same argument applies to edges (cells with the form $[T_{0},T_{1}]$).
Hence no two cells of our fundamental domain are in the same orbit, that is, we have a strict fundamental domain.
\end{proof}

\subsection{Stabilisers of Vertices in $\mathcal{D}_{n}$}\label{stabilisers}

To move through our complex $\mathcal{C}_{n}$, we
consider `collapse--expansion' paths, since two vertices (graphs of groups) are adjacent in $\mathcal{C}_{n}$ if and only if one is a collapse of the other.
If $T_{1}\dash T_{2}\dash T_{3}$ is a path in $\mathcal{C}_{n}$ such that $T_{2}$ is a collapse of both $T_{1}$ and $T_{3}$, and $T_{1}$ and $T_{3}$ have the same underlying graph structure, then we will have that $T_{3}=T_{1}\cdot\psi$ for some $\psi\in\stab(T_{2})$.
Thus understanding vertex stabilisers is key to understanding adjacency in $\mathcal{C}_{n}$.
We will also need to understand vertex stabilisers in order to apply Brown's Theorem (Theorem \ref{brown thm strict}).

Recall that given a point $T=(T,(H_{1},\dots,H_{n}))\in\mathcal{C}_{n}^{(0)}$, we have that $\psi\in\outs(G)$ is in the stabiliser $\stab(T)$ of $T$ if and only if $T\cdot\psi=T$, that is, $(H_{1},\dots,H_{n})$ and $(\psi(H_{1}),\dots,\psi(H_{n}))$ are equivalent as labellings of $T$.
Recall from Definition \ref{defn equivalent labellings} that this means for each $i=1,\dots,n$ there exists $g_{i}\in G$ and $\varphi_{i}\in\aut(H_{i})$ so that $\psi(H_{i})=\varphi_{i}(H_{i})^{g_{i}}$, and moreover, for every edge
\begin{tikzpicture}
\draw[thick,->-] (0,0) -- (1,0);
\filldraw (0,0) circle [radius=0.06];
\filldraw (1,0) circle [radius=0.06];
\node at (0,0.25) {$u$};
\node at (1,0.25) {$v$};
\end{tikzpicture}
of $T$ we have $g_{v}g_{u}^{-1}\in H_{u}$.

We will only compute stabilisers of vertices in $\mathcal{D}_{n}$.
However, if $T\cdot\chi\in\mathcal{C}_{n}$ (with $T\in\mathcal{D}_{n}$ and $\chi\in\outs(G)$), then $\stab(T\cdot\chi)=\chi^{-1}\stab(T)\chi=\stab(T)^{\chi}$.
As such, we will assume for now that any graph of groups $T$ has $(G_{1},\dots,G_{n})$ as a labelling.

We will present several viewpoints on the stabiliser of a vertex.

\subsubsection*{The Guirardel--Levitt Approach}

Recall from Definition \ref{defn factor autos} that $\Phi\le\outs(G)$ is the group of factor automorphisms of $G_{1}\ast\dots\ast G_{n}$, with $\Phi=\prod_{i=1}^{n}\aut(G_{i})$.

Given a vertex $v_{i}$ of a point (graph of groups) $T\in\mathcal{C}_{n}$, with vertex group $G_{v_{i}}$ (assuming $G_{v_{i}}\ne\{1\}$, that is, $v_{i}$ is not a trivial vertex), let 
$\mu_{i}$ be the valency of $v_{i}$ in $T$.

In \cite[Section 5]{Guirardel2007}, Guirardel and Levitt give the stabiliser of $T$ in $\outs(G)$ as being isomorphic to:
\[\prod_{i=1}^{n}\left(G_{i}^{\mu_{i}-1}\rtimes\aut(G_{i})\right)=\left(\prod_{i=1}^{n}G_{i}^{\mu_{i}-1}\right)\rtimes\Phi\]
where the semidirect product relation is given by the natural action of $\Phi$ on each $G_{i}$.

Using this, we recover Table \ref{table stabs G-L} showing the stabilisers (upto isomorphism) of points in $\mathcal{D}_{n}$. Recall that the graph structures of these points are shown in Table \ref{table n>=5 points}.

\begin{table}[h]
\centering
\begin{tabular}{ | c | r | }
\hline
Vertex 			&	Stabiliser								
\\	
\hline 
$\rho_{jk}$			&	$\Phi$									
\\
$\sigma_{i,jk,lm}$		&	$G_{i}^{n-4}\rtimes\Phi$						
\\
$\tau_{j,k,lm}$		&	$G_{j}\rtimes\Phi$							
\\
$\alpha$			&	$\Phi$									
\\
$\beta_{j,k}$		&	$G_{j}\rtimes\Phi$							
\\
$\gamma_{i,jk}$		&	$G_{i}^{n-3}\rtimes\Phi$						
\\
$\delta_{i,j,k,lm}$		&	$(G_{i}^{n-4}\times G_{j})\rtimes\Phi$				
\\
$\varepsilon_{j,k,l,m}$	&	$(G_{j}\times G_{l})\rtimes\Phi$					
\\
$A_{i}$			&	$G_{i}^{n-2}\rtimes\Phi$						
\\
$B_{i,j,k}$			&	$(G_{i}^{n-3}\times G_{j})\rtimes\Phi$				
\\
$C_{i,j,k,l,m}$		&	$(G_{i}^{n-4}\times G_{j}\times G_{l})\rtimes\Phi$		
\\
\hline
\end{tabular}
\caption{Vertex Stabilisers (up to isomorphism) using Guirardel--Levitt}
\label{table stabs G-L}
\end{table}

This point of view corresponds to fixing a particular edge of a graph of groups $T$, and then twisting the remaining edges (outwards from the fixed edge).
We demonstrate this with an example:

\begin{example}\label{eg stab sigma}
Consider the graph of groups $\sigma$:
\begin{tikzpicture}[scale=0.8] 
\draw[thick] (0,0) -- (-1,0);
\draw[thick] (-1,0) -- (-1.5,0.866);
\draw[thick] (-1,0) -- (-1.5,-0.866);
\draw[thick] (0,0) -- (1,0);
\draw[thick] (1,0) -- (1.5,0.866);
\draw[thick] (1,0) -- (1.5,-0.866);
\draw[thick] (0,0) -- (0.5,-0.866);
\draw[thick] (0,0) -- (-0.5,-0.866);
\draw (0,-0.2) -- (0,-0.8);
\draw (0.0625,-0.242) -- (0.2,-0.775);
\draw (-0.0625,-0.242) -- (-0.2,-0.775);
\draw[red, fill] (-1.5,0.866) circle [radius=0.09]; 
\draw[red, fill] (-1.5,-0.866) circle [radius=0.09]; 
\draw[yellow, fill] (-1,0) circle [radius=0.065]; 
\draw[red, fill] (0,0) circle [radius=0.09]; 
\draw[yellow, fill] (1,0) circle [radius=0.065]; 
\draw[red, fill] (1.5,0.866) circle [radius=0.09]; 
\draw[red, fill] (1.5,-0.866) circle [radius=0.09]; 
\draw[red, fill] (-0.5,-0.866) circle [radius=0.09]; 
\draw[red, fill] (0.5,-0.866) circle [radius=0.09]; 
\draw[red, fill] (-0.25,-0.968) circle [radius=0.05];
\draw[red, fill] (0,-1) circle [radius=0.05];
\draw[red, fill] (0.25,-0.968) circle [radius=0.05];
\node at (0,0.3) {$G_{i}$};
\node at (-1.85,0.9) {$G_{j}$};
\node at (-1.85,-0.9) {$G_{k}$};
\node at (1.85,0.9) {$G_{l}$};
\node at (1.95,-0.9) {$G_{m}$};
\node at (-0.6,-1.2) {$G_{v_{1}}$};
\node at (1,-1.2) {$G_{v_{n-5}}$};
\end{tikzpicture}
(the graph $\sigma_{i,jk,lm}$ with labelling $\hat{G}:=(G_{1},\dots, G_{n})$).
We will compute the stabiliser of $\sigma$ by `fixing' an edge in the graph of groups.

Let $\hat{H}=\left(G_{1}^{h_{1}},\dots,G_{n}^{h_{n}}\right)$ be an arbitrary labelling in the equivalence class of $\hat{G}$ of labellings of $\sigma$.
Note that by Definition \ref{defn equivalent labellings} we must also have elements $h_{jk},h_{lm}\in G$ corresponding to the two trivial vertices of $\sigma$, and that $h_{j}=h_{k}=h_{jk}$ and $h_{l}=h_{m}=h_{lm}$.
As noted in \ref{inner equivalent} of Example \ref{eg equivalent B labellings}, inner automorphisms of $G$ preserve equivalence classes of labellings.
Thus the labelling $\hat{H}'$ achieved by replacing each $G_{a}^{h_{a}}$ by $G_{a}^{h_{a}h_{lm}^{-1}}$ is equivalent to $\hat{H}$.
Note that in this labelling, the conjugator of $G_{i}$ is $h_{i}h_{lm}^{-1}=(h_{lm}h_{i}^{-1})^{-1}\in G_{i}$ by Definition \ref{defn equivalent labellings}.
Since this is an (inner) factor automorphism of $G_{i}$, then the labelling $\hat{H}''$ achieved by replacing $G_{i}^{h_{i}h_{lm}^{-1}}$ in $\hat{H}'$ with simply $G_{i}$ is also equivalent to $\hat{H}$.
We have now essentially `fixed' the edge $i\dash lm$ (i.e. the edge which separates $G_{l}$ and $G_{m}$ from $G_{i}$ and all the other vertex groups) in $\sigma$.

Note that for any $a$ we have $h_{a}h_{lm}^{-1}=(h_{a}h_{i}^{-1})(h_{lm}h_{i}^{-1})^{-1}\in G_{i}$.
So aside from factor automorphisms (i.e. replacing $G_{a}$ with $\varphi(G_{a})$ for $\varphi\in\Phi$),
the only freedom we have left is to `twist' along the remaining edges incident to $i$ (the vertex in $\sigma$ whose vertex group is $G_{i}$);
that is, given each remaining edge $e$ incident to $i$, to conjugate all vertex groups separated from $i$ by $e$ by an element of $G_{i}$.
Note that for $g_{i}\in G_{i}$ and $\varphi\in\Phi$, we have $\varphi(G_{a}^{g_{i}})=\varphi(G_{a})^{\varphi(g_{i})}=G_{a}^{\varphi(g_{i})}$.
We will let $G_{i_{v}}$ denote the group of Whitehead automorphisms which conjugate the vertex group $G_{v}$ by elements of $G_{i}$ (for $v=v_{1},\dots,v_{n-5}$), and similarly denote by $G_{i_{jk}}$ the group of Whitehead automorphisms which conjugate the vertex groups $G_{j}$ and $G_{k}$ simultaneously by elements of $G_{i}$.
Note that $G_{i_{jk}}\cong G_{i_{v}}\cong G_{i}$ (for $v=v_{1},\dots,v_{n-5}$).
Since twists along edges from $i$ happen independently of each other, we then have that $\stab(\sigma)=(G_{i_{jk}}\times G_{i_{v_{1}}}\times\dots\times G_{i_{v_{n-5}}})\rtimes\Phi\cong G_{i}^{n-4}\rtimes\Phi$, as listed in Table \ref{table stabs G-L}.
\end{example}

Note that by choosing an edge to `fix'  in a graph of groups $T$ and indexing the remaining edges as in the above example, we can similarly expand all the stabilisers listed in Table \ref{table stabs G-L}.
While this works well for some graphs, it does lead to a lack of symmetry, and in graphs such as $A_{i}$, such a choice can feel entirely arbitrary.

Even in the above example, we could have chosen to fix the edge from $G_{i}$ leading to $G_{j}$ and $G_{k}$ (or even an edge $G_{i}\dash G_{v}$) instead of the edge from $G_{i}$ leading to $G_{l}$ and $G_{m}$ .
This means we must have that $\stab(\sigma)=(G_{i_{jk}}\times G_{i_{v_{1}}}\times\dots\times G_{i_{v_{n-5}}})\rtimes\Phi=(G_{i_{lm}}\times G_{i_{v_{1}}}\times\dots\times G_{i_{v_{n-5}}})\rtimes\Phi\left(=(G_{i_{jk}}\times G_{i_{lm}}\times G_{i_{v_{2}}}\times\dots\times G_{i_{v_{n-5}}})\rtimes\Phi\right)$ as subgroups of $\outs(G)$.

To see why this holds, let $g_{i}\in G_{i}$, let $\iota_{g_{i}^{-1}}$ be the element of $\inn(G)$ which conjugates every element of $G$ by $g_{i}^{-1}$, and for groups $G_{w_{1}},\dots,G_{w_{a}}$ ($\{w_{1},\dots,w_{a}\}\subseteq\{1,\dots,n\}$) let $\left(\{G_{w_{1}},\dots, G_{w_{k}}\},g_{i}\right)$ be the Whitehead automorphism which conjugates each $G_{w}$ by $g_{i}$ (for $w=w_{1},\dots,w_{a}$).
Then for $\left(\{G_{l},G_{m}\},g_{i}\right)$ an arbitrary element of $G_{i_{lm}}$, we have that $\left(\{G_{l},G_{m}\},g_{i}\right)\iota_{g_{i}^{-1}}=\left(\{G_{i}\},g_{i}\right)\left(\{G_{j},G_{k},G_{v_{1}},\dots,G_{v_{n-5}}\},g_{i}\right)$, where $\left(\{G_{i}\},g_{i}\right)\in\inn(G_{i}) \\ \noindent \le\aut(G_{i})\le\Phi$ and $\left(\{G_{j},G_{k},G_{v_{1}},\dots,G_{v_{n-5}}\},g_{i}\right)\in G_{i_{jk}}\times G_{i_{v_{1}}}\times\dots\times G_{i_{v_{n-5}}}$.
That is, $G_{i_{lm}}\le(G_{i_{jk}}\times G_{i_{v_{1}}}\times\dots\times G_{i_{v_{n-5}}})\rtimes\Phi$.
One can similarly show that $G_{i_{jk}}\le(G_{i_{lm}}\times G_{i_{v_{1}}}\times\dots\times G_{i_{v_{n-5}}})\rtimes\Phi$, as well as the inclusions required for the third claimed equality.

Thus while correct, and simple to write down, this method of determining stabilisers can obscure subgroups and other structure.
To remedy this, we may consider initially fixing just a vertex, rather than an edge, in our graph of groups, or more generally, not `fixing' anything at all.

\subsubsection*{The Bass--Jiang Approach}

Given a vertex $v_{i}$ of a point (graph of groups) $T\in\mathcal{C}_{n}^{(0)}$, with vertex group $G_{i}$ (assuming $G_{i}\ne\{1\}$, that is, $v_{i}$ is not a trivial vertex), let
$E(v_{i})$ be the set of edges of $T$ with $v_{i}$ as an endpoint.
We will index these edges by the vertices they separate from $v_{i}$
(so for example in \Tgamma{$n-3$}{$i$}{$j$}{$k$}, the edges incident to the vertex $i$ are indexed by $v_{1},\dots,v_{n-3}$ and $(jk)$).

Bass and Jiang \cite[Theorem 8.1]{Bass1996} give a filtration explicitly describing the $\out(G)$ stabiliser of a graph of groups.
Since we are restricting to pure symmetric (outer) automorphisms, we have trivial edge stabilisers in our graphs of groups and no graph automorphisms (as our graphs of groups are trees, and we do not permit permutation of the vertex groups). So this filtration simplifies to a short exact sequence:

\begin{center}
\begin{tikzcd}
1
\ar[r]
	&
		\displaystyle\prod_{i=1}^{n}\left(\faktor{\left(\displaystyle\prod_{e\in E(v_{i})}G_{i_{e}}\right)}{Z(G_{i})}\right)
		\ar[r]
			&
				\stab(T)
				\ar[r]
					&
						\displaystyle\prod_{i=1}^{n}\out(G_{i})
						\ar[r]
							&
								1
\end{tikzcd}
\end{center}
where $G_{i_{e}}\cong G_{i}$ is the group of Whitehead automorphisms which conjugate the vertex groups of vertices separated from $v_{i}$ by $e$ by elements of the vertex group $G_{i}$, and $Z(G_{i})$ is the centre of $G_{i}$, with diagonal embedding.
For brevity, given $T\in\mathcal{D}_{n}^{(0)}$ we will write $M_{T}$ for the $\prod_{i=1}^{n}\left(\faktor{\left(\prod_{e\in E(v_{i})}G_{i_{e}}\right)}{Z(G_{i})}\right)$ term of the above short exact sequence.
This term corresponds to `twisting' along each edge incident to each vertex in $T$.

Writing $v$ for the vertex group $G_{v}$, and $v_{e}$ for the automorphism group $G_{v_{e}}$ (with the indexing described above), we recover Table \ref{table stabs B-J}, showing the $M_{T}$ term of the Bass--Jiang short exact sequence for each tree $T$ of Table \ref{table n>=5 points}.

\begin{table}[h]
\centering
\begin{tabular}{ | c | l | }
\hline 				&																							\\
$\displaystyle T$		&	$\displaystyle M_{T}$																			\\
\hline 				&																							\\
$\rho_{jk}$			&	
					\hspace{0.185cm}$\displaystyle\prod_{v}\inn(v)$ 															\\
$\sigma_{i,jk,lm}$		&	
					\hspace{0.185cm}$\displaystyle\prod_{v\ne i}\inn(v)\hspace{0.185cm}\times\faktor{(i_{v_{1}}\times\dots\times i_{v_{n-5}}\times i_{jk}\times i_{lm})}{Z(i)}$																											\\
$\tau_{j,k,lm}$		&
					\hspace{0.175cm}$\displaystyle\prod_{v\ne j}\inn(v)\hspace{0.1725cm}\times\faktor{(j_{v_{1}\dots v_{n-4} l m}\times j_{k})}{Z(j)}$			\\
$\alpha$			&
					\hspace{0.185cm}$\displaystyle\prod_{v}\inn(v)$															\\
$\beta_{j,k}$			&	
					\hspace{0.175cm}$\displaystyle\prod_{v\ne j}\inn(v)\hspace{0.1725cm}\times\faktor{(j_{v_{1}\dots v_{n-2}}\times j_{k})}{Z(j)}$			\\
$\gamma_{i,jk}$		&
					\hspace{0.185cm}$\displaystyle\prod_{v\ne i}\inn(v)\hspace{0.185cm}\times\faktor{(i_{v_{1}}\times\dots\times i_{v_{n-3}}\times i_{jk})}{Z(i)}$	\\
$\delta_{i,j,k,lm}$		&	
					\hspace{0.075cm}$\displaystyle\prod_{v\ne i,j}\inn(v)\hspace{0.075cm}\times\faktor{(i_{v_{1}}\times\dots\times i_{v_{n-5}}\times i_{jk}\times i_{lm})}{Z(i)} \times\faktor{(j_{v_{1}\dots v_{n-5} i l m}\times j_{k})}{Z(j)}$																	\\
$\varepsilon_{j,k,l,m}$	&
					\hspace{0.09cm}$\displaystyle\prod_{v\ne j,l}\inn(v)\hspace{0.0775cm}\times\faktor{(j_{v_{1}\dots v_{n-4} l m}\times j_{k})}{Z(j)} \times\faktor{(l_{v_{1}\dots v_{n-4} j k}\times l_{m})}{Z(l)}$																		\\
$A_{i}$			&
					\hspace{0.185cm}$\displaystyle\prod_{v\ne i}\inn(v)\hspace{0.185cm}\times\faktor{(i_{v_{1}}\times\dots i_{v_{n-1}})}{Z(i)}$				\\
$B_{i,j,k}$			&
					\hspace{0.075cm}$\displaystyle\prod_{v\ne i,j}\inn(v)\hspace{0.08cm}\times\faktor{(i_{v_{1}}\times\dots\times i_{v_{n-3}}\times i_{jk})}{Z(i)} \times\faktor{(j_{i v_{1}\dots v_{n-3}}\times j_{k})}{Z(j)}$																			\\
$C_{i,j,k,l,m}$		&
					$\displaystyle\prod_{v\ne i,j,l}\inn(v)\times\faktor{(i_{v_{1}}\times\dots\times i_{v_{n-5}}\times i_{jk}\times i_{lm})}{Z(i)}\times\faktor{(j_{i v_{1}\dots v_{n-5} l m}\times j_{k})}{Z(j)}$																							\\
				&	\hspace{2.08cm}$\displaystyle\times\faktor{(l_{i v_{1}\dots v_{n-5} j k}\times l_{m})}{Z(l)}$									\\
\hline
\end{tabular}
\caption{$M_{T}$ Terms of Vertices $T$ from Bass--Jiang Short Exact Sequence}
\label{table stabs B-J}
\end{table}

Observe that for any $a=2,\dots,n-1$ we can embed the group $i_{v_{1}\dots v_{a}}$ diagonally into the direct product $i_{v_{1}}\times\dots\times i_{v_{a}}$.
We write $i_{v_{1}\dots v_{a}}\diag i_{v_{1}}\times\dots\times i_{v_{a}}$ to indicate that we consider $i_{v_{1}\dots v_{a}}$ to be the diagonal subgroup of $i_{v_{1}}\times\dots\times i_{v_{a}}$. 
If $g_{i}\in i_{v_{1}\dots v_{n-1}}$ and $\iota_{g_{i}}$ is the inner automorphism which conjugates all elements of $G$ by $g_{i}$,
then $g_{i}\iota_{g_{i}}$ conjugates all elements of $G_{i}$ by $g_{i}^{-1}$ and fixes all other elements of $G$.
That is, $g_{i}\iota_{g_{i}}\in\inn(G_{i})$. So we have that $i_{v_{1}\dots v_{n-1}}\inn(G)=\inn(G_{i})\inn(G)$ as cosets in $\outs(G)$.

More generally, if $A\sqcup B$ partitions $\{1,\dots,n\}-\{i\}$ then $G_{i_{A}}=G_{i_{iB}}$ in $\outs(G)$.

\begin{example}\label{eg stab A}
Consider the graph of groups $A$:
\begin{tikzpicture} 
\draw[thick] (0,0) -- (0,-1);
\draw[thick] (0,0) -- (0.866,0.5);
\draw[thick] (0,0) -- (-0.866,0.5);
\draw (0.2,0) -- (0.8,0);
\draw (0.1732,-0.1) -- (0.6928,-0.4);
\draw (0.1,-0.1732) -- (0.4,-0.6928);
\draw (-0.1,-0.1732) -- (-0.4,-0.6928);
\draw (-0.1732,-0.1) -- (-0.6928,-0.4);
\draw (-0.2,0) -- (-0.8,0);
\draw[red, fill] (0,0) circle [radius=0.075]; 
\draw[red, fill] (0.866,0.5) circle [radius=0.075]; 
\draw[red, fill] (1,0) circle [radius=0.05];
\draw[red, fill] (0.866,-0.5) circle [radius=0.05];
\draw[red, fill] (0.5,-0.866) circle [radius=0.05];
\draw[red, fill] (0,-1) circle [radius=0.075]; 
\draw[red, fill] (-0.5,-0.866) circle [radius=0.05];
\draw[red, fill] (-0.866,-0.5) circle [radius=0.05];
\draw[red, fill] (-1,0) circle [radius=0.05];
\draw[red, fill] (-0.866,0.5) circle [radius=0.075]; 
\node at (0,0.35) {$G_{i}$};
\node at (0.1,-1.4) {$G_{v_{j}}$};
\node at (1.05,0.775) {$G_{v_{1}}$};
\node at (-0.6,0.775) {$G_{v_{n-1}}$};
\end{tikzpicture}
(the graph $A_{i}$ with labelling $\hat{G}:=\left(G_{1},\dots,G_{n}\right)$).
We explore two ways of determining $\stab(A)$.

\begin{enumerate}
\item We will first deduce $\stab(A)$ from the Bass--Jiang filtration.
Note that this filtration allows us to compute the $M_{A}$ term of the Bass--Jiang short exact sequence, rather than the stabiliser of $A$ itself.

We begin by considering `twisting' from a vertex $v_{j}$ (with vertex group $G_{v_{j}}$) for some $j=1,\dots,n-1$.
Since this is a vertex of valency 1, we conjugate all other vertex groups by the same element $g_{j}$ of $G_{v_{j}}$, to achieve a labelling (upto appropriate reordering) $\left(G_{v_{j}}, G_{i}^{g_{j}},G_{v_{1}}^{g_{j}},\dots,G_{v_{n-1}}^{g_{j}}\right)$.
However, by applying the inner automorphism $\iota_{g_{j}^{-1}}\in\inn(G)$ which conjugates all elements of $G$ by $g_{j}^{-1}$, we see that this is equivalent to the labelling $\left(G_{j}^{g_{j}^{-1}},G_{i},G_{v_{1}},\dots,G_{v_{n-1}}\right)$, which equates to having applied the inner factor automorphism which conjugates $G_{v_{j}}$ by an element of itself.
Thus twisting from a valency 1 vertex $v_{j}$ simply yields $\inn(G_{v_{j}})$ at this stage of the Bass--Jiang filtration.

We now consider twists from the vertex $i$ with vertex group $G_{i}$.
This has valency $n-1$, and so we get $n-1$ groups $i_{v_{j}}\cong G_{i}$, which equate to conjugating the vertex group $G_{v}$ by elements of $G_{i}$.
Since twisting along edges incident to $i$ is independent of the order in which we twist them, this forms a direct product of groups.
Note that if we were to twist all edges by the same element $g_{i}$ of $G_{i}$, this is equivalent (up to an inner automorphism of $G$) to just conjugating $G_{i}$ by $g_{i}^{-1}$.
That is, $\inn(G_{i})=i_{v_{1}\dots v_{n-1}}\diag i_{v_{1}}\times\dots\times i_{v_{n-1}}$, the diagonal subgroup.
Moreover, if $g_{i}\in Z(G_{i})$ is central in $G_{i}$, then conjugation of $G_{i}$ by $g_{i}^{-1}$ is the identity map on $G_{i}$, and so twisting along all edges incident to $i$ by $g_{i}$ is equivalent to the identity automorphism.
Thus we must quotient out by the centre of $G_{i}$, embedded diagonally into $i_{v_{1}\dots v_{n-1}}\diag i_{v_{1}}\times\dots\times i_{v_{n-1}}$.

Hence we have the short exact sequence:
\[1\to \prod_{v\ne i}\inn(v)\times\faktor{(i_{v_{1}}\times\dots i_{v_{n-1}})}{Z(i)}\to \stab(A)\to \prod_{v\ne i}\out(v)\times\out(i)\to 1\]
We deduce from this that $\stab(A)$ is generated by $\Phi=\prod_{a=1}^{n}\aut(G_{a})$ and groups $G_{i_{a}}\cong G_{i}$ for $a\in\left\{1,\dots,n\right\}-\left\{i\right\}$,
so that $[G_{i_{a}},G_{i_{b}}]=1$ (each element of $G_{i_{a}}$ commutes with each element of $G_{i_{b}}$) for every $a,b\in\left\{1,\dots,n\right\}-\left\{i\right\}$,
and where $G_{i_{v_{1}\dots v_{n-1}}}\diag G_{i_{v_{1}}}\times\dots\times G_{i_{v_{n-1}}}$ is the diagonal subgroup, we have $\faktor{G_{i_{v_{1}\dots v_{n-1}}}}{Z(G_{i})}\\ \noindent =\inn(G_{i})$.
We observe that for $a,b\ne i$, $[\aut(G_{a}),G_{i_{b}}]=1$.
However, for $a\ne i$, $\left(\{G_{a}\},g_{i}\right)\in G_{i_{a}}$, and $\varphi_{i}\in\aut(G_{i})$, we have $\left(\{G_{a}\},g_{i}\right)\varphi_{i}=\varphi_{i}\left(\{G_{a}\},\varphi_{i}(g_{i})\right)$.
These relations on the given generators are enough to fully determine $\stab(A)$ as a subgroup of $\out(G)$.

\item Alternatively, we can calculate $\stab(A)$ from Definition \ref{defn equivalent labellings} by considering equivalent labellings on $A$ and the automorphisms which lead to these.
This corresponds to `twisting outwards from $i$'.
Note that in this method, we implicitly `fix' the `basepoint' $i$.
Recall that $A$ has labelling $\hat{G}:=\left(G_{1},\dots,G_{n}\right)$, and let $\hat{H}:=\left(G_{1}^{h_{1}},\dots,G_{n}^{h_{n}}\right)$ be an arbitrary labelling in the equivalence class of $\hat{G}$ of labellings of $A$.

Observe that as in \ref{inner equivalent} of Example \ref{eg equivalent B labellings}, we can apply the inner automorphism $\iota_{h_{i}^{-1}}\in\inn(G)$ which conjugates each element of $G$ by $h_{i}^{-1}$.
Thus we obtain the labelling $\hat{H}':=\left(G_{i},G_{v_{1}}^{h_{v_{1}}h_{i}^{-1}},\dots,G_{v_{n-1}}^{h_{v_{n-1}}h_{i}^{-1}}\right)$ which (upto appropriate reordering) is equivalent to $\hat{H}$.

By Definition \ref{defn equivalent labellings}, we have that for each $a\in\left\{1,\dots,n-1\right\}$, $h_{v_{a}}h_{i}^{-1}\in G_{i}$.
Hence (aside from factor automorphisms) our only freedom in labellings is to conjugate each non-$G_{i}$ vertex group by an element of $G_{i}$, i.e. to `twist' along each of the edges incident to the vertex $i$ (with vertex group $G_{i}$) in $A$.

Thus $\stab(A)$ is generated by $\Phi$ and by $n-1$ groups $G_{i_{a}}\cong G_{i}$.
Note that this is exactly as determined above, and the same arguments can be made to determine relations, resulting in the same presentation for $\stab(A)$.
\end{enumerate}
\end{example}

While the Bass--Jiang approach deals with the removal of symmetry which occurs by making specific choices in the Guirardel--Levitt approach, we lose the ability to concisely write down stabilisers.
As such, we do not wish to replace the Guirardel--Levitt approach with this one, but rather enhance it.

\subsubsection*{Vertex Stabilisers for Common Use}

We will now detail presentations for the stabilisers of vertices in $\mathcal{D}_{n}$ (graphs of groups with structure listed in Table \ref{table n>=5 points} and labelling $(G_{1},\dots,G_{n})$) which will be useful throughout the paper, but especially in Sections \ref{presentation 5} and \ref{pairwise intersections}.
For brevity, we write $i_{j}$ for the group $G_{i_{j}}\cong G_{i}$ of Whitehead automorphisms which conjugate the factor group $G_{j}$ by elements of the factor group $G_{i}$.
Recall that there is an isomorphism $f_{i_{j}}:G_{i}\to i_{j}$ given by $f_{i_{j}}(g)=(\{G_{j}\},g^{-1})$.
We will consider the centre $Z(G_{i})$ of $G_{i}$ to be embedded in $i_{v_{1}\dots v_{k}}\diag i_{v_{1}}\times\dots\times i_{v_{k}}$ via the isomorphisms $f_{i_{j}}$.

We divide the vertices of $\mathcal{D}_{n}$ into three categories, according to which method(s) we will use to compute their stabilisers.
\smallskip

First, are vertices which have graph structure well-suited to the Guirardel--Levitt approach:

\begin{prop}\label{prop brown stabs 1} 
As subgroups of $\outs(G)$ we have:
\begin{itemize}
\item $\stab(\rho_{ij})=\stab(\alpha)=\Phi$
\item $\stab(\tau_{i,j,kl})=\stab(\beta_{i,j})=i_{j}\rtimes\Phi$
\item $\stab(\varepsilon_{i,j,k,l})=(i_{j}\times k_{l})\rtimes\Phi$
\end{itemize}
where the semidirect relation is given by $\left(\{G_{j}\},g_{i}\right)\circ\varphi=\varphi\circ\left(\{G_{j}\},\varphi(g)\right)$ for any $i,j$ with $\left(\{G_{j}\},g\right)\in i_{j}$ and $\varphi\in\Phi$.
In other words, $\varphi^{-1}f_{i_{j}}(g)\varphi=f_{i_{j}}(\varphi(g))$ for $g\in G_{i}$.
\end{prop}

\begin{proof}
These are lifted directly from Table \ref{table stabs G-L}, utilising Example \ref{eg stab sigma} which follows it to index the groups of automorphisms by the vertices they act on.
\end{proof}

Our second category is that of vertices whose graph structures are well-suited to the Bass--Jiang approach.
We write $i_{v_{1}\dots v_{k}}$ for the group of automorphims \\ \noindent $\left\{ \left( \{G_{v_{1}},\dots,G_{v_{k}}\},g \right) | g\in G_{i} \right\}$.

\begin{prop}\label{prop brown stabs 2} 
As subgroups of $\outs(G)$ we have:
\begin{itemize}
\item $\stab(\sigma_{i,jk,lm})$ is generated by $\faktor{(i_{jk}\times i_{lm}\times i_{v_{1}}\times\dots\times i_{v_{n-5}})}{Z(G_{i})}$ and $\Phi$
\item $\stab(\gamma_{i,jk})$ is generated by $\faktor{(i_{jk}\times i_{v_{1}}\times\dots\times i_{v_{n-3}})}{Z(G_{i})}$ and $\Phi$
\item $\stab(A_{i})$ is generated by $\faktor{(i_{v_{1}}\times\dots\times i_{v_{n-1}})}{Z(G_{i})}$ and $\Phi$
\end{itemize} 
each subject to the relations
$f_{i_{w_{1}}}(g)\dots f_{i_{w_{n-1}}}(g)=\ad_{G_{i}}(g)$ (with $\ad_{G_{i}}(g)$ as in Notation \ref{notation ad(g)}) and $\varphi^{-1}f_{i_{v}}(g)\varphi=f_{i_{v}}(\varphi(g))$,
where $\{w_{1},\dots,w_{n-1}\}=\{1,\dots,n\}-\{i\}$, $v\in\{w_{1},\dots,w_{n-1}\}$, and $\varphi\in\Phi$.
That is,
$\faktor{i_{w_{1}\dots w_{n-1}}}{Z(G_{i})}=\inn(G_{i})$ and
$i_{v}^{\varphi}=\varphi(i_{v})$.
\end{prop}

\begin{proof}
These are deduced from the $M_{T}$ terms of the Bass--Jiang short exact sequences listed in Table \ref{table stabs B-J}.
 Example \ref{eg stab A} explicitly details how to recover relations for $\stab(A_{i})$, and the others follow similarly.
\end{proof}

Finally, our third category is that of vertices whose graph structures are not well-suited to either approach, and we thus work directly from Definition \ref{defn equivalent labellings}:

\begin{prop}\label{prop brown stabs 3} 
As subgroups of $\outs(G)$ we have:
	\begin{itemize}
	\item $\stab(\delta_{i,j,k,lm})$ is generated by $j_{k}$, $\faktor{(i_{jk}\times i_{lm}\times i_{v_{1}}\times\dots\times i_{v_{n-5}})}{Z(G_{i})}$, and $\Phi$, subject to the relations:
		\begin{enumerate}
		\item $\faktor{i_{j k l m v_{1} \dots v_{n-5}}}{Z(G_{i})}=\inn(G_{i})$
		\item $\left[j_{k}, \faktor{(i_{jk}\times i_{lm}\times i_{v_{1}}\times\dots\times i_{v_{n-5}})}{Z(G_{i})}\right]=1$
		\item $j_{k}^{\varphi}=\varphi(j_{k})$ and $i_{x}^{\varphi}=\varphi(x)$ for each $x\in\{jk, lm, v_{1}, \dots, v_{n-5}\}$
		\end{enumerate}
	\item $\stab(B_{i,j,k})$ is generated by $j_{k}$, $\faktor{(i_{jk}\times i_{v_{1}}\times\dots\times i_{v_{n-3}})}{Z(G_{i})}$, and $\Phi$, subject to the relations:
		\begin{enumerate}
		\item $\faktor{i_{j k v_{1} \dots v_{n-3}}}{Z(G_{i})}=\inn(G_{i})$
		\item $\left[j_{k} , \faktor{(i_{jk}\times i_{v_{1}}\times\dots\times i_{v_{n-3}})}{Z(G_{i})}\right]=1$
		\item $j_{k}^{\varphi}=\varphi(j_{k})$ and $i_{x}^{\varphi}=\varphi(x)$ for each $x\in\{jk, v_{1}, \dots, v_{n-3}\}$
		\end{enumerate}
	\item $\stab(C_{i,j,k,l,m})$ is generated by $j_{k}$, $l_{m}$, $\faktor{(i_{jk}\times i_{lm}\times i_{v_{1}}\times\dots\times i_{v_{n-5}})}{Z(G_{i})}$, and $\Phi$, subject to the relations:
		\begin{enumerate}
		\item $\faktor{i_{j k l m v_{1} \dots v_{n-5}}}{Z(G_{i})}=\inn(G_{i})$
		\item $\left[a_{b} , \faktor{(i_{jk}\times i_{lm}\times i_{v_{1}}\times\dots\times i_{v_{n-5}})}{Z(G_{i})}\right]=1$ for each of $a_{b}=j_{k}$ and $a_{b}= l_{m}$
		\item $j_{k}^{\varphi}=\varphi(j_{k})$, $l_{m}^{\varphi}=\varphi(l_{m})$, and $i_{x}^{\varphi}=\varphi(x)$ for each $x\in\{jk, lm, v_{1}, \dots, v_{n-5}\}$
		\item $\left[j_{k} , l_{m}\right]=1$
		\end{enumerate}
	\end{itemize} 
\end{prop}

\begin{proof}
We calculate these using Definition \ref{defn equivalent labellings}.
The equivalence class of labellings for $B_{i,j,k}$ is explicitly described in Example \ref{eg equivalent B labellings}.
The classes for $\delta_{i,j,k,lm}$ and $C_{i,j,k,l,m}$ follow similarly.
From here, deduction of the $\outs(G)$ stabilisers follows in much the same way as in the second part of Example \ref{eg stab A}.
\end{proof}

\section{Properties of the Fundamental Domain $\mathcal{D}_{n}$}\label{section fun dom connectivity}

Before proving any structural statements about the fundamental domain $\mathcal{D}_{n}$, we provide some illustrations to aid in understanding this subcomplex.
We describe some of the substructures found within the 1-skeleton $\mathcal{D}_{n}^{(1)}$ of the fundamental domain of $\mathcal{C}_{n}$.
These will be particularly useful in determining edge inclusion relations in Theorem \ref{thm n>=5 presentation}, as well as in showing that $\mathcal{D}_{n}$ is simply connected in Section \ref{section fun dom sc}.

The subcomplex of the fundamental domain obtained by restricting to collapses of a given graph \Taleph{$n-4$}{$i$}{$j$}{$k$}{$l$} will be referred to as a `spike' of the fundamental domain.

\begin{figure}[h]
\centering
\begin{tikzpicture} 
\draw[Bittersweet,fill] (0,-2.9) circle [radius=0.075]; 
\draw[Bittersweet,fill] (0.3,-2.95) circle [radius=0.075]; 
\draw[Bittersweet,fill] (-0.3,-2.95) circle [radius=0.075]; 
\draw[thick] (0,-2) -- (-0.45,-3.8);
\draw[thick] (0,-2) -- (0.45,-3.8);
\draw[thick] (0,-2) -- (-0.7,-3.3);
\draw[thick] (0,-2) -- (0.7,-3.3);
\draw[gray] (0,-2) -- (-0.3,-2.95);
\draw[gray] (0,-2) -- (0,-2.9);
\draw[gray] (0,-2) -- (0.3,-2.95);
\draw[Red,fill] (0,-2) circle [radius=0.2]; 
\draw[Bittersweet,fill] (-0.45,-3.8) circle [radius=0.2]; 
\draw[Bittersweet,fill] (0.45,-3.8) circle [radius=0.2]; 
\draw[Bittersweet,fill] (-0.7,-3.3) circle [radius=0.15]; 
\draw[Bittersweet,fill] (0.7,-3.3) circle [radius=0.15]; 
\node at (0,-2) {\textcolor{white}{$\alpha$}}; 
\node at (-0.45,-3.8) {\textcolor{white}{$i$}}; 
\node at (0.45,-3.8) {\textcolor{white}{$j$}}; 
\node at (-0.7,-3.3) {\textcolor{white}{{\small$1$}}}; 
\node at (0.7,-3.3) {\textcolor{white}{{\footnotesize$n_{2}$}}}; 
\end{tikzpicture}
\caption{The $\alpha$-$A$--Star}
\label{alpha A star}
\end{figure}
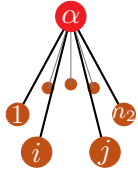
Figure \ref{alpha A star} shows the $\alpha$-$A$--Star. The circle with label `$i$' represents the vertex $A_{i}$. The circle with label `$n_{2}$' represents the vertex $A_{v_{n-2}}$.
Where the 3 small dots at the back are, one should imagine ``many'', that is, that there are in fact $n-4$ $A$-vertices there.
This structure appears precisely once in the fundamental domain, and is present in every `spike' (of which there are $\frac{n!}{8(n-4)!}$).

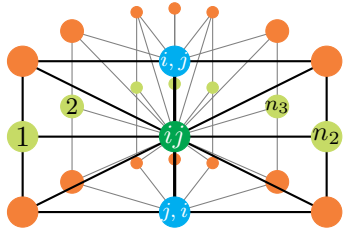
\begin{figure}[h]
\centering
\begin{tikzpicture} 
\draw[thin,gray] (0,0) -- (0,0.7);
\draw[thin,gray] (0,0) -- (-0.5,0.65);
\draw[thin,gray] (0,0) -- (0.5,0.65);
\draw[thin,gray] (0,0) -- (0,1.7);
\draw[thin,gray] (0,0) -- (-0.5,1.65);
\draw[thin,gray] (0,0) -- (0.5,1.65);
\draw[thin,gray] (0,0) -- (0,-0.3);
\draw[thin,gray] (0,0) -- (-0.5,-0.35);
\draw[thin,gray] (0,0) -- (0.5,-0.35);
\draw[thin,gray] (0,0.7) -- (0,1.7);
\draw[thin,gray] (0,0.7) -- (0,-0.3);
\draw[thin,gray] (0.5,0.65) -- (0.5,1.65);
\draw[thin,gray] (0.5,0.65) -- (0.5,-0.35);
\draw[thin,gray] (-0.5,0.65) -- (-0.5,1.65);
\draw[thin,gray] (-0.5,0.65) -- (-0.5,-0.35);
\draw[thin,gray] (0,1) -- (0,1.7);
\draw[thin,gray] (0,1) -- (-0.5,1.65);
\draw[thin,gray] (0,1) -- (0.5,1.65);
\draw[thin,gray] (0,-1) -- (0,-0.3);
\draw[thin,gray] (0,-1) -- (-0.5,-0.35);
\draw[thin,gray] (0,-1) -- (0.5,-0.35);
\draw[gray] (0,0) -- (-1.35,0.4);
\draw[gray] (0,0) -- (1.35,0.4);
\draw[gray] (0,0) -- (-1.35,1.4);
\draw[gray] (0,0) -- (1.35,1.4);
\draw[gray] (0,0) -- (-1.35,-0.6);
\draw[gray] (0,0) -- (1.35,-0.6);
\draw[gray] (-1.35,0.4) -- (-1.35,1.4);
\draw[gray] (-1.35,0.4) -- (-1.35,-0.6);
\draw[gray] (1.35,0.4) -- (1.35,1.4);
\draw[gray] (1.35,0.4) -- (1.35,-0.6);
\draw[gray] (0,1) -- (-1.35,1.4);
\draw[gray] (0,1) -- (1.35,1.4);
\draw[gray] (0,-1) -- (-1.35,-0.6);
\draw[gray] (0,-1) -- (1.35,-0.6);
\draw[SpringGreen,fill] (0,0.7) circle [radius=0.075]; 
\draw[SpringGreen,fill] (-0.5,0.65) circle [radius=0.075]; 
\draw[SpringGreen,fill] (0.5,0.65) circle [radius=0.075]; 
\draw[Orange,fill] (-1.35,-0.6) circle [radius=0.15]; 
\draw[Orange,fill] (1.35,-0.6) circle [radius=0.15]; 
\draw[Orange,fill] (0,-0.3) circle [radius=0.075]; 
\draw[Orange,fill] (-0.5,-0.35) circle [radius=0.075]; 
\draw[Orange,fill] (0.5,-0.35) circle [radius=0.075]; 
\draw[very thick] (0,0) -- (0,1);
\draw[very thick] (0,0) -- (0,-1);
\draw[thick] (0,0) -- (-2,0);
\draw[thick] (0,0) -- (2,0);
\draw[thick] (0,0) -- (-2,1);
\draw[thick] (0,0) -- (-2,-1);
\draw[thick] (0,0) -- (2,1);
\draw[thick] (0,0) -- (2,-1);
\draw[thick] (0,1) -- (-2,1);
\draw[thick] (0,1) -- (2,1);
\draw[thick] (0,-1) -- (-2,-1);
\draw[thick] (0,-1) -- (2,-1);
\draw[thick] (-2,0) -- (-2,1);
\draw[thick] (2,0) -- (2,1);
\draw[thick] (-2,0) -- (-2,-1);
\draw[thick] (2,0) -- (2,-1);
\draw[Green,fill] (0,0) circle [radius=0.2]; 
\draw[Cyan,fill] (0,1) circle [radius=0.2]; 
\draw[Cyan,fill] (0,-1) circle [radius=0.2]; 
\draw[SpringGreen,fill] (-2,0) circle [radius=0.2]; 
\draw[SpringGreen,fill] (2,0) circle [radius=0.2]; 
\draw[Orange,fill] (-2,1) circle [radius=0.2]; 
\draw[Orange,fill] (-2,-1) circle [radius=0.2]; 
\draw[Orange,fill] (2,1) circle [radius=0.2]; 
\draw[Orange,fill] (2,-1) circle [radius=0.2]; 
\draw[SpringGreen,fill] (-1.35,0.4) circle [radius=0.15]; 
\draw[SpringGreen,fill] (1.35,0.4) circle [radius=0.15]; 
\draw[Orange,fill] (-1.35,1.4) circle [radius=0.15]; 
\draw[Orange,fill] (1.35,1.4) circle [radius=0.15]; 
\draw[Orange,fill] (0,1.7) circle [radius=0.075]; 
\draw[Orange,fill] (-0.5,1.65) circle [radius=0.075]; 
\draw[Orange,fill] (0.5,1.65) circle [radius=0.075]; 
\node at (0,0) {{\small\textcolor{white}{$ij$}}}; 
\node at (0,1) {\scriptsize\textcolor{white}{$i,j$}}; 
\node at (0,-1) {\scriptsize\textcolor{white}{$j,i$}}; 
\node at (-2,0) {$1$}; 
\node at (2,0) {{\small$n_{2}$}}; 
\node at (-1.35,0.4) {\footnotesize$2$}; 
\node at (1.35,0.4) {\scriptsize$n_{3}$}; 
\end{tikzpicture}
\caption{The $\rho$--Book}
\label{rho book}
\end{figure}
Figure \ref{rho book} shows a $\rho$--Book (\textbf{the} $\rho$--Book associated to the graph $\rho_{ij}$). The circle with label `$1$' represents the vertex $\gamma_{v_{1},ij}$. The circle with label `$ij$' represents the vertex $\rho_{ij}$, and the circle with label `$i,j$' represents the vertex $\beta_{i,j}$.
The orange circle adjacent to both $\beta_{i,j}$ and $\gamma_{v_{1},ij}$ represents $B_{v_{1},i,j}$.
This structure appears $\frac{n(n-1)}{2}$ times in the fundamental domain. There are two $\rho$--Books per `spike', and each $\rho$--Book appears in $\frac{(n-2)(n-3)}{2}$ spikes.
Two distinct $\rho$--Books appear in only one spike together, and only if they are associated to $\rho_{ij}$ and $\rho_{kl}$ where $i,j,k,l$ are all distinct.

\begin{figure}[h]
\centering
\begin{tikzpicture} 
\draw[gray] (-2,-1) -- (-1,1) -- (2,1);
\draw[gray] (0,-2) -- (-1.333,-2.5);
\draw[gray] (0,-2) -- (-0.333,-2.5);
\draw[gray] (0,-2) -- (0.667,-2.5);
\draw[gray] (0,-2) -- (1,-2);
\draw[gray] (0,-2) -- (1.333,-1.5);
\draw[gray] (0,-2) -- (0.333,-1.5);
\draw[gray] (0,-2) -- (-0.667,-1.5);
\draw[gray] (0,-2) -- (-1,-2);
\draw[gray] (-1.333,-2.5) -- (-0.667,-1.5) -- (1.333,-1.5);
\draw[gray] (-1,2) -- (-1.5,0) -- (-1,-2);
\draw[gray] (-0.667,2.5) -- (-1,1) -- (-0.667,-1.5);
\draw[gray] (0.333,2.5) -- (0.5,1) -- (0.333,-1.5);
\draw[gray] (-2,-1) -- (-1,2) -- (-1,1) -- (-1,-2) -- cycle;
\draw[gray] (-1,1) -- (0.333,2.5) -- (2,1) -- (0.333,-1.5) -- cycle;
\draw[Lavender,fill] (-1.5,0) circle [radius=0.2]; 
\draw[Blue,fill] (-1,1) circle [radius=0.15]; 
\draw[Lavender,fill] (0.5,1) circle [radius=0.15]; 
\draw[Goldenrod,fill] (0,-2) circle [radius=0.175]; 
\draw[Dandelion,fill] (0.333,-1.5) circle [radius=0.125]; 
\draw[Purple,fill] (-0.667,-1.5) circle [radius=0.125]; 
\draw[Dandelion,fill] (-1,-2) circle [radius=0.175]; 
\draw[Goldenrod,fill] (0,1.25) circle [radius=0.125];
\draw[Goldenrod,fill] (0,0.5) circle [radius=0.075];
\draw[Goldenrod,fill] (0,0) circle [radius=0.075];
\draw[Goldenrod,fill] (0,-0.5) circle [radius=0.075];
\draw[Goldenrod,fill] (0,-1.25) circle [radius=0.125];
\draw[thick] (0,2) -- (-1.333,1.5);
\draw[thick] (0,2) -- (-0.333,1.5);
\draw[thick] (0,2) -- (0.667,1.5);
\draw[thick] (0,2) -- (1,2);
\draw[thick] (0,2) -- (1.333,2.5);
\draw[thick] (0,2) -- (0.333,2.5);
\draw[thick] (0,2) -- (-0.667,2.5);
\draw[thick] (0,2) -- (-1,2);
\draw[thick] (-1.333,1.5) -- (0.667,1.5) -- (1.333,2.5) -- (-0.667,2.5) -- cycle;
\draw[thick] (-2,-1) -- (1,-1) -- (2,1);
\draw[thick] (-1.333,-2.5) -- (0.667,-2.5) -- (1.333,-1.5);
\draw[thick] (-1.333,1.5) -- (-2,-1) -- (-1.333,-2.5);
\draw[thick] (-0.333,1.5) -- (-0.5,-1) -- (-0.333,-2.5);
\draw[thick] (0.667,1.5) -- (1,-1) -- (0.667,-2.5);
\draw[thick] (1,2) -- (1.5,0) -- (1,-2);
\draw[thick] (1.333,2.5) -- (2,1) -- (1.333,-1.5);
\draw[thick] (-2,-1) -- (-0.333,1.5) -- (1,-1) -- (-0.333,-2.5) -- cycle;
\draw[thick] (1,-1) -- (1,2) -- (2,1) -- (1,-2) -- cycle;
\draw[Goldenrod,fill] (0,2) circle [radius=0.175]; 
\draw[Purple,fill] (-1.333,1.5) circle [radius=0.225]; 
\draw[Dandelion,fill] (-0.333,1.5) circle [radius=0.225]; 
\draw[Purple,fill] (0.667,1.5) circle [radius=0.225]; 
\draw[Dandelion,fill] (1,2) circle [radius=0.175]; 
\draw[Purple,fill] (1.333,2.5) circle [radius=0.125]; 
\draw[Dandelion,fill] (0.333,2.5) circle [radius=0.125]; 
\draw[Purple,fill] (-0.667,2.5) circle [radius=0.125]; 
\draw[Dandelion,fill] (-1,2) circle [radius=0.175]; 
\draw[Blue,fill] (-2,-1) circle [radius=0.25]; 
\draw[Lavender,fill] (-0.5,-1) circle [radius=0.25]; 
\draw[Blue,fill] (1,-1) circle [radius=0.25]; 
\draw[Lavender,fill] (1.5,0) circle [radius=0.2]; 
\draw[Blue,fill] (2,1) circle [radius=0.15]; 
\draw[Purple,fill] (-1.333,-2.5) circle [radius=0.225]; 
\draw[Dandelion,fill] (-0.333,-2.5) circle [radius=0.225]; 
\draw[Purple,fill] (0.667,-2.5) circle [radius=0.225]; 
\draw[Dandelion,fill] (1,-2) circle [radius=0.175]; 
\draw[Purple,fill] (1.333,-1.5) circle [radius=0.125]; 
\node at (-2,-1) {\small\textcolor{white}{$ij$}}; 
\node at (-0.5,-1) {\tiny$ijkl$}; 
\node at (1,-1) {\small\textcolor{white}{$kl$}}; 
\node at (2,1) {\tiny\textcolor{white}{$ji$}}; 
\node at (-1,1) {\tiny\textcolor{white}{$lk$}}; 
\node at (0,2) {\small$1$}; 
\node at (0,1.25) {\scriptsize$2$}; 
\node at (0,-1.25) {\tiny$n_{5}$}; 
\node at (0,-2) {\scriptsize$n_{4}$}; 
\end{tikzpicture}
\caption{The $\tau$-$\varepsilon$--Box}
\label{tau epsilon box}
\end{figure}
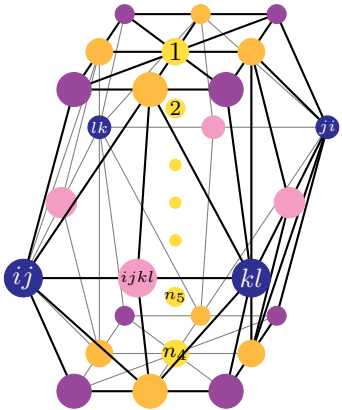
Figure \ref{tau epsilon box} shows a $\tau$-$\varepsilon$--Box.
It appears precisely once in each spike, and is unique to its spike.
It has $n-4$ layers, each associated to a $\sigma$-vertex. A given layer will be called a $\sigma$--Slice.
The cycle left when removing all $\sigma$--Slices is called a $\tau$-$\varepsilon$--Square.
The yellow circle with label `$n_{5}$' represents $\sigma_{v_{n-5},ij,kl}$.
The blue circle with label `$ij$' represents $\tau_{i,j,kl}$, and the pink circle with label `$ijkl$' represents $\varepsilon_{i,j,k,l}$.

\subsection{Connectedness of the Fundamental Domain} \label{section connected}

We first show that the fundamental domain $\mathcal{D}_{n}$ is (path) connected.
We will do this by finding a path from an arbitrary vertex $T\in\mathcal{D}_{n}$ to the vertex $\alpha\in\mathcal{D}_{n}$ (see Table \ref{table n>=5 points}).
Then any two arbitrary vertices of $\mathcal{D}_{n}$ will be connected via $\alpha$.

\begin{lemma}\label{lemma fun dom pc}
The fundamental domain $\mathcal{D}_{n}$ of $\mathcal{C}_{n}$ is path connected.
\end{lemma}

\begin{proof}
We refer to Table \ref{table n>=5 points} for the naming convention of vertices in $\mathcal{D}_{n}$ (where each vertex group is precisely one of the factor groups $G_{1},\dots,G_{n}$).
Note that $\alpha$ is adjacent to every $\rho_{ij}$ and every $A_{k}$ in $\mathcal{D}_{n}$. 
Every $\beta$ and $\gamma$ graph is a collapse of some $\rho$ graph, hence any $\beta$ or $\gamma$ vertex in $\mathcal{D}_{n}$ is path connected to $\alpha$.
Additionally, every $B$ graph is the collapse of some $\beta$ graph, so $B$ vertices are also connected.
Note that $\sigma_{i,jk,lm}$ and $\tau_{i,j,kl}$ both collapse to $A_{i}$ (which is adjacent to $\alpha$).
Any $C$ graph is the collapse of some $\sigma$ graph, and every $\delta$ and $\varepsilon$ graph is the collapse of some $\tau$ graph.
Hence any vertex in $\mathcal{D}_{n}$ has a path in the fundamental domain to $\alpha$.
Thus the fundamental domain is (path) connected.
Explicit paths are listed in Table \ref{paths in fun dom}.
\begin{table}[h]
\centering
\begin{tabular}{ | c | }								\hline
$\alpha$									\\	\hline
$\rho_{ij}\dash \alpha$							\\	\hline
$A_{i}\dash \alpha$							\\	\hline
$\beta_{i,j}\dash \rho_{ij}\dash \alpha$				\\	\hline
$\gamma_{i,jk}\dash \rho_{jk}\dash \alpha$				\\	\hline
$B_{i,j,k}\dash \beta_{j,k}\dash \rho_{jk}\dash \alpha$		\\	\hline
$\sigma_{i,jk,lm}\dash A_{i}\dash \alpha$				\\	\hline
$\tau_{i,j,kl}\dash A_{i}\dash \alpha$					\\	\hline
$C_{i,j,k,l,m}\dash \sigma_{i,jk,lm}\dash A_{i}\dash \alpha$	\\	\hline
$\delta_{i,j,k,lm}\dash \tau_{j,k,lm}\dash A_{j}\dash \alpha$	\\	\hline
$\varepsilon_{i,j,k,l}\dash \tau_{i,j,kl}\dash A_{i}\dash \alpha$	\\	\hline
\end{tabular}
\caption{Paths Between Vertices in the Fundamental Domain}
\label{paths in fun dom}
\end{table}
\end{proof}

\begin{cor}\label{cor copies of the fun dom are pc}
Any vertex $T$ in the complex $\mathcal{C}_{n}$ is connected via an edge path to some $\alpha$ graph.
\end{cor}

\begin{proof}
This follows by noting that $T$ sits in at least one copy of the fundamental domain, and that the action of $\outs(G)$ preserves adjacency, so the above argument applies.
\end{proof}

In \cite[Section 4]{Gilbert1987}, Gilbert gives a summary of Fouxe-Rabinovitch's presentation for $\aut(G)$ described in \cite{F-R1940} and \cite{F-R1941}.
Restricting to the pure symmetric automorphisms $\aut_{\mathfrak{S}}(G)$ of a splitting $\mathfrak{S}$ of $G$, this states that 
$\aut_{\mathfrak{S}}(G)$ is generated by factor automorphisms and Whitehead automorphisms only (see Definitions \ref{defn factor autos} and \ref{defn whitehead autos}).
We can use this to give a quick proof that our full complex $\mathcal{C}_{n}$ is path connected.

\begin{prop}\label{prop alpha graphs are path connected} 
Any two $\alpha$-graphs in the complex $\mathcal{C}_{n}$ are connected via a path which travels only via $\alpha$ and $A$ graphs.
\end{prop}

\begin{proof}
Let $\alpha_{0}$ be the $\alpha$-graph in the fundamental domain $\mathcal{D}_{n}$ and let $\alpha_{0}\cdot\hat{\psi}$ be an arbitrary $\alpha$-graph in $\mathcal{C}_{n}$ (with $\hat{\psi}\in\outs(G)$ and $\psi\in\aut_{\mathfrak{S}}(G)$ a representative for $\hat{\psi}$).
By \cite{F-R1940} and \cite{F-R1941}, transcribed in \cite[Section 4]{Gilbert1987}, we can write $\psi$ as $\psi_{0}\psi_{1}\dots\psi_{m}$ for some $m\in\mathbb{N}$ where $\psi_{0}\in\Phi$ is a factor automorphism, and for each $1\le i\le m$, $\psi_{i}$ is a Whitehead automorphism of the form $(S_{i},x_{i})$ with $x_{i}\in G_{j_{i}}$ for some $j_{i}\in\{1,\dots,n\}$ and $S_{i}\subseteq\{G_{1},\dots,G_{n}\}-\{G_{j_{i}}\}$.

By Proposition \ref{prop brown stabs 1}, $\psi_{0}\in\stab(\alpha_{0})$, that is, $\alpha_{0}\cdot\psi_{0}=\alpha_{0}$.
We now write $\psi_{1}\dots\psi_{m}=\psi_{m}(\psi_{m-1}^{\psi_{m}})(\psi_{m-2}^{\psi_{m-1}\psi_{m}})\dots(\psi_{2}^{\psi_{3}\dots\psi_{m}})(\psi_{1}^{\psi_{2}\dots\psi_{m}})$.
Observe that for each $1\le i<m$, $\psi_{i}^{\psi_{i+1}\dots\psi_{m}}$ acts on $\alpha_{0}\cdot\psi_{i+1}\dots\psi_{m}$ to produce the graph $\alpha_{0}\cdot\psi_{i}\dots\psi_{m}$ (and $\psi_{m}$ acts on $\alpha_{0}=\alpha_{0}\cdot\psi_{0}$ producing $\alpha_{0}\cdot\psi_{m}$).
Moreover, if $\psi_{i}$ has operating factor $G_{j_{i}}$ and $A_{j_{i}}$ is the $A$-graph in $\mathcal{D}_{n}$ whose central vertex has stabiliser $G_{j_{i}}$, then by Proposition \ref{prop brown stabs 2}, $\psi_{i}\in\stab(A_{j_{i}})$ and thus $\psi_{i}^{\psi_{i+1}\dots\psi_{m}}\in\stab(A_{j_{i}}\cdot\psi_{i+1}\dots\psi_{m})$ (and $\psi_{m}\in\stab(A_{j_{m}})$).
Thus both the graphs $\alpha_{0}\cdot\psi_{i}\dots\psi_{m}$ and $\alpha_{0}\cdot\psi_{i+1}\dots\psi_{m}$ collapse to the graph $A_{j_{i}}\cdot\psi_{i+1}\dots\psi_{m}$.

We therefore have a path $\alpha_{0} \dash A_{i_{m}} \dash \alpha_{0}\cdot\psi_{m} \dash A_{i_{m-1}}\cdot\psi_{m} \dash \alpha_{0}\cdot\psi_{m-1}\psi_{m} \dash \dots \dash \\ \noindent \alpha_{0}\cdot\psi_{2}\dots\psi_{m} \dash A_{j_{1}}\cdot\psi_{2}\dots\psi_{m} \dash \alpha_{0}\cdot\psi_{1}\dots\psi_{m}=\alpha_{0}\cdot\psi_{0}\psi_{1}\dots\psi_{m}=\alpha_{0}\cdot\psi$, as required.
\end{proof}

We give an alternative proof of this in \cite{Iveson2025}, which does not rely on already having a presentation for $\aut_{\mathfrak{S}}(G)$ or $\outs(G)$, and rather uses the geometry of $\mathcal{C}_{n}$.

\begin{cor}\label{cor Cn pc}
The complex $\mathcal{C}_{n}$ is (path) connected.
\end{cor}

\begin{proof}
This follows immediately by combining Corollary \ref{cor copies of the fun dom are pc} with Proposition \ref{prop alpha graphs are path connected}.
\end{proof}


\subsection{Simple Connectivity of the Fundamental Domain} \label{section fun dom sc}

We will now show that the fundamental domain $\mathcal{D}_{n}$ of the space $\mathcal{C}_{n}$ (and hence each $\outs(G)$-image of $\mathcal{D}_{n}$) is simply connected.
This will be the main result of this section, and is given as Theorem \ref{thm fun dom sc}.

We will consider nested subcomplexes of $\mathcal{D}_{n}$, adding `types' of 0-cell at each stage.
We will show that the first of these subcomplexes is simply connected, and then apply a corollary of the Seifert--van Kampen Theorem to see that each successive subcomplex is also simply connected.

\begin{cor}\label{cor gluing 2-cells}
Let $X$ and $Y$ both be simply connected (simplicial) complexes.
If we (suitably\footnote{i.e. so that $X\cup Y$ is still a simplicial complex}) glue $X$ and $Y$ together along a path connected collection of edges, then $X\cup Y$ is simply connected.
\end{cor}

\begin{proof}
Note that in the complex $X\cup Y$, the subset $X\cap Y$ is precisely the collection of edges we have glued along.
Since this is stipulated to be path connected, then by the Seifert--van Kampen Theorem (Theorem \ref{s van k}), we have
$\pi_{1}(X\cup Y)\cong\pi_{1}(X)\ast_{\pi_{1}(X\cap Y)}\pi_{1}(Y)=\{1\}\ast_{\pi_{1}(X\cap Y)}\{1\}=\{1\}$.
\end{proof}

Recall that we describe a vertex $[T]$ of $\mathcal{C}_{n}$ as `collapsing' to another vertex $[S]$ if a graph represented by $[T]$ has an edge (or edges) which can be collapsed to form a graph associated to $[S]$. That is, the vertices $[T]$ and $[S]$ are adjacent in $\mathcal{C}_{n}$.

\begin{defn}\label{defn DnT}
Let $\mathcal{T}$ be a subset of the graph structures shown in Table \ref{table n>=5 points}.
Denote by $\mathcal{D}_{n}[\mathcal{T}]$ the subcomplex of the fundamental domain of $\mathcal{C}_{n}$ obtained by restricting to simplices whose 0-cells are those associated to graph structures in $\mathcal{T}$.
\end{defn}

\begin{lemma}
$\mathcal{D}_{n}[\{\alpha,\rho,A\}]$ is simply connected.
\end{lemma}

\begin{proof}
Note that any $\rho$-vertex \Trho{$n-2$}{$i$}{$j$} collapses to $\alpha$, \Talpha{$n$} (by collapsing the edge whose endpoints are both trivial), and that $\alpha$ in turn collapses to any $A$-vertex \TA{$n-1$}{$k$} (including $k=i$ or $k=j$).
According to how we constructed the space $\mathcal{C}_{n}$, this means that for each pair $(\rho, A)$ we have a $2$-cell $[\rho, \alpha, A]$.
So the subcomplex $\mathcal{D}_{n}[\{\alpha,\rho,A\}]$ comprises these 2-cells, glued along `matching' edges.
Given a particular $\rho_{ij}$-graph, we have a cone on a star at $\alpha$ (where the leaves of the star are the various $A$ graphs).
As we vary $\rho$, we get copies of this cone, all glued along the star formed by the $\alpha$ and $A$ vertices.
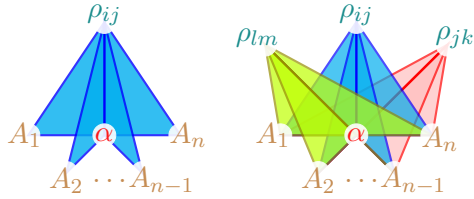
\begin{center}
\begin{tikzpicture} 
\transparent{0.8}
\filldraw[blue,  thick, fill=cyan] (0,0) -- (1,0) -- (0,1.5) -- cycle;
\filldraw[blue,  thick, fill=cyan] (0,0) -- (-1,0) -- (0,1.5) -- cycle;
\filldraw[blue,  thick, fill=cyan] (0,0) -- (0.5,-0.5) -- (0,1.5) -- cycle;
\filldraw[blue,  thick, fill=cyan] (0,0) -- (-0.5,-0.5) -- (0,1.5) -- cycle;
\transparent{0.9}
\draw[white,fill] (0,1.5) circle [radius=0.15cm];
\node[teal] at (0,1.6) {$\rho_{ij}$}; 
\draw[white,fill] (0,0) circle [radius=0.15cm];
\node[red] at (0,0) {$\alpha$}; 
\draw[white,fill] (-0.5,-0.5) circle [radius=0.15cm];
\node[brown] at (-0.5,-0.6) {$A_{2}$}; 
\draw[white,fill] (-1,0) circle [radius=0.15cm];
\node[brown] at (-1.05,0) {$A_{1}$}; 
\draw[white,fill] (1,0) circle [radius=0.15cm];
\node[brown] at (1.1,0) {$A_{n}$}; 
\draw[white,fill] (0.5,-0.5) circle [radius=0.15cm];
\node[brown] at (0.75,-0.6) {$A_{n-1}$}; 
\node[brown] at (0.1,-0.6) {$\dots$};
\end{tikzpicture}
\begin{tikzpicture} 
\transparent{0.7}
\filldraw[red,  thick, fill=pink] (0,0) -- (1,0) -- (1.2,1.2) -- cycle;
\filldraw[red,  thick, fill=pink] (0,0) -- (-1,0) -- (1.2,1.2) -- cycle;
\filldraw[red,  thick, fill=pink] (0,0) -- (0.5,-0.5) -- (1.2,1.2) -- cycle;
\filldraw[red,  thick, fill=pink] (0,0) -- (-0.5,-0.5) -- (1.2,1.2) -- cycle;
\transparent{0.9}
\draw[white,fill] (1.2,1.2) circle [radius=0.15cm];
\node[teal] at (1.3,1.3) {$\rho_{jk}$}; 
\transparent{0.7}
\filldraw[blue,  thick, fill=cyan] (0,0) -- (1,0) -- (0,1.5) -- cycle;
\filldraw[blue,  thick, fill=cyan] (0,0) -- (-1,0) -- (0,1.5) -- cycle;
\filldraw[blue,  thick, fill=cyan] (0,0) -- (0.5,-0.5) -- (0,1.5) -- cycle;
\filldraw[blue,  thick, fill=cyan] (0,0) -- (-0.5,-0.5) -- (0,1.5) -- cycle;
\transparent{0.9}
\draw[white,fill] (0,1.5) circle [radius=0.15cm];
\node[teal] at (0,1.6) {$\rho_{ij}$}; 
\transparent{0.7}
\filldraw[olive,  thick, fill=lime] (0,0) -- (1,0) -- (-1.2,1.2) -- cycle;
\filldraw[olive,  thick, fill=lime] (0,0) -- (-1,0) -- (-1.2,1.2) -- cycle;
\filldraw[olive,  thick, fill=lime] (0,0) -- (0.5,-0.5) -- (-1.2,1.2) -- cycle;
\filldraw[olive,  thick, fill=lime] (0,0) -- (-0.5,-0.5) -- (-1.2,1.2) -- cycle;
\transparent{0.9}
\draw[white,fill] (-1.2,1.2) circle [radius=0.15cm];
\node[teal] at (-1.3,1.3) {$\rho_{lm}$}; 
\draw[white,fill] (0,0) circle [radius=0.15cm];
\node[red] at (0,0) {$\alpha$}; 
\draw[white,fill] (-0.5,-0.5) circle [radius=0.15cm];
\node[brown] at (-0.5,-0.6) {$A_{2}$}; 
\draw[white,fill] (-1,0) circle [radius=0.15cm];
\node[brown] at (-1.1,0) {$A_{1}$}; 
\draw[white,fill] (1,0) circle [radius=0.15cm];
\node[brown] at (1.1,-0.05) {$A_{n}$}; 
\draw[white,fill] (0.5,-0.5) circle [radius=0.15cm];
\node[brown] at (0.75,-0.6) {$A_{n-1}$}; 
\node[brown] at (0.1,-0.6) {$\dots$};
\end{tikzpicture}
\end{center}
Clearly each cone is simply connected.
Since the intersection of these cones is (the star based at $\alpha$ with leaves the $A$ vertices, which is) path connected, we can iteratively apply Corollary \ref{cor gluing 2-cells} to `add in' each cone (of which there are finitely many).
Hence the structure $\mathcal{D}_{n}[\{\alpha,\rho,A\}]$ is simply connected.
\end{proof}

\begin{lemma}
$\mathcal{D}_{n}[\{\alpha,\rho,A,\beta,\gamma\}]$ is simply connected.
\end{lemma}

\begin{proof}
By the previous lemma, we have that $\pi_{1}(\mathcal{D}_{n}[\{\alpha,\rho,A\}])=\{1\}$.
We will (iteratively) apply Corollary \ref{cor gluing 2-cells} to $\mathcal{D}_{n}[\{\alpha,\rho,A,\beta,\gamma\}]$ taking $X$ to be $\mathcal{D}_{n}[\{\alpha,\rho,A\}]$ (or the union of this with successive $Y$'s) and $Y$ to be the neighbourhood in $\mathcal{D}_{n}[\{\alpha,\rho,A,\beta,\gamma\}]$ of $\beta$ or $\gamma$.

Each $\rho$ graph  collapses to $2$ $\beta$ graphs, and $n-2$ $\gamma$ graphs. These $\beta$s and $\gamma$s are unique to the given $\rho$ (that is, two distinct $\rho$ graphs cannot both collapse to the same $\beta$ or $\gamma$.
Specifically, $\rho_{ij}$ collapses to $\beta_{i,j}$, $\beta_{j,i}$ and $\gamma_{v,ij}$ for $v\ne i,j$.
In turn, $\beta_{i,j}$ collapses to $A_{i}$, $\beta_{j,i}$ to $A_{j}$, and $\gamma_{v,ij}$ to $A_{v}$.
Thus the neighbourhood of $\beta_{i,j}$ (or $\gamma_{k,ij}$) is a 2-cell  $[\rho_{ij}, \beta_{i,j}, A_{i}]$ (or $[\rho_{ij}, \gamma_{k,ij}, A_{k}]$, respectively).
Note that any 2-cell is simply connected.
The intersection of each of these neighbourhoods with any of the spaces $X$ is an edge $\rho\dash A$ (which is, in particular, path connected).
So by Corollary \ref{cor gluing 2-cells}, $\mathcal{D}_{n}[\{\alpha,\rho,A,\beta,\gamma\}]$ is simply connected.
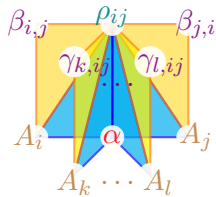
\begin{center}
\begin{tikzpicture} 
\transparent{0.7}
\filldraw[blue,  thick, fill=cyan] (0,0) -- (1,0) -- (0,1.5) -- cycle;
\filldraw[blue,  thick, fill=cyan] (0,0) -- (-1,0) -- (0,1.5) -- cycle;
\filldraw[blue,  thick, fill=cyan] (0,0) -- (0.5,-0.5) -- (0,1.5) -- cycle;
\filldraw[blue,  thick, fill=cyan] (0,0) -- (-0.5,-0.5) -- (0,1.5) -- cycle;
\transparent{0.7}
\filldraw[orange, thick, fill=Goldenrod] (1,0) -- (1,1.5) -- (0,1.5) -- cycle;
\filldraw[orange, thick, fill=Goldenrod] (-1,0) -- (-1,1.5) -- (0,1.5) -- cycle;
\transparent{0.9}
\draw[white,fill] (1,1.5) circle [radius=0.15cm];
\node[violet] at (1.1,1.5) {$\beta_{j,i}$}; 
\draw[white,fill] (-1,1.5) circle [radius=0.15cm];
\node[violet] at (-1.1,1.5) {$\beta_{i,j}$}; 
\transparent{0.7}
\filldraw[orange, thick, fill=yellow] (0.5,-0.5) -- (0.5,1) -- (0,1.5) -- cycle;
\filldraw[orange, thick, fill=yellow] (-0.5,-0.5) -- (0,1.5) -- (-0.5,1) -- cycle;
\transparent{0.9}
\draw[white,fill] (-0.5,1) circle [radius=0.2cm];
\node[violet] at (-0.35,0.95) {$\gamma_{k,ij}$}; 
\draw[white,fill] (0.5,1) circle [radius=0.2cm];
\node[violet] at (0.65,0.95) {$\gamma_{l,ij}$}; 
\transparent{0.9}
\draw[white,fill] (0,1.5) circle [radius=0.15cm];
\node[teal] at (0,1.6) {$\rho_{ij}$}; 
\draw[white,fill] (0,0) circle [radius=0.15cm];
\node[red] at (0,0) {$\alpha$}; 
\draw[white,fill] (-0.5,-0.5) circle [radius=0.15cm];
\node[brown] at (-0.5,-0.6) {$A_{k}$}; 
\draw[white,fill] (-1,0) circle [radius=0.15cm];
\node[brown] at (-1.1,0) {$A_{i}$}; 
\draw[white,fill] (1,0) circle [radius=0.15cm];
\node[brown] at (1.1,0) {$A_{j}$}; 
\draw[white,fill] (0.5,-0.5) circle [radius=0.15cm];
\node[brown] at (0.6,-0.6) {$A_{l}$}; 
\transparent{1}
\node[brown] at (0.1,-0.6) {$\dots$};
\node[violet] at (0.1,0.7) {$\dots$};
\end{tikzpicture}
\end{center}
\end{proof}

We now consider how to attach the $B$-vertices. 

\begin{lemma}
$\mathcal{D}_{n}[\{\alpha,\rho,A,\beta,\gamma,B\}]$ is simply connected.
\end{lemma}

\begin{proof}
Each $B_{i,j,k}$ \TB{$n-3$}{$i$}{$j$}{$k$} grows to a unique $\beta_{j,k}$ \Tbeta{$n-2$}{$j$}{$k$}, a unique $\gamma_{k,ij}$ \Tgamma{$n-3$}{$k$}{$i$}{$j$}, and a unique $\rho_{jk}$ \Trho{$n-2$}{$j$}{$k$}.
So $\mathcal{D}_{n}[\{\alpha,\rho,A,\beta,\gamma,B\}]$ is the subcomplex $\mathcal{D}_{n}[\{\alpha,\rho,A,\beta,\gamma\}]$with the addition of $2$-cells $[\rho_{jk}, \beta_{j,k}, B_{i,j,k}]$ and $[\rho_{jk}, \gamma_{i,jk}, B_{i,j,k}]$ for all possible values of $i$, $j$ and $k$ (by gluing along the $\rho$-$\beta$ and $\rho$-$\gamma$ edges that already exist, and additionally gluing our two 2-cells along the $\rho$-$\beta$ edge they both share):
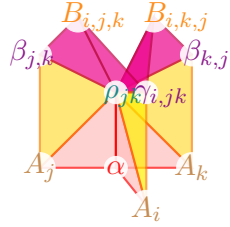
\begin{center}
\begin{tikzpicture} 
\transparent{0.7}
\filldraw[red, thick, fill=pink] (0,0) -- (1,0) -- (0,1) -- cycle;
\filldraw[red, thick, fill=pink] (0,0) -- (-1,0) -- (0,1) -- cycle;
\filldraw[orange, thick, fill=Goldenrod] (1,1.5) -- (1,0) -- (0,1) -- cycle;
\filldraw[orange, thick, fill=Goldenrod] (-1,1.5) -- (-1,0) -- (0,1) -- cycle;
\filldraw[purple, thick, fill=magenta] (0,1) -- (1,1.5) -- (0.5,2) -- cycle;
\filldraw[purple, thick, fill=magenta] (0,1) -- (-1,1.5) -- (-0.5,2) -- cycle;
\filldraw[purple, thick, fill=magenta] (0,1) -- (0.4,1) -- (0.5,2) -- cycle;
\filldraw[purple, thick, fill=magenta] (0,1) -- (0.4,1) -- (-0.5,2) -- cycle;
\filldraw[red, thick, fill=pink] (0,0) -- (0.4,-0.5) -- (0,1) -- cycle;
\filldraw[orange, thick, fill=yellow] (0,1) -- (0.4,-0.5) -- (0.4,1) -- cycle;
\transparent{0.9}
\draw[white,fill] (0,0) circle [radius=0.15cm];
\node[red] at (0,0) {$\alpha$}; 
\draw[white,fill] (-1,0) circle [radius=0.2cm];
\node[brown] at (-1,0) {$A_{j}$}; 
\draw[white,fill] (1,0) circle [radius=0.2cm];
\node[brown] at (1,0) {$A_{k}$}; 
\draw[white,fill] (0.4,-0.5) circle [radius=0.2cm];
\node[brown] at (0.4,-0.55) {$A_{i}$}; 
\draw[white,fill] (-1,1.5) circle [radius=0.15cm];
\node[violet] at (-1.1,1.5) {$\beta_{j,k}$}; 
\draw[white,fill] (1,1.5) circle [radius=0.15cm];
\node[violet] at (1.2,1.5) {$\beta_{k,j}$}; 
\draw[white,fill] (0.4,1) circle [radius=0.15cm];
\node[violet] at (0.6,0.95) {$\gamma_{i,jk}$}; 
\draw[white,fill] (0,1) circle [radius=0.15cm];
\node[teal] at (0.1,0.95) {$\rho_{jk}$}; 
\draw[white,fill] (0.5,2) circle [radius=0.2cm];
\node[orange] at (0.8,2) {$B_{i,k,j}$}; 
\draw[white,fill] (-0.5,2) circle [radius=0.2cm];
\node[orange] at (-0.3,2) {$B_{i,j,k}$}; 
\end{tikzpicture}
\end{center}

\noindent That is, given a specific $\rho$-vertex in our structure, for every $\beta$ and $\gamma$ we see adjacent to said $\rho$, we glue in a `fin'
\begin{tikzpicture}
\transparent{0.7}
\filldraw[purple, thick, fill=magenta] (0,1) -- (-1,0) -- (0,-1) -- (1,0) -- cycle;
\draw[thick,purple] (0,1) -- (0,-1);
\draw[very thick,dotted,blue] (-1,0) -- (0,-1) -- (1,0);
\node[orange] at (0,1.2) {$B$};
\node[violet] at (-1.2,0) {$\beta$};
\node[violet] at (1.2,0) {$\gamma$};
\node[teal] at (0,-1.2) {$\rho$};
\end{tikzpicture}
along the dotted line.
The edge path $\beta\dash \rho\dash \gamma$ is path connected, so by repeated applications of Corollary \ref{cor gluing 2-cells}, we see that  $\mathcal{D}_{n}[\{\alpha,\rho,A,\beta,\gamma,B\}]$ is simply connected.
\end{proof}

\begin{lemma}
$\mathcal{D}_{n}[\{\alpha,\rho,A,\beta,\gamma,B,\sigma\}]$ is simply connected.
\end{lemma}

\begin{proof}
Given a vertex $\sigma_{i,jk,lm}$ in $\mathcal{D}_{n}[\{\alpha,\rho,A,\beta,\gamma,B,\sigma\}]$, with associated graph \\ \Tsigma{$n-5$}{$i$}{$j$}{$k$}{$l$}{$m$}, we see there is precisely one $A$-vertex it is adjacent to, \TA{$n-1$}{$i$}~.
We also see that $\sigma_{i,jk,lm}$ collapses to two $\gamma$-vertices, $\gamma_{i,jk}$ and $\gamma_{i,lm}$.
Further, $\gamma_{i,jk}$ collapses to both $B_{i,j,k}$ and $B_{i,k,j}$ (similarly for $\gamma_{i,lm}$).
To create $\mathcal{D}_{n}[\{\alpha,\rho,A,\beta,\gamma,B,\sigma\}]$, in $\mathcal{D}_{n}[\{\alpha,\rho,A,\beta,\gamma,B\}]$ at any $A_{i}$-vertex we attach two 2-cells $[\sigma_{i,jk,lm},\gamma_{i,jk}, A_{i}]$ and $[\sigma_{i,jk,lm},\gamma_{i,lm}, A_{i}]$ (glued to each other along the shared edge $A_{i}$-$\sigma_{i,jk,lm}$) wherever we see two vertices $\gamma_{i,jk}$ and $\gamma_{i,lm}$ adjacent to $A_{i}$ with $j$, $k$, $l$ and $m$ (and $i$) distinct.
We then glue in additional 2-cells of the form $[\sigma,\gamma,B]$ wherever we see a path $\sigma$--$\gamma$--$B$.

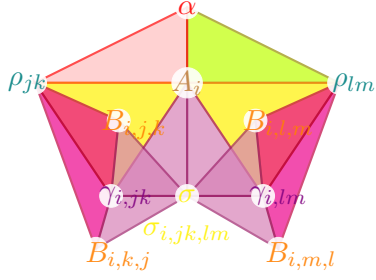
\begin{center}
\begin{tikzpicture} 
\transparent{0.7}
\filldraw[olive, thick, fill=lime] (0,0) -- (0,1) -- (2,0) -- cycle; 
\filldraw[red, thick, fill=pink] (0,0) -- (0,1) -- (-2,0) -- cycle; 
\filldraw[orange, thick, fill=yellow] (0,0) -- (1,-1.5) -- (2,0) -- cycle; 
\filldraw[orange, thick, fill=yellow] (0,0) -- (-1,-1.5) -- (-2,0) -- cycle; 
\filldraw[purple, thick, fill=magenta] (0.9,-0.5) -- (1,-1.5) -- (2,0) -- cycle;  
\filldraw[purple, thick, fill=magenta] (-0.9,-0.5) -- (-1,-1.5) -- (-2,0) -- cycle; 
\filldraw[purple, thick, fill=magenta] (1.2,-2.2) -- (1,-1.5) -- (2,0) -- cycle; 
\filldraw[purple, thick, fill=magenta] (-1.2,-2.2) -- (-1,-1.5) -- (-2,0) -- cycle; 
\filldraw[RedViolet, thick, fill=Thistle] (0,0) -- (0,-1.5) -- (1,-1.5) -- cycle;
\filldraw[RedViolet, thick, fill=Thistle] (0,0) -- (0,-1.5) -- (-1,-1.5) -- cycle;
\filldraw[RedViolet, thick, fill=Thistle] (1.2,-2.2) -- (0,-1.5) -- (1,-1.5) -- cycle;
\filldraw[RedViolet, thick, fill=Thistle] (-1.2,-2.2) -- (0,-1.5) -- (-1,-1.5) -- cycle;
\filldraw[RedViolet, thick, fill=Thistle] (0.9,-0.5) -- (0,-1.5) -- (1,-1.5) -- cycle;
\filldraw[RedViolet, thick, fill=Thistle] (-0.9,-0.5) -- (0,-1.5) -- (-1,-1.5) -- cycle;
\transparent{0.9}
\draw[white,fill] (0,0) circle [radius=0.2cm];
\node[brown] at (0,0) {$A_{i}$}; 
\draw[white,fill] (0,1) circle [radius=0.15cm];
\node[red] at (0,1) {$\alpha$}; 
\draw[white,fill] (-2,0) circle [radius=0.15cm];
\node[teal] at (-2.1,0) {$\rho_{jk}$}; 
\draw[white,fill] (2,0) circle [radius=0.15cm];
\node[teal] at (2.2,0) {$\rho_{lm}$}; 
\draw[white,fill] (1,-1.5) circle [radius=0.15cm];
\node[violet] at (1.2,-1.52) {$\gamma_{i,lm}$}; 
\draw[white,fill] (-1,-1.5) circle [radius=0.15cm];
\node[violet] at (-0.8,-1.52) {$\gamma_{i,jk}$}; 
\draw[white,fill] (0,-1.5) circle [radius=0.15cm];
\node[yellow] at (0,-2) {$\sigma_{i,jk,lm}$}; 
\node[yellow] at (0,-1.5) {$\sigma$};
\draw[white,fill] (-0.9,-0.5) circle [radius=0.15cm];
\node[orange] at (-0.7,-0.55) {$B_{i,j,k}$}; 
\draw[white,fill] (0.9,-0.5) circle [radius=0.15cm];
\node[orange] at (1.2,-0.55) {$B_{i,l,m}$}; 
\draw[white,fill] (-1.2,-2.2) circle [radius=0.15cm];
\node[orange] at (-0.9,-2.3) {$B_{i,k,j}$}; 
\draw[white,fill] (1.2,-2.2) circle [radius=0.15cm];
\node[orange] at (1.5,-2.3) {$B_{i,m,l}$}; 
\end{tikzpicture}
\end{center}
That is, each $\sigma_{i,jk,lm}$ has a simply connected neighbourhood, and the intersection of this neighbourhood with $\mathcal{D}_{n}[\{\alpha,\rho,A,\beta,\gamma,B\}]$ is the collection of edges $\gamma_{i,jk}\dash B_{i,j,k}$, $\gamma_{i,jk}\dash B_{i,k,j}$, $\gamma_{i,lm}\dash B_{i,l,m}$, $\gamma_{i,lm}\dash B_{i,m,l}$, $\gamma_{i,jk}\dash A_{i}$, and $\gamma_{i,lm}\dash A_{i}$, which is path conected.
The result follows from Corollary \ref{cor gluing 2-cells}.
\end{proof}

The process of adding in $\tau$-vertices to our structure will be very similar to that for $\sigma$-vertices.

\begin{lemma}
$\mathcal{D}_{n}[\{\alpha,\rho,A,\beta,\gamma,B,\sigma,\tau\}]$ is simply connected.
\end{lemma}

\begin{proof}
The vertex $\tau_{i,l,jk}$ in $\mathcal{D}_{n}[\{\alpha,\rho,A,\beta,\gamma,B,\sigma,\tau\}]$ with graph \\ \noindent \Ttau{$n-4$}{$i$}{$l$}{$j$}{$k$} collapses to $\gamma_{i,jk}$ \Tgamma{$n-3$}{$i$}{$j$}{$k$} and $\beta_{i,l}$ \Tbeta{$n-2$}{$i$}{$l$}.
In turn, $\beta_{i,l}$ and $\gamma_{i,jk}$ both collapse to $A_{i}$.
Additionally $\gamma_{i,jk}$ collapses to $B_{i,j,k}$ and $B_{i,k,j}$, and $\beta_{i,l}$ collapses to $n-2$ vertices of the form $B_{v,i,l}$ (for $v\ne i,l$).
and in our structure so far, wherever we see a path $\gamma_{i,jk}$--$A_{i}$--$\beta_{i,l}$ with $i$, $j$, $k$, $l$ distinct, we glue in a pair of 2-cells of the form $[\tau,\gamma,A]$ and $[\tau,\beta,A]$ (glued together along their common edge $A$-$\tau$).
As before with $\sigma$, we must also glue in all possible 2-cells of the form $[\tau,\gamma,B]$ and $[\tau,\beta,B]$ as determined by the relative collapses of $\gamma$ and $\beta$.
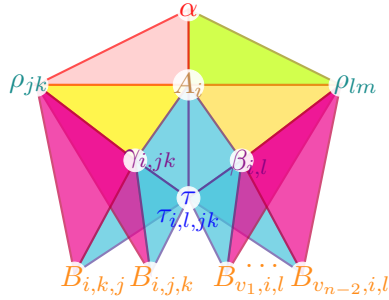
\begin{center}
\begin{tikzpicture} 
\transparent{0.7}
\filldraw[olive, thick, fill=lime] (0,0) -- (0,1) -- (2,0) -- cycle;
\filldraw[red, thick, fill=pink] (0,0) -- (0,1) -- (-2,0) -- cycle;
\filldraw[orange, thick, fill=Goldenrod] (0,0) -- (0.7,-1) -- (2,0) -- cycle;
\filldraw[orange, thick, fill=yellow] (0,0) -- (-0.7,-1) -- (-2,0) -- cycle;
\filldraw[Violet, thick, fill=SkyBlue] (1.5,-2.5) -- (0,-1.5) -- (0.7,-1) -- cycle;
\filldraw[Violet, thick, fill=SkyBlue] (-1.5,-2.5) -- (0,-1.5) -- (-0.7,-1) -- cycle;
\filldraw[purple, thick, fill=magenta] (0.5,-2.5) -- (0.7,-1) -- (2,0) -- cycle;
\filldraw[purple, thick, fill=magenta] (-0.5,-2.5) -- (-0.7,-1) -- (-2,0) -- cycle;
\filldraw[Violet, thick, fill=SkyBlue] (0,0) -- (0,-1.5) -- (0.7,-1) -- cycle;
\filldraw[Violet, thick, fill=SkyBlue] (0,0) -- (0,-1.5) -- (-0.7,-1) -- cycle;
\filldraw[purple, thick, fill=magenta] (1.5,-2.5) -- (0.7,-1) -- (2,0) -- cycle;
\filldraw[purple, thick, fill=magenta] (-1.5,-2.5) -- (-0.7,-1) -- (-2,0) -- cycle;
\filldraw[Violet, thick, fill=SkyBlue] (0.5,-2.5) -- (0,-1.5) -- (0.7,-1) -- cycle;
\filldraw[Violet, thick, fill=SkyBlue] (-0.5,-2.5) -- (0,-1.5) -- (-0.7,-1) -- cycle;
\transparent{0.9}
\draw[white,fill] (0,0) circle [radius=0.2cm];
\node[brown] at (0,0) {$A_{i}$}; 
\draw[white,fill] (0,1) circle [radius=0.15cm];
\node[red] at (0,1) {$\alpha$}; 
\draw[white,fill] (-2,0) circle [radius=0.15cm];
\node[teal] at (-2.1,0) {$\rho_{jk}$}; 
\draw[white,fill] (2,0) circle [radius=0.15cm];
\node[teal] at (2.2,0) {$\rho_{lm}$}; 
\draw[white,fill] (0.7,-1) circle [radius=0.15cm];
\node[violet] at (0.8,-1.05) {$\beta_{i,l}$}; 
\draw[white,fill] (-0.7,-1) circle [radius=0.15cm];
\node[violet] at (-0.5,-1) {$\gamma_{i,jk}$}; 
\draw[white,fill] (0,-1.5) circle [radius=0.15cm];
\node[blue] at (0,-1.8) {$\tau_{i,l,jk}$}; 
\node[blue] at (0,-1.5) {$\tau$};
\draw[white,fill] (-0.5,-2.5) circle [radius=0.15cm];
\node[orange] at (-0.35,-2.6) {$B_{i,j,k}$}; 
\draw[white,fill] (0.5,-2.5) circle [radius=0.15cm];
\node[orange] at (0.8,-2.6) {$B_{v_{1},i,l}$}; 
\draw[white,fill] (-1.5,-2.5) circle [radius=0.15cm];
\node[orange] at (-1.25,-2.6) {$B_{i,k,j}$}; 
\draw[white,fill] (1.5,-2.5) circle [radius=0.15cm];
\node[orange] at (2,-2.6) {$B_{v_{n-2},i,l}$}; 
\node[orange] at (1,-2.4) {$\dots$};
\end{tikzpicture}
\end{center}
We have then identified the neighbourhood of $\tau_{i,l,jk}$ inside $\mathcal{D}_{n}[\{\alpha,\rho,A,\beta,\gamma,B,\sigma,\tau\}]$, and found the intersection of this neighbourhood with $\mathcal{D}_{n}[\{\alpha,\rho,A,\beta,\gamma,B,\sigma\}]$.
By the previous lemma, $\mathcal{D}_{n}[\{\alpha,\rho,A,\beta,\gamma,B,\sigma\}]$ is simply connected, and clearly the neighbourhood of $\tau$ (shown in pale blue) is simply connected.
Moreover, the intersection of these subsets is the collection of edges $\gamma_{i,jk}\dash A_{i}$, $\gamma_{i,jk}\dash B_{i,j,k}$, $\gamma_{i,jk}\dash B_{i,k,j}$, $\beta_{i,l}\dash A_{i}$, $\beta_{i,l}\dash B_{v,i,l}$ (for $v\ne i,j,k,l$).
Since this is path connected, then by Corollary \ref{cor gluing 2-cells}, $\mathcal{D}_{n}[\{\alpha,\rho,A,\beta,\gamma,B,\sigma,\tau\}]$ is simply connected.
\end{proof}

\begin{lemma}
$\mathcal{D}_{n}[\{\alpha,\rho,A,\beta,\gamma,B,\sigma,\tau,\delta\}]$ is simply connected.
\end{lemma}

\begin{proof}
Each $\delta$ in $\mathcal{D}_{n}[\{\alpha,\rho,A,\beta,\gamma,B,\sigma,\tau,\delta\}]$ collapses to a unique $B$ and `grows' to a unique $\sigma$ and a unique $\tau$.\\ \noindent
Precisely, $\delta_{i,j,k,lm}$ \Tdelta{$n-5$}{$i$}{$j$}{$k$}{$l$}{$m$} collapses to $B_{i,j,k}$ \TB{$n-3$}{$i$}{$j$}{$k$} and grows to $\sigma_{i,jk,lm}$ \Tsigma{$n-5$}{$i$}{$j$}{$k$}{$l$}{$m$} and $\tau_{j,k,lm}$ \Ttau{$n-4$}{$j$}{$k$}{$l$}{$m$}.
As such, we see that a given $\sigma_{i,jk,lm}$ and $\tau_{j,k,lm}$ are both adjacent to $B_{i,j,k}$ (and share no other common $B$ adjacency).
So whenever we see a path $\sigma_{i,jk,lm}$--$B_{i,j,k}$--$\tau_{j,k,lm}$ in our structure, we will glue in 2-cells $[\sigma,\delta,B]$ and $[\tau,\delta,B]$ (gluing them along their shared $\delta$-$B$ edge).
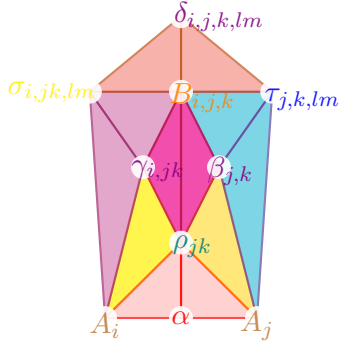
\begin{center}
\begin{tikzpicture} 
\transparent{0.7}
\filldraw[red, thick, fill=pink] (0,0) -- (0,1) -- (1,0) -- cycle;
\filldraw[red, thick, fill=pink] (0,0) -- (0,1) -- (-1,0) -- cycle;
\filldraw[orange, thick, fill=Goldenrod] (0,1) -- (0.5,2) -- (1,0) -- cycle;
\filldraw[orange, thick, fill=yellow] (0,1) -- (-0.5,2) -- (-1,0) -- cycle;
\filldraw[Violet, thick, fill=SkyBlue] (1.2,3) -- (0.5,2) -- (1,0) -- cycle;
\filldraw[RedViolet, thick, fill=Thistle] (-1.2,3) -- (-0.5,2) -- (-1,0) -- cycle;
\filldraw[Violet, thick, fill=SkyBlue] (1.2,3) -- (0.5,2) -- (0,3) -- cycle;
\filldraw[RedViolet, thick, fill=Thistle] (-1.2,3) -- (-0.5,2) -- (0,3) -- cycle;
\filldraw[purple, thick, fill=magenta] (0,1) -- (0.5,2) -- (0,3) -- cycle;
\filldraw[purple, thick, fill=magenta] (0,1) -- (-0.5,2) -- (0,3) -- cycle;
\filldraw[Bittersweet, thick, fill=Salmon] (0,4) -- (1.2,3) -- (0,3) -- cycle;
\filldraw[Bittersweet, thick, fill=Salmon] (0,4) -- (-1.2,3) -- (0,3) -- cycle;
\transparent{0.9}
\draw[white,fill] (0,0) circle [radius=0.15cm];
\draw[white,fill] (0,1) circle [radius=0.15cm];
\draw[white,fill] (1,0) circle [radius=0.15cm];
\draw[white,fill] (0.5,2) circle [radius=0.15cm];
\draw[white,fill] (-1,0) circle [radius=0.15cm];
\draw[white,fill] (-0.5,2) circle [radius=0.15cm];
\draw[white,fill] (1.2,3) circle [radius=0.15cm];
\draw[white,fill] (-1.2,3) circle [radius=0.15cm];
\draw[white,fill] (0,3) circle [radius=0.15cm];
\draw[white,fill] (0,4) circle [radius=0.15cm];
\node[red] at (0,0) {$\alpha$};
\node[brown] at (-1,-0.1) {$A_{i}$};
\node[brown] at (1,-0.1) {$A_{j}$};
\node[teal] at (0.15,0.95) {$\rho_{jk}$};
\node[violet] at (-0.3,1.95) {$\gamma_{i,jk}$};
\node[violet] at (0.65,1.95) {$\beta_{j,k}$};
\node[yellow] at (-1.7,3) {$\sigma_{i,jk,lm}$};
\node[blue] at (1.6,2.9) {$\tau_{j,k,lm}$};
\node[orange] at (0.28,2.92) {$B_{i,j,k}$};
\node[violet] at (0.5,4) {$\delta_{i,j,k,lm}$};
\end{tikzpicture}
\end{center}
That is, each $\delta_{ij,k,lm}$ has a neighbourhood in $\mathcal{D}_{n}[\{\alpha,\rho,A,\beta,\gamma,B,\sigma,\tau,\delta\}]$ comprising two 2-cells glued along a single edge.
The intersection of this neighbourhood with $\mathcal{D}_{n}[\{\alpha,\rho,A,\beta,\gamma,B,\sigma,\tau\}]$ is the edge path $\sigma_{i,jk,lm}\dash B_{i,j,k}\dash \tau_{j,k,lm}$.
So successive applications of Corollary \ref{cor gluing 2-cells} tells us that $\mathcal{D}_{n}[\{\alpha,\rho,A,\beta,\gamma,B,\sigma,\tau,\delta\}]$ is simply connected.
\end{proof}

\begin{lemma}
$\mathcal{D}_{n}[\{\alpha,\rho,A,\beta,\gamma,B,\sigma,\tau,\delta,C\}]$ is simply connected.
\end{lemma}

\begin{proof}
Each $C_{i,j,k,l,m}$ \TC{$n-5$}{$i$}{$j$}{$k$}{$l$}{$m$} grows to $\delta_{i,j,k,lm}$, $\delta_{i,l,m,jk}$, $\tau_{l,m,jk}$, $\tau_{j,k,lm}$, and $\sigma_{i,jk,lm}$.\\ \noindent
So the neighbourhood of $C_{i,j,k,l,m}$ inside $\mathcal{D}_{n}[\{\alpha,\rho,A,\beta,\gamma,B,\sigma,\tau,\delta,C\}]$ is four 2-cells $[\tau_{l,m,jk},\delta_{i,l,m,jk},C_{i,j,k,l,m}]$, $[\sigma_{i,jk,lm},\delta_{i,l,m,jk},C_{i,j,k,l,m}]$, \\ \noindent $[\tau_{j,k,lm},\delta_{i,j,k,lm},C_{i,j,k,l,m}]$, and $[\sigma_{i,jk,lm},\delta_{i,j,k,lm},C_{i,j,k,l,m}]$, glued along their common edges.\\ \noindent
The intersection of this neighbourhood with $\mathcal{D}_{n}[\{\alpha,\rho,A,\beta,\gamma,B,\sigma,\tau,\delta\}]$ is the edge path $\tau_{l,m,jk}$--$\delta_{i,l,m,jk}$--$\sigma_{i,jk,lm}$--$\delta_{i,j,k,lm}$--$\tau_{j,k,lm}$. \\ \noindent
By repeated applications of Corollary \ref{cor gluing 2-cells} (and by the previous lemma), \\ $\mathcal{D}_{n}[\{\alpha,\rho,A,\beta,\gamma,B,\sigma,\tau,\delta,C\}]$ is simply connected.
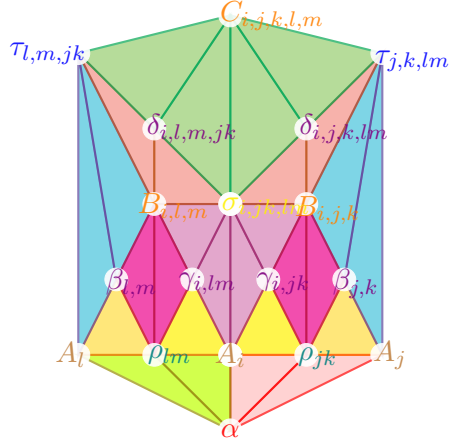
\begin{center}
\begin{tikzpicture} 
\transparent{0.7}
\filldraw[red, thick, fill=pink] (0,0) -- (1,1) -- (2,1) -- cycle;
\filldraw[red, thick, fill=pink] (0,0) -- (1,1) -- (0,1) -- cycle;
\filldraw[orange, thick, fill=Goldenrod] (1,1) -- (1.5,2) -- (2,1) -- cycle;
\filldraw[orange, thick, fill=yellow] (1,1) -- (0.5,2) -- (0,1) -- cycle;
\filldraw[Violet, thick, fill=SkyBlue] (2,5) -- (1.5,2) -- (2,1) -- cycle;
\filldraw[RedViolet, thick, fill=Thistle] (0,3) -- (0.5,2) -- (0,1) -- cycle;
\filldraw[Violet, thick, fill=SkyBlue] (2,5) -- (1.5,2) -- (1,3) -- cycle;
\filldraw[RedViolet, thick, fill=Thistle] (0,3) -- (0.5,2) -- (1,3) -- cycle;
\filldraw[purple, thick, fill=magenta] (1,1) -- (1.5,2) -- (1,3) -- cycle;
\filldraw[purple, thick, fill=magenta] (1,1) -- (0.5,2) -- (1,3) -- cycle;
\filldraw[Bittersweet, thick, fill=Salmon] (1,4) -- (2,5) -- (1,3) -- cycle;
\filldraw[Bittersweet, thick, fill=Salmon] (1,4) -- (0,3) -- (1,3) -- cycle;
\filldraw[olive, thick, fill=lime] (0,0) -- (-1,1) -- (-2,1) -- cycle;
\filldraw[olive, thick, fill=lime] (0,0) -- (-1,1) -- (0,1) -- cycle;
\filldraw[orange, thick, fill=Goldenrod] (-1,1) -- (-1.5,2) -- (-2,1) -- cycle;
\filldraw[orange, thick, fill=yellow] (-1,1) -- (-0.5,2) -- (0,1) -- cycle;
\filldraw[Violet, thick, fill=SkyBlue] (-2,5) -- (-1.5,2) -- (-2,1) -- cycle;
\filldraw[RedViolet, thick, fill=Thistle] (0,3) -- (-0.5,2) -- (0,1) -- cycle;
\filldraw[Violet, thick, fill=SkyBlue] (-2,5) -- (-1.5,2) -- (-1,3) -- cycle;
\filldraw[RedViolet, thick, fill=Thistle] (0,3) -- (-0.5,2) -- (-1,3) -- cycle;
\filldraw[purple, thick, fill=magenta] (-1,1) -- (-1.5,2) -- (-1,3) -- cycle;
\filldraw[purple, thick, fill=magenta] (-1,1) -- (-0.5,2) -- (-1,3) -- cycle;
\filldraw[Bittersweet, thick, fill=Salmon] (-1,4) -- (-2,5) -- (-1,3) -- cycle;
\filldraw[Bittersweet, thick, fill=Salmon] (-1,4) -- (0,3) -- (-1,3) -- cycle;
\filldraw[Green, thick, fill=YellowGreen] (0,5.5) -- (0,3) -- (1,4) -- cycle;
\filldraw[Green, thick, fill=YellowGreen] (0,5.5) -- (2,5) -- (1,4) -- cycle;
\filldraw[Green, thick, fill=YellowGreen] (0,5.5) -- (0,3) -- (-1,4) -- cycle;
\filldraw[Green, thick, fill=YellowGreen] (0,5.5) -- (-2,5) -- (-1,4) -- cycle;
\transparent{0.9}
\draw[white,fill] (0,0) circle [radius=0.15cm];
\draw[white,fill] (1,1) circle [radius=0.15cm];
\draw[white,fill] (2,1) circle [radius=0.15cm];
\draw[white,fill] (1.5,2) circle [radius=0.15cm];
\draw[white,fill] (0,1) circle [radius=0.15cm];
\draw[white,fill] (0.5,2) circle [radius=0.15cm];
\draw[white,fill] (2,5) circle [radius=0.15cm];
\draw[white,fill] (0,3) circle [radius=0.15cm];
\draw[white,fill] (1,3) circle [radius=0.15cm];
\draw[white,fill] (1,4) circle [radius=0.15cm];
\draw[white,fill] (-1,1) circle [radius=0.15cm];
\draw[white,fill] (-2,1) circle [radius=0.15cm];
\draw[white,fill] (-1.5,2) circle [radius=0.15cm];
\draw[white,fill] (0,1) circle [radius=0.15cm];
\draw[white,fill] (-0.5,2) circle [radius=0.15cm];
\draw[white,fill] (-2,5) circle [radius=0.15cm];
\draw[white,fill] (0,3) circle [radius=0.15cm];
\draw[white,fill] (-1,3) circle [radius=0.15cm];
\draw[white,fill] (-1,4) circle [radius=0.15cm];
\draw[white,fill] (0,5.5) circle [radius=0.15cm];
\node[red] at (0,0) {$\alpha$};
\node[brown] at (2.1,1) {$A_{j}$};
\node[teal] at (1.15,0.95) {$\rho_{jk}$};
\node[violet] at (0.7,1.95) {$\gamma_{i,jk}$};
\node[violet] at (1.65,1.95) {$\beta_{j,k}$};
\node[blue] at (2.4,4.9) {$\tau_{j,k,lm}$};
\node[orange] at (1.28,2.92) {$B_{i,j,k}$};
\node[violet] at (1.5,4) {$\delta_{i,j,k,lm}$};
\node[brown] at (0,1) {$A_{i}$};
\node[brown] at (-2.1,1) {$A_{l}$};
\node[teal] at (-0.8,1) {$\rho_{lm}$};
\node[violet] at (-0.3,1.95) {$\gamma_{i,lm}$};
\node[violet] at (-1.3,1.95) {$\beta_{l,m}$};
\node[yellow] at (0.47,2.95) {$\sigma_{i,jk,lm}$};
\node[blue] at (-2.4,5) {$\tau_{l,m,jk}$};
\node[orange] at (-0.75,2.95) {$B_{i,l,m}$};
\node[violet] at (-0.5,4) {$\delta_{i,l,m,jk}$};
\node[orange] at (0.55,5.5) {$C_{i,j,k,l,m}$};
\end{tikzpicture}
\end{center}
\end{proof}

Finally, we add $\varepsilon$-vertices to our complex.

\begin{lemma}\label{adding epsilons is sc}
$\mathcal{D}_{n}[\{\alpha,\rho,A,\beta,\gamma,B,\sigma,\tau,\delta,C,\varepsilon\}]$ is simply connected.
\end{lemma}

\begin{proof}
Note that the subcomplex neighbourhood around $\varepsilon_{j,k,l,m}$ in our fundamental domain (equal to $\mathcal{D}_{n}[\{\alpha,\rho,A,\beta,\gamma,B,\sigma,\tau,\delta,C,\varepsilon\}]$) is:
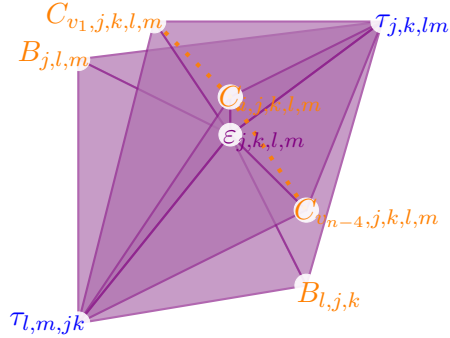
\begin{center}
\begin{tikzpicture} 
\transparent{0.7}
\filldraw[Plum, thick, fill=Orchid] (0,0.5) -- (-2,-2) -- (1,-1.5) -- cycle; 
\filldraw[Plum, thick, fill=Orchid] (0,0.5) -- (2,2) -- (1,-1.5) -- cycle; 
\filldraw[Plum, thick, fill=Orchid] (0,0.5) -- (-2,1.5) -- (2,2) -- cycle; 
\filldraw[Plum, thick, fill=Orchid] (0,0.5) -- (-2,-2) -- (-2,1.5) -- cycle; 
\transparent{0.9}
\draw[white,fill] (-2,1.5) circle [radius=0.15cm]; 
\draw[white,fill] (1,-1.5) circle [radius=0.15cm]; 
\transparent{1}
\node[orange] at (-2.32,1.55) {$B_{j,l,m}$}; 
\node[orange] at (1.3,-1.65) {$B_{l,j,k}$}; 
\transparent{0.7}
\filldraw[Plum, thick, fill=Orchid] (0,0.5) -- (-2,-2) -- (-1,2) -- cycle; 
\filldraw[Plum, thick, fill=Orchid] (0,0.5) -- (2,2) -- (-1,2) -- cycle; 
\filldraw[Plum, thick, fill=Orchid] (0,0.5) -- (-2,-2) -- (0,1) -- cycle; 
\filldraw[Plum, thick, fill=Orchid] (0,0.5) -- (2,2) -- (0,1) -- cycle; 
\filldraw[Plum, thick, fill=Orchid] (0,0.5) -- (-2,-2) -- (1,-0.5) -- cycle; 
\filldraw[Plum, thick, fill=Orchid] (0,0.5) -- (2,2) -- (1,-0.5) -- cycle; 
\transparent{0.9}
\draw[orange, loosely dotted, ultra thick] (-1,2) -- (0,1) -- (1,-0.5);
\draw[white,fill] (0,0.5) circle [radius=0.15cm]; 
\draw[white,fill] (0,1) circle [radius=0.16cm]; 
\draw[white,fill] (-1,2) circle [radius=0.15cm]; 
\draw[white,fill] (1,-0.5) circle [radius=0.16cm]; 
\draw[white,fill] (-2,-2) circle [radius=0.15cm]; 
\draw[white,fill] (2,2) circle [radius=0.15cm]; 
\transparent{1}
\node[violet] at (0.45,0.42) {$\varepsilon_{j,k,l,m}$};
\node[blue] at (2.4,1.95) {$\tau_{j,k,lm}$};
\node[blue] at (-2.4,-2) {$\tau_{l,m,jk}$};
\node[orange] at (-1.65,2.05) {$C_{v_{1},j,k,l,m}$}; 
\node[orange] at (0.53,0.95) {$C_{i,j,k,l,m}$}; 
\node[orange] at (1.8,-0.55) {$C_{v_{n-4},j,k,l,m}$}; 
\end{tikzpicture}
\end{center}

\noindent That is, $\varepsilon_{j,k,l,m}$ grows to $\tau_{j,k,lm}$ and $\tau_{l,m,jk}$ and collapses to $B_{j,l,m}$, $B_{l,j,k}$, and $n-4$ vertices of the form $C_{v,j,k,l,m}$ (for $v\ne j,k,l,m$).

The intersection of this neighbourhood with $\mathcal{D}_{n}[\{\alpha,\rho,A,\beta,\gamma,B,\sigma,\tau,\delta,C\}]$ is the boundary of the neighbourhood.
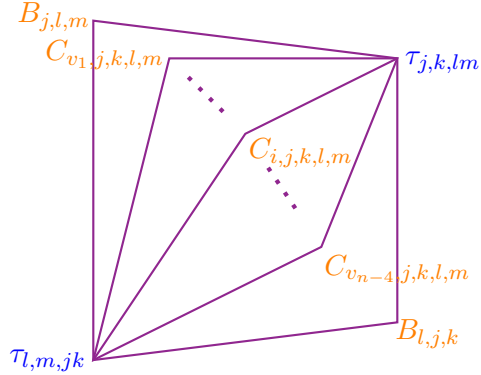
\begin{center}
\begin{tikzpicture} 
\draw[Plum,thick] (-2,-2) -- (-2,2.5) -- (2,2) -- (2,-1.5) -- cycle;
\draw[Plum,thick] (-2,-2) -- (-1,2) -- (2,2) -- (1,-0.5) -- cycle;
\draw[Plum,thick] (-2,-2) -- (0,1) -- (2,2);
\draw[Plum, loosely dotted, ultra thick] (-0.75,1.75) -- (-0.25,1.25);
\draw[Plum, loosely dotted, ultra thick] (0.3,0.55) -- (0.7,-0.05);
\node[blue] at (2.6,1.95) {$\tau_{j,k,lm}$};
\node[blue] at (-2.6,-2) {$\tau_{l,m,jk}$};
\node[orange] at (-1.85,2.05) {$C_{v_{1},j,k,l,m}$}; 
\node[orange] at (0.73,0.75) {$C_{i,j,k,l,m}$}; 
\node[orange] at (2,-0.8) {$C_{v_{n-4},j,k,l,m}$}; 
\node[orange] at (-2.52,2.55) {$B_{j,l,m}$}; 
\node[orange] at (2.4,-1.65) {$B_{l,j,k}$}; 
\end{tikzpicture}
\end{center}
Since this is path connected, Corollary \ref{cor gluing 2-cells} applies, and $\mathcal{D}_{n}[\{\alpha,\rho,A,\beta,\gamma,B,\sigma,\tau,\delta,C,\varepsilon\}]$ is simply connected.
\end{proof}

Finally, we have proved:

\begin{thm}\label{thm fun dom sc}
The fundamental domain $\mathcal{D}_{n}$ of the complex  $\mathcal{C}_{n}$ is simply connected.
\end{thm}

\begin{proof}
By Lemma \ref{adding epsilons is sc}, $\mathcal{D}_{n}[\{\alpha,\rho,A,\beta,\gamma,B,\sigma,\tau,\delta,C,\varepsilon\}]$ is simply connected.
But $\{\alpha,\rho,A,\beta,\gamma,B,\sigma,\tau,\delta,C,\varepsilon\}$ covers all of the graph structures in Table \ref{table n>=5 points}.\\
So by Definitions \ref{defn DnT} and \ref{defn Dn} $\mathcal{D}_{n}[\{\alpha,\rho,A,\beta,\gamma,B,\sigma,\tau,\delta,C,\varepsilon\}]$ is precisely the fundamental domain $\mathcal{D}_{n}$ of $\mathcal{C}_{n}$.
\end{proof}


\section{A Presentation for $\outs(G)$}\label{section presentation}

This section is the main result of the paper.

We recall from Section \ref{section background} the theorem of Brown \cite{Brown1984} which we will use to determine a presentation for $\outs(G)$ (where $G=G_{1}\ast\dots\ast G_{n}$, $\mathfrak{S}=(G_{1},\dots,G_{n})$, and $\mathcal{G}=\outs(G)$):

\brown*

The complex $X$ we will use is $\mathcal{C}_{n}$, and the subcomplex $W$ is $\mathcal{D}_{n}$.
Since $\mathcal{C}_{3}$ and $\mathcal{C}_{4}$ are just the barycentric spine of Guirardel and Levitt's Outer Space (for $n=3$ and $n=4$ respectively), results satisfying the restrictions on $X$ and $W$ are assumed from \cite{Guirardel2007}. 
It is highly non-trivial to show that $\mathcal{C}_{n}$ is simply connected for $n\ge5$, so we delay the proof of this to Sections \ref{section space of domains} and \ref{section peak reduction}.
The required result here is:

\begin{restatable*}{cor}{simplyconnected}\label{Cn is sc}
The space $\mathcal{C}_{n}$ (for $n\ge5$) is simply connected.
\end{restatable*}

That $\mathcal{D}_{n}$ for $n\ge5$ satisfies the strictness condition on $W$ is the result of Propositions \ref{prop Dn is fun dom} and \ref{prop strict fun dom}.

The isotropy subgroups $\mathcal{G}_{v}$ here are the vertex stabilisers $\stab(T)$ for $T\in\mathcal{D}_{n}^{(0)}$,
which are detailed in Propositions \ref{prop brown stabs 1}, \ref{prop brown stabs 2}, and \ref{prop brown stabs 3}.

Note that Lemma \ref{lem stabs include} implies that the edge relations in Brown's Theorem become:
`$g=\iota_{o(e)}^{-1}\left(\iota_{t(e)}(g)\right)$ for all $g\in \mathcal{G}_{o(e)}$'
(that is, that vertex stabilisers $\stab(T)$ are identified with their natural images under inclusion in the stabilisers $\stab(S)$ of any vertices $S$ to which the original vertex $T$ collapses).

\begin{notation}\label{notation for presentation}
We summarise the notational shorthand we have adopted thus far:
\begin{description}
\item[{$\left[A, B\right]=1$:}] For subgroups $A$ and $B$, `$A$ commutes with $B$', in the sense that for all $a\in A$ and for all $b\in B$ we have $ab=ba$.
\item[$G_{i_{j}}$:] The group of (outer) automorphisms which act by conjugating all elements of the factor group $G_{j}$ by an element of the factor group $G_{i}$.
\item[$f_{i_{j}}$:] The isomorphism $f_{i_{j}}:G_{i}\to G_{i_{j}}$ which maps an element $g\in G_{i}$ to the element in $G_{i_{j}}$ which conjugates each element of $G_{j}$ by $g^{-1}$.
\item[$G_{i_{v_{1}\dots v_{k}}}$:] The group $\left\{ \left. \left(f_{i_{v_{1}}}(g_{i}),\dots,f_{i_{v_{k}}}(g_{i})\right) \right\vert g_{i}\in G_{i}\right\}$ which is the diagonal subgroup of $f_{i_{v_{1}}}(G_{i})\times\dots\times f_{i_{v_{k}}}(G_{i})=G_{i_{v_{1}}}\times\dots\times G_{i_{v_{k}}}$. \\
	\noindent We often denote this by $G_{i_{v_{1}\dots v_{k}}}\diag G_{i_{v_{1}}}\times\dots\times G_{i_{v_{k}}}$.
\item[$Z(G_{i})$:] The centre of the group $G_{i}$, i.e. the subgroup $\left\{ g\in G_{i}  \left\vert \ gh=hg \ \forall h\in G_{i} \right. \right\}$.
We will often identify $Z(G_{i})$ with its images $f_{i_{j}}(Z(G_{i}))$.
\item[$\aut(G_{i})$:] Often considered to be the subgroup $\left\{ \left( 1,\dots,1,\aut(G_{i}), 1,\dots,1 \right) \right\}$ of $\Phi=\prod_{j=1}^{n}\aut(G_{j})$.
\item[$\ad_{G_{i}}(g)$:] The element of $\inn(G_{i})$ which conjugates each element of $G_{i}$ by $g$ (where $g\in G_{i}$).
\item[$G_{i_{j}}^{\varphi}$:] The group of automorphisms $\left\{ \left. \varphi^{-1} \circ f_{i_{j}}(g) \circ \varphi \ \right\vert g\in G_{i} \right\}$ (where $\varphi\in\Phi$).
\item[$\varphi(G_{i_{j}})$:] The group of automorphisms $\left\{ \left. f_{i_{j}}\left( \varphi\left(g\right) \right) \right\vert g\in G_{i} \right\}$ (where $\varphi\in\Phi$).
\end{description}
\end{notation}

We now split into cases dependent on the number $n$ of factors in our splitting $G=G_{1}\ast\dots\ast G_{n}$.

\subsection{The Case $n\ge5$} \label{presentation 5} 

We have all the pieces required to build our presentation for $\outs(G)$.
\presentation

\begin{proof}
We apply Brown's Theorem (Theorem \ref{brown thm strict}) to the fundamental domain $\mathcal{D}_{n}$ of the action of $\outs(G)$ on $\mathcal{C}_{n}$.
As previously noted, Proposition \ref{prop strict fun dom} states that the strictness requirement on $\mathcal{D}_{n}$ to apply Brown's Theorem is satisfied.
We also require $\mathcal{C}_{n}$ to be simply connected. We delay the proof of this until after this section. The desired result here is Corollary \ref{Cn is sc}.

We now have that $\outs(G)$ is generated by $\left\{ \stab(T) \left\vert T\in\mathcal{D}_{n}^{(0)} \right. \right\}$, such that if $[S,T]$ is an edge in $\mathcal{D}_{n}^{(1)}$ (i.e. $S,T\in\mathcal{D}_{n}^{(0)}$ with $T$ a collapse of $S$) then we have inclusions $\stab(S)\hookrightarrow\stab(T)$.
We use the descriptions of $\stab(T)$ from Propositions \ref{prop brown stabs 1}, \ref{prop brown stabs 2}, and \ref{prop brown stabs 3}.
We proceed by examining the structure of $\mathcal{D}_{n}^{(1)}$. Recall that graph shapes for $T\in\mathcal{D}_{n}^{(0)}$ are listed in Table \ref{table n>=5 points}.

We first observe that every $\rho_{ij}$ collapses to $\alpha$. Since abstractly, $\stab(\rho_{ij})=\stab(\alpha)$ (Proposition \ref{prop brown stabs 1}), then each $\stab(\rho_{ij})$ is identified with $\stab(\alpha)=\Phi$ in $\outs(G)$.
Similarly, for each $i,j$ we have that every $\stab(\tau_{i,j,kl})$ ($k,l\in[n]-\{i,j\}$) is identified with $\stab(\beta_{i,j})$ in $\outs(G)$.

Since $\rho_{jk}$ collapses to $\beta_{j,k}$, $\gamma_{i,jk}$, $A_{v}$, and $B_{i,j,k}$ (for any $v\in\{1,\dots,n\}$ and any $i\not\in \{j,k\}$), we immediately deduce that the $\Phi$ contribution from any $\stab(\beta)$, $\stab(\gamma)$, $\stab(A)$, or $\stab(B)$ is identified with $\stab(\alpha)$.

We now consider the $\tau$-$\varepsilon$--Square (Figure \ref{tau epsilon box}). Note that for each $\varepsilon_{i,j,k,l}$ there are precisely two $\tau$ graphs, $\tau_{i,j,kl}$ and $\tau_{k,l,ij}$, which collapse to $\varepsilon_{i,j,k,l}$. 
We may `replace' $\tau_{i,j,kl}$ with $\beta_{i,j}$ and $\tau_{k,l,ij}$ with $\beta_{k,l}$, and recalling that $\rho_{ij}$ collapses to $\beta_{i,j}$ and $\rho_{kl}$ to $\beta_{k,l}$, also `replace' $\rho_{ij}$ and $\rho_{kl}$ with $\alpha$. We thus have a diagram:
\begin{center}
\begin{tikzcd}
\stab(\tau_{i,j,kl})
\ar[r,hookrightarrow]			
			&	\stab(\varepsilon_{i,j,k,l})
									&	\stab(\tau_{k,l,ij})
										\ar[l,hook']
													\\
\stab(\beta_{i,j})
\ar[u,equal]
\ar[ur,dashed,hook]
			&						&	\stab(\beta_{k,l})
										\ar[u,equal]
										\ar[ul,dashed,hook']	
													\\
\stab(\rho_{ij})
\ar[u,hookrightarrow]
			&	\stab(\alpha)
				\ar[l,equal]
				\ar[r,equal]
				\ar[ul,dashed,hook']
				\ar[ur,dashed,hook]
									&	\stab(\rho_{kl})
										\ar[u,hook']
\end{tikzcd}
\end{center}
where the dashed inclusions are naturally induced by the `replacements' we made.
Note that $\stab(\beta_{i,j})$ and $\stab(\beta_{k,l})$ `cover' $\stab(\varepsilon_{i,j,k,l})$ in the sense that $\stab(\varepsilon_{i,j,k,l})\subseteq \stab(\beta_{i,j})\times\stab(\beta_{k,l})$.
Since $\stab(\beta_{i,j})\cap\stab(\beta_{k,l})=\Phi=\stab(\alpha)$, the dashed inclusions form something akin to a pushout diagram, and we may conclude that $\stab(\varepsilon)=\stab(\beta_{i,j})\times_{\Phi}\stab(\beta_{k,l})$.

Next, we consider $\stab(A_{i})$.
Observe that given $i\in[n]$, every $\beta_{i,j}$ for $j\in[n]-\{i\}$ collapses to $A_{i}$, thus we have $n-1$ inclusions $\stab(\beta_{i,j})\hookrightarrow\stab(A_{i})$.
Then the $G_{i_{j}}$ contribution from $\stab(A_{i})$ is identified in $\outs(G)$ with the $G_{i_{j}}$ contribution from $\stab(\beta_{i,j})$, and we can consider $\stab(A_{i})$ to be generated by $\stab(\beta_{i_{v_{1}}})\times_{\Phi}\dots\times_{\Phi}\stab(\beta_{i_{v_{n-1}}})$ subject to the relations in Proposition \ref{prop brown stabs 2}, as well as the relation $Z(G_{i_{v_{1}\dots v_{n-1}}})=\{1\}$ (since the $\beta_{i,j}$'s `cover' $A_{i}$).

A similar principle applies to $\stab(\gamma_{i,jk})$.
Given $i,j,k\in[n]$, we have that for any $l\in[n]-\{i,j,k\}$, the graph $\tau_{i,l,jk}$ collapses to $\gamma_{i,jk}$.
Noting that $\stab(\tau_{i,l,jk})$ is identified in $\outs(G)$ with $\stab(\beta_{i,j})$, we have $n-3$ inclusions of the form $\stab(\beta_{i,l})\hookrightarrow\stab(\gamma_{i,jk})$.
Recall from Proposition \ref{prop brown stabs 2} that abstractly, $\stab(\gamma_{i,jk})$ is generated by 
\\ \noindent $\faktor{G_{i_{jk}}\times G_{i_{l_{1}}}\times\dots\times G_{i_{l_{n-3}}} }{Z(G_{i})}$ and $\Phi$.
However, by manipulation of relations in $\stab(\gamma_{i,jk})$ (or by considering the Guirardel--Levitt approach to computing stabilisers), we have that the $G_{i_{jk}}$ component  is redundant as a generator.
Specifically, for $f_{i_{jk}}(g_{i})\in G_{i_{jk}}$ (with $g_{i}\in G_{i}$), we have that $f_{i_{jk}}(g_{i})=\iota_{g_{i}}f_{i_{l_{1}\dots l_{n-3}}}(g_{i}^{-1})f_{i}(g_{i}^{-1})$,
where $\iota_{g_{i}}\in\inn(G)$ conjugates every element of $G$ by $g_{i}$, and $f_{i}:G_{i}\to\inn(G_{i})$ is the canonical homomorphism.
Thus we can consider $\stab(\gamma_{i,jk})$ to be generated by $\stab(\beta_{i_{l_{1}}})\times_{\Phi}\dots\times_{\Phi}\stab(\beta_{i_{l_{n-3}}})$ (with relations similar to $\stab(A_{i})$).

Given a `top' vertex in $\mathcal{D}_{n}^{(0)}$ (i.e. a $\rho$, $\sigma$, or $\tau$ graph), we can reach a `bottom' vertex ($A$, $B$, or $C$ graph) by successively collapsing two edges.
By changing the order in which we collapse these edges, we produce square (or `diamond') diagrams of inclusions
\begin{tikzcd}[cramped,sep=small]
		&	W
			\ar[dl,hook']
			\ar[dr,hook]
				 	&			\\
X
\ar[dr,hook]
		&			&	Y
						\ar[dl,hook']	\\
		&	Z 		&
\end{tikzcd}
where so long as $X$ and $Y$ `cover' $Z$ (in the sense that $\stab(Z)\subseteq\stab(X)\times\stab(Y)$),
we will have $\stab(Z)=\stab(X)\times_{\stab(W)}\stab(Y)$, where $\stab(W)=\stab(X)\cap\stab(Y)$.
Figure \ref{fig pushouts} illustrates some such diagrams which are of particular use.
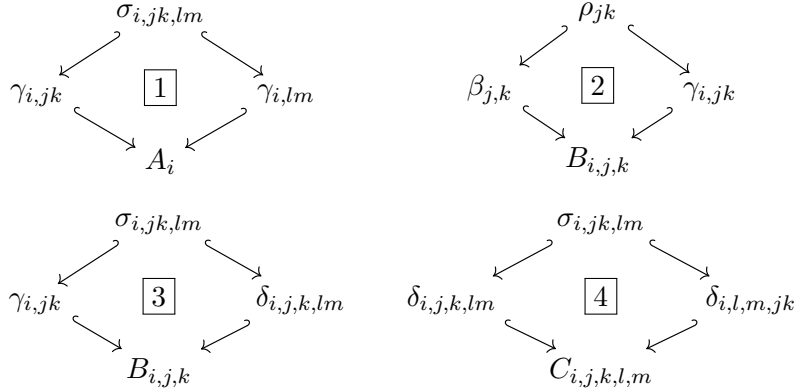
\begin{figure}[h]
\centering
\begin{tabular}{c c}
{
	\begin{tikzcd}[cramped,sep=small]
			&	\sigma_{i,jk,lm}
				\ar[dl,hook']
				\ar[dr,hook]
					 	&			\\
	\gamma_{i,jk}
	\ar[dr,hook]
			&	\boxed{1}
						&	\gamma_{i,lm}
							\ar[dl,hook']	\\
			&	A_{i} 		&
	\end{tikzcd}
}
&
{
	\begin{tikzcd}[cramped,sep=small]
			&	\rho_{jk}
				\ar[dl,hook']
				\ar[dr,hook]
					 	&			\\
	\beta_{j,k}
	\ar[dr,hook]
			&	\boxed{2}		&	\gamma_{i,jk}
								\ar[dl,hook']	\\
			&	B_{i,j,k} 		&
	\end{tikzcd}
}
\\[1.5cm]
{ \ \ \
	\begin{tikzcd}[cramped,sep=small]
			&	\sigma_{i,jk,lm}
				\ar[dl,hook']
				\ar[dr,hook]
					 	&			\\
	\gamma_{i,jk}
	\ar[dr,hook]
			&	\boxed{3}		&	\delta_{i,j,k,lm}
								\ar[dl,hook']	\\
			&	B_{i,j,k} 		&
	\end{tikzcd}
}
&
{ \
	\begin{tikzcd}[cramped,sep=small]
			&	\sigma_{i,jk,lm}
				\ar[dl,hook']
				\ar[dr,hook]
					 	&			\\
	\delta_{i,j,k,lm}
	\ar[dr,hook]
			&	\boxed{4}		&	\delta_{i,l,m,jk}
								\ar[dl,hook']	\\
			&	C_{i,j,k,l,m} 		&
	\end{tikzcd}
}
\end{tabular}
\caption{Inclusion Diagrams in $\mathcal{D}_{n}^{(1)}$}
\label{fig pushouts}
\end{figure}

Diagram $\boxed{1}$ may be thought of as akin to a pullback diagram, in that $\sigma_{i,jk,lm}$ is uniquely determined by $\gamma_{i,jk}$ and $\gamma_{i,lm}$.
Since $\stab(A_{i})\subseteq\stab(\gamma_{i,jk})\times\stab(\gamma_{i,lm})$, we deduce that $\stab(\sigma_{i,jk,lm})=\stab(\gamma_{i,jk})\cap\stab(\gamma_{i,lm})$ in $\outs(G)$.
The remaining diagrams are more akin to pushouts than pullbacks.

Diagram $\boxed{2}$ is from the $\rho$--Book of Figure \ref{rho book}, and is enough to uniquely determine a given $B_{i,j,k}$.
Recalling that $\stab(\rho_{jk})=\stab(\alpha)=\Phi$ in $\outs(G)$, we conclude that $\stab(B_{i,j,k})=\stab(\beta_{j,k})\times_{\Phi}\stab(\gamma_{i,jk})$.
Diagram $\boxed{3}$ implies that $\stab(B_{i,j,k})=\stab(\delta_{i,j,k,lm})\times_{\stab(\sigma_{i,jk,lm})}\stab(\gamma_{i,jk})$,
thus $\stab(\beta_{j,k})\times_{\Phi}\stab(\gamma_{i,jk})=$ \\ \noindent $\stab(\delta_{i,j,k,lm})\times_{\stab(\sigma_{i,jk,lm})}\stab(\gamma_{i,jk})$.
From this, we deduce that $\stab(\delta_{i,j,k,lm})=\stab(\beta_{j,k})\times_{\Phi}\stab(\sigma_{i,jk,lm})$.

Diagram $\boxed{4}$ is from the $\sigma$--Slice of Figure \ref{tau epsilon box}, and is enough to uniquely determine a given $C_{i,j,k,l,m}$.
We then have that:
\begin{align*}
\stab(C_{i,j,k,l,m}) 	&	= 	\stab(\delta_{i,j,k,lm})\times_{\stab(\sigma_{i,jk,lm})}\stab(\delta_{i,l,m,jk}) 		\\
			&	= 	\left( \stab(\beta_{jk})\times_{\Phi}\stab(\sigma_{i,jk,lm}) \right) \times_{\stab(\sigma_{i,jk,lm})} \left( \stab(\beta_{l,m})\times_{\Phi}\stab(\sigma_{i,jk,lm}) \right) 															\\
			&	=	\stab(\beta_{j,k})\times_{\Phi}\stab(\beta_{l,m})\times_{\Phi}\stab(\sigma_{i,jk,lm})	
\end{align*}

We have now shown that any $\stab(T)$ for $T\in\mathcal{D}_{n}^{(0)}$ can be written in terms of $\Phi=\stab(\alpha)$ and $\stab(\beta_{i,j})$ (allowing $i$ and $j$ to vary over $\{1,\dots,n\}$).
Thus $\outs(G)$ is generated by $\left\{ \left. \stab(\beta_{i,j}) \right\vert i\in[n], j\in[n]-\{i\} \right\}$,
that is, $\outs(G)$ is generated by \\ \noindent $\left\{ \left. G_{i_{j}} \right\vert i \in \left\{1,\dots,n\right\}, j \in \left\{1,\dots,n\right\}-\left\{ i \right\} \right\} \cup \Phi$.

From $\stab(A_{i})$ (Proposition \ref{prop brown stabs 2}), we see that $[G_{i_{j}},G_{i_{k}}]=1$ and $\faktor{G_{i_{v_{1}\dots v_{n-1}}}}{Z(G_{i})}=\inn(G_{i})\cong \faktor{G_{i}}{Z(G_{i})}$. Using the isomorphisms $f_{i_{j}}$ for preciseness, we recover Relations 1 and 4. Note that if $g\in Z(G_{i})$ then Relation 4 gives $f_{i_{v_{1}}}(g)\dots f_{i_{v_{n-1}}}(g)=1$.
We deduce from $\stab(\varepsilon_{i,j,k,l})$ (Proposition \ref{prop brown stabs 1}) that $[G_{i_{j}},G_{k_{l}}]=1$, from $\stab(B_{i,j,k})$ (Proposition \ref{prop brown stabs 3}) that $[G_{j_{k}},G_{i_{jk}}]=1$, and from $\stab(\beta_{i,j})$ (Proposition \ref{prop brown stabs 1}) that $G_{i_{j}}^{\varphi}=\varphi(G_{i_{j}})$.
We now recover Relations 2,3, and 5 by substituting the appropriate isomorphisms $f_{i_{j}}$ into the above formulae.
All other relations found in vertex stabilisers are subsumed by these five.
\end{proof}

\begin{rem}
Note that $n$ of these generators are `redundant', in that for each $i\in[n]$ and $j\in[n]-\{i\}$, $G_{i_{j}}\le(G_{i_{v_{1}}}\times\dots\times G_{i_{v_{n-5}}})\rtimes\Phi$.
However, consistently choosing generators to remove without overcomplicating the relations is tricky, so we elect  not to do this.
\end{rem}

\presentationcor

\begin{proof}
Note that the splitting $G_{1}\ast\dots\ast G_{n}$ described is a Grushko decomposition for $G$, and so every automorphism must preserve the conjugacy classes of the factor groups.
That is, $\out(G;G_{1},\dots,G_{n})=\out(G)$.
\end{proof}

\subsection{The Case $n=4$} \label{presentation 4} 

Let $G=G_{1}\ast G_{2}\ast G_{3}\ast G_{4}$ 
be a free splitting of a group $G$, and let $\mathfrak{S}=(G_{1},G_{2},G_{3},G_{4})$.

By Definition \ref{defn c3 and c4}, our complex $\mathcal{C}_{4}$ is the barycentric spine of Guirardel and Levitt's Outer Space relative to $\mathfrak{S}$, which we build by taking only the graph shapes from Table \ref{table n>=5 points} which have at most four non-trivial (red) vertices.
This leaves us with graph shapes $\rho$, $\alpha$, $\beta$, $\gamma$, $A$, and $B$. Note however that for $\{i,j,k,l\}=\{1,2,3,4\}$, we have $\gamma_{i,kl}=\beta_{i,j}$. Also, $\rho_{ij}=\rho_{kl}$ and $B_{i,j,k}=B_{j,i,l}$.
Note additionally that we did not take $\tau$ or $\varepsilon$ graphs (despite these only displaying four non-trivial vertices) since the trivial `basepoint' must have valency at least three here, implying at least one suppressed non-trivial vertex in each case.

Thus the vertex set of the fundamental domain $\mathcal{D}_{4}^{(0)}$ of $\mathcal{C}_{4}$ consists of $\frac{1}{2}{{4}\choose{2}}=3$ $\rho$ vertices, $1$ $\alpha$ vertex, $4\times 3=12$ $\beta$ vertices, $4$ $A$ vertices, and $\frac{4!}{2}=12$ $B$ vertices.
These form three $\rho_{ij}$ `spikes' (for $\{i,j\}\subseteq\{1,2,3,4\}$) in $\mathcal{D}_{4}^{(1)}$, shown in Figure \ref{fig n=4 spike},
\begin{figure}[h]
\centering
\begin{tikzpicture}
\draw[ultra thick, white] (0,-5) -- (-0.75,-4); \draw[thick,->-] (0,-5) -- (-0.75,-4);
\draw[ultra thick, white] (0,-5) -- (1.25,-4); \draw[thick,->-] (0,-5) -- (1.25,-4);
\draw[ultra thick, white] (0,-5) -- (-1.25,-6); \draw[thick,->-] (0,-5) -- (-1.25,-6);
\draw[ultra thick, white] (0,-5) -- (0.75,-6); \draw[thick,->-] (0,-5) -- (0.75,-6);
\draw[ultra thick, white] (-0.75,1) -- (-0.75,-4); \draw[thick,->-] (-0.75,1) -- (-0.75,-4);
\draw[ultra thick, white] (1.25,1) -- (1.25,-4); \draw[thick,->-] (1.25,1) -- (1.25,-4);
\draw[ultra thick, white] (0,0) -- (-0.75,-4); \draw[thick,->-] (0,0) -- (-0.75,-4);
\draw[ultra thick, white] (0,0) -- (1.25,-4); \draw[thick,->-] (0,0) -- (1.25,-4);
\draw[ultra thick, white] (0,0) -- (0,-5); \draw[thick,->-] (0,0) -- (0,-5);
\draw[ultra thick, white] (0,0) -- (-1.25,-6); \draw[thick,->-] (0,0) -- (-1.25,-6);
\draw[ultra thick, white] (0,0) -- (0.75,-6); \draw[thick,->-] (0,0) -- (0.75,-6);
\draw[ultra thick, white] (-1.25,-1) -- (-1.25,-5.5); \draw[thick,->-] (-1.25,-1) -- (-1.25,-6);
\draw[ultra thick, white] (0.75,-1) -- (0.75,-5.5); \draw[thick,->-] (0.75,-1) -- (0.75,-6);
\draw[ultra thick, white] (0,0) -- (-1,0); \draw[thick,->-] (0,0) -- (-1,0);
\draw[ultra thick, white] (0,0) -- (1,0); \draw[thick,->-] (0,0) -- (1,0);
\draw[ultra thick, white] (0,0) -- (-0.75,1); \draw[thick,->-] (0,0) -- (-0.75,1);
\draw[ultra thick, white] (0,0) -- (0.25,1); \draw[thick,->-] (0,0) -- (0.25,1);
\draw[ultra thick, white] (0,0) -- (1.25,1); \draw[thick,->-] (0,0) -- (1.25,1);
\draw[ultra thick, white] (0,0) -- (-1.25,-1); \draw[thick,->-] (0,0) -- (-1.25,-1);
\draw[ultra thick, white] (0,0) -- (-0.25,-1); \draw[thick,->-] (0,0) -- (-0.25,-1);
\draw[ultra thick, white] (0,0) -- (0.75,-1); \draw[thick,->-] (0,0) -- (0.75,-1);
\draw[ultra thick, white] (-0.75,1) -- (0.25,1); \draw[thick,->-] (-0.75,1) -- (0.25,1);
\draw[ultra thick, white] (1.25,1) -- (0.25,1); \draw[thick,->-] (1.25,1) -- (0.25,1);
\draw[ultra thick, white] (-1.25,-1) -- (-0.25,-1); \draw[thick,->-] (-1.25,-1) -- (-0.25,-1);
\draw[ultra thick, white] (0.75,-1) -- (-0.25,-1); \draw[thick,->-] (0.75,-1) -- (-0.25,-1);
\draw[ultra thick, white] (-0.75,1)-- (-1,0); \draw[thick,->-] (-0.75,1) -- (-1,0);
\draw[ultra thick, white] (-1.25,-1) -- (-1,0); \draw[thick,->-] (-1.25,-1) -- (-1,0);
\draw[ultra thick, white] (1.125,0.5) -- (1,0) ; \draw[thick,->-] (1.25,1) -- (1,0);
\draw[ultra thick, white] (0.75,-1)-- (1,0); \draw[thick,->-] (0.75,-1) -- (1,0);
\draw[ultra thick, white] (-0.75,0.175) -- (-0.75,-0.175);
\draw[ultra thick, white] (1.25,0.175) -- (1.25,-0.175);
\draw[ultra thick,white] (-0.3,-1.15) -- (0.1,-1.15);
\draw[ultra thick,white] (-0.185,-0.85) -- (0.2,-0.85);
\draw[ultra thick, white] (0.1275,-1.02) -- (0.1275,-1.175);
\node[rectangle, inner sep=1pt, draw, fill=white] at (0,0) {\tiny$\rho_{ij}$};
\node[rectangle, inner sep=1pt, draw, fill=white] at (-0.75,1) {\tiny$\beta_{k,l}$};
\node[rectangle, inner sep=1pt, draw, fill=white] at (0.25,1) {\tiny$B_{i,k,l}$};
\node[rectangle, inner sep=1pt, draw, fill=white] at (1.25,1) {\tiny$\beta_{i,j}$};
\node[rectangle, inner sep=1pt, draw, fill=white] at (-1.25,-1) {\tiny$\beta_{j,i}$};
\node[rectangle, inner sep=1pt, draw, fill=white] at (-0.25,-1) {\tiny$B_{j,l,k}$};
\node[rectangle, inner sep=1pt, draw, fill=white] at (0.75,-1) {\tiny$\beta_{l,k}$};
\node[rectangle, inner sep=1pt, draw, fill=white] at (-1,0) {\tiny$B_{j,k,l}$};
\node[rectangle, inner sep=1pt, draw, fill=white] at (1,0) {\tiny$B_{i,l,k}$};
\node[rectangle, inner sep=1pt, draw, fill=white] at (0,-5) {\tiny$\alpha$};
\node[rectangle, inner sep=1pt, draw, fill=white] at (-0.75,-4) {\tiny$A_{k}$};
\node[rectangle, inner sep=1pt, draw, fill=white] at (1.25,-4) {\tiny$A_{i}$};
\node[rectangle, inner sep=1pt, draw, fill=white] at (-1.25,-6) {\tiny$A_{j}$};
\node[rectangle, inner sep=1pt, draw, fill=white] at (0.75,-6) {\tiny$A_{l}$};
\draw[ultra thick, white] (-0.8,-3.84) -- (-0.875,-4.2);
\draw[thick] (-0.68,-3.264) -- (-0.9,-4.32);
\end{tikzpicture}
\caption{A `Spike' in $\mathcal{D}_{4}^{(1)}$}
\label{fig n=4 spike}
\end{figure}
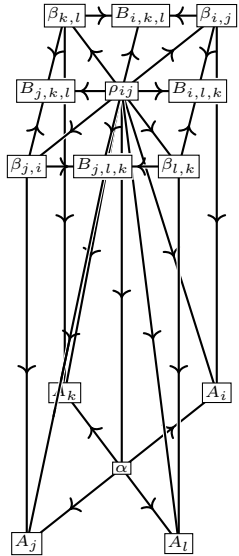
which are identified along the $\alpha$-$A$--Star
\begin{tikzpicture}
\draw[thick] (0,0) -- (0.5,0.5);
\draw[thick] (0,0) -- (0.5,-0.5);
\draw[thick] (0,0) -- (-0.5,-0.5);
\draw[thick] (0,0) -- (-0.5,0.5);
\filldraw[white] (0,0) circle [radius=0.15cm];
\filldraw[white] (0.5,0.5) circle [radius=0.15cm];
\filldraw[white] (0.5,-0.5) circle [radius=0.15cm];
\filldraw[white] (-0.5,-0.5) circle [radius=0.15cm];
\filldraw[white] (-0.5,0.5) circle [radius=0.15cm];
\node at (0,0) {$\alpha$};
\node at (0.5,0.6) {$A_{1}$};
\node at (0.5,-0.5) {$A_{2}$};
\node at (-0.5,-0.5) {$A_{3}$};
\node at (-0.5,0.6) {$A_{4}$};
\end{tikzpicture}.

\begin{prop}\label{prop n=4 stabs} 
We have the following for $\{i,j,k,l\}=\{1,2,3,4\}$:
\begin{enumerate}
\item $\stab(\rho_{ij})=\stab(\alpha)=\Phi$
\item $\stab(\beta_{i,j})=G_{i_{j}}\rtimes\Phi$
\item $\stab(A_{i})$ is generated by $\faktor{(G_{i_{j}}\times G_{ i_{k}}\times G_{i_{l}})}{Z(G_{i})}$ and $\Phi$ with relations $\faktor{G_{i_{jkl}}}{Z(G_{i})}=\inn(G_{i})$ and $(\{G_{a}\},g_{i})\circ\varphi=\varphi\circ(\{G_{a}\},\varphi(g_{i}))$ for $\varphi\in\Phi$ and $(\{G_{a}\},g_{i})\in i_{a}$ for each $a\in\{j,k,l\}$
\item $\stab(B_{i,k,l})=(G_{i_{j}}\times G_{k_{l}})\rtimes\Phi$
\end{enumerate}
\end{prop}

\begin{proof}
The first three items are the same as the presentations listed for the $n\ge5$ case. The arguments there also hold for $n=4$.
We use the Guirardel--Levitt approach (see Example \ref{eg stab sigma}) to compute $\stab(B_{i,k,l})$, noting that $B_{i,k,l}=$
\begin{tikzpicture}[scale=0.8]
\draw[thick] (1,0) -- (2,0) -- (3,0) -- (4,0);
\filldraw[red] (1,0) circle [radius=0.09cm];
\filldraw[red] (2,0) circle [radius=0.09cm];
\filldraw[red] (3,0) circle [radius=0.09cm];
\filldraw[red] (4,0) circle [radius=0.09cm];
\node at (1.05,0.3) {$G_{j}$};
\node at (2.05,0.3) {$G_{i}$};
\node at (3.05,0.3) {$G_{k}$};
\node at (4.05,0.3) {$G_{l}$};
\end{tikzpicture}
and fixing the edge between vertex groups $G_{i}$ and $G_{k}$.
\end{proof}

\begin{thm}\label{thm n=4 presentation}
Let $G_{1}\ast G_{2}\ast G_{3}\ast G_{4}$ 
be a free splitting of a group $G$ where each $G_{i}$ is non-trivial, and let $\mathfrak{S}=(G_{1},G_{2},G_{3},G_{4})$.
Writing $f_{i_{j}}$ for the isomorphism $G_{i}\to G_{i_{j}}$,
$\outs(G)$ is generated by the twelve groups $G_{i_{j}}$ for $i,j\in\{1,2,3,4\}$ distinct, and $\Phi=\prod_{k=1}^{4}\aut(G_{k})$, subject to relations:
\begin{enumerate}
\item $\left[ f_{i_{j}}(g),f_{i_{k}}(h) \right]=1 \ \forall g,h\in G_{i}$
\item $\left[ f_{i_{j}}(g),f_{k_{l}}(h) \right]=1 \ \forall g\in G_{i}, h\in G_{k}$
\item $f_{i_{j}}(g)f_{i_{k}}(g)f_{i_{l}}(g)=\ad_{G_{i}}(g) \ \forall g\in G_{i}$
\item $\varphi^{-1}f_{i_{j}}(g)\varphi=f_{i_{j}}(\varphi(g)) \ \forall g\in G_{i}$ for all $\varphi\in\Phi$
\end{enumerate}
where $\{i,j,k,l\}=\{1,2,3,4\}$.
\end{thm}

\begin{proof}
We apply Theorem \ref{brown thm strict} to the complex $\mathcal{C}_{4}$. Note that Lemma \ref{lemma c3 and c4 sc} tells us $\mathcal{C}_{4}$ is simply connected, and it is not hard to deduce that the fundamental domain $\mathcal{D}_{4}$ meets the strictness requirements here.
Then $\outs(G)$ is generated by the stabiliser groups described in Proposition \ref{prop n=4 stabs} for each combination of $\{i,j,k,l\}=\{1,2,3,4\}$.
We now consider the edge relations from $\mathcal{D}_{4}$.
Since each $\rho_{ij}$ collapses to $\alpha$, we have that $\stab(\rho_{12})$, $\stab(\rho_{13})$, $\stab(\rho_{14})$, and $\stab(\alpha)$ all generate the same subgroup $\Phi$ of $\out(G)$.
Moreover, since every vertex is a collapse of some $\rho_{ij}$ then the $\Phi$ contribution from each stabiliser in Proposition \ref{prop n=4 stabs} is identified with $\stab(\alpha)$.
Each `spike' of $\mathcal{D}_{4}$ gives a diagram of inclusions:
\\ \noindent
\begin{center}
\begin{tikzcd}[cramped]
(G_{k_{l}})\rtimes\Phi \ar[r] \ar[d]		& (G_{i_{j}}\times G_{k_{l}})\rtimes\Phi					& (G_{i_{j}})\rtimes\Phi \ar[l] \ar[d]		\\
(G_{j_{i}}\times G_{k_{l}})\rtimes\Phi		& \Phi \ar[l] \ar[r] \ar[u] \ar[d] \ar[ul] \ar[ur] \ar[dl] \ar[dr]		& (G_{i_{j}}\times G_{l_{k}})\rtimes\Phi	\\
(G_{j_{i}})\rtimes\Phi \ar[u] \ar[r]		& (G_{j_{i}}\times G_{l_{k}})\rtimes\Phi					& (G_{l_{k}})\rtimes\Phi \ar[u] \ar[l]		\\	
\end{tikzcd}
\end{center}
and each $A_{i}$ neighbourhood gives a further diagram of inclusions:
\\ 
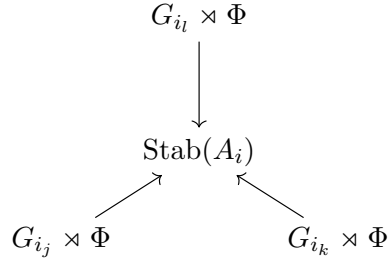
\noindent
\begin{center}
\begin{tikzcd}[cramped, column sep=0.5em, row sep=1.5em] 
					&	G_{i_{l}}\rtimes\Phi \ar[dd]	&						\\
					&						&						\\				
					&	\stab(A_{i})				&						\\
G_{i_{j}}\rtimes\Phi \ar[ur]	&						&	G_{i_{k}}\rtimes\Phi \ar[ul]	\\
\end{tikzcd}
\end{center}
Hence any $G_{i_{j}}$ contribution from $\stab(A_{i})$ or $\stab(B_{i,k,l})$ is identified with the $G_{i_{j}}$ contribution from $\beta_{i,j}$.
Thus $\outs(G)$ is generated $\stab(\beta_{i,j})$ for each pair $(i,j)$ and $\alpha$, with the $\Phi$ contribution from each $\stab(\beta_{i,j})$ identified with $\stab(\alpha)$.
That is, $\outs(G)$ is generated by $G_{i_{j}}$ for each pair $(i,j)$ and $\Phi$.
The Relations 1--4 come from the groups in Proposition \ref{prop n=4 stabs} (specifically, 1 and 3 are from $\stab(A_{i})$, 2 is from $\stab(B_{i,k,l})$, and 4 is from $\stab(\beta_{i,j})$).
\end{proof}

Observe that the Relation `$\left[ f_{j_{k}}(g),f_{i_{j}}(h)f_{i_{k}}(h) \right]=1$' from the case $n\ge5$ may be deduced from Relations 2 and 3 here (since writing $k_{k}$ for $\inn(k)$ yields $k_{ij}=k_{kl}$).
Thus Theorem \ref{thm n>=5 presentation} in fact holds for $n\ge4$.

One should note that one of each of the groups $G_{i_{j}}\cong G_{i}$ is `redundant' in the sense that only any $2$ of the groups $G_{i_{j}}, G_{i_{k}}, G_{i_{l}}$ are independent of each other.
It is possible to consistently choose $4$ such groups to eliminate from the list of generators, but doing so would somewhat complicate the relations.

\subsection{The Case $n=3$} \label{presentation 3} 

For $G=G_{1}\ast G_{2}\ast G_{3}$, the only graphs which respect the splitting $\mathfrak{S}=(G_{1},G_{2},G_{3})$ are $\alpha$ and $A$ graphs. Thus the fundamental domain $\mathcal{D}_{3}$ of our complex $\mathcal{C}_{3}$ is the following tripod:
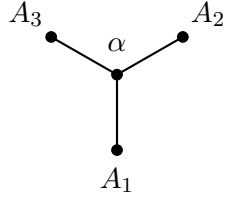
\begin{center}
\begin{tikzpicture} 
\draw[fill] (0,0) circle [radius=0.07];
\draw[fill] (0,-1) circle [radius=0.07];
\draw[fill] (0.866,0.5) circle [radius=0.07];
\draw[fill] (-0.866,0.5) circle [radius=0.07];
\draw[thick] (0,0) -- (0,-1);
\draw[thick] (0,0) -- (0.866,0.5);
\draw[thick] (0,0) -- (-0.866,0.5);
\node at (0,0.4) {$\alpha$};
\node at (0,-1.4) {$A_{1}$};
\node at (1.2,0.8) {$A_{2}$};
\node at (-1.2,0.8) {$A_{3}$};
\end{tikzpicture}
\end{center}
where $A_{i}$ is the graph
\begin{tikzpicture}[scale=0.8]
\draw[thick] (0,0) -- (-1,0);
\draw[thick] (0,0) -- (1,0);
\draw[red, fill] (0,0) circle [radius=0.09];
\draw[red, fill] (-1,0) circle [radius=0.09];
\draw[red, fill] (1,0) circle [radius=0.09];
\node at (1.025,-0.375) {$G_{j}$};
\node at (0,-0.35) {$G_{i}$};
\node at (-0.95,-0.35) {$G_{k}$};
\end{tikzpicture}
for $i=1,2,3$ and $\{j,k\}=\{1,2,3\}-\{i\}$.

As before, we have that $\stab(\alpha)=\Phi=\aut(G_{1})\ast\aut(G_{2})\ast\aut(G_{3})$, and that for each $i\in\{1,2,3\}$ with $\{j,k\}=\{1,2,3\}-\{i\}$, $\stab(A_{i})$ is generated by $\faktor{(G_{i_{j}}\times G_{i_{k}})}{Z(G_{i})}$ and $\Phi$ such that $\faktor{G_{i_{jk}}}{Z(G_{i})}=\inn(G_{i})$ and $\left(\{G_{a}\},g_{i}\right)\circ\varphi=\varphi\circ\left(\{G_{a}\},\varphi(g_{i})\right)$ (for $a=j,k$) for all $\left(\{G_{a}\},g_{i}\right)\in G_{i_{a}}$ and $\varphi\in\Phi$.

\begin{thm}\label{thm n=3 presentation}
Let $G=G_{1}\ast G_{2}\ast G_{3}$ 
be a free splitting of a group $G$ where each $G_{i}$ is non-trivial, and let $\mathfrak{S}=(G_{1},G_{2},G_{3})$.
Then writing $f_{i_{j}}$ for the isomorphism $G_{i}\to G_{i_{j}}$, $\outs(G)$ is generated by the six groups $G_{i_{j}}$ for $i,j\in\{1,2,3\}$ distinct, and $\Phi$, subject to relations:
\begin{enumerate}
\item $\left[f_{i_{j}}(g), f_{i_{k}}(h)\right]=1 \ \forall g,h\in G_{i}$
\item $f_{i_{j}}(g)f_{i_{k}}(g)=\ad_{G_{i}}(g) \ \forall g\in G_{i}$
\item $\varphi^{-1}f_{i_{j}}(g)\varphi=f_{i_{j}}(\varphi(g)) \ \forall g\in G_{i}$ for all $\varphi\in\Phi$
\end{enumerate}
for all $i=1,2,3$ and $\{j,k\}=\{1,2,3\}-\{i\}$.
\end{thm}

\begin{proof}
By Theorem \ref{brown thm strict}, $\outs(G)$ is generated by $\stab(\alpha)$, $\stab(A_{1})$, $\stab(A_{2})$, and $\stab(A_{3})$.
The structure of $\mathcal{D}_{3}$ means that the $\Phi$ contribution from each $\stab(A_{i})$ is identified with $\stab(\alpha)$.
That is, $\outs(G)=\stab(A_{1})\ast_{\Phi}\stab(A_{2})\ast_{\Phi}\stab(A_{3})$.
The result then follows by examining each $\stab(A_{i})$.
\end{proof}

Note that $\faktor{G_{i_{jk}}}{Z(i)}=\inn(G_{i})$ means $G_{i_{j}}\inn(G_{i})=G_{i_{k}}\inn(G_{i})$ as cosets in $\stab(A_{i})$.
Thus we can write $\stab(A_{i})=G_{i_{j}}\rtimes\Phi=G_{i_{k}}\rtimes\Phi$ (using the Guirardel--Levitt approach to computing stabilisers demonstrated in Example \ref{eg stab sigma}).
Since the only relation between vertex groups here is the amalgamation over $\Phi$, we can in this case obtain a much simpler presentation:

\begin{cor}
For $G=G_{1}\ast G_{2}\ast G_{3}$ as above, we have:
\[\outs(G)=(G_{1_{2}}\ast G_{2_{3}}\ast G_{3_{1}})\rtimes\Phi\cong G\rtimes\Phi\]
\end{cor}

\begin{proof}
As noted above, in this case we have that $\outs(G)=\stab(A_{1})\ast_{\Phi}\stab(A_{2})\ast_{\Phi}\stab(A_{3})$.
We now observe that $(G_{1_{2}}\rtimes\Phi)\ast_{\Phi}(G_{2_{3}}\rtimes\Phi)\ast_{\Phi}(G_{3_{1}}\rtimes\Phi)=(G_{1_{2}}\ast G_{2_{3}}= G_{3_{1}})\rtimes\Phi$.
\end{proof}

If each of $G_{1}$, $G_{2}$, and $G_{3}$ is additionally freely indecomposable, not infinite cyclic, and pairwise non-isomorphic, then we have that $\outs(G)=\out(G)$.
In this case, this is exactly the presentation for $\out(G)$ given by Collins and Gilbert \cite[Propositions 4.1,4.2]{Collins1990}.


\section{The Space of Domains} \label{section space of domains}

The rest of the paper will be spent proving that for $n\ge5$, $\mathcal{C}_{n}$ is simply connected.

In order to study global properties of our complex $\mathcal{C}_{n}$, we will define a new space, akin to a nerve complex (first introduced by Alexandroff in \cite{Alexandroff1928}).

A nerve complex is an abstract simplicial complex built using information on the intersections within a family of sets.
The sets we will choose are copies of the fundamental domain $\mathcal{D}_{n}$, which form a (closed) cover of $\mathcal{C}_{n}$.
However, in order to keep the dimension low, we will only consider $k$-wise intersections of sets for $k\le3$.

It is not necessary to have background knowledge of nerve complexes in order to understand our space or arguments.

\subsection{Defining the Space of Domains} \label{subsection space of domains}

We will call a subset of our complex a \emph{domain} if it is of the form $\mathcal{D}_{n}\cdot\psi$ for some $\psi\in\outs(G)$, where we think of $\mathcal{D}_{n}$ as a set and $\mathcal{D}_{n}\cdot\psi=\{x\cdot\psi|x\in\mathcal{D}_{n}\}$. Since the action of $\outs(G)$ on $\mathcal{C}_{n}$ preserves adjacency, then $\mathcal{D}_{n}\cdot\psi$ has the same topological structure as $\mathcal{D}_{n}$.

\begin{defn}
The \emph{Graph of Domains} is a graph whose vertex set contains one vertex for every domain in our complex, and whose edge set contains an edge joining distinct vertices $u$ and $v$ if and only if the intersection of the domains associated to $u$ and $v$ is non-empty.
\end{defn}

Since the graph \Talpha{$n$} (denoted $\alpha$) occurs precisely once per domain, taking one vertex per domain equates to taking one vertex for every point of the form \Talpha{$n$} in our complex $\mathcal{C}_{n}$.
We will thus often denote vertices in the Graph of Domains by $\alpha$.
Now any two distinct vertices $\alpha_{1}$ and $\alpha_{2}$ in the Graph of Domains are joined by an edge precisely when the intersection of the domain containing the graph $\alpha_{1}$ and the domain containing the graph $\alpha_{2}$ is non-empty.

Note that the action of $\outs(G)$ on $\mathcal{C}_{n}$ induces a natural $\outs(G)$-action on the Graph of Domains.

\begin{defn}\label{defn splitting of domain}
The \emph{splitting associated to a domain} $\mathcal{D}_{n}\cdot\psi$ is the labelling \\ \noindent $\left( (G_{1})\psi,\dots,(G_{n})\psi \right)$, which is equivalent to any labelling $(H_{1},\dots,H_{n})$ of the $\alpha$-graph contained within the domain $\mathcal{D}_{n}\cdot\psi$.
\end{defn}

Note that $\psi\in\outs(G)$ is only unique up to factor automorphisms --- that is, if $\varphi\in\Phi$ then $\mathcal{D}_{n}\cdot\psi=\mathcal{D}_{n}\cdot\varphi\psi$.

\begin{cor}
The Graph of Domains is (path) connected.
\end{cor}

\begin{proof}
This follows immediately from Proposition \ref{prop alpha graphs are path connected}.
\end{proof}

\begin{rem}
Note that the Graph of Domains is \textbf{not} locally finite.
\end{rem}

To show our complex $\mathcal{C}_{n}$ is simply connected, we will need the following construction:

\begin{defn}\label{Space of Domains defn}
Wherever we have a 3-cycle 
\begin{tikzpicture}
\node at (0,0) {$\alpha_{1}$};
\node at (1,1) {$\alpha_{2}$};
\node at (2,0) {$\alpha_{3}$};
\draw[thick] (0.2,0) -- (1.75,0);
\draw[thick] (0.1,0.1) -- (0.9,0.9);
\draw[thick] (1.15,0.85) -- (1.85,0.15);
\end{tikzpicture}
in our Graph of Domains, we will insert a 2-simplex
\begin{tikzpicture}
\transparent{0.7}
\filldraw[white, fill=lightgray] (0,0) -- (2,0) -- (1,1) -- cycle;
\transparent{1}
\node at (0,0) {$\alpha_{1}$};
\node at (1,1) {$\alpha_{2}$};
\node at (2,0) {$\alpha_{3}$};
\draw[thick] (0.2,0) -- (1.75,0);
\draw[thick] (0.1,0.1) -- (0.9,0.9);
\draw[thick] (1.15,0.85) -- (1.85,0.15);
\end{tikzpicture}
if and only if $\alpha_{1}\cap\alpha_{2}\cap\alpha_{3}\ne\emptyset$ (in the complex $\mathcal{C}_{n}$).
We call the resulting CW-complex the \emph{Space of Domains}.

Note that given a cycle $\alpha_{1}\dash \dots\dash \alpha_{n-1}\dash \alpha_{n}=\alpha_{1}$ with $\alpha_{1}\cap\dots\cap\alpha_{n-1}\ne\emptyset$, we can split this up into 3-cycles
\begin{tikzpicture}
\node at (-0.05,0) {$\alpha_{1}$};
\node at (1,1) {$\alpha_{i}$};
\node at (2.2,0) {$\alpha_{i+1}$};
\draw[thick] (0.2,0) -- (1.75,0);
\draw[thick] (0.1,0.1) -- (0.9,0.9);
\draw[thick] (1.15,0.85) -- (1.85,0.15);
\end{tikzpicture}
for $i=2,\dots,n-2$, with each $\alpha_{1}\cap\alpha_{i}\cap\alpha_{i+1}\ne\emptyset$, so any such loop is contractible in our Space of Domains.
\end{defn}

The idea behind this definition is that, since we have shown that the fundamental domain $\mathcal{D}_{n}$ is simply connected, then if we had that any pairwise intersection of domains is either empty or path-connected,
then any non-trivial loop in $\mathcal{C}_{n}$ would be projected to a non-trivial loop in the Space of Domains.

In particular, if $\alpha_{1}\cup\alpha_{2}$ were simply connected (assuming $\alpha_{1}\cap\alpha_{2}\ne\emptyset$) then there would be no non-trivial loops in $\mathcal{C}_{n}$ which would appear as a forwards-and-backwards traversal of an edge when projected to the Space of Domains,
and similarly for $\alpha_{1}\cup\alpha_{2}\cup\alpha_{3}$ (or the 2-cell $\alpha_{1}\dash \alpha_{2}\dash \alpha_{3}\dash \alpha_{1}$ in the Space of Domains).

Then to show that our space $\mathcal{C}_{n}$ is simply connected, it would suffice to show that our Space of Domains is simply connected.

Unfortunately, it will not be quite this simple, but the general idea will remain the same. 
We will formalise (and resolve) this in Sections \ref{pairwise intersections} and \ref{section map SoD to Cn}.

\subsection{Pairwise Intersections} \label{pairwise intersections}

Here we would hope to show that the intersection of two adjacent domains is path-connected.

Then we could deduce using the Seifert--van Kampen Theorem that the union of two adjacent domains is simply connected.
This would ensure, for example, that there are no non-trivial loops in $\mathcal{C}_{n}$ of the form $\alpha_{1}\dash A\dash \alpha_{2}\dash A\dash \alpha_{1}$, and would justify the use of a single edge between adjacent vertices in our Graph of Domains.

As it turns out, not quite all such intersections are path-connected, but we show that the case where this does not hold can be circumvented.
This is deduced in Propositions \ref{pw int pc case not C} and \ref{pw int pc case C}, the main results of this subsection.

To avoid confusion regarding domains and graphs within domains, we will temporarily break from the convention of naming domains $\alpha$.
So let $\aleph_{1}$ and $\aleph_{2}$ be two arbitrary domains. Assume $\Int\ne\emptyset$.
Without loss of generality, we may assume $\aleph_{1}$ is the fundamental domain $\mathcal{D}_{n}$.

\begin{obs}\label{obs pw int stab}
Note that if $\aleph_{2}=\aleph_{1}\cdot\psi$ then for $T\in\aleph_{1}$, we have $T\in\aleph_{2}$ if and only if $T=T'\cdot\psi$ for some $T'\in\aleph_{1}$.
But each domain contains precisely one element of each orbit, so we must have $T=T'$. Then $\psi\in\stab(T)$.
Moreover, for $T\in \aleph_{1}$ and $\psi\in\stab(T)$, we have $T=T\cdot\psi\in\aleph_{1}\cdot\psi$.
That is, for $T\in\aleph_{1}$, we have $T\in\aleph_{2}$ if and only if $\aleph_{2}=\aleph_{1}\cdot\psi$ for some $\psi\in\stab(T)$.
\end{obs}

\begin{lemma}
Let $T_{1}$, $T_{2}$, and $T_{3}$ be vertices in the complex $\mathcal{C}_{n}$. If $T_{1},T_{2}\in\Int$ and $T_{3}\in\aleph_{1}$ with $\stab(T_{1})\cap\stab(T_{2})\subseteq\stab(T_{3})$ then $T_{3}\in\Int$.
\end{lemma}

\begin{proof}
By Observation \ref{obs pw int stab}, $T\in\Int$ if and only if $\aleph_{2}=\aleph_{1}\cdot\psi$ for some $\psi\in\stab{T}$.
Thus if $T_{1}, T_{2}\in\Int$ then $\aleph_{2}=\aleph_{1}\cdot\psi$ for some $\psi\in\stab{T_{1}}\cap\stab{T_{2}}$.
Additionally, if for some $T_{3}\in\aleph_{1}$ we have $\psi\in\stab(T_{3})$, then $T_{3}\in\Int$.
This last condition holds if (but not only if) $\stab(T_{1})\cap\stab(T_{2})\subseteq\stab(T_{3})$.
\end{proof}

\begin{cor}
Let $T_{1}$ and $T_{2}$ be vertices in the intersection $\Int\subseteq\mathcal{C}_{n}$.
If $\stab(T_{1})\cap\stab(T_{2})=\Phi$, then $\Int=\aleph_{1}$. 
\end{cor}

\begin{proof}
Recall from Proposition \ref{prop brown stabs 1} that $\stab(\alpha)=\Phi$.
Thus if $\stab(T_{1})\cap\stab(T_{2})=\Phi$ for some $T_{1},T_{2}\in\Int$, then $\alpha\in\Int$.
But each $\alpha$ vertex appears in exactly one domain, hence $\aleph_{1}=\aleph_{2}$.
\end{proof}

Suppose $T_{1}$ and $T_{2}$ are distinct vertices in the complex $\mathcal{C}_{n}$.
By Lemma \ref{lem stabs include}, if $T_{1}$ is a collapse of $T_{2}$, then $\stab(T_{2})\subseteq\stab(T_{1})$, so $T_{2}\in\Int\implies T_{1}\in\Int$. 
Since every graph collapses to at least one of $A_{i}$, $B_{i,j,k}$, or $C_{i,j,k,l,m}$ (for some $i,j,k,l,m$), then to show path connectivity of intersections, it suffices to find paths in the intersection $\Int$ with endpoints as the following six cases:
\begin{enumerate}
\item $A_{i}\dash A_{p}$
\item $B_{i,j,k}\dash A_{p}$
\item $B_{i,j,k}\dash B_{p,q,r}$
\item $C_{i,j,k,l,m}\dash A_{p}$
\item $C_{i,j,k,l,m}\dash B_{p,q,r}$
\item $C_{i,j,k,l,m}\dash C_{p,q,r,s,t}$
\end{enumerate}
(where $i,j,k,l,m,p,q,r,s,t$ need not be distinct, unless appearing together as indices of a single vertex.)
In our proofs, we will assume the `left' vertex has fixed indices, and allow the indices of the `right' vertex to vary.

Writing $i_{j}$ for the group $G_{i_{j}}$, we
recall that the stabiliser of $A_{i}$ is a quotient of 	$(i_{v_{1}}\times\dots\times i_{v_{n-1}})\rtimes\Phi$	 
, the stabiliser of $B_{i,j,k}$ is a quotient of 	$(i_{jk}\times j_{k}\times i_{v_{1}}\times\dots\times i_{v_{n-3}})\rtimes\Phi$	 
, and the stabiliser of $C_{i,j,k,l,m}$ is a quotient of 	$(i_{lm}\times l_{m} \times i_{jk}\times j_{k}\times i_{v_{1}}\times\dots\times i_{v_{n-5}})\rtimes\Phi$.	 
In the group of automorphisms $G_{i_{j}}$, we call $G_{i}$ the \emph{operating factor} and $G_{j}$ the \emph{dependent factor}.
We say a graph has operating and dependent factors if the same is true of its stabiliser.
So $A_{i}$ has one operating factor, $B_{i,j,k}$ has two, and $C_{i,j,k,l,m}$ has three distinct operating factors.
In each case, only one operating factor has more than one dependent factor.
We now proceed through the Cases 1--6:

\begin{lemma}[Case 1]\label{lemma pw int case 1}
If $A_{i}$ and $A_{p}$ are vertices in $\aleph_{1}\cap\aleph_{2}$, then there is a path in $\aleph_{1}\cap\aleph_{2}$ from $A_{i}$ to $A_{p}$.
\end{lemma}

\begin{proof}
We have $\stab(A_{i})\cap\stab(A_{p})\ne\Phi$ if and only if $i=p$. But each domain contains only one $A_{i}$-graph for each $i\in\{1,\dots,n\}$. So either $\alpha\in\Int$ (in which case $\alpha_{1}=\alpha_{2}$), or $A_{i}=A_{p}$.

So if $A_{i}$ and $A_{p}$ are points in $\Int$ for any $i,p\in\{1,\dots,n\}$ then there is a path in $\Int$ connecting them.
\end{proof}

\begin{lemma}[Case 2]\label{lemma pw int case 2}
If $B_{i,j,k}$ and $A_{p}$ are vertices in $\aleph_{1}\cap\aleph_{2}$, then there is a path in $\aleph_{1}\cap\aleph_{2}$ from $B_{i,j,k}$ to $A_{p}$.
\end{lemma}

\begin{proof}
If $p\not\in\{i,j\}$ then $B_{i,j,k}$ and $A_{p}$ share no common operating factors, hence \\ $\stab(B_{i,j,k})\cap\stab(A_{p})=\Phi$ and $\alpha\in\Int$.

If $p=i$ then $\stab(B_{i,j,k})\cap\stab(A_{p})$ contains only one operating factor, with $n~-~3$ dependent factors.
That is, $\stab(B_{i,j,k})\cap\stab(A_{i})=(i_{v_{1}}\times\dots\times i_{v_{n-3}})\rtimes\Phi$ for $\{v_{1},\dots,v_{n-3}\}=\{1,\dots,n\}-\{i,j,k\}$. This is precisely the stabiliser of $\gamma_{i,jk}$, hence $\gamma_{i,jk}\in\Int$.
Moreover, $\gamma_{i,jk}$ collapses to both $A_{i}$ and $B_{i,j,k}$, so we have a path $B_{i,j,k}\dash \gamma_{i,jk}\dash A_{i}$.

For $p=j$, we have $\stab(B_{i,j,k})\cap\stab(A_{j})=j_{k}\rtimes\Phi=\stab(\beta_{j,k})$, and so $B_{i,j,k}\dash \beta_{j,k}\dash A_{j}$ is a path in $\Int$.

So if $B_{i,j,k}$ and $A_{p}$ are points in $\Int$ for any $i,j,k,l\in\{1,\dots,n\}$ then there is a path in $\Int$ connecting them.
\end{proof}

\begin{lemma}[Case 3]\label{lemma pw int case 3}
If $B_{i,j,k}$ and $B_{p,q,r}$ are vertices in $\aleph_{1}\cap\aleph_{2}$, then there is a path in $\aleph_{1}\cap\aleph_{2}$ from $B_{i,j,k}$ to $B_{p,q,r}$.
\end{lemma}

\begin{proof}
In order to have $\stab(B_{i,j,k})\cap\stab(B_{p,q,r})\ne\Phi$, we must have that $\{i,j\}\cap\{p,q\}\ne\emptyset$. We will thus assume this holds.

If $\{i,j\}=\{p,q\}$ and additionally $r=k$, then either $B_{p,q,r}=B_{i,j,k}$ and we are done, or we have $B_{p,q,r}=B_{j,i,k}$, in which case $\stab(B_{i,j,k})\cap\stab(B_{p,q,r})=\Phi$.
So we may assume $r\ne k$ in this case.

If $p=j$ and $q=i$ (with $r\ne k$) then $\stab(B_{i,j,k})\cap\stab(B_{p,q,r})=(i_{r}\times j_{k})\rtimes\Phi=\stab(\varepsilon_{i,r,j,k})$, and $B_{i,j,k}-\varepsilon_{i,r,j,k}-B_{j,i,r}$ is a path in $\Int$.

If $p=i$ and $q=j$ (with $r\ne k$) then $\stab(B_{i,j,k})\cap\stab(B_{p,q,r})=(i_{v_{1}}\times\dots\times i_{v_{n-4}})\rtimes\Phi$ where $\{i_{v_{1}},\dots,i_{v_{n-4}}\}=\{1,\dots,n\}-\{i,j,k,r\}$.
This is contained within $\stab(A_{i})$, hence $A_{i}\in\Int$.
Then by Case 2 (Lemma \ref{lemma pw int case 2}), there is some path from $B_{i,j,k}$ to $A_{i}$ and some path from $A_{i}$ to $B_{i,j,r}$ in $\Int$.

We will now consider $\{i,j\}\ne\{p,q\}$.
Then $|\{i,j\}\cap\{p,q\}|=1$ and so $\stab(B_{i,j,k})\cap\stab(B_{p,q,r})$ has at most one operating factor (with at most $n-4$ dependent factors).
If there is no common operating factor, then $\stab(B_{i,j,k})\cap\stab(B_{p,q,r})=\Phi$. Otherwise, $\stab(B_{i,j,k})\cap\stab(B_{p,q,r})\subset A_{v}$ for some $v\in\{i,j,p,q\}$.
Then we are reduced to Case 2.

So if $B_{i,j,k}$ and $B_{p,q,r}$ are points in $\Int$ for any $i,j,k,p,q,r\in\{1,\dots,n\}$ then there is a path in $\Int$ connecting them.
\end{proof}

\begin{lemma}[Case 4]\label{lemma pw int case 4}
If $C_{i,j,k,l,m}$ and $A_{p}$ are vertices in $\aleph_{1}\cap\aleph_{2}$, then there is a path in $\aleph_{1}\cap\aleph_{2}$ from $C_{i,j,k,l,m}$ to $A_{p}$.
\end{lemma}

\begin{proof}
In order to satisfy $\stab(C_{i,j,k,l,m})\cap\stab(A_{p})\ne\Phi$, we require that $p\in\{i,j,l\}$.
Note that by symmetry, $C_{i,j,k,l,m}=C_{i,l,m,j,k}$, so we only need to consider one of $p=j$ and $p=l$.

If $p=j$ then $\stab(C_{i,j,k,l,m})\cap\stab(A_{p})=j_{k}\rtimes\Phi=\stab(\tau_{j,k,lm})$.
We have that $\tau_{j,k,lm}$ collapses to both $A_{j}$ and $C_{i,j,k,l,m}$, so $C_{i,j,k,l,m}\dash \tau_{j,k,lm}\dash A_{j}$ is a path in $\Int$.

If $p=i$ then $\stab(C_{i,j,k,l,m})\cap\stab(A_{p})=(i_{v_{1}}\times\dots\times i_{v_{n-5}})\rtimes\Phi$ (where $\{v_{1},\dots,v_{n-5}\}=\{1,\dots,n\}-\{i,j,k,l,m\}$).
This is contained in the stabiliser of $\sigma_{i,jk,lm}$, which is a graph that collapses to both $A_{i}$ and $C_{i,j,k,l,m}$, hence $C_{i,j,k,l,m}\dash \sigma_{i,jk,lm}\dash A_{i}$ is a path in $\Int$.

So if $C_{i,j,k,l,m}$ and $A_{p}$ are points in $\Int$ for any $i,j,k,l,m,p\in\{1,\dots,n\}$ then there is a path in $\Int$ connecting them.
\end{proof}

\begin{lemma}[Case 5]\label{lemma pw int case 5}
If $C_{i,j,k,l,m}$ and $B_{p,q,r}$ are vertices in $\aleph_{1}\cap\aleph_{2}$, then there is a path in $\aleph_{1}\cap\aleph_{2}$ from $C_{i,j,k,l,m}$ to $B_{p,q,r}$.
\end{lemma}

\begin{proof}
To satisfy $\stab(C_{i,j,k,l,m})\cap\stab(B_{p,q,r})\ne\Phi$, we require $\{p,q\}\cap\{i,j,l\}\ne\emptyset$.

Suppose $p=j$ and $q=l$. If $r=k$ then $\stab(C_{i,j,k,l,m})\cap\stab(B_{p,q,r})=\Phi$.
If $r=m$ then $\stab(C_{i,j,k,l,m})\cap\stab(B_{p,q,r})=(j_{k}\times l_{m})\rtimes\Phi=\stab(\varepsilon(j,k,l,m)$, and $\varepsilon_{j,k,l,m}$ collapses to both $B_{j,l,m}$ and $C_{i,j,k,l,m}$, so these points are connected by a path in $\Int$.
If $r\not\in\{k,m\}$ then $\stab(C_{i,j,k,l,m})\cap\stab(B_{p,q,r})=j_{k}\rtimes\Phi\subset\stab(A_{j})$. Thus $A_{j}\in\Int$.
By Case 2 (Lemma \ref{lemma pw int case 2})there is a path in $\Int$ from $B_{j,l,r}$ to $A_{j}$, and by Case 4 (Lemma \ref{lemma pw int case 4}) there is a path in $\Int$ from $A_{j}$ to $C_{i,j,k,l,m}$.

By symmetry of $C_{i,j,k,l,m}=C_{i,l,m,j,k}$, we do not need to consider the case $p=l$ and $q=j$.

We may now assume $\{p,q\}\ne\{j,l\}$. Again by the symmetry of $C$-vertices, we need only consider $\{p,q\}\cap\{i,j\}\ne\emptyset$.

Suppose $p=i$ and $q=j$ (or $q=l$ by symmetry).
If $r=k$ then $\stab(C_{i,j,k,l,m})\cap\stab(B_{p,q,r})=(j_{k}\times i_{v_{1}}\times\dots\times i_{v_{n-5}})\rtimes\Phi\subseteq\stab(\delta_{i,j,k,lm})$.
Since $\delta_{i,j,k,lm}$ collapses to both $B_{i,j,k}$ and $C_{i,j,k,l,m}$ then this provides a path in $\Int$.
If $r\ne k$ then the only operating factor in $\stab(C_{i,j,k,l,m})\cap\stab(B_{p,q,r})$ is $i$, hence $A_{i}\in\Int$ and thus by Cases 2 and 4 there is a path from $C_{i,j,k,l,m}$ to $B_{i,j,r}$ in $\Int$.

Suppose $q=i$ and $p=j$ (or $p=l$ by symmetry).
If $r=k$ then $\stab(C_{i,j,k,l,m})\cap\stab(B_{p,q,r})=\Phi$ and $\alpha\in\Int$.
If $r\in\{l,m\}$ then $\stab(C_{i,j,k,l,m})\cap\stab(B_{p,q,r})=j_{k}\rtimes\Phi\subset\stab(A_{j})$, hence we are reduced to Cases 2 and 4.
If $r\not\in\{k,l,m\}$ then $\stab(C_{i,j,k,l,m})\cap\stab(B_{p,q,r})=(i_{r}\times j_{k})\rtimes\Phi\subset\stab(B_{i,j,k}$.
In the previous paragraph we showed there is a path in $\Int$ from $C_{i,j,k,l,m}$ to $B_{i,j,k}$ via a $\delta$-graph, and by Case 3 (Lemma \ref{lemma pw int case 3}), there is a path in $\Int$ from $B_{i,j,k}$ to $B_{j,i,r}=B_{p,q,r}$.

If $|\{p,q\}\cap\{i,j,l\}|=1$ then $\stab(C_{i,j,k,l,m})\cap\stab(B_{p,q,r})$ has at most one operating factor (with at most $n-6$ dependent factors). So either $\alpha\in\Int$, or there is some $A$-vertex in $\Int$, and we are reduced to Cases 2 and 4.

So if $C_{i,j,k,l,m}$ and $B_{p,q,r}$ are points in $\Int$ for any $i,j,k,l,m,p,q,r\in\{1,\dots,n\}$ then there is a path in $\Int$ connecting them.
\end{proof}

\begin{lemma}[Case 6]\label{lemma pw int case 6}
If $C_{i,j,k,l,m}$ and $C_{p,q,r,s,t}$ are vertices in $\aleph_{1}\cap\aleph_{2}$, then there is a path in $\aleph_{1}\cap\aleph_{2}$ from $C_{i,j,k,l,m}$ to $C_{p,q,r,s,t}$ if and only if $r\ne k \implies t\ne m$.
\end{lemma}

\begin{proof}
Suppose we have $C_{i,j,k,l,m}\in\Int$ and $C_{p,q,r,s,t}\in\Int$.

If $\{p,q,s\}\cap\{i,j,l\}=\emptyset$ then $\stab(C_{i,j,k,l,m})\cap\stab(C_{p,q,r,s,t})=\Phi$ and $\alpha\in\Int$.

If $|\{p,q,s\}\cap\{i,j,l\}|=1$ then we have at most one operating factor in $\stab(C_{i,j,k,l,m})\cap\stab(C_{p,q,r,s,t})$ (with at most $n-7$ dependent factors), so either $\alpha\in\Int$ (and we are done) or there is some $A$-vertex in $\Int$, and we are reduced to Case 4 (Lemma \ref{lemma pw int case 4}).

Suppose $|\{p,q,s\}\cap\{i,j,l\}|=2$. Then $\stab(C_{i,j,k,l,m})\cap\stab(C_{p,q,r,s,t})$ has at most two operating factors.
If $p\ne i$ then all of these operating factors have at most one dependent factor.
Since any $B$-vertex has two operating factors, each with at least one dependent, then $\stab(C_{i,j,k,l,m})\cap\stab(C_{p,q,r,s,t})\subset \stab(B)$ for some $B$-vertex.
Then we are reduced to Case 5 (Lemma \ref{lemma pw int case 5}).
If $p=i$ then we have at most one operating factor with at most $n-5$ dependent factors, and the possible other operating factor has at most one dependent factor. So again there is some $B$-vertex in $\Int$, and by Case 5, there must be some path between our two $C$-vertices.

Now suppose $\{p,q,s\}=\{i,j,l\}$.
If $p=i$ then by symmetry of $C$-vertices we may assume $q=j$ and $s=l$.
If in addition we have $r=k$ and $t=m$ then $C_{p,q,r,s,t}=C_{i,j,k,l,m}$. So suppose $\{r,t\}\ne\{k,m\}$.
If $|\{r,t\}\cap\{k,m\}|=1$ then either $B_{i,j,k}$ or $B_{i,l,m}$ is in $\Int$, which reduces us to Case 5.
If $|\{r,t\}\cap\{k,m\}|=0$ then $A_{i}\in\Int$, and we are reduced to Case 4.

Finally, suppose $p=j$, $q=i$, and $s=l$.
By permuting indices in accordance with the symmetries of $C_{i,j,k,l,m}$ and $C_{p,q,r,s,t}$, this covers all cases where $\{p,q,s\}=\{i,j,l\}$ and $p\ne i$.
If $r=k$ then $\stab(C_{i,j,k,l,m})\cap\stab(C_{p,q,r,s,t})\subseteq l_{t}\rtimes\Phi\subset\stab(A_{l})$, so by Case 4 we are done.
If $r\ne k$ and $t\ne m$ then $\stab(C_{i,j,k,l,m})\cap\stab(C_{p,q,r,s,t})\subseteq(j_{k}\times i_{r})\rtimes\Phi\subset\stab(B_{i,j,k})$, so by Case 5 we are done.

A problem arises when $r\ne k$ but $t=m$.
Then $\stab(C_{i,j,k,l,m})\cap\stab(C_{p,q,r,s,t})=\stab(C_{i,j,k,l,m})\cap\stab(C_{j,i,r,l,m})=(j_{k}\times l_{m}\times i_{r})\rtimes\Phi$. The only graph $T$ in our complex (besides $C_{i,j,k,l,m}$ and $C_{j,i,r,l,m}$) with $(j_{k}\times l_{m}\times i_{r})\rtimes\Phi\subseteq\stab(T)$ is $C_{l,i,r,j,k}$.
So in this case $\Int$ consists of three distinct non-adjacent points.
Note that this is the only case where $\Int$ is not path-connected.

So if $C_{i,j,k,l,m}$ and $C_{p,q,r,s,t}$ are points in $\Int$ for any $i,j,k,l,m,p,q,r,s,t\in\{1,\dots,n\}$, and we don't have that $r\ne k$ and $t=m$, then there is a path in $\Int$ connecting $C_{p,q,r,s,t}$ to $C_{i,j,k,l,m}$.
\end{proof}

Having dealt with all six cases, we may now conclude:

\begin{prop}\label{pw int pc case not C}
If $\aleph_{1}$ and $\aleph_{2}$ are two domains with $\Int\ne\emptyset$ such that $\Int$ contains some vertex which is \textbf{not} of the form $C_{i,j,k,l,m}$, then $\aleph_{1}\cap\aleph_{2}$ is path-connected.
\end{prop}

\begin{proof}
Note that in the proof of Lemma \ref{lemma pw int case 6} we showed that if $C_{i,j,k,l,m}\in\Int$ and $C_{p,q,r,s,t}\in\Int$ then there exists some non-$C$ vertex in $\Int$ if and only if $r\ne k \implies t\ne m$.
The result then follows from Lemmas \ref{lemma pw int case 1}, \ref{lemma pw int case 2}, \ref{lemma pw int case 3}, \ref{lemma pw int case 4}, \ref{lemma pw int case 5}, and \ref{lemma pw int case 6}. 
\end{proof}

\begin{rem}
Note that this statement is not saying that $\aleph_{1}\cap\aleph_{2}$ cannot contain a vertex $C_{i,j,k,l,m}$ if it is to be path connected, just that it must also contain some other vertex as well which is not a $C$-vertex.
\end{rem}

While it is not ideal that such an intersection containing only $C$-vertices is not path-connected, this can be handled via the following:

\begin{prop}\label{pw int pc case C}
If $\aleph_{1}$ and $\aleph_{2}$ are two domains with $\Int\ne\emptyset$ such that $\Int$ contains \textbf{only} vertices of the form $C_{i,j,k,l,m}$, then $\Int=\{C_{i,j,k,l,m}, C_{j,l,m,i,p}, C_{l,i,p,j,k}\}$ for some (distinct) $i,j,k,l,m,p\in\{1,\dots,n\}$.
Moreover, there exists a domain $\aleph_{3}$ with $\Int\subset\aleph_{3}$, such that $\aleph_{1}\cup\aleph_{2}\cup\aleph_{3}$ is simply connected.
\end{prop}

\begin{proof}
As noted in Case 6 (Lemma \ref{lemma pw int case 6}) above, if $\Int$ does not contain any vertices except $C$-vertices, then there exist some distinct $i,j,k,l,m,p\in\{1,\dots,n\}$ such that 
$\Int=\{C_{i,j,k,l,m}, C_{j,l,m,i,p}, C_{l,i,p,j,k}\}$.
Further, we have that $\aleph_{2}=\aleph_{1}\psi$ for some $\psi\in(i_{p}\times j_{k}\times l_{m})\rtimes\Phi$ (and $\psi\not\in (i_{p}\times j_{k})\rtimes\Phi \cup (j_{k}\times l_{m})\rtimes\Phi \cup (l_{m}\times i_{p})\rtimes\Phi$).
That is, $\psi=(G_{p},i_{o})(G_{k},j_{0})(G_{m},l_{0})\phi$ for some $i_{0}\in G_{i}$, $j_{0}\in G_{j}$, $l_{0}\in G_{l}$, and $\phi\in\Phi$.
Define $\aleph_{3}$ to be $\aleph_{1}(G_{p},i_{0})$.
Note that $(G_{p},i_{0})\in i_{p}\rtimes\Phi=\stab(A_{i})$, so we have $A_{i}\in\aleph_{1}\cap\aleph_{3}$.
Then $\aleph_{2}=\aleph_{1}(G_{p},i_{o})(G_{k},j_{0})(G_{m},l_{0})\phi=\aleph_{3}(G_{k},j_{0})(G_{m},l_{0})\phi$.
Note that $(G_{k},j_{0})(G_{m},l_{0})\phi\in(j_{k}\times l_{m})\rtimes\Phi\subseteq\stab(B_{j,l,m})$, so we have $B_{j,l,m}\in\aleph_{2}\cap\aleph_{3}$.
Moreover, $(G_{p},i_{o}), (G_{k},j_{0})(G_{m},l_{0})\phi \in (i_{p}\times j_{k}\times l_{m})\rtimes\Phi$. So $C_{i,j,k,l,m}, C_{j,l,m,i,p}, C_{l,i,p,j,k} \in \aleph_{3}$.

By Proposition \ref{pw int pc case not C}, we have that both $\aleph_{1}\cap\aleph_{3}$ and $\aleph_{2}\cap\aleph_{3}$ are path-connected,
and by Theorem \ref{thm fun dom sc}, $\aleph_{1}$, $\aleph_{2}$ and $\aleph_{3}$ are each simply connected.
Then by the Seifert--Van Kampen Theorem (Theorem \ref{s van k}), $\aleph_{1}\cup\aleph_{3}$ and $\aleph_{2}\cup\aleph_{3}$ are both simply connected.
Then we can again apply the Seifert--Van Kampen Theorem to the sets $A=\aleph_{1}\cup\aleph_{3}$ and $B=\aleph_{2}\cup\aleph_{3}$.
Since we have that $\pi_{1}(A)=\pi_{1}(B)=\{1\}$, and $A\cap B=(\Int)\cup\aleph_{3}=\aleph_{3}$ is path-connected by Corollary \ref{cor copies of the fun dom are pc},
then we get that $\pi_{1}(\aleph_{1}\cup\aleph_{2}\cup\aleph_{3})=\pi_{1}(A\cup B)=\{1\}$.
That is, $\aleph_{1}\cup\aleph_{2}\cup\aleph_{3}$ is simply connected.
\end{proof}

In plain language, this means if there is an edge in the Space of Domains which does not represent a simply connected subset of the complex $\mathcal{C}_{n}$, then it must be in the boundary of a 2-cell which \textbf{does} represent a simply connected subset of $\mathcal{C}_{n}$.

\begin{cor}\label{cor pw int has sc nbhd}
Let $\aleph_{1}$ and $\aleph_{2}$ be two domains so that $\aleph_{1}\cap\aleph_{2}\ne\emptyset$. Then there exists $U\subset\mathcal{C}_{n}$ with $\aleph_{1}\cup\aleph_{2}\subseteq U$ such that $U$ is simply connected.
\end{cor}

\begin{proof}
Note that any pair of domains with non-empty intersection fall under precisely one of Proposition \ref{pw int pc case not C} or Proposition \ref{pw int pc case C}.
In the second case, we take $U=\aleph_{1}\cup\aleph_{2}\cup\aleph_{3}$ with $\aleph_{3}$ as in Proposition \ref{pw int pc case C}.
In the first case, we take $U=\aleph_{1}\cup\aleph_{2}$.
By Theorem \ref{thm fun dom sc} we have that each of $\aleph_{1}$ and $\aleph_{2}$ is simply connected, and since in this case $\aleph_{1}\cap\aleph_{2}$ is path connected, then by the Seifert--van Kampen Theorem (Theorem \ref{s van k}), we must have that $\aleph_{1}\cup\aleph_{2}$ is simply connected.
\end{proof}


\subsection{A Map from the Space of Domains to the Complex $\mathcal{C}_{n}$}\label{section map SoD to Cn}

In this subsection, we will define a (continuous) map from the Space of Domains to the complex $\mathcal{C}_{n}$.
This will allow us to conclude that simple connectivity of $\mathcal{C}_{n}$ can be deduced from simple connectivity of the Space of Domains, as desired.

\begin{defn}\label{defn map F}
We define a map $F$ from the 1-skeleton of the Space of Domains (equivalently, from the Graph of Domains) to $\mathcal{C}_{n}$ as follows:
\begin{itemize}
\item A vertex $\aleph$ in the 0-skeleton of the Space of Domains is mapped under $F$ to the $\alpha$-graph contained in the domain $\aleph\subseteq\mathcal{C}_{n}$.
\item Let $\aleph_{1}\dash \aleph_{2}$ be some edge in the Space of Domains, and set $\alpha_{1}=F(\aleph_{1})$ and $\alpha_{2}=F(\aleph_{2})$.
Choose an edge path $\lambda_{12}:[0,1]\to\aleph_{1}\cup\aleph_{2}\subset\mathcal{C}_{n}$ with $\lambda_{12}(0)=\alpha_{1}$ and $\lambda_{12}(1)=\alpha_{2}$ and let $|\lambda_{12}|$ be its image.
We define $F(\aleph_{1}\dash \aleph_{2}):=|\lambda_{12}|$.
\item Given $\lambda_{12}$, we require that $\lambda_{21}$ (the chosen path in $\aleph_{1}\cup\aleph_{2}$ from $\alpha_{2}$ to $\alpha_{1}$) be defined by $\lambda_{21}(t):=\lambda_{12}(1-t)$ for all $t\in[0,1]$; that is, that $F(\aleph_{1}\dash \aleph_{2})=F(\aleph_{2}\dash \aleph_{1})$.
\end{itemize}
\end{defn}

\begin{obs}
Note that by Definition \ref{Space of Domains defn}, if $\aleph_{1}\dash \aleph_{2}$ is an edge in the Space of Domains, then we have that $\aleph_{1}\cap\aleph_{2}\ne\emptyset$ as a subset of $\mathcal{C}_{n}$.
By Lemma \ref{lemma fun dom pc}, each of $\aleph_{1}$ and $\aleph_{2}$ is a path-connected subset of $\mathcal{C}_{n}$, hence so too is $\aleph_{1}\cup\aleph_{2}\subset\mathcal{C}_{n}$.
Thus there must exist some path in $\aleph_{1}\cup\aleph_{2}$ connecting $F(\aleph_{1})$ and $F(\aleph_{2})$.
\end{obs}

\begin{rem}
If one wishes to be more explicit, given an edge $\aleph_{1}\dash \aleph_{2}$, one may choose $T\in\aleph_{1}\cap\aleph_{2}\subset\mathcal{C}_{n}$ and $\psi,\varphi\in\outs(G)$ with $\aleph_{1}=\mathcal{D}_{n}\cdot\psi$ and $\aleph_{2}=\mathcal{D}_{n}\cdot\varphi$ so that $\psi^{-1}\varphi\in\stab(T)$, where $\mathcal{D}_{n}$ is the fundamental domain.
Denote the path in Table \ref{paths in fun dom} from $T\cdot\psi^{-1}=T\cdot\varphi^{-1}$ to $\alpha_{0}$ by $p$, where $\alpha_{0}$ is the unique $\alpha$-graph contained in $\mathcal{D}_{n}$, and set $\lambda_{12}:=(p^{-1}\cdot\psi)(p\cdot\varphi)$.
Then $F(\aleph_{1}\dash \aleph_{2})=|p|\cdot\psi\cup|p|\cdot\varphi$.
It may well be that this construction leads to an $\outs(G)$-equivariant map, but one would need to carefully handle the choice of $T\in\aleph_{1}\cap\aleph_{2}$ to make it so.
\end{rem}

\begin{lemma}\label{loops homotope to image of F}
Any loop in $\mathcal{C}_{n}$ is homotopic to a loop which is the image under $F$ of a loop in the Space of Domains.
\end{lemma}

\begin{proof}
Let $\lambda'=T'_{0}\dash T'_{1}\dash \cdots\dash T'_{m'}\dash T'_{0}$ be a loop in the complex $\mathcal{C}_{n}$.
By Corollary \ref{cor Cn pc}, we can write down a path $p$ in $\mathcal{C}_{n}$ from $T'_{0}$ to the $\alpha$ graph in our fundamental domain (which we call $\alpha_{0}$).
By setting $\lambda=p^{-1}\lambda' p$ we now have a based loop in $\mathcal{C}_{n}$ (i.e. a loop containing the `basepoint' $\alpha_{0}$) which is homotopic to $\lambda'$. Say $\lambda=T_{0}\dash T_{1}\dash \cdots\dash T_{m}$, where $T_{m}=T_{0}=\alpha_{0}$.
\smallskip

We will now describe how to associate a domain $\aleph_{i}$ to each vertex $T_{i}$ in $\lambda$ with $T_{i}\in\aleph_{i}$.
This will allow us to construct a path in the Space of Domains whose image under ${F}$ we will show to be homotopic to $\lambda$.

First set $\aleph_{0}$ to be $\mathcal{D}_{n}$, the fundamental domain.
If $T_{i}$ is associated to $\aleph_{i}$ and $T_{i+1}\in\aleph_{i}$, we set $\aleph_{i+1}:=\aleph_{i}$. In particular, this is the case whenever $T_{i+1}$ is a collapse of $T_{i}$.
Note that for any edge $S\dash T$ in $\mathcal{C}_{n}$ we have that either $S$ is a collapse of $T$, or $T$ is a collapse of $S$.
Now suppose we have a domain $\aleph_{i}\ni T_{i}$ and $T_{i+1}\not\in\aleph_{i}$. We must then have that $T_{i}$ is a collapse of $T_{i+1}$.
Choose any domain $\aleph_{i+1}$ containing $T_{i+1}$. Then $T_{i}\in\aleph_{i+1}$, hence $T_{i}\dash T_{i+1}\subseteq\aleph_{i+1}$, and $T_{i}\in\aleph_{i}\cap\aleph_{i+1}$.

For a given domain $\aleph_{i}$, let $\alpha_{i}\in\mathcal{C}_{n}$ be the unique $\alpha$-graph contained in $\aleph_{i}$.
Since $\aleph_{i}\cap\aleph_{i+1}\ne\emptyset$, then either $\aleph_{i}=\aleph_{i+1}$ or $\aleph_{i}\dash \aleph_{i+1}$ is an edge in the Space of Domains.
We define paths $\mu_{i,i+1}$ as follows:
if $\aleph_{i}=\aleph_{i+1}$, set $\mu_{i,i+1}$ to be the constant path at $\alpha_{i}=\alpha_{i+1}$;
otherwise, set $\mu_{i,i+1}$ to be the path $\lambda_{i,i+1}$ from $\alpha_{i}$ to $\alpha_{i+1}$ such that $F(\aleph_{i}\dash \aleph_{i+1})=|\lambda_{i,i+1}|$.

Then the concatenation $\mu:=\mu_{0,1}\mu_{1,2}\dots\mu_{m-1,m}$ is equal to the concatenation \\ \noindent $F(\aleph_{\sigma(0)}\dash \aleph_{\sigma(1)})F(\aleph_{\sigma(1)}\dash \aleph_{\sigma(2)})\dots F(\aleph_{\sigma(k)}\dash \aleph_{\sigma(0)})$
which is \\ \noindent $F(\aleph_{\sigma(0)}\dash \aleph_{\sigma(1)}\dash \aleph_{\sigma(2)}\dash \cdots\dash \aleph_{\sigma(k)}\dash \aleph_{\sigma(0)})$,
where $\sigma(0)=0$, and given $\sigma(i)$, $\sigma(i+1)$ is the next index such that $\alpha_{\sigma(i+1)}\ne\alpha_{\sigma(i)}$.

We now prove that $\lambda$ (and hence also $\lambda'$) is homotopic in $\mathcal{C}_{n}$ to $\mu$.

Given $T_{i}$ and its associated domain $\aleph_{i}$ with graph $\alpha_{i}$, let $\nu_{i}$ be a path contained in $\aleph_{i}$ from $T_{i}$ to $\alpha_{i}$ (to be explicit, one may take the correct $\out(G)$-image of the relevant path listed in Table \ref{paths in fun dom}).
Additionally, let $e_{i}$ denote the (oriented) edge (path) $T_{i}\dash T_{i+1}$.
By Corollary \ref{cor pw int has sc nbhd}, there exists a simply connected neighbourhood $U_{i}\subset\mathcal{C}_{n}$ containing $\aleph_{i}\cup\aleph_{i+1}$ (if $\aleph_{i+1}=\aleph_{i}$, instead set $U_{i}=\aleph_{i}$, and note that by Theorem \ref{thm fun dom sc} this is simply connected).

Then the loop 
$\nu_{i}\mu_{i}\nu_{i+1}^{-1}\overbar{e_{i}}$ is contained in $U_{i}$,
hence said loop is contractible in $\mathcal{C}_{n}$. 
In other words, the edge $T_{i}\dash T_{i+1}$ is homotopic in $\mathcal{C}_{n}$ to the path
$\nu_{i}\mu_{i}\nu_{i+1}^{-1}$.
Since this holds for all $i$, it follows that $\lambda$ is homotopic in $\mathcal{C}_{n}$ to $\mu$, the image under $F$ of a loop in the Space of Domains.
\end{proof}

\begin{lemma}\label{lemma map F extends}
The map $F$ described in Definition \ref{defn map F} extends to a continuous map $\bar{F}$ from the Space of Domains to $\mathcal{C}_{n}$.
\end{lemma}

\begin{proof}
Let $[\aleph_{1},\aleph_{2},\aleph_{3}]$ be a face (2-cell) in the Space of Domains. Then as subsets of $\mathcal{C}_{n}$, we have that $\aleph_{1}\cap\aleph_{2}\cap\aleph_{3}\ne\emptyset$.
Let $\lambda_{12}$, $\lambda_{23}$, and $\lambda_{31}$ be paths such that each $\lambda_{ij}$ is an edge path in $\aleph_{i}\cup\aleph_{j}$ from $\alpha_{i}$ to $\alpha_{j}$ with $F(\aleph_{i}\dash \aleph_{j})=|\lambda_{ij}|$ (for $i,j\in\{1,2,3\}$ distinct).
To show that $F$ extends to a map $\bar{F}$, it will suffice to show that the concatenated path $\lambda_{12}\lambda_{23}\lambda_{31}$ is the boundary of some simply connected subset of $\mathcal{C}_{n}$, that is, that the loop $\lambda_{12}\lambda_{23}\lambda_{31}$ is contractible in $\mathcal{C}_{n}$.
We illustrate the following process in Figure \ref{fig map extends}.
\\ \noindent
\textbf{Step 1:}
Since $[\aleph_{1},\aleph_{2},\aleph_{3}]$ is a 2-cell in the Space of Domains, then by Definition \ref{Space of Domains defn}, there must exist some point $x\in\aleph_{1}\cap\aleph_{2}\cap\aleph_{3}\subseteq\mathcal{C}_{n}$.
Let $\lambda_{03}$ be a path in $\aleph_{3}\subseteq\mathcal{C}_{n}$ from $x$ to $\alpha_{3}$ and let $\lambda_{30}=\overbar{\lambda_{03}}$ be its reverse path (i.e. the same set of edges, but read from $\alpha_{3}$ to $x$).
By Theorem \ref{thm fun dom sc}, each domain is simply connected, so $\lambda_{12}\lambda_{23}\lambda_{30}\lambda_{03}\lambda_{31}$ is a path in $\mathcal{C}_{n}$ which is homotopic to the image $\lambda_{12}\lambda_{23}\lambda_{31}$ of the loop $\aleph_{1}\dash \aleph_{2}\dash \aleph_{3}\dash \aleph_{1}$ in the Space of Domains.
\\ \noindent
\textbf{Step 2:}
By Corollary \ref{cor pw int has sc nbhd}, there exists a simply connected neighbourhood in $\mathcal{C}_{n}$ containing $\aleph_{2}\cup\aleph_{3}$. Thus the subpath $\lambda_{23}\lambda_{30}$ from $\alpha_{2}$ to $x$ is contained within a simply connected subset of $\mathcal{C}_{n}$ containing $\aleph_{2}$, hence is homotopic to some path $\lambda_{2}$ from $\alpha_{2}$ to $x$ fully contained in $\aleph_{2}$.
\\ \noindent
\textbf{Step 3:}
By the same reasoning as in Step 2, the subpath $\lambda_{03}\lambda_{31}$ from $x$ to $\alpha_{1}$ is homotopic in $\mathcal{C}_{n}$ to some path $\lambda_{1}$ from $x$ to $\alpha_{1}$ fully contained in $\aleph_{1}$.
We have so far shown that $\lambda_{12}\lambda_{23}\lambda_{31}$ is homotopic in $\mathcal{C}_{n}$ to the loop $\lambda_{12}\lambda_{2}\lambda_{1}$.
\\ \noindent
\textbf{Step 4:}
We have that $\lambda_{12}\lambda_{2}\lambda_{1}$ is a loop contained in $\aleph_{1}\cup\aleph_{2}$.
Again, by Corollary \ref{cor pw int has sc nbhd}, the loop $\lambda_{12}\lambda_{2}\lambda_{1}$ is contained within a simply connected subset of $\mathcal{C}_{n}$, hence it must be contractible.
\end{proof}

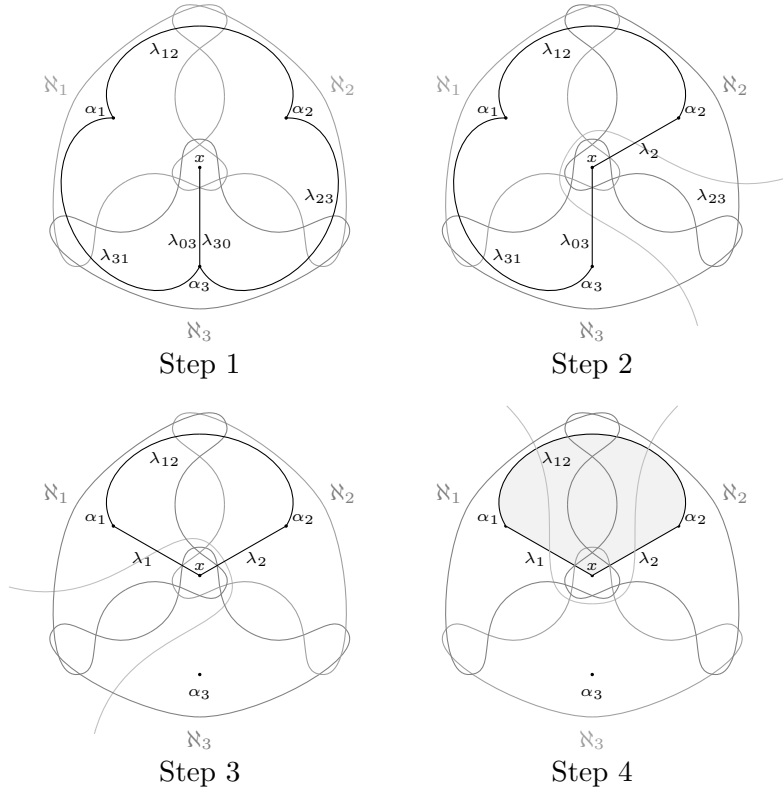
\begin{figure}[h!]
\centering
\begin{tikzpicture}[scale=1.5]
\draw[fill] (0,1.25) circle [radius=0.01cm]; 
\draw[fill] (0,0.375) circle [radius=0.01cm]; 
\draw[fill] (0.75775,1.6875) circle [radius=0.01cm]; 
\draw[fill] (-0.75775,1.6875) circle [radius=0.01cm]; 
\draw[white] (-0.75,2.75) .. controls (0,2) and (-0.75,1) .. (0,1);
\draw[white] (0.75,2.75) .. controls (0,2) and (0.75,1) .. (0,1);
\draw[white] (-0.924,-0.1495) .. controls (-0.6495,0.875) and (0.5915,0.7255) .. (0.2165,1.375);
\draw[white] (-1.674,1.1495) .. controls (-0.6495,0.875) and (-0.1585,2.0245) .. (0.2165,1.375);
\draw[white] (0.924,-0.1495) .. controls (0.6495,0.875) and (-0.5915,0.7255) .. (-0.2165,1.375);
\draw[white] (1.674,1.1495) .. controls (0.6495,0.875) and (0.1585,2.0245) .. (-0.2165,1.375);
\draw[gray!75] (0.6495,1.125) .. controls (1.0825,0.875) and (0.7915,0.3791) .. (1.0825,0.375) .. controls (1.3745,0.3693) and (1.3325,1.442) .. (1.0825,1.875);
\draw[gray!75] (0.6495,1.125) .. controls (0.2165,1.375) and (-0.0665,0.8625) .. (-0.2165,1.125) .. controls (-0.3665,1.3848) and (0.2165,1.375) .. (0.2165,1.875);
\draw[gray!75] (0.2165,1.875) .. controls (0.2165,2.375) and (-0.3585,2.3709) .. (-0.2165,2.625) .. controls (-0.0755,2.8807) and (0.8325,2.308) .. (1.0825,1.875);
\draw[gray!75] (-0.6495,1.125) .. controls (-1.0825,0.875) and (-0.7915,0.3791) .. (-1.0825,0.375) .. controls (-1.3745,0.3693) and (-1.3325,1.442) .. (-1.0825,1.875);
\draw[gray!75] (-0.6495,1.125) .. controls (-0.2165,1.375) and (0.0665,0.8625) .. (0.2165,1.125) .. controls (0.3665,1.3848) and (-0.2165,1.375) .. (-0.2165,1.875);
\draw[gray!75] (-0.2165,1.875) .. controls (-0.2165,2.375) and (0.3585,2.3709) .. (0.2165,2.625) .. controls (0.0755,2.8807) and (-0.8325,2.308) .. (-1.0825,1.875);
\draw[gray] (-0.433,0.75) .. controls (-0.866,0.5) and (-1.15,1) .. (-1.299,0.75) .. controls (-1.45,0.5) and (-0.5,0) .. (0,0);
\draw[gray] (-0.433,0.75) .. controls (0,1) and (-0.3,1.5) .. (0,1.5) .. controls (0.3,1.5) and (0,1) .. (0.433,0.75);
\draw[gray] (0.433,0.75) .. controls (0.866,0.5) and (1.15,1) .. (1.299,0.75) .. controls (1.45,0.5) and (0.5,0) .. (0,0);
\node[gray] at (0,-0.2) {\footnotesize$\aleph_{3}$}; 
\node[gray!75] at (1.2557,1.975) {\footnotesize$\aleph_{2}$}; 
\node[gray!75] at (-1.2557,1.975) {\footnotesize$\aleph_{1}$}; 
\node at (0,0.2) {\tiny$\alpha_{3}$}; 
\node at (0.9093,1.75) {\tiny$\alpha_{2}$};  
\node at (-0.9093,1.75) {\tiny$\alpha_{1}$}; 
\node at (0,1.325) {\tiny$x$}; 
\node at (1.05,1) {\tiny$\lambda_{23}$}; 
\node at (-0.3085,2.2843) {\tiny$\lambda_{12}$}; 
\node at (-0.7415,0.4657) {\tiny$\lambda_{31}$}; 
\node at (0.15,0.6) {\tiny$\lambda_{30}$}; 
\node at (-0.15,0.6) {\tiny$\lambda_{03}$}; 
\draw (0,1.25) -- (0,0.375); 
\draw (0,0.375) .. controls (0.2,0) and (0.80475,0.144) .. (1.0825,0.625); 
\draw (0,0.375) .. controls (-0.2,0) and (-0.80475,0.144) .. (-1.0825,0.625); 
\draw (0.75775,1.6875) .. controls (0.9825,2.0482) and (0.555,2.5) .. (0,2.5); 
\draw (0.75775,1.6875) .. controls (1.1825,1.7018) and (1.360,1.106) .. (1.0825,0.625); 
\draw (-0.75775,1.6875) .. controls (-0.9825,2.0482) and (-0.555,2.5) .. (0,2.5); 
\draw (-0.75775,1.6875) .. controls (-1.1825,1.7018) and (-1.360,1.106) .. (-1.0825,0.625); 
%
\node at (0,-0.5) {Step 1};
\end{tikzpicture}
\begin{tikzpicture}[scale=1.5]
\draw[fill] (0,1.25) circle [radius=0.01cm]; 
\draw[fill] (0,0.375) circle [radius=0.01cm]; 
\draw[fill] (0.75775,1.6875) circle [radius=0.01cm]; 
\draw[fill] (-0.75775,1.6875) circle [radius=0.01cm]; 
\draw[white] (-0.75,2.75) .. controls (0,2) and (-0.75,1) .. (0,1);
\draw[white] (0.75,2.75) .. controls (0,2) and (0.75,1) .. (0,1);
\draw[white] (-0.924,-0.1495) .. controls (-0.6495,0.875) and (0.5915,0.7255) .. (0.2165,1.375);
\draw[white] (-1.674,1.1495) .. controls (-0.6495,0.875) and (-0.1585,2.0245) .. (0.2165,1.375);
\draw[gray!50] (0.924,-0.1495) .. controls (0.6495,0.875) and (-0.5915,0.7255) .. (-0.2165,1.375);
\draw[gray!50] (1.674,1.1495) .. controls (0.6495,0.875) and (0.1585,2.0245) .. (-0.2165,1.375);
\draw[gray] (-0.433,0.75) .. controls (-0.866,0.5) and (-1.15,1) .. (-1.299,0.75) .. controls (-1.45,0.5) and (-0.5,0) .. (0,0);
\draw[gray] (-0.433,0.75) .. controls (0,1) and (-0.3,1.5) .. (0,1.5) .. controls (0.3,1.5) and (0,1) .. (0.433,0.75);
\draw[gray] (0.433,0.75) .. controls (0.866,0.5) and (1.15,1) .. (1.299,0.75) .. controls (1.45,0.5) and (0.5,0) .. (0,0);
\draw[gray] (0.6495,1.125) .. controls (1.0825,0.875) and (0.7915,0.3791) .. (1.0825,0.375) .. controls (1.3745,0.3693) and (1.3325,1.442) .. (1.0825,1.875);
\draw[gray] (0.6495,1.125) .. controls (0.2165,1.375) and (-0.0665,0.8625) .. (-0.2165,1.125) .. controls (-0.3665,1.3848) and (0.2165,1.375) .. (0.2165,1.875);
\draw[gray] (0.2165,1.875) .. controls (0.2165,2.375) and (-0.3585,2.3709) .. (-0.2165,2.625) .. controls (-0.0755,2.8807) and (0.8325,2.308) .. (1.0825,1.875);
\draw[gray!75] (-0.6495,1.125) .. controls (-1.0825,0.875) and (-0.7915,0.3791) .. (-1.0825,0.375) .. controls (-1.3745,0.3693) and (-1.3325,1.442) .. (-1.0825,1.875);
\draw[gray!75] (-0.6495,1.125) .. controls (-0.2165,1.375) and (0.0665,0.8625) .. (0.2165,1.125) .. controls (0.3665,1.3848) and (-0.2165,1.375) .. (-0.2165,1.875);
\draw[gray!75] (-0.2165,1.875) .. controls (-0.2165,2.375) and (0.3585,2.3709) .. (0.2165,2.625) .. controls (0.0755,2.8807) and (-0.8325,2.308) .. (-1.0825,1.875);
\node[gray] at (0,-0.2) {\footnotesize$\aleph_{3}$}; 
\node[gray] at (1.2557,1.975) {\footnotesize$\aleph_{2}$}; 
\node[gray!75] at (-1.2557,1.975) {\footnotesize$\aleph_{1}$}; 
\node at (0,0.2) {\tiny$\alpha_{3}$}; 
\node at (0.9093,1.75) {\tiny$\alpha_{2}$};  
\node at (-0.9093,1.75) {\tiny$\alpha_{1}$}; 
\node at (0,1.325) {\tiny$x$}; 
\node at (1.05,1) {\tiny$\lambda_{23}$}; 
\node at (-0.3085,2.2843) {\tiny$\lambda_{12}$}; 
\node at (-0.7415,0.4657) {\tiny$\lambda_{31}$}; 
\node at (-0.15,0.6) {\tiny$\lambda_{03}$}; 
\node at (0.5,1.4) {\tiny$\lambda_{2}$}; 
\draw (0,1.25) -- (0,0.375); 
\draw (0,0.375) .. controls (-0.2,0) and (-0.80475,0.144) .. (-1.0825,0.625); 
\draw (0.75775,1.6875) .. controls (0.9825,2.0482) and (0.555,2.5) .. (0,2.5); 
\draw (-0.75775,1.6875) .. controls (-0.9825,2.0482) and (-0.555,2.5) .. (0,2.5); 
\draw (-0.75775,1.6875) .. controls (-1.1825,1.7018) and (-1.360,1.106) .. (-1.0825,0.625); 
\draw (0,1.25) -- (0.75775,1.6875); 
%
\node at (0,-0.5) {Step 2};
\end{tikzpicture}
\begin{tikzpicture}[scale=1.5]
\draw[fill] (0,1.25) circle [radius=0.01cm]; 
\draw[fill] (0,0.375) circle [radius=0.01cm]; 
\draw[fill] (0.75775,1.6875) circle [radius=0.01cm]; 
\draw[fill] (-0.75775,1.6875) circle [radius=0.01cm]; 
\draw[white] (-0.75,2.75) .. controls (0,2) and (-0.75,1) .. (0,1);
\draw[white] (0.75,2.75) .. controls (0,2) and (0.75,1) .. (0,1);
\draw[white] (0.924,-0.1495) .. controls (0.6495,0.875) and (-0.5915,0.7255) .. (-0.2165,1.375);
\draw[white] (1.674,1.1495) .. controls (0.6495,0.875) and (0.1585,2.0245) .. (-0.2165,1.375);
\draw[gray!50] (-0.924,-0.1495) .. controls (-0.6495,0.875) and (0.5915,0.7255) .. (0.2165,1.375);
\draw[gray!50] (-1.674,1.1495) .. controls (-0.6495,0.875) and (-0.1585,2.0245) .. (0.2165,1.375);
\draw[gray] (-0.433,0.75) .. controls (-0.866,0.5) and (-1.15,1) .. (-1.299,0.75) .. controls (-1.45,0.5) and (-0.5,0) .. (0,0);
\draw[gray] (-0.433,0.75) .. controls (0,1) and (-0.3,1.5) .. (0,1.5) .. controls (0.3,1.5) and (0,1) .. (0.433,0.75);
\draw[gray] (0.433,0.75) .. controls (0.866,0.5) and (1.15,1) .. (1.299,0.75) .. controls (1.45,0.5) and (0.5,0) .. (0,0);
\draw[gray!75] (0.6495,1.125) .. controls (1.0825,0.875) and (0.7915,0.3791) .. (1.0825,0.375) .. controls (1.3745,0.3693) and (1.3325,1.442) .. (1.0825,1.875);
\draw[gray!75] (0.6495,1.125) .. controls (0.2165,1.375) and (-0.0665,0.8625) .. (-0.2165,1.125) .. controls (-0.3665,1.3848) and (0.2165,1.375) .. (0.2165,1.875);
\draw[gray!75] (0.2165,1.875) .. controls (0.2165,2.375) and (-0.3585,2.3709) .. (-0.2165,2.625) .. controls (-0.0755,2.8807) and (0.8325,2.308) .. (1.0825,1.875);
\draw[gray] (-0.6495,1.125) .. controls (-1.0825,0.875) and (-0.7915,0.3791) .. (-1.0825,0.375) .. controls (-1.3745,0.3693) and (-1.3325,1.442) .. (-1.0825,1.875);
\draw[gray] (-0.6495,1.125) .. controls (-0.2165,1.375) and (0.0665,0.8625) .. (0.2165,1.125) .. controls (0.3665,1.3848) and (-0.2165,1.375) .. (-0.2165,1.875);
\draw[gray] (-0.2165,1.875) .. controls (-0.2165,2.375) and (0.3585,2.3709) .. (0.2165,2.625) .. controls (0.0755,2.8807) and (-0.8325,2.308) .. (-1.0825,1.875);
\node[gray] at (0,-0.2) {\footnotesize$\aleph_{3}$}; 
\node[gray] at (1.2557,1.975) {\footnotesize$\aleph_{2}$}; 
\node[gray] at (-1.2557,1.975) {\footnotesize$\aleph_{1}$}; 
\node at (0,0.2) {\tiny$\alpha_{3}$}; 
\node at (0.9093,1.75) {\tiny$\alpha_{2}$};  
\node at (-0.9093,1.75) {\tiny$\alpha_{1}$}; 
\node at (0,1.325) {\tiny$x$}; 
\node at (-0.3085,2.2843) {\tiny$\lambda_{12}$}; 
\node at (0.5,1.4) {\tiny$\lambda_{2}$}; 
\node at (-0.5,1.4) {\tiny$\lambda_{1}$}; 
\draw (0.75775,1.6875) .. controls (0.9825,2.0482) and (0.555,2.5) .. (0,2.5); 
\draw (-0.75775,1.6875) .. controls (-0.9825,2.0482) and (-0.555,2.5) .. (0,2.5); 
\draw (0,1.25) -- (0.75775,1.6875); 
\draw (0,1.25) -- (-0.75775,1.6875); 
\node at (0,-0.5) {Step 3};
\end{tikzpicture}
\begin{tikzpicture}[scale=1.5]
\draw[fill] (0,1.25) circle [radius=0.01cm]; 
\draw[fill] (0,0.375) circle [radius=0.01cm]; 
\draw[fill] (0.75775,1.6875) circle [radius=0.01cm]; 
\draw[fill] (-0.75775,1.6875) circle [radius=0.01cm]; 
\draw[white] (-0.924,-0.1495) .. controls (-0.6495,0.875) and (0.5915,0.7255) .. (0.2165,1.375);
\draw[white] (-1.674,1.1495) .. controls (-0.6495,0.875) and (-0.1585,2.0245) .. (0.2165,1.375);
\draw[white] (0.924,-0.1495) .. controls (0.6495,0.875) and (-0.5915,0.7255) .. (-0.2165,1.375);
\draw[white] (1.674,1.1495) .. controls (0.6495,0.875) and (0.1585,2.0245) .. (-0.2165,1.375);
\filldraw[fill=gray!10] (-0.75775,1.6875) .. controls (-0.9825,2.0482) and (-0.555,2.5) .. (0,2.5) .. controls (0.555,2.5) and (0.9825,2.0482) .. (0.75775,1.6875) -- (0,1.25) -- (-0.75775,1.6875);
\draw[gray!50] (-0.75,2.75) .. controls (0,2) and (-0.75,1) .. (0,1);
\draw[gray!50] (0.75,2.75) .. controls (0,2) and (0.75,1) .. (0,1);
\draw[gray!75] (-0.433,0.75) .. controls (-0.866,0.5) and (-1.15,1) .. (-1.299,0.75) .. controls (-1.45,0.5) and (-0.5,0) .. (0,0);
\draw[gray!75] (-0.433,0.75) .. controls (0,1) and (-0.3,1.5) .. (0,1.5) .. controls (0.3,1.5) and (0,1) .. (0.433,0.75);
\draw[gray!75] (0.433,0.75) .. controls (0.866,0.5) and (1.15,1) .. (1.299,0.75) .. controls (1.45,0.5) and (0.5,0) .. (0,0);
\draw[gray] (0.6495,1.125) .. controls (1.0825,0.875) and (0.7915,0.3791) .. (1.0825,0.375) .. controls (1.3745,0.3693) and (1.3325,1.442) .. (1.0825,1.875);
\draw[gray] (0.6495,1.125) .. controls (0.2165,1.375) and (-0.0665,0.8625) .. (-0.2165,1.125) .. controls (-0.3665,1.3848) and (0.2165,1.375) .. (0.2165,1.875);
\draw[gray] (0.2165,1.875) .. controls (0.2165,2.375) and (-0.3585,2.3709) .. (-0.2165,2.625) .. controls (-0.0755,2.8807) and (0.8325,2.308) .. (1.0825,1.875);
\draw[gray] (-0.6495,1.125) .. controls (-1.0825,0.875) and (-0.7915,0.3791) .. (-1.0825,0.375) .. controls (-1.3745,0.3693) and (-1.3325,1.442) .. (-1.0825,1.875);
\draw[gray] (-0.6495,1.125) .. controls (-0.2165,1.375) and (0.0665,0.8625) .. (0.2165,1.125) .. controls (0.3665,1.3848) and (-0.2165,1.375) .. (-0.2165,1.875);
\draw[gray] (-0.2165,1.875) .. controls (-0.2165,2.375) and (0.3585,2.3709) .. (0.2165,2.625) .. controls (0.0755,2.8807) and (-0.8325,2.308) .. (-1.0825,1.875);
\node[gray!75] at (0,-0.2) {\footnotesize$\aleph_{3}$}; 
\node[gray] at (1.2557,1.975) {\footnotesize$\aleph_{2}$}; 
\node[gray] at (-1.2557,1.975) {\footnotesize$\aleph_{1}$}; 
\node at (0,0.2) {\tiny$\alpha_{3}$}; 
\node at (0.9093,1.75) {\tiny$\alpha_{2}$};  
\node at (-0.9093,1.75) {\tiny$\alpha_{1}$}; 
\node at (0,1.325) {\tiny$x$}; 
\node at (-0.3085,2.2843) {\tiny$\lambda_{12}$}; 
\node at (0.5,1.4) {\tiny$\lambda_{2}$}; 
\node at (-0.5,1.4) {\tiny$\lambda_{1}$}; 
\node at (0,-0.5) {Step 4};
\end{tikzpicture}
\caption{Illustration Contracting the Image Under ${F}$ of a 3-Cycle}
\label{fig map extends}
\end{figure}

\begin{prop}\label{prop Cn is sc if SoD is} 
Suppose that the Space of Domains is simply connected. Then so too is the complex $\mathcal{C}_{n}$.
\end{prop}

\begin{proof}
Let $\lambda:\mathbb{S}^{1}\to\mathcal{C}_{n}$ be an arbitrary loop in $\mathcal{C}_{n}$.
We temporarily denote the Space of Domains by $\mathcal{S}$, and its 1-skeleton (the Graph of Domains) by $\mathcal{S}^{(1)}$.
By Lemma \ref{loops homotope to image of F}, $\lambda$ is homotopic in $\mathcal{C}_{n}$ to some loop $\mu:\mathbb{S}^{1}\to F(\mathcal{S}^{(1)})$ which lifts to a loop $M:\mathbb{S}^{1}\to\mathcal{S}$ in the Space of Domains with $\mu=F\circ M$, where $F:\mathcal{S}^{(1)}\to\mathcal{C}_{n}$ is the map described in Definition \ref{defn map F}.
If $\mathcal{S}$ is simply connected, then the loop $M$ is contractible.
That is, there exists a continuous map $f:\mathbb{D}^{2}\to\mathcal{S}$ so that $f|_{\mathbb{S}^{1}}=M$.
By Lemma \ref{lemma map F extends}, $F$ extends to a continuos map $\bar{F}:\mathcal{S}\to\mathcal{C}_{n}$.
Now $\bar{F}\circ f:\mathbb{D}^{2}\to\mathcal{C}_{n}$ is a continuous map, and $\bar{F}\circ f|_{\mathbb{S}^{1}}=F\circ M=\mu:\mathbb{S}^{1}\to\mathcal{C}_{n}$.
Thus $\mu$ is contractible, and hence $\lambda$ is as well.
\end{proof}


\subsection{Edges in the Space of Domains} 

Here we will consider adjacency in the Graph/Space of Domains. We will use $\alpha$ to represent both a vertex in the Space of Domains (i.e. a domain, as a subspace of $\mathcal{C}_{n}$), and a graph in a domain (i.e. a vertex in $\mathcal{C}_{n}^{(0)}$).

Let $\alpha_{1}\dash\alpha_{2}$ be an edge in the Space of Domains, that is, let $\alpha_{1}\subset\mathcal{C}_{n}$ and $\alpha_{2}\subset\mathcal{C}_{n}$ be two domains such that $\alpha_{1}\cap\alpha_{2}\ne\emptyset$.
Then there is some vertex $T\in\mathcal{C}_{n}$ in the intersection $\alpha_{1}\cap\alpha_{2}\subset\mathcal{C}_{n}$ and, as noted in Section \ref{pairwise intersections}, some $\varphi\in\stab(T)$ such that $\alpha_{2}=(\alpha_{1})\varphi$. 

Thus the edge $\alpha_{1}\dash\alpha_{2}$ in the Space of Domains can be described in two ways: according to some vertex $T\in\alpha_{1}\cap\alpha_{2}$, or according to some (pure symmetric outer) automorphism $\varphi\in\outs(G)$ satisfying $\alpha_{2}=(\alpha_{1})\varphi$ (and hence also $\alpha_{1}=(\alpha_{2})\varphi^{-1}$).
We may then label the edge $\alpha_{1}\dash\alpha_{2}$ by either
\begin{tikzcd}[cramped,sep=small]
\alpha_{1} \ar[r,dash,"T"]	& \alpha_{2}
\end{tikzcd}
or
\begin{tikzcd}[cramped,sep=small]
\alpha_{1} \ar[r,dash,->-,"\varphi"]	& \alpha_{2}
\end{tikzcd},
depending on our viewpoint.

This subsection considers the former viewpoint, i.e. points in the intersection $\alpha_{1}\cap\alpha_{2}$.
As such, this subsection can be considered the Space of Domains analogue to Section \ref{pairwise intersections}.
The latter viewpoint (automorphisms satisfying $\alpha_{2}=(\alpha_{1})\varphi$) will be discussed in Section \ref{whitehead autos}

\begin{defn} \label{defn type}
We will say an edge $\alpha_{1}\dash \alpha_{2}$ in the Graph/Space of Domains is of \emph{Type $T$} if there is some tree $T$ in the intersection $\alpha_{1}\cap\alpha_{2}$ in the complex $\mathcal{C}_{n}$.
\end{defn}

Note that edges can be of more than one Type.
In particular, if an edge is of Type $T_{1}$, and $T_{2}$ is a collapse of $T_{1}$, then the edge is also of Type $T_{2}$.
However an edge can be of Type $T_{1}$ and Type $T_{2}$ even if neither is a collapse of the other.
Recall from Section \ref{pairwise intersections} that if $T\in\alpha_{1}\cap\alpha_{2}$, then at least one of $A_{i}$, $B_{i,j,k}$, or $C_{i,j,k,l,m}$ is in $\alpha_{1}\cap\alpha_{2}$ for some $i,j,k,l,m\in\{1,\dots,n\}$, hence every edge in the Space of Domains is at least one of Type $A$, Type $B$, or Type $C$.

\smallskip
Some earlier results may be summarised using this new terminology:
\begin{description}
\item[Proposition \ref{prop alpha graphs are path connected}:] Any two vertices in the Graph of Domains are connected via a path whose edges are all of Type A.
\item[Proposition \ref{pw int pc case not C}:] If an edge $\alpha_{1}\dash \alpha_{2}$ in the Graph of Domains is of Type A or Type B, then $\alpha_{1}\cap\alpha_{2}$ is path connected in the complex $\mathcal{C}_{n}$.
\item[Proposition \ref{pw int pc case C}:] If an edge $\alpha_{1}\dash \alpha_{2}$ in the Graph of Domains is of Type C but not of Type A or Type B, then $\alpha_{1}\cap\alpha_{2}$ is not path connected in the complex $\mathcal{C}_{n}$. However, there exists a domain $\alpha_{3}\supseteq\alpha_{1}\cap\alpha_{2}$ so that $[\alpha_{1},\alpha_{2},\alpha_{3}]$ is a 2-cell in the Space of Domains and $\alpha_{1}\cup\alpha_{2}\cup\alpha_{3}$ is simply connected in the complex $\mathcal{C}_{n}$.
\end{description}

We will now deduce that simple connectivity of the Space of Domains can be proved considering only edges of Type A.

\begin{prop}\label{prop type B edges}
Suppose $\alpha_{1}\dash \alpha_{2}$ is an edge in the Space of Domains of Type $B$ but not of Type $A$.
Then there exists some domain $\alpha_{3}$ and edges $\alpha_{1}\dash \alpha_{3}$ and $\alpha_{3}\dash \alpha_{2}$  of Type $A$ such that  $\alpha_{1}\dash \alpha_{2}$  is homotopic to $\alpha_{1}\dash \alpha_{3}\dash \alpha_{2}$ in the Space of Domains.
\end{prop}

\begin{proof}
Suppose $B_{i,j,k}$ is a $B$-graph in the intersection $\alpha_{1}\cap\alpha_{2}\subset\mathcal{C}_{n}$.
\\ \noindent Let $\{i,j,k,v_{1},\dots,v_{n-3}\}=\{1,\dots,n\}$ and let $\left(H_{i},H_{j},H_{k},H_{v_{1}},\dots,H_{v_{n-3}}\right)$ be an $\mathfrak{S}$-labelling for the $\alpha$-graph in the domain $\alpha_{1}$ of $\mathcal{C}_{n}$.
Then $\left(H_{i},H_{j},H_{k},H_{v_{1}},\dots,H_{v_{n-3}}\right)$ is an $\mathfrak{S}$-labelling for $B_{i,j,k}$, and by Example \ref{eg equivalent B labellings}, any equivalent labelling for $B_{i,j,k}$ must be of the form 
\\ \noindent $\left(H_{i}^{g_{i}g},H_{j}^{g_{j}i_{jk}g},H_{k}^{g_{k}j_{k}i_{jk}g},H_{v_{1}}^{g_{v_{1}}i_{v_{1}}g},\dots,H_{v_{n-3}}^{g_{v_{n-3}}i_{v_{n-3}}g}\right)$ for some $g\in G=H_{1}\ast\dots H_{n}$, $g_{i}\in H_{i}$, $g_{j}\in H_{j}$, $g_{k}\in H_{k}$, $j_{k}\in H_{j}$, $i_{jk}\in H_{i}$, with $g_{v}\in H_{v}$ and $i_{v}\in H_{i}$ for each $v=v_{1},\dots,v_{n-3}$.

We may now assume that the $\alpha$-graph in $\alpha_{2}$ has an $\mathfrak{S}$-labelling of the form 
\\ \noindent $\left(H_{i},H_{j}^{i_{jk}},H_{k}^{j_{k}i_{jk}},H_{v_{1}}^{i_{v_{1}}},\dots,H_{v_{n-3}}^{i_{v_{n-3}}}\right)$ (since both inner automorphisms and relative factor automorphisms stabilise $\alpha$).
 Since the edge $\alpha_{1}\dash\alpha_{2}$ is stipulated to not be of Type $A$, then we must have that $j_{k}\ne 1$ and $i_{a}\ne1$ for some $a\in\{jk,v_{1},\dots,v_{n-3}\}$.

Let $\alpha_{3}$ be the domain whose $\alpha$-graph has $\mathfrak{S}$-labelling $\left(H_{i},H_{j}^{i_{jk}},H_{k}^{i_{jk}},H_{v_{1}}^{i_{v_{1}}},\dots,H_{v_{n-3}}^{i_{v_{n-3}}}\right)$, and let $A_{i}$ and $A_{j}$ be the $A$-graphs in $\alpha_{3}$ with central vertex $H_{i}$ and $H_{j}^{i_{jk}}$, respectively.
Observe that $A_{i}\in\alpha_{3}\cap\alpha_{1}$, thus $\alpha_{1}\dash\alpha_{3}$ is an edge of Type $A$ in the Space of Domains.
Further, note that since $\left(H_{i},H_{j}^{i_{jk}},H_{k}^{i_{jk}},H_{v_{1}}^{i_{v_{1}}},\dots,H_{v_{n-3}}^{i_{v_{n-3}}}\right)$ is an $\mathfrak{S}$-labelling for $A_{j}$, then by Definition \ref{defn equivalent labellings}, so too is $\left(H_{i},H_{j}^{i_{jk}},H_{k}^{i_{jk}\left(j_{k}^{i_{jk}}\right)},H_{v_{1}}^{i_{v_{1}}},\dots,H_{v_{n-3}}^{i_{v_{n-3}}}\right)$.
But $H_{k}^{i_{jk}\left(j_{k}^{i_{jk}}\right)}=H_{k}^{i_{jk}\left(i_{jk}^{-1}j_{k}i_{jk}\right)}=H_{k}^{j_{k}i_{jk}}$.
Thus $A_{j}\in\alpha_{3}\cap\alpha_{2}$, and so $\alpha_{2}\dash\alpha_{3}$ is an edge of Type $A$ in the Space of Domains.

Finally, since $B_{i,j,k}\in\alpha_{1}\cap\alpha_{2}\cap\alpha_{3}$ then $[\alpha_{1},\alpha_{2},\alpha_{3}]$ is a 2-cell in the Space of Domains, and so $\alpha_{1}\dash\alpha_{2}$ is homotopic to the path $\alpha_{1}\dash\alpha_{3}\dash\alpha_{2}$ whose edges are both of Type $A$.
\end{proof}

\begin{prop}\label{prop type C edges}
Suppose $\alpha_{1}\dash \alpha_{2}$ is an edge in the Space of Domains of Type $C$, but which is not of Type $A$ or Type $B$.
Then there exist domains $\alpha_{3}$ and $\alpha_{4}$ and edges $\alpha_{1}\dash \alpha_{3}$, $\alpha_{3}\dash \alpha_{4}$ and $\alpha_{4}\dash \alpha_{2}$  of Type $A$ such that  $\alpha_{1}\dash \alpha_{2}$  is homotopic in the Space of Domains to $\alpha_{1}\dash \alpha_{3}\dash \alpha_{4}\dash \alpha_{2}$.
\end{prop}

\begin{proof}
By Proposition \ref{pw int pc case C}, there exists a domain $\alpha_{3}$ so that $\alpha_{1}\dash \alpha_{2}$ is homotopic to $\alpha_{1}\dash \alpha_{3}\dash \alpha_{2}$, where $\alpha_{1}\dash \alpha_{3}$ is an edge of Type $A$ and $\alpha_{3}\dash \alpha_{2}$ is an edge of Type $B$.
Then by Proposition \ref{prop type B edges}, there exists a domain $\alpha_{4}$ with $\alpha_{3}\dash \alpha_{2}$ homotopic to $\alpha_{3}\dash \alpha_{4}\dash \alpha_{2}$, where $\alpha_{3}\dash \alpha_{4}$ and $\alpha_{4}\dash \alpha_{2}$ are both edges of Type $A$.
Now $\alpha_{1}\dash \alpha_{2}$ is homotopic to $\alpha_{1}\dash \alpha_{3}\dash \alpha_{4}\dash \alpha_{2}$.
\end{proof}

\begin{cor} \label{type A edges} 
Any path in our Graph of Domains is homotopic (in the Space of Domains) to a path whose edges are all of Type $A$.
\end{cor}

\begin{proof}
This follows from Propositions \ref{prop type B edges} and \ref{prop type C edges}, recalling that any edge in the Graph of Domains is at least one of Type $A$, Type $B$, or Type $C$.
\end{proof}


\subsection{Relative Whitehead Automorphisms} \label{whitehead autos} 

In this subsection, we consider how to move through the Graph of Domains (i.e. how to move along edges in the Space of Domains).
By Corollary \ref{type A edges}, we need only consider edges of Type $A$.

In Definition \ref{defn whitehead autos}, we define Whitehead automorphisms and multiple Whitehead automorphisms, which provide a convenient way to discuss elements of the stabilisers of vertices in the fundamental domain $\mathcal{D}_{n}$.
These are the same (differing only in notation) as those used by Gilbert \cite{Gilbert1987} and Collins and Zieschang \cite{Collins1984}.
However, to examine domains other than the fundamental domain (i.e. vertices in the Space of Domains), it will prove useful to discuss `relative' Whitehead automorphisms, which may act on groups beyond just the factor groups $G_{i}$ of the splitting $\mathfrak{S}$.

We will find that moving along edges of Type $A$ in the Space of Domains is achieved by applying `relative multiple Whitehead automorphisms', or equivalently, by collapsing and expanding edges of graphs of groups in $\mathcal{C}_{n}$ (i.e. travelling along $\alpha\dash A\dash \alpha$ paths in $\mathcal{C}_{n}$).

\begin{defn}[Relative Whitehead Automorphism]\label{defn rel whitehead auto}
Let $H_{1}\ast\dots\ast H_{n}$ be an $\mathfrak{S}$ free factor splitting for $G=G_{1}\ast\dots\ast G_{n}$. 
A \emph{relative Whitehead automorphism} (with respect to the splitting $H_{1}\ast\dots\ast H_{n}$) is
a map $\psi$ for which there exists $x\in H_{i}$ for some $i$ and $A\subseteq\{H_{1},\dots,H_{n}\}-\{H_{i}\}$ so that $\psi$ pointwise conjugates $H_{j}$ by $x$ for each $H_{j}\in A$, and pointwise fixes $H_{k}$ for each $H_{k}\not\in A$.
We denote such a map $\psi$ by $(A,x)$.
If $|A|=1$, i.e. $A=\{H_{j}\}$ for some $j$, we may abuse notation and write $(H_{j},x)$ for $(\{H_{j}\},x)$.

If $\mathbf{x}=(x_{1},\dots,x_{k})\subset H_{i}$ for some $i$ and $\mathbf{A}=(A_{1},\dots,A_{k})$ where each $A_{j}\subseteq\{H_{1},\dots,H_{n}\}-\{H_{i}\}$ and $A_{j_{1}}\cap A_{j_{2}}=\emptyset$ for $j_{1}\ne j_{2}$, then we denote by $(\mathbf{A},\mathbf{x})$ the composition $(A_{1},x_{1})\dots(A_{k},x_{k})$.
Such a map is called a \emph{relative multiple Whitehead automorphism}.
We denote the union $A_{1}\cup\dots\cup A_{k}$ by $\hat{A}$, or, for longer expressions, by $\bigcup\mathbf{A}$.
\end{defn}

The Whitehead automorphisms of Definition \ref{defn whitehead autos} may be thought of as relative Whitehead automorphisms with respect to the initial splitting $G_{1}\ast\dots\ast G_{n}$ of $G$.
However it should be noted that they behave quite differently under composition.

\begin{lemma}
Let $H_{1}\ast\dots\ast H_{n}$ be an $\mathfrak{S}$ free factor splitting for $G_{1}\ast\dots\ast G_{n}$. If $\psi\in\out(G)$ is a relative Whitehead automorphism (with respect to $H_{1}\ast\dots\ast H_{n}$), then $\psi\in\outs(G)$.
\end{lemma}

\begin{proof}
By Lemma \ref{lemma automorphisms exist between splittings}, there exists $\varphi\in\outs(G)$ with $\varphi(G_{i})=H_{i}=G_{i}^{g_{i}}$ for each $i$ ( for some $g_{i}\in G$).
Since $\psi$ is a relative Whitehead automorphism for $H_{1}\ast\dots\ast H_{n}$, for each $i$ there exists $h_{i}\in G$ such that $\psi(H_{i})=H_{i}^{h_{i}}$.
Now let $\chi\in\outs(G)$ be such that for each $i$, $\chi:g\mapsto g^{\varphi^{-1}(h_{i})}$ for all $g\in G_{i}$.
Then for each $i$, $\varphi(\chi(G_{i}))=\varphi(\varphi^{-1}(h_{i}^{-1})G_{i}\varphi^{-1}(h_{i})=h_{i}^{-1}\varphi(G_{i})h_{i}=G_{i}^{g_{i}h_{i}}$, and the following diagram commutes:

\begin{tikzcd}
G_{1}\ast\dots\ast G_{n}
\ar[r,"\varphi"]
\ar[d,"\chi"]
					&	G_{1}^{g_{1}}\ast\dots\ast G_{n}^{g_{n}}
						\ar[r,equal]
						\ar[d,"\psi"]
														&	H_{1}\ast\dots\ast H_{n}
															\ar[d,"\psi"]
																				\\
G_{1}^{\varphi^{-1}(h_{1})}\ast\dots\ast G_{n}^{\varphi^{-1}(h_{n})}
\ar[r,"\varphi"]
					&	G_{1}^{g_{1}h_{1}}\ast\dots\ast G_{n}^{g_{n}h_{n}}
						\ar[r,equal]
														&	H_{1}^{h_{1}}\ast\dots\ast H_{n}^{h_{n}}
																				\\
\end{tikzcd}

\noindent
Thus in $\out(G)$ we have $\psi=\varphi^{-1}\chi\varphi$ and since $\varphi,\chi\in\outs(G)\le\out(G)$ then $\psi\in\outs(G)$.
\end{proof}

\begin{lemma}\label{A edges have (A,x) autos}
If $\alpha_{1}\dash\alpha_{2}$ is an edge in the Space of Domains of Type $A$, then there exists some relative multiple Whitehead automorphism $(\mathbf{A},\mathbf{x})$ such that $\alpha_{2}=\alpha_{1}\cdot(\mathbf{A},\mathbf{x})$.
\end{lemma}

\begin{proof}
If $\alpha_{1}\dash\alpha_{2}$ is an edge of Type $A$, then there is some $A$-graph $A_{i}\in\alpha_{1}\cap\alpha_{2}\subset\mathcal{C}_{n}$.
Suppose the $\alpha$-graph in the domain $\alpha_{1}$ has $\mathfrak{S}$-labelling $(H_{1},\dots,H_{n})=(G_{1}^{g_{1}},\dots,G_{n}^{g_{n}})$ for some $g_{1},\dots,g_{n}\in G$. Then so too does $A_{i}$, and the $\alpha$-graph in $\alpha_{2}$ must have labelling $\left( \chi(H_{1}),\dots,\chi(H_{n}) \right)$ for some $[\chi]\in\stab(A_{i})$.

Let $[\psi]\in\outs(G)$ be such that $H_{k}=\psi(G_{k})$ for all $k$ (that is, $\alpha_{1}=\mathcal{D}_{n}\cdot\psi$). If $\underline{A}_{i}$ is the $A_{i}$-graph in $\mathcal{D}_{n}$, the fundamental domain, then $\stab(A_{i})=\psi^{-1}\stab(\underline{A}_{i})\psi$.
Recall from Proposition \ref{prop brown stabs 2} and Example \ref{eg stab A} that $\stab(\underline{A}_{i})$ comprises elements of the form $(G_{v_{1}},y_{1})\dots(G_{v_{n-1}},y_{n-1})\varphi$ where $y_{1},\dots,y_{n-1}\in G_{i}$ and $\varphi\in\Phi=\prod_{k=1}^{n}\aut(G_{k})$.

Since $\varphi\in\stab(\underline{\alpha})$ (where $\underline{\alpha}$ is the $\alpha$-graph in $\mathcal{D}_{n}$), then we may assume that $\alpha_{2}\cdot\psi^{-1}=\mathcal{D}_{n}\cdot(G_{v_{1}},y_{1})\dots(G_{v_{n-1}},y_{n-1})$ for some $y_{1},\dots,y_{n-1}\in G_{i}$.
Now for a given factor group $G_{j}$ ($j\ne i$), we have $(G_{j})(G_{j},y_{j})\psi=(G_{j}^{y_{j}})\psi=\left( \psi(G_{j})\right)^{\psi(y_{j}))}=\left(G_{j}^{g_{j}}\right)^{\left(y_{j}^{g_{i}}\right)}=(H_{j})\left(H_{j},y_{j}^{g_{i}}\right)$.
Observe that since $y_{j}\in G_{i}$ then $y_{j}^{g_{i}}\in G_{i}^{g_{i}}=H_{i}$.

Thus setting $x_{j}=y_{j}^{g_{i}}$ for all $j\ne i$, we have $\alpha_{2}=\alpha_{1}\cdot\psi^{-1}(G_{v_{1}},y_{1})\dots(G_{v_{n-1}},y_{n-1})\psi=\alpha_{1}\cdot(H_{v_{1}},x_{1})\dots(H_{v_{n-1}},x_{v_{n-1}})$.
Grouping together terms for which $x_{j}=x_{k}$, we may then write this in the form $\alpha_{2}=\alpha_{1}\cdot(\mathbf{A},\mathbf{x})$
with $\hat{A}=\bigcup\mathbf{A}\subseteq\{H_{1},\dots,H_{n}\}-\{H_{i}\}$ and $\mathbf{x}\subseteq H_{i}$, as required.
Moreover, $(\mathbf{A},\mathbf{x})\in\stab(A_{i})$, as one would expect.
\end{proof}

\begin{example}\label{eg (H3,h1)}
Let $\alpha_{1}$ be the $\alpha$-graph with $\mathfrak{S}$-labelling $(H_{1},\dots,H_{n})$, and let $h_{1}\in H_{1}$.
Then $\alpha_{2}:=\alpha_{1}\cdot(H_{3},h_{1})$ is the $\alpha$-graph with $\mathfrak{S}$-labelling $(H_{1},H_{2},H_{3}^{h_{1}},H_{4},\dots,H_{n})$.
Moreover, $\alpha_{1}\dash\alpha_{2}$ is an edge in the Space of Domains of Type $A$ (writing $\alpha_{i}$ for the domain containing the $\alpha$-graph $\alpha_{i}$).
We demonstrate this via the following collapse--expansion path in $\mathcal{C}_{n}$:
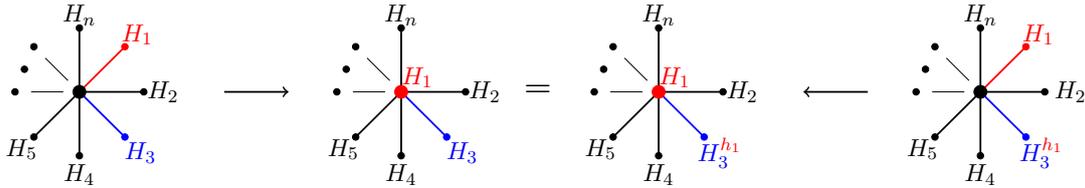
\begin{center}
\resizebox{\linewidth}{!}{
\begin{tikzpicture}
\draw[thick] (0,0) -- (0,1); 
\draw[red,thick] (0,0) -- (0.707,0.707); 
\draw[thick] (0,0) -- (1,0); 
\draw[blue,thick] (0,0) -- (0.707,-0.707); 
\draw[thick] (0,0) -- (0,-1); 
\draw[thick] (0,0) -- (-0.707,-0.707); 
\draw[thin] (-0.25,0) -- (-0.75,0);
\draw[thin] (-0.177,0.177) -- (-0.53,0.53);
\filldraw (0,1) circle [radius=0.05cm]; 
\filldraw[red] (0.707,0.707) circle [radius=0.05cm]; 
\filldraw (1,0) circle [radius=0.05cm]; 
\filldraw[blue] (0.707,-0.707) circle [radius=0.05cm]; 
\filldraw (0,-1) circle [radius=0.05cm]; 
\filldraw (-0.707,-0.707) circle [radius=0.05cm]; 
\filldraw (-1,0) circle [radius=0.05cm];
\filldraw (-0.707,0.707) circle [radius=0.05cm];
\filldraw (-0.854,0.354) circle [radius=0.05cm];
\filldraw (0,0) circle [radius=0.1cm]; 
\node at (0,1.2) {$H_{n}$};
\node[red] at (0.9,0.9) {$H_{1}$};
\node at (1.3,0) {$H_{2}$};
\node[blue] at (0.95,-0.95) {$H_{3}$};
\node at (0,-1.3) {$H_{4}$};
\node at (-0.9,-0.9) {$H_{5}$};
\draw[thick,-to] (2.25,0) -- (3.25,0);
\draw[thick] (5,0) -- (5,1); 
\draw[thick] (5,0) -- (6,0); 
\draw[blue,thick] (5,0) -- (5.707,-0.707); 
\draw[thick] (5,0) -- (5,-1); 
\draw[thick] (5,0) -- (4.293,-0.707); 
\draw[thin] (4.75,0) -- (4.25,0);
\draw[thin] (4.823,0.177) -- (4.47,0.53);
\filldraw (5,1) circle [radius=0.05cm]; 
\filldraw[red] (5,0) circle [radius=0.1cm]; 
\filldraw (6,0) circle [radius=0.05cm]; 
\filldraw[blue] (5.707,-0.707) circle [radius=0.05cm]; 
\filldraw (5,-1) circle [radius=0.05cm]; 
\filldraw (4.293,-0.707) circle [radius=0.05cm]; 
\filldraw (4,0) circle [radius=0.05cm];
\filldraw (4.293,0.707) circle [radius=0.05cm];
\filldraw (4.146,0.354) circle [radius=0.05cm];
\node at (5,1.2) {$H_{n}$};
\node[red] at (5.25,0.225) {$H_{1}$};
\node at (6.3,0) {$H_{2}$};
\node[blue] at (5.95,-0.95) {$H_{3}$};
\node at (5,-1.3) {$H_{4}$};
\node at (4.1,-0.9) {$H_{5}$};
\node at (7.125,0) {\LARGE{$=$}};
\draw[thick] (9,0) -- (9,1); 
\draw[thick] (9,0) -- (10,0); 
\draw[blue,thick] (9,0) -- (9.707,-0.707); 
\draw[thick] (9,0) -- (9,-1); 
\draw[thick] (9,0) -- (8.293,-0.707); 
\draw[thin] (8.75,0) -- (8.25,0);
\draw[thin] (8.823,0.177) -- (8.47,0.53);
\filldraw (9,1) circle [radius=0.05cm]; 
\filldraw[red] (9,0) circle [radius=0.1cm]; 
\filldraw (10,0) circle [radius=0.05cm]; 
\filldraw[blue] (9.707,-0.707) circle [radius=0.05cm]; 
\filldraw (9,-1) circle [radius=0.05cm]; 
\filldraw (8.293,-0.707) circle [radius=0.05cm]; 
\filldraw (8,0) circle [radius=0.05cm];
\filldraw (8.293,0.707) circle [radius=0.05cm];
\filldraw (8.146,0.354) circle [radius=0.05cm];
\node at (9,1.2) {$H_{n}$};
\node[red] at (9.25,0.225) {$H_{1}$};
\node at (10.3,0) {$H_{2}$};
\node[blue] at (9.95,-0.95) {$H_{3}^{\textcolor{red}{h_{1}}}$};
\node at (9,-1.3) {$H_{4}$};
\node at (8.1,-0.9) {$H_{5}$};
\draw[thick,to-] (11.25,0) -- (12.25,0);
\draw[thick] (14,0) -- (14,1); 
\draw[red,thick] (14,0) -- (14.707,0.707); 
\draw[blue,thick] (14,0) -- (14.707,-0.707); 
\draw[thick] (14,0) -- (15,0); 
\draw[thick] (14,0) -- (14,-1); 
\draw[thick] (14,0) -- (13.293,-0.707); 
\draw[thin] (13.75,0) -- (13.25,0);
\draw[thin] (13.823,0.177) -- (13.47,0.53);
\filldraw (14,1) circle [radius=0.05cm]; 
\filldraw[red] (14.707,0.707) circle [radius=0.05cm]; 
\filldraw[blue] (14.707,-0.707) circle [radius=0.05cm]; 
\filldraw (15,0) circle [radius=0.05cm]; 
\filldraw (14,-1) circle [radius=0.05cm]; 
\filldraw (13.293,-0.707) circle [radius=0.05cm]; 
\filldraw (13,0) circle [radius=0.05cm];
\filldraw (13.293,0.707) circle [radius=0.05cm];
\filldraw (13.146,0.354) circle [radius=0.05cm];
\filldraw (14,0) circle [radius=0.1cm]; 
\node at (14,1.2) {$H_{n}$};
\node[red] at (14.9,0.9) {$H_{1}$};
\node[blue] at (14.95,-0.95) {$H_{3}^{\textcolor{red}{h_{1}}}$};
\node at (15.4,0) {$H_{2}$};
\node at (14,-1.3) {$H_{4}$};
\node at (13.1,-0.9) {$H_{5}$};
\end{tikzpicture}
}
\end{center}
It is not hard to see using this example how to extend Lemma \ref{A edges have (A,x) autos} to an ``if and only if'' statement.
Additionally, we see that relative (multiple) Whitehead automorphisms must obey the relations of the stabiliser of the relevant $A$-graph in $\mathcal{C}_{n}$.
\end{example}

\begin{obs}\label{obs geometric whitehead argument}
In order to apply this kind of geometric argument, we must ensure that all automorphisms are written relative to the domain on which they act. This includes automorphisms written within a composition.
Thus, continuing Example \ref{eg (H3,h1)}, if $h_{2}\in H_{2}$ say, we have $(H_{3},h_{1}h_{2})=(H_{3},h_{1})(H_{3}^{h_{1}},h_{2})$.
\end{obs}

\begin{lemma}\label{lemma whitehead properties} 
Let $(H_{1},\dots, H_{n})$ be an $\mathfrak{S}$-labelling for some $\alpha$-graph $\alpha\in\mathcal{C}_{n}$.
Suppose $x\in H_{i}$ for some $i$, and $A,B\subseteq\hat{H}=\{H_{1},\dots,H_{n}\}$ with $H_{i}\not\in A\cup B$.
If $A=\{H_{a_{1}},\dots,H_{a_{m}}\}$, set $A^{x}:=\{H_{a_{1}}^{x},\dots,H_{a_{m}}^{x}\}$.
\begin{enumerate}
\item For $x_{1},x_{2}\in H_{i}$, we have $\alpha\cdot(A,x_{1})(A^{x_{1}},x_{2})=\alpha\cdot(A,x_{1}x_{2})$.
\item We have $\alpha\cdot(A,x)(A^{x},x^{-1})=\alpha$. We will thus write $(A^{x},x^{-1})=(A,x)^{-1}$.
\item If $A\cap B=\emptyset$, then $\alpha\cdot(A,x)(B,x)=\alpha\cdot(B,x)(A,x)$, which we may write as $\alpha\cdot(A\cup B,x)$.
\end{enumerate}
\end{lemma}

\begin{proof}
Let $\alpha$ be the $\alpha$-graph in $\mathcal{C}_{n}$ with $(H_{1},\dots,H_{n})$ as a labelling.
We partition $\{H_{1},\dots,H_{n}\}$ as $\{H_{a_{1}},\dots,H_{a_{m}}\}\cup\{H_{i_{1}},\dots,H_{i_{s}}\}$ where $m+s=n$ and $H_{a}\in A$ for each $a\in\{a_{1},\dots,a_{m}\}$.
\begin{enumerate}
\item By Definition \ref{defn action of out(G)}, we have that $\alpha\cdot(A,x_{1})$ is the $\alpha$-graph, say $\alpha_{1}$, with labelling $(H_{a_{1}}^{x_{1}},\dots,H_{a_{m}}^{x_{1}},H_{i_{1}},\dots,H_{i_{s}})$.
Then $\alpha_{1}\cdot(A^{x_{1}},x_{2})$ is the $\alpha$-graph, say $\alpha_{2}$, with labelling $((H_{a_{1}}^{x_{1}})^{x_{2}},\dots,(H_{a_{m}}^{x_{1}})^{x_{2}},H_{i_{1}},\dots,H_{i_{s}})=(H_{a_{1}}^{x_{1}x_{2}},\dots,H_{a_{m}}^{x_{1}x_{2}},H_{i_{1}},\dots,H_{i_{s}})$.
On the other hand, we clearly have that $\alpha\cdot(A,x_{1}x_{2})=\alpha_{2}$.
\item This follows immediately from 1 by setting $x_{1}=x$ and $x_{2}=x^{-1}$ and noting that $xx^{-1}=1$ and $(A,1)$ is the identity for any $A$.
\item Since $A\cap B=\emptyset$, we will partition $\{H_{1},\dots,H_{n}\}$ as $\{H_{a_{1}},\dots,H_{a_{p}}\}\cup\{H_{b_{1}},\dots,H_{b_{q}}\}\cup\{H_{i_{1}},\dots,H_{i_{r}}\}$ where $p+q+r=n$, $H_{a}\in A$ for each $a\in\{a_{1},\dots,a_{p}\}$, and $H_{b}\in B$ for each $b\in\{b_{1},\dots,b_{q}\}$.
Set $\alpha_{1}:=\alpha\cdot(A,x)$ and $\alpha_{2}:=\alpha\cdot(B,x)$.
Then $\alpha_{1}$ is the $\alpha$-graph in $\mathcal{C}_{n}$ with labelling $(H_{a_{1}}^{x},\dots,H_{a_{p}}^{x},H_{b_{1}},\dots,H_{b_{q}},H_{i_{1}},\dots,H_{i_{r}})$, and $\alpha_{2}$ is the $\alpha$-graph in $\mathcal{C}_{n}$ with labelling $(H_{a_{1}},\dots,H_{a_{p}},H_{b_{1}}^{x},\dots,H_{b_{q}}^{x},H_{i_{1}},\dots,H_{i_{r}})$.
Now $\alpha_{3}:=\alpha_{1}\cdot(B,x)$ is the $\alpha$-graph in $\mathcal{C}_{n}$ with labelling $(H_{a_{1}}^{x},\dots,H_{a_{p}}^{x},H_{b_{1}}^{x},\dots,H_{b_{q}}^{x},H_{i_{1}},\dots,H_{i_{r}})$.
But $\alpha_{4}:=\alpha_{2}\cdot(A,x)$ is also an $\alpha$-graph in $\mathcal{C}_{n}$, with the same labelling as $\alpha_{3}$.
Thus $\alpha_{3}$ and $\alpha_{4}$ belong to the same $\out(G)$-orbit of the fundamental domain, $\mathcal{D}_{n}$.
Since $\mathcal{D}_{n}$ contains a unique $\alpha$-graph, then we must have that $\alpha_{3}=\alpha_{4}$.
That is, $\alpha\cdot(A,x)(B,x)=\alpha\cdot(B,x)(A,x)$.
\end{enumerate}
\end{proof}

\begin{rem}
Let $\hat{H}=\{H_{1},\dots,H_{n}\}$.
For $I=\hat{H}-(A\cup B\cup \{H_{i}\})$ with $A\cap B=\emptyset$ and $H_{i}\not\in A\cup B$, we can `partition' the $\alpha$-graph with labelling $(H_{1},\dots,H_{n})$
\\ \noindent \Talphaexp{$H_{1}$}{$H_{2}$}{$H_{3}$}{$H_{n-1}$}{$H_{n}$} as
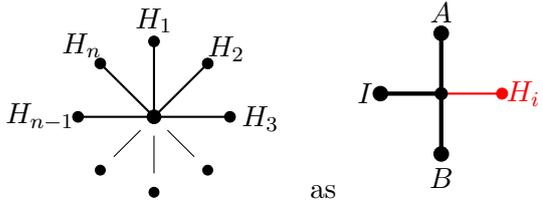
\begin{tikzpicture}[scale=0.8]
\draw[ultra thick] (0,0) -- (-1,0); 
\draw[ultra thick] (0,0) -- (0,1); 
\draw[ultra thick] (0,0) -- (0,-1); 
\draw[thick,red] (0,0) -- (1,0); 
\filldraw (0,0) circle [radius=0.1cm]; 
\filldraw (-1,0) circle [radius=0.12cm];
\filldraw (0,1) circle [radius=0.12cm];
\filldraw (0,-1) circle [radius=0.12cm];
\filldraw[red] (1,0) circle [radius=0.09cm];
\node at (-1.25,0) {$I$};
\node at (0,1.35) {$A$};
\node at (0,-1.4) {$B$};
\node[red] at (1.35,0) {$H_{i}$};
\end{tikzpicture}
~. One should think of this still as an $\alpha$-graph, but `abbreviated' --- instead of drawing individual edges for each leaf $H_{j}$ of $A$, we draw one wider edge (similarly for $B$ and $I$).
\end{rem}

\begin{lemma} \label{disjoint whitehead autos}
Let $H_{1}\ast\dots\ast H_{n}$ be an $\mathfrak{S}$ free factor splitting for $G$, and let $\alpha$ be the $\alpha$-graph (and the domain containing it) with $\mathfrak{S}$-labelling $(H_{1},\dots,H_{n})$.
Suppose there are elements $x\in H_{i}$ and $y\in H_{j}$ for some $i\ne j$, and subsets $A,B\subseteq\{H_{1},\dots,H_{n}\}-\{H_{i},H_{j}\}$.
If $A\cap B=\emptyset$, then $\alpha\cdot(A,x)(B,y)=\alpha\cdot(B,y)(A,x)$.
\end{lemma}

\begin{proof}
We have that $\alpha$ is the graph of groups \Talphaexp{$H_{1}$}{$H_{2}$}{$H_{3}$}{$H_{n-1}$}{$H_{n}$}.
For $I=\{H_{1},\dots,H_{n}\}-(A\cup B\cup\{H_{i},H_{j}\})$, we `partition' $\alpha$ as
\begin{tikzpicture}[scale=0.8]
\draw[ultra thick] (0,0) -- (-1,0); 
\draw[ultra thick] (0,0) -- (0,1); 
\draw[ultra thick] (0,0) -- (0,-1); 
\draw[red,thick] (0,0) -- (0.866,0.5); 
\draw[blue,thick] (0,0) -- (0.866,-0.5); 
\filldraw (0,0) circle [radius=0.09cm]; 
\filldraw (-1,0) circle [radius=0.12cm];
\filldraw (0,1) circle [radius=0.12cm];
\filldraw (0,-1) circle [radius=0.12cm];
\filldraw[red] (0.866,0.5) circle [radius=0.0625cm];
\filldraw[blue] (0.866,-0.5) circle [radius=0.0625cm];
\node at (-1.25,0) {$I$};
\node at (0,1.35) {$A$};
\node at (0,-1.4) {$B$};
\node[red] at (1.2,0.5) {$H_{i}$};
\node[blue] at (1.2,-0.6) {$H_{j}$};
\end{tikzpicture},
noting that by construction $A \sqcup B \sqcup \{H_{i},H_{j}\} \sqcup I$ forms a disjoint partition of the labelling $(H_{1},\dots,H_{n})$ for $\alpha$.
Then the diagram in Figure \ref{fig (A,x)(B,y)=(B,y)(A,x)} commutes.

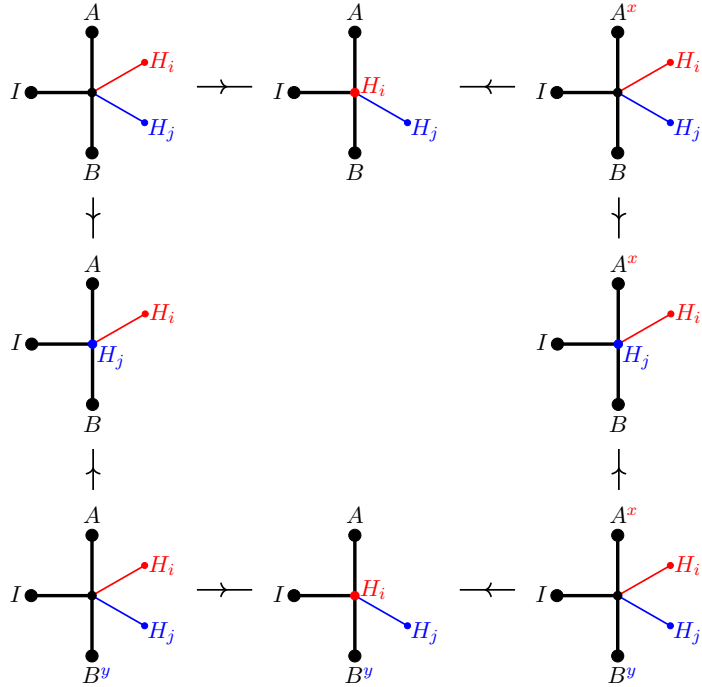
\begin{figure}[h]
\centering
\adjustbox{scale=0.8}{
\begin{tikzcd}
\begin{tikzpicture} 
\draw[ultra thick] (0,0) -- (-1,0); 
\draw[ultra thick] (0,0) -- (0,1); 
\draw[ultra thick] (0,0) -- (0,-1); 
\draw[red,thick] (0,0) -- (0.866,0.5); 
\draw[blue,thick] (0,0) -- (0.866,-0.5); 
\filldraw (0,0) circle [radius=0.07cm]; 
\filldraw (-1,0) circle [radius=0.1cm];
\filldraw (0,1) circle [radius=0.1cm];
\filldraw (0,-1) circle [radius=0.1cm];
\filldraw[red] (0.866,0.5) circle [radius=0.05cm];
\filldraw[blue] (0.866,-0.5) circle [radius=0.05cm];
\node at (-1.25,-0.1) {$I$};
\node at (0,1.2) {$A$};
\node at (0,-1.45) {$B$};
\node[red] at (1.15,0.4) {$H_{i}$};
\node[blue] at (1.15,-0.7) {$H_{j}$};
\end{tikzpicture}
\ar[d,->-,dash,thick]
\ar[r,->-,dash,thick]
&
\begin{tikzpicture} 
\draw[ultra thick] (0,0) -- (-1,0); 
\draw[ultra thick] (0,0) -- (0,1); 
\draw[ultra thick] (0,0) -- (0,-1); 
\draw[blue,thick] (0,0) -- (0.866,-0.5); 
\filldraw[red] (0,0) circle [radius=0.07cm]; 
\filldraw (-1,0) circle [radius=0.1cm];
\filldraw (0,1) circle [radius=0.1cm];
\filldraw (0,-1) circle [radius=0.1cm];
\filldraw[blue] (0.866,-0.5) circle [radius=0.05cm];
\node at (-1.25,-0.1) {$I$};
\node at (0,1.2) {$A$};
\node at (0,-1.45) {$B$};
\node[red] at (0.3,0) {$H_{i}$};
\node[blue] at (1.15,-0.7) {$H_{j}$};
\end{tikzpicture}
&
\begin{tikzpicture} 
\draw[ultra thick] (0,0) -- (-1,0); 
\draw[ultra thick] (0,0) -- (0,1); 
\draw[ultra thick] (0,0) -- (0,-1); 
\draw[red,thick] (0,0) -- (0.866,0.5); 
\draw[blue,thick] (0,0) -- (0.866,-0.5); 
\filldraw (0,0) circle [radius=0.07cm]; 
\filldraw (-1,0) circle [radius=0.1cm];
\filldraw (0,1) circle [radius=0.1cm];
\filldraw (0,-1) circle [radius=0.1cm];
\filldraw[red] (0.866,0.5) circle [radius=0.05cm];
\filldraw[blue] (0.866,-0.5) circle [radius=0.05cm];
\node at (-1.25,-0.1) {$I$};
\node at (0.1,1.2) {$A^{\textcolor{red}{x}}$};
\node at (0,-1.45) {$B$};
\node[red] at (1.15,0.4) {$H_{i}$};
\node[blue] at (1.15,-0.7) {$H_{j}$};
\end{tikzpicture}
\ar[l,->-,dash,thick]
\ar[d,->-,dash,thick]
\\
\begin{tikzpicture} 
\draw[ultra thick] (0,0) -- (-1,0); 
\draw[ultra thick] (0,0) -- (0,1); 
\draw[ultra thick] (0,0) -- (0,-1); 
\draw[red,thick] (0,0) -- (0.866,0.5); 
\filldraw[blue] (0,0) circle [radius=0.07cm]; 
\filldraw (-1,0) circle [radius=0.1cm];
\filldraw (0,1) circle [radius=0.1cm];
\filldraw (0,-1) circle [radius=0.1cm];
\filldraw[red] (0.866,0.5) circle [radius=0.05cm];
\node at (-1.25,-0.1) {$I$};
\node at (0,1.2) {$A$};
\node at (0,-1.45) {$B$};
\node[red] at (1.15,0.4) {$H_{i}$};
\node[blue] at (0.3,-0.3) {$H_{j}$};
\end{tikzpicture}
&
&
\begin{tikzpicture} 
\draw[ultra thick] (0,0) -- (-1,0); 
\draw[ultra thick] (0,0) -- (0,1); 
\draw[ultra thick] (0,0) -- (0,-1); 
\draw[red,thick] (0,0) -- (0.866,0.5); 
\filldraw[blue] (0,0) circle [radius=0.07cm]; 
\filldraw (-1,0) circle [radius=0.1cm];
\filldraw (0,1) circle [radius=0.1cm];
\filldraw (0,-1) circle [radius=0.1cm];
\filldraw[red] (0.866,0.5) circle [radius=0.05cm];
\node at (-1.25,-0.1) {$I$};
\node at (0.1,1.2) {$A^{\textcolor{red}{x}}$};
\node at (0,-1.45) {$B$};
\node[red] at (1.15,0.4) {$H_{i}$};
\node[blue] at (0.3,-0.3) {$H_{j}$};
\end{tikzpicture}
\\
\begin{tikzpicture} 
\draw[ultra thick] (0,0) -- (-1,0); 
\draw[ultra thick] (0,0) -- (0,1); 
\draw[ultra thick] (0,0) -- (0,-1); 
\draw[red,thick] (0,0) -- (0.866,0.5); 
\draw[blue,thick] (0,0) -- (0.866,-0.5); 
\filldraw (0,0) circle [radius=0.07cm]; 
\filldraw (-1,0) circle [radius=0.1cm];
\filldraw (0,1) circle [radius=0.1cm];
\filldraw (0,-1) circle [radius=0.1cm];
\filldraw[red] (0.866,0.5) circle [radius=0.05cm];
\filldraw[blue] (0.866,-0.5) circle [radius=0.05cm];
\node at (-1.25,-0.1) {$I$};
\node at (0,1.2) {$A$};
\node at (0.075,-1.45) {$B^{\textcolor{blue}{y}}$};
\node[red] at (1.15,0.4) {$H_{i}$};
\node[blue] at (1.15,-0.7) {$H_{j}$};
\end{tikzpicture}
\ar[u,->-,dash,thick]
\ar[r,->-,dash,thick]
&
\begin{tikzpicture} 
\draw[ultra thick] (0,0) -- (-1,0); 
\draw[ultra thick] (0,0) -- (0,1); 
\draw[ultra thick] (0,0) -- (0,-1); 
\draw[blue,thick] (0,0) -- (0.866,-0.5); 
\filldraw[red] (0,0) circle [radius=0.07cm]; 
\filldraw (-1,0) circle [radius=0.1cm];
\filldraw (0,1) circle [radius=0.1cm];
\filldraw (0,-1) circle [radius=0.1cm];
\filldraw[blue] (0.866,-0.5) circle [radius=0.05cm];
\node at (-1.25,-0.1) {$I$};
\node at (0,1.2) {$A$};
\node at (0.075,-1.45) {$B^{\textcolor{blue}{y}}$};
\node[red] at (0.3,0) {$H_{i}$};
\node[blue] at (1.15,-0.7) {$H_{j}$};
\end{tikzpicture}
&
\begin{tikzpicture} 
\draw[ultra thick] (0,0) -- (-1,0); 
\draw[ultra thick] (0,0) -- (0,1); 
\draw[ultra thick] (0,0) -- (0,-1); 
\draw[red,thick] (0,0) -- (0.866,0.5); 
\draw[blue,thick] (0,0) -- (0.866,-0.5); 
\filldraw (0,0) circle [radius=0.07cm]; 
\filldraw (-1,0) circle [radius=0.1cm];
\filldraw (0,1) circle [radius=0.1cm];
\filldraw (0,-1) circle [radius=0.1cm];
\filldraw[red] (0.866,0.5) circle [radius=0.05cm];
\filldraw[blue] (0.866,-0.5) circle [radius=0.05cm];
\node at (-1.25,-0.1) {$I$};
\node at (0.1,1.2) {$A^{\textcolor{red}{x}}$};
\node at (0.075,-1.45) {$B^{\textcolor{blue}{y}}$};
\node[red] at (1.15,0.4) {$H_{i}$};
\node[blue] at (1.15,-0.7) {$H_{j}$};
\end{tikzpicture}
\ar[l,->-,dash,thick]
\ar[u,->-,dash,thick]
\\
\end{tikzcd}
}
\caption{Commuting Diagram of $\alpha$ and $A$ Graphs}
\label{fig (A,x)(B,y)=(B,y)(A,x)}
\end{figure}

That is, in the Space of Domains there is a loop
\begin{tikzcd}
\alpha
\ar[d,->-,dash,"(B{,}y)"]
\ar[r,->-,dash,"(A{,}x)"] 
			& 	\alpha\cdot(A,x) 
				\ar[d,->-,dash,"(B{,}y)"]
							\\
\alpha\cdot(B,y)
\ar[r,->-,dash,"(A{,}x)"]
			&	\alpha\cdot(A,x)(B,y)
\end{tikzcd}
which can be seen algebraically by considering the labellings of each $\alpha$-graph.
Thus we have $\alpha\cdot(A,x)(B,y)=\alpha\cdot(B,y)(A,x)$.
\end{proof}

\begin{rem}
With notation as in Lemma \ref{disjoint whitehead autos}, if instead we have $H_{j}\in A$, then $\alpha\cdot(A,x)(B,y)$ is \textbf{not} well-defined geometrically --- we would instead need to write $\alpha\cdot(A,x)(B,y^{x})$. 
On the other hand, if $A\cap B\ne\emptyset$, say $B\subseteq A$, then $\alpha\cdot(A,x)(B,y)$ would also not be well-defined geometrically.
Instead, we would write $\alpha\cdot(A,x)(B^{x},y)$.
Note that this differs from the notation used for non-relative Whitehead automorphisms, where if $\alpha_{0}$ is the $\alpha$-graph in $\mathcal{D}_{n}$ (i.e. with labelling $(G_{1},\dots,G_{n})$) and $x,y\in G_{n}$, we would have that $\alpha_{0}\cdot(G_{1},x)(G_{1},y)$ is the $\alpha$-graph with labelling $(G_{1}^{yx},G_{2},\dots,G_{n})$.
For the remainder of the paper, all automorphisms will be assumed to be relative (multiple) Whitehead automorphisms, unless otherwise specified.
\end{rem}

Despite our relative (multiple) Whitehead automorphisms being different objects than the Whitehead automorphisms used by Gilbert \cite{Gilbert1987}, we borrow some notation introduced in \cite[Section 2]{Gilbert1987}:
\begin{notation}\label{notation whitehead autos}
Let $H_{1}\ast\dots\ast H_{n}$ be an $\mathfrak{S}$ free factor splitting for $G$ and set $\hat{H}:=\{H_{1},\dots,H_{n}\}$.
Given subsets $A_{1},\dots,A_{k}\subseteq\hat{H}$ where $A_{i}\cap A_{j}=\emptyset$ for $i\ne j$,
set $\mathbf{A}=(A_{1},\dots,A_{k})$
and let $\hat{A}:=A_{1}\cup\dots\cup A_{k}$.
Suppose $B$ is an arbitrary subset of $\hat{H}$, $\mathbf{x}=(x_{1},\dots,x_{k})\subset H_{i}$ (where $H_{i}\not\in\hat{A}$), and $y\in H_{i}$.
Also set $\bar{A}:=(\hat{H}-\hat{A})-\{H_{i}\}$.
We will adopt the following notation:
\begin{itemize}
\item $(\mathbf{A}\cap B,\mathbf{x})		:=		(A_{1}\cap B,x_{1})\dots(A_{k}\cap B,x_{k})$
\item $(\mathbf{A}-B,\mathbf{x})			:=		(A_{1}-B,x_{1})\dots(A_{k}-B,x_{k})$
\item $\begin{aligned}[t]
	(\mathbf{A}+_{j}B,\mathbf{x})		:=	&	(A_{1}-B,x_{1})\dots(A_{j-1}-B,x_{j-1})(A_{j}\cup B,x_{j})\\
									&	(A_{j+1}-B,x_{j+1})\dots(A_{k}-B,x_{k})\\
								=	&	(A_{1}-B,x_{1})\dots(A_{j}\cup B,x_{j})\dots(A_{k}-B,x_{k})
	\end{aligned}$
\item $(\mathbf{A},y\mathbf{x})			:=		(A_{1},yx_{1})\dots(A_{k},yx_{k})$
\item $(\mathbf{A},\mathbf{x}y)			:=		(A_{1},x_{1}y)\dots(A_{k},x_{k}y)$
\item $\begin{aligned}[t]
	(\mathbf{\bar{A}}_{j},\mathbf{x})	:=	&	(A_{1},x_{1})\dots(A_{j-1},x_{j-1})(\bar{A},x_{j})(A_{j+1},x_{j+1})\dots(A_{k},x_{k})\\
								=	&	(A_{1},x_{1})\dots(\bar{A},x_{j})\dots(A_{k},x_{k})
	\end{aligned}$
\item $\begin{aligned}[t]
	(\mathbf{A},\mathbf{\tilde{x}}_{j})	:=	&	(A_{1},x_{1})\dots(A_{j-1},x_{j-1})(A_{j},1)(A_{j+1},x_{j+1})\dots(A_{k},x_{k})\\
								=	&	(A_{1},x_{1})\dots(A_{j},1)\dots(A_{k},x_{k})
	\end{aligned}$
\item $[\mathbf{A}]_{j}:=A_{j}$, $[\mathbf{x}]_{j}:=x_{j}$, and $[(\mathbf{A},\mathbf{x})]_{j}:=(A_{j},x_{j})$
\item If $A=\{H_{a_{1}},\dots,H_{a_{m}}\}$ then $A^{x}:=\{H_{a_{1}}^{x},\dots,H_{a_{m}}^{x}\}$ and we define \\ \noindent $\mathbf{A}^{\mathbf{x}}:=(A_{1}^{x_{1}},\dots,A_{k}^{x_{k}})$
\end{itemize}
\end{notation}

We find that similar (though not identical) properties hold for us as are used by Gilbert \cite{Gilbert1987}. In particular, part 4 of the following Proposition is adapted from \cite[Lemma 2.10]{Gilbert1987}.

\begin{prop}\label{whitehead notational properties}
With the above notation, we have that:
\begin{enumerate}
\item 
$
	(\mathbf{A},\mathbf{x})	=	(\mathbf{A}-B,\mathbf{x})	(\mathbf{A}\cap B,\mathbf{x})	
					=	(\mathbf{A}\cap B,\mathbf{x})	(\mathbf{A}-B,\mathbf{x})
$
\item 
$	(\mathbf{A}+_{j}B,\mathbf{x})	=	(\mathbf{A}-B,\mathbf{x})	(B,x_{j})	$
\item 
$	(\mathbf{\bar{A}}_{j},x_{j}^{-1}\mathbf{\tilde{x}}_{j})	=	(\mathbf{A},x_{j}^{-1}\mathbf{x})	(\bar{A},x_{j}^{-1})	$
and $	(\mathbf{\bar{A}}_{j},\mathbf{\tilde{x}}_{j}x_{j}^{-1})	=	(\mathbf{A},\mathbf{x}x_{j}^{-1})	(\bar{A},x_{j}^{-1})	$
\item 
$
	(\mathbf{A},\mathbf{x})	=	(\mathbf{A}+_{j}B,\mathbf{x})		( (\mathbf{\bar{A}}_{j}\cap B)^{x_{j}},x_{j}^{-1}\mathbf{\tilde{x}}_{j})
					=	(\mathbf{\bar{A}}_{j}\cap B, \mathbf{\tilde{x}}_{j}x_{j}^{-1})		\left( (\mathbf{A}+_{j}B)',\mathbf{x}\right)	
$
\smallskip \\ \noindent where $[(\mathbf{A}+_{j}B)']_{a}:=[(\mathbf{A}+_{j}B)]_{a}=A_{a}-B$ for $a\in\{1,\dots,k\}-\{j\}$ and 
$\begin{aligned}[t]
[(\mathbf{A}+_{j}B)']_{j}	:&=\left( [(\mathbf{A}+_{j}B)]_{j}-\bigcup(\mathbf{\bar{A}}_{j}\cap B) \right) \cup \bigcup\left( [(\mathbf{A}+_{j}B)]_{j}\cap(\mathbf{\bar{A}}_{j}\cap B)\right)^{\mathbf{\tilde{x}}_{j}x_{j}^{-1}}	\\
				&=(A_{j}-B)\cup\bigcup(\mathbf{A}\cap B)^{\mathbf{x}x_{j}^{-1}}\cup(B-\hat{A}^{x_{j}^{-1}})
\end{aligned}$
\end{enumerate}
\end{prop}

\begin{proof}
\begin{enumerate}
\item Given arbitrary sets $A$ and $B$, we have $A=(A-B)\cup(A\cap B)$, which is a disjoint partition of the set $A$.
	Now \\	$\begin{aligned}[t]
	(\mathbf{A},\mathbf{x})	&=	(A_{1},x_{1})\dots(A_{k},x_{k})	\\
					&=	((A_{1}-B)\cup(A_{1}\cap B),x_{1})	\dots	((A_{k}-B)\cup(A_{k}\cap B),x_{k})	\\
					&=	(A_{1}-B,x_{1})	(A_{1}\cap B,x_{1})	\dots	(A_{k}-B,x_{k})	(A_{k}\cap B,x_{k})	\\
					&=	(A_{1}-B,x_{1})	\dots	(A_{k}-B,x_{k})	(A_{1}\cap B,x_{1})	\dots	(A_{k}\cap B,x_{k})	\\
					&=	(\mathbf{A}-B,\mathbf{x})	(\mathbf{A}\cap B,\mathbf{x})	.
	\end{aligned}$
	
	Moreover, we also have $A=(A\cap B)\cup(A-B)$, and a similar argument yields $(\mathbf{A},\mathbf{x})=(\mathbf{A}\cap B,\mathbf{x})(\mathbf{A}-B,\mathbf{x})$.
\item Given arbitrary sets $A$ and $B$, we have $A\cup B=(A-B)\cup B$, which is a disjoint partition of the set $A\cup B$.
	Now \\	$\begin{aligned}[t]
	(\mathbf{A}+_{j}B,\mathbf{x})	&=	(A_{1}-B,x_{1})	\dots	(A_{j}\cup B,x_{j})	\dots	(A_{k}-B,x_{k})	\\
						&=	(A_{1}-B,x_{1})	\dots	((A_{j}-B)\cup B,x_{j})	\dots	(A_{k}-B,x_{k})	\\
						&=	(A_{1}-B,x_{1})	\dots	(A_{j}-B,x_{j})	(B,x_{j})	\dots	(A_{k}-B,x_{k})	\\
						&=	(A_{1}-B,x_{1})	\dots	(A_{j}-B,x_{j})	\dots	(A_{k}-B,x_{k})	(B,x_{j})	\\
						&=	(\mathbf{A}-B,\mathbf{x})	(B,x_{j})	.
	\end{aligned}$
\item We have \\	$\begin{aligned}[t]
	(\mathbf{\bar{A}}_{j},x_{j}^{-1}\mathbf{\tilde{x}}_{j})
			&=	(A_{1},x_{j}^{-1}x_{1})	\dots	(\bar{A},x_{j}^{-1}1)	\dots	(A_{k},x_{j}^{-1}x_{k})	\\
			&=	(A_{1},x_{j}^{-1}x_{1})	\dots	(A_{j},1)	(\bar{A},x_{j}^{-1})	\dots	(A_{k},x_{j}^{-1}x_{k})	\\
			&=	(A_{1},x_{j}^{-1}x_{1})	\dots	(A_{j},1)	\dots	(A_{k},x_{j}^{-1}x_{k})	(\bar{A},x_{j}^{-1})	\\
			&=	(A_{1},x_{j}^{-1}x_{1})	\dots	(A_{j},x_{j}^{-1}x_{j})	\dots	(A_{k},x_{j}^{-1}x_{k})	(\bar{A},x_{j}^{-1})	\\
			&=	(\mathbf{A},x_{j}^{-1}\mathbf{x})	(\bar{A},x_{j}^{-1})	.
	\end{aligned}$
	
Similarly, \\ 	$\begin{aligned}[t]
	 (\mathbf{\bar{A}}_{j},\mathbf{\tilde{x}}_{j}x_{j}^{-1})
	 		&=	(A_{1},x_{1}x_{j}^{-1})	\dots	(\bar{A},x_{j}^{-1})	\dots	(A_{k},x_{k}x_{j}^{-1})	\\
	 		&=	(A_{1},x_{1}x_{j}^{-1})	\dots	(A_{j},x_{j}x_{j}^{-1}(\bar{A},x_{j}^{-1})	\dots	(A_{k},x_{k}x_{j}^{-1})	\\
	 		&=	(A_{1},x_{1}x_{j}^{-1})	\dots	(A_{j},x_{j}x_{j}^{-1}	\dots	(A_{k},x_{k}x_{j}^{-1})(\bar{A},x_{j}^{-1})	\\
	 		&=	(\mathbf{A},\mathbf{x}x_{j}^{-1})(\bar{A},x_{j}^{-1})	.
	\end{aligned}$
\item  By 2. and 3. above, we have
 \\	$\begin{aligned}[t]
			&	(\mathbf{A}+_{j}B,\mathbf{x})	((\mathbf{\bar{A}}_{j}\cap B)^{x_{j}},x_{j}^{-1}\mathbf{\tilde{x}}_{j})	\\
			&=	(\mathbf{A}-B,\mathbf{x}) (B,x_{j}) ((\mathbf{A}\cap B)^{x_{j}},x_{j}^{-1}\mathbf{x}) ((\bar{A}\cap B)^{x_{j}},x_{j}^{-1})	\\
			&=	(\mathbf{A}-B,\mathbf{x}) (B\cap\hat{A},x_{j}) (B-\hat{A},x_{j}) ((\mathbf{A}\cap B)^{x_{j}},x_{j}^{-1}\mathbf{x}) ((B-\hat{A})^{x_{j}},x_{j}^{-1})	\\
			&=	(\mathbf{A}-B,\mathbf{x}) (\hat{A}\cap B,x_{j}) ((\mathbf{A}\cap B)^{x_{j}},x_{j}^{-1}\mathbf{x}) (B-\hat{A},x_{j}) ((B-\hat{A})^{x_{j}},x_{j}^{-1})	\\
			&=	(\mathbf{A}-B,\mathbf{x}) (\mathbf{A}\cap B,x_{j}x_{j}^{-1}\mathbf{x}) (B-\hat{A},x_{j}x_{j}^{-1})	\\
			&=	(\mathbf{A}-B,\mathbf{x}) (\mathbf{A}\cap B,\mathbf{x})	\\
			&=	(\mathbf{A},\mathbf{x})	.
	\end{aligned}$

Similarly, \\	$\begin{aligned}[t]
			&	(\mathbf{\bar{A}}_{j}\cap B,x_{j}^{-1}\mathbf{\tilde{x}}_{j})	((\mathbf{A}+_{j}B)',\mathbf{x})	\\
			&=	(\mathbf{A}\cap B,\mathbf{x}x_{j}^{-1}) (B-\hat{A},x_{j}^{-1}) (\mathbf{A}-B,\mathbf{x}) ((\bigcup (\mathbf{A}\cap B)^{\mathbf{x}x_{j}^{-1}})\cup(B-\hat{A})^{x_{j}^{-1}}),x_{j})	\\
			&=	(\mathbf{A}-B,\mathbf{x}) (\mathbf{A}\cap B,\mathbf{x}x_{j}^{-1}) (B-\hat{A},x_{j}^{-1}) (\bigcup (\mathbf{A}\cap B)^{\mathbf{x}x_{j}^{-1}},x_{j}) (B-\hat{A})^{x_{j}^{-1}},x_{j})	\\
			&=	(\mathbf{A}-B,\mathbf{x}) (\mathbf{A}\cap B,\mathbf{x}x_{j}^{-1}) (\bigcup (\mathbf{A}\cap B)^{\mathbf{x}x_{j}^{-1}},x_{j}) (B-\hat{A},x_{j}^{-1}) (B-\hat{A})^{x_{j}^{-1}},x_{j})	\\
			&=	(\mathbf{A}-B,\mathbf{x}) (\mathbf{A}\cap B,\mathbf{x}) 	\\
			&=	(\mathbf{A},\mathbf{x})	.
	\end{aligned}$
\end{enumerate}
\end{proof}


\section{Peak Reduction in the Space of Domains} \label{section peak reduction}	

In this section we will prove that the Space of Domains is simply connected.
We will do this via `peak reduction'.
The idea of this is that given a (based) loop in the Space of Domains, any `peaks' in the loop can be reduced, until the loop is just the basepoint.
This is roughly illustrated in Figure \ref{peak reduction picture}.

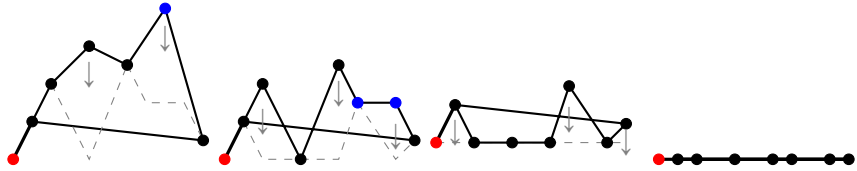
\begin{figure}[h]
\centering
\begin{tikzpicture}
\draw[very thick] (0,0) -- (0.25,0.5);
\draw[thick] (0.25,0.5) -- (0.5,1) -- (1,1.5) -- (1.5,1.25) -- (2,2) -- (2.5,0.25) -- cycle;
\draw[gray,dashed] (0.5,1) -- (1,0) -- (1.5,1.25);
\draw[gray,dashed] (1.5,1.25) -- (1.75,0.75) -- (2.25,0.75) -- (2.5,0.25);
\filldraw[red] (0,0) circle [radius=0.07];
\filldraw (0.25,0.5) circle [radius=0.07];
\filldraw (0.5,1) circle [radius=0.07];
\filldraw (1,1.5) circle [radius=0.07];
\filldraw (1.5,1.25) circle [radius=0.07];
\filldraw[blue] (2,2) circle [radius=0.07];
\filldraw (2.5,0.25) circle [radius=0.07];
\node[gray] at (1,1.125) {$\downarrow$};
\node[gray] at (2,1.6) {$\downarrow$};
\end{tikzpicture}
\begin{tikzpicture}
\draw[very thick] (0,0) -- (0.25,0.5);
\draw[thick] (0.25,0.5) -- (0.5,1) -- (1,0) -- (1.5,1.25) -- (1.75,0.75) -- (2.25,0.75) -- (2.5,0.25) -- cycle;
\draw[gray,dashed] (0.25,0.5) -- (0.5,0) -- (1,0);
\draw[gray,dashed] (1,0) -- (1.5,0) -- (1.75,0.75);
\draw[gray,dashed] (1.75,0.75) -- (2.25,0) -- (2.5,0.25);
\filldraw[red] (0,0) circle [radius=0.07];
\filldraw (0.25,0.5) circle [radius=0.07];
\filldraw (0.5,1) circle [radius=0.07];
\filldraw (1,0) circle [radius=0.07];
\filldraw (1.5,1.25) circle [radius=0.07];
\filldraw[blue] (1.75,0.75) circle [radius=0.07];
\filldraw[blue] (2.25,0.75) circle [radius=0.07];
\filldraw (2.5,0.25) circle [radius=0.07];
\node[gray] at (0.5,0.5) {$\downarrow$};
\node[gray] at (1.5,0.87) {$\downarrow$};
\node[gray] at (2.25,0.3) {$\downarrow$};
\end{tikzpicture}
\begin{tikzpicture}
\draw[gray,dashed] (0,0) -- (2.5,0);
\draw[very thick] (0,0) -- (0.25,0.5);
\draw[thick] (0.25,0.5) -- (0.5,0) -- (1,0) -- (1.5,0) -- (1.75,0.75) -- (2.25,0) -- (2.5,0.25) -- cycle;
\filldraw[red] (0,0) circle [radius=0.07];
\filldraw (0.25,0.5) circle [radius=0.07];
\filldraw (0.5,0) circle [radius=0.07];
\filldraw (1,0) circle [radius=0.07];
\filldraw (1.5,0) circle [radius=0.07];
\filldraw (1.75,0.75) circle [radius=0.07];
\filldraw (2.25,0) circle [radius=0.07];
\filldraw (2.5,0.25) circle [radius=0.07];
\node[gray] at (0.25,0.125) {$\downarrow$};
\node[gray] at (1.75,0.3) {$\downarrow$};
\node[gray] at (2.5,0) {$\downarrow$};
\end{tikzpicture}
\begin{tikzpicture}
\draw[very thick] (0,0) -- (0.25,0);
\draw[very thick] (0.25,0) -- (0.5,0) -- (1,0) -- (1.5,0) -- (1.75,0) -- (2.25,0) -- (2.5,0) -- cycle;
\filldraw[red] (0,0) circle [radius=0.07];
\filldraw (0.25,0) circle [radius=0.07];
\filldraw (0.5,0) circle [radius=0.07];
\filldraw (1,0) circle [radius=0.07];
\filldraw (1.5,0) circle [radius=0.07];
\filldraw (1.75,0) circle [radius=0.07];
\filldraw (2.25,0) circle [radius=0.07];
\filldraw (2.5,0) circle [radius=0.07];
\end{tikzpicture}
\caption{The Idea of `Squashing Loops' by Reducing Peaks}
\label{peak reduction picture}
\end{figure}

We will follow the outline below, which is largely based on the method used by Gilbert \cite[Section 2]{Gilbert1987} (which in turned is based on the work of Collins and Zieschang \cite[Section 2]{Collins1984})\footnote{Note however that the objects Gilbert as well as Collins and Zieschang study are words in a group, as opposed to geometric objects, thus while the overall structure of the idea is similar, the details vary greatly.}:

\begin{itemize} 
\item Define a concept of `height' of a vertex/domain $\alpha$, and define a `peak' of a loop in the Graph/Space of Domains
\item By Corollary \ref{type A edges}, we may solely consider loops comprising edges of Type $A$, so a `peak' looks like 
	\begin{tikzcd}[cramped,sep=small]
	\alpha \ar[r,dash,"A"]	& \alpha \ar[r,dash,"A"]	& \alpha
	\end{tikzcd}
\item Split into four cases of 
	\begin{tikzcd}[cramped,sep=small]
	\alpha \ar[r,dash,"A_{i}"]	& \alpha \ar[r,dash,"A_{j}"]	& \alpha
	\end{tikzcd}
	for various conditions on $i$ and $j$
\item For a given path
	\begin{tikzcd}[cramped,sep=small]
	\alpha \ar[r,dash,"A_{i}"]	& \alpha \ar[r,dash,"A_{j}"]	& \alpha
	\end{tikzcd}
	show there is either a 4-cycle or 5-cycle in the Graph of Domains (whose edges are all of Type $A$) with
	\begin{tikzcd}[cramped,sep=small]
	\alpha \ar[r,dash,"A_{i}"]	& \alpha \ar[r,dash,"A_{j}"]	& \alpha
	\end{tikzcd}
	as a subpath
\item Given such a loop in the Graph of Domains, show that it is contractible in the Space of Domains (that is, that 
	\begin{tikzcd}[cramped,sep=small]
	\alpha \ar[r,dash,"A_{i}"]	& \alpha \ar[r,dash,"A_{j}"]	& \alpha
	\end{tikzcd}
	is homotopic to a path of length 2 or 3 with the same endpoints)
\item Show that if
	\begin{tikzcd}[cramped,sep=small]
	\alpha \ar[r,dash,"A_{i}"]	& \alpha \ar[r,dash,"A_{j}"]	& \alpha
	\end{tikzcd}
	was a peak in some loop in the Space of Domains, then it is homotopic to a path whose `middle' is `smaller' than that of
	\begin{tikzcd}[cramped,sep=small]
	\alpha \ar[r,dash,"A_{i}"]	& \alpha \ar[r,dash,"A_{j}"]	& \alpha
	\end{tikzcd}
\end{itemize}

Note that we are only interested in loops in the Space of Domains whose endpoints are vertices and who strictly follow edge paths (with no backtracking etc.).
This is permissible, as any loop can be deformed into such a loop quite easily.
It also means we can easily consider the loop in the Graph of Domains (which is just the one-skeleton of the Space of Domains).

\subsection{Defining Height} \label{height} 

\begin{notation}\label{defn alpha hat}
Let $\alpha$ be an $\alpha$-graph \Talpha{$n$} in $\mathcal{C}_{n}$ with $\mathfrak{S}$-labelling $\left(H_{1},\dots,H_{n}\right)$.
We denote by $\hat{\alpha}$ the $G$-tree which is the universal cover relative to $\alpha$ (according to Serre \cite{Serre1980}), that is, the $G$-tree satisfying (up to equivariant isometry) $\faktor{\hat{\alpha}}{G}=\alpha$ viewing $\alpha$ here as a quotient graph of groups (via Bass--Serre theory).
We will label vertices of $\hat{\alpha}$ by their stabiliser in $G$, so, for example, $G_{i}\cdot x=G_{i}^{x}$ for $x\in G$ and $G_{i}\in\hat{\alpha}$.
When edges in $\hat{\alpha}$ are given labels, we will write the action of $G$ multiplicatively, so, for example, $e\cdot x=ex$.
\end{notation}

We consider $\alpha$ to be a subgraph  of $\hat{\alpha}$ acting as a fundamental domain for the action of $G$.

Let $\alpha_{0}$ be the fundamental domain $\mathcal{D}_{n}$ (the domain with the graph in Figure \ref{fig alpha 0} at its centre).
\begin{figure}[h]
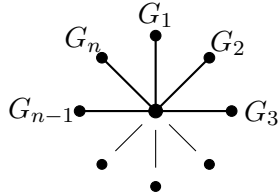

\centering
\Talphaexp{$G_{1}$}{$G_{2}$}{$G_{3}$}{$G_{n-1}$}{$G_{n}$}
\caption{The $\alpha$ Graph at the Centre of the Fundamental Domain}
\label{fig alpha 0}
\end{figure}
Let $\alpha$ be an arbitrary domain (with the graph in Figure \ref{fig arbitrary alpha} at its centre).
\begin{figure}[h]
\centering
\begin{tikzcd}[cramped,sep=small]
\Talphaexp{$H_{1}$}{$H_{2}$}{$H_{3}$}{$H_{n-1}$}{$H_{n}$}
\ar[r,equals,thick]
&
\Talphaexp{$G_{1}^{g_{1}}$}{$G_{2}^{g_{2}}$}{$G_{3}^{g_{3}}$}{$G_{n-1}^{g_{n-1}}$}{$G_{n}^{g_{n}}$}
\\
\end{tikzcd}
\caption{An Arbitrary $\alpha$ Graph}
\label{fig arbitrary alpha}
\end{figure}
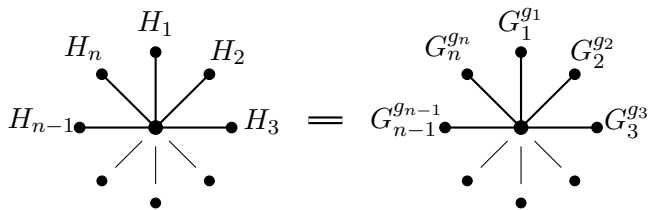

Let $\mathcal{W}$ be the set of pairs $\{G_{i},G_{j}\}$ of (distinct) elements of $\{G_{1},\dots,G_{n}\}$.
Note that $|\mathcal{W}|={n\choose2}=\frac{1}{2}n(n-1)$.

Set $\left\vert\{G_{i},G_{j}\}\right\vert_{\alpha}$ to be the length of the edge path from the vertex labelled $G_{i}$ to the vertex labelled $G_{j}$ in $\hat{\alpha}$ (note that this is symmetric, so we don't need to worry about the order of our pair).

\begin{defn}
We define the \emph{height} of the domain $\alpha$ to be
\[||\alpha||:=\sum_{w\in\mathcal{W}}\left(|w|_{\alpha}-2\right)\]
\end{defn}

Note that this can only take (non-negative) integer values.

\begin{lemma} \label{zero height is basepoint}
We have $||\alpha||=0$ if and only if $\alpha=\alpha_{0}$.
\end{lemma}

\begin{proof}
On the one hand, in $\hat{\alpha_{0}}$ we have that $|w|_{\alpha_{0}}=2$ $\forall w\in \mathcal{W}$, so clearly $||\alpha_{0}||=0$.

On the other hand, observe that for any $w\in\mathcal{W}$, $|w|_{\alpha}\ge2$ $\forall \alpha$.
So
\begin{align*}
||\alpha||=||\alpha_{0}||	& \implies  {\sum_{w\in\mathcal{W}}\left(|w|_{\alpha}-2\right)}	= 	0	\\
						& \implies  {\sum_{w\in\mathcal{W}}|w|_{\alpha}}	= 	2|\mathcal{W}|	\\
						& \implies  |w|_{\alpha}=2 \quad\forall w\in\mathcal{W} \\
\end{align*}

We claim that this implies $\alpha=\alpha_{0}$.
Indeed, if there is some $\alpha$ such that for all $i,j\in\{1,\dots,n\}$ we have $|\{G_{i},G_{j}\}|_{\alpha}=2$, then the path in $\hat{\alpha}$ between the vertex labelled $G_{i}$ and the vertex labelled $G_{j}$ must be
\begin{tikzpicture}
\draw[thick] (0,0) -- (1,0) -- (2,0);
\draw[white,fill] (0,0) circle [radius=0.25cm];
\draw[black,fill] (1,0) circle [radius=0.07cm];
\draw[white,fill] (2,0) circle [radius=0.25cm];
\node at (0,0) {$G_{i}$};
\node at (2,0) {$G_{j}$};
\end{tikzpicture}
(for each pair $\{G_{i},G_{j}\}$).
Suppose for some $i,j,k\in\{1\dots,n\}$ that $\hat{\alpha}$ does not contain the tripod 
\begin{tikzpicture}
\draw[thick] (0,0) -- (0,1);
\draw[thick] (0,0) -- (0.866,-0.5);
\draw[thick] (0,0) -- (-0.866,-0.5);
\draw[fill] (0,0) circle [radius=0.07];
\draw[white,fill] (0,1) circle [radius=0.25]; \node at (0,1) {$G_{i}$};
\draw[white,fill] (0.866,-0.5) circle [radius=0.25]; \node at (0.866,-0.5) {$G_{j}$};
\draw[white,fill] (-0.866,-0.5) circle [radius=0.25]; \node at (-0.866,-0.5) {$G_{k}$};
\end{tikzpicture}
. Then $\hat{\alpha}$ must contain the path
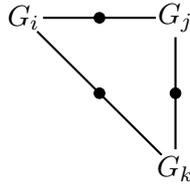
\begin{tikzpicture}
\draw[thick] (0,0) -- (1,0) -- (2,0);
\draw[thick] (2,0) -- (2,-1) -- (2,-2);
\draw[white,fill] (0,0) circle [radius=0.25cm];
\draw[black,fill] (1,0) circle [radius=0.07cm];
\draw[white,fill] (2,0) circle [radius=0.25cm];
\draw[black,fill] (2,-1) circle [radius=0.07cm];
\draw[white,fill] (2,-2) circle [radius=0.25cm];
\node at (0,0) {$G_{i}$};
\node at (2,0) {$G_{j}$};
\node at (2,-2) {$G_{k}$};
\end{tikzpicture}
. But the only way for the length of $(G_{i},G_{k})$ to be $2$ now is to have a cycle
\begin{tikzpicture}
\draw[thick] (0,0) -- (1,0) -- (2,0);
\draw[thick] (2,0) -- (2,-1) -- (2,-2);
\draw[thick] (0,0) -- (1,-1) -- (2,-2);
\draw[white,fill] (0,0) circle [radius=0.25cm];
\draw[black,fill] (1,0) circle [radius=0.07cm];
\draw[white,fill] (2,0) circle [radius=0.25cm];
\draw[black,fill] (2,-1) circle [radius=0.07cm];
\draw[white,fill] (2,-2) circle [radius=0.25cm];
\draw[black,fill] (1,-1) circle [radius=0.07cm];
\node at (0,0) {$G_{i}$};
\node at (2,0) {$G_{j}$};
\node at (2,-2) {$G_{k}$};
\end{tikzpicture}
. But $\hat{\alpha}$ is a tree, so this cannot happen.

Hence $\hat{\alpha}$ must contain every tripod of the form
\begin{tikzpicture}
\draw[thick] (0,0) -- (0,1);
\draw[thick] (0,0) -- (0.866,-0.5);
\draw[thick] (0,0) -- (-0.866,-0.5);
\draw[fill] (0,0) circle [radius=0.07];
\draw[white,fill] (0,1) circle [radius=0.25]; \node at (0,1) {$G_{i}$};
\draw[white,fill] (0.866,-0.5) circle [radius=0.25]; \node at (0.866,-0.5) {$G_{j}$};
\draw[white,fill] (-0.866,-0.5) circle [radius=0.25]; \node at (-0.866,-0.5) {$G_{k}$};
\end{tikzpicture}
for all $i,j,k\in\{1,\dots,n\}$.
But this precisely means that $\hat{\alpha}$ contains as a subgraph the star \scalebox{0.8}{\Talphaexp{$G_{1}$}{$G_{2}$}{$G_{3}$}{$G_{n-1}$}{$G_{n}$}}.

Thus if $|w|_{\alpha}=2$ $\forall w\in\mathcal{W}$, then we must have $\alpha=\alpha_{0}$.

\end{proof}

Note that we think of $\{G_{i},G_{j}\}$ as both a pair of groups, and a pair of vertices in the universal cover of the graph of groups.
Since the universal cover is a tree, there is a unique path from the vertex whose stabiliser is $G_{i}$ to the vertex whose stabiliser is $G_{j}$, so we may also use $\{G_{i},G_{j}\}$ to refer to the edge path connecting them (in a given $\hat{\alpha}$).

\begin{defn}
Let $w$ be a sequence of edges forming a path in a $G$-tree $\hat{\alpha}$, and let $u$ be any subpath of $w$.
Denote by $\Lambda_{w}(u)$ the number of times the subword $u$ (or some $G$-translation $u\cdot{z}$, or inverse $\overbar{u\cdot{z}}$) appears in $w$.
Define $|w|_{\alpha}$ to be the reduced path length of $w$ in $\hat{\alpha}$.
\end{defn}

\begin{con}\label{convention edge labels and map for alpha hat}
Let $\alpha_{1}$ be the $\alpha$-graph with $\mathfrak{S}$-labelling $(H_{1},\dots,H_{n})$, let $\psi\in\outs(G)$, and set $\alpha_{2}:=\alpha_{1}\cdot\psi$.
Recall that $\hat{\alpha}_{1}$ and $\hat{\alpha}_{2}$ are the $G$-trees associated to $\alpha_{1}$ and $\alpha_{2}$, respectively, with vertices labelled by their $G$-stabiliser.
Call the vertex with trivial stabiliser in the convex hull of the vertices $H_{1},\dots,H_{n}$ in $\hat{\alpha}_{1}$ `$v$', and the vertex with trivial stabiliser in the convex hull of the vertices $(H_{1})\psi,\dots,(H_{n})\psi$ in $\hat{\alpha}_{2}$ `$v'$'.
We equivariantly label the edges of $\hat{\alpha}_{1}$ by assigning an edge 
\begin{tikzcd}[cramped,sep=small]
v	\ar[r,dash,->-]	&	H_{j}
\end{tikzcd}
the label `$e_{j}$', and its $G$-images `$e_{j}x$` where $x\in G$.
Similarly, we label the edges of $\hat{\alpha}_{2}$ of the form 
\begin{tikzcd}[cramped,sep=small]
v'	\ar[r,dash,->-]	&	(H_{j})\psi
\end{tikzcd}
`$f_{j}$', and equivariantly extend this to a labelling of all the edges of $\hat{\alpha}_{2}$.
We now define an equivariant map $\varphi_{\psi}:\hat{\alpha}_{1}\to\hat{\alpha}_{2}$ with $\varphi_{\psi}(v)=v'$ so that the vertex in $\hat{\alpha}_{1}$ whose stabiliser is $H_{j}$ is mapped to the vertex in $\hat{\alpha}_{2}$ whose stabiliser is $H_{j}$.
\end{con}

\begin{example}\label{eg change in word length}
Let $\alpha_{1}$ be
\begin{tikzpicture}
\draw[thick] (0,0) -- (0,1);
\draw[thick] (0,0) -- (-0.707,0.707);
\draw[thick] (0,0) -- (0.866,-0.5);
\draw[thick] (0,0) -- (-0.5,-0.866);
\filldraw (0,0) circle [radius=0.075cm]; 
\filldraw[white] (0,1) circle [radius=0.275cm]; 
\filldraw[white] (0.866,-0.5) circle [radius=0.275cm]; 
\filldraw[white] (-0.5,-0.866) circle [radius=0.275cm]; 
\filldraw[white] (-0.707,0.707) circle [radius=0.275cm]; 
\node at (0,1) {$H_{1}$};
\node at (0.866,-0.5) {$H_{a}$};
\node at (-0.5,-0.866) {$H_{i}$};
\node at (-0.707,0.707) {$H_{n}$};
\filldraw (1,0) circle [radius=0.05cm];
\filldraw (0.5,0.866) circle [radius=0.05cm];
\filldraw (0.866,0.5) circle [radius=0.05cm];
\filldraw (0.25,-0.968) circle [radius=0.05cm];
\filldraw (-0.917,-0.4) circle [radius=0.05cm];
\filldraw (-0.980,0.2) circle [radius=0.05cm];
\draw[thin] (0.25,0) -- (0.75,0);
\draw[thin] (0.125,0.217) -- (0.375,0.650);
\draw[thin] (0.217,0.125) -- (0.650,0.375);
\draw[thin] (0.063,-0.242) -- (0.188,-0.726);
\draw[thin] (-0.229,-0.1) -- (-0.688,-0.3);
\draw[thin] (-0.245,0.05) -- (-0.735,0.15);
\end{tikzpicture}
and $\alpha_{2}=\alpha_{1}(H_{a},x)$ (with $x\in H_{i}$) be
\begin{tikzpicture}
\draw[thick] (0,0) -- (0,1);
\draw[thick] (0,0) -- (-0.707,0.707);
\draw[thick] (0,0) -- (0.866,-0.5);
\draw[thick] (0,0) -- (-0.5,-0.866);
\filldraw (0,0) circle [radius=0.075cm]; 
\filldraw[white] (0,1) circle [radius=0.275cm]; 
\filldraw[white] (0.866,-0.5) circle [radius=0.275cm]; 
\filldraw[white] (-0.5,-0.866) circle [radius=0.275cm]; 
\filldraw[white] (-0.707,0.707) circle [radius=0.275cm]; 
\node at (0,1) {$H_{1}$};
\node at (0.866,-0.5) {$H_{a}^{x}$};
\node at (-0.5,-0.866) {$H_{i}$};
\node at (-0.707,0.707) {$H_{n}$};
\filldraw (1,0) circle [radius=0.05cm];
\filldraw (0.5,0.866) circle [radius=0.05cm];
\filldraw (0.866,0.5) circle [radius=0.05cm];
\filldraw (0.25,-0.968) circle [radius=0.05cm];
\filldraw (-0.917,-0.4) circle [radius=0.05cm];
\filldraw (-0.980,0.2) circle [radius=0.05cm];
\draw[thin] (0.25,0) -- (0.75,0);
\draw[thin] (0.125,0.217) -- (0.375,0.650);
\draw[thin] (0.217,0.125) -- (0.650,0.375);
\draw[thin] (0.063,-0.242) -- (0.188,-0.726);
\draw[thin] (-0.229,-0.1) -- (-0.688,-0.3);
\draw[thin] (-0.245,0.05) -- (-0.735,0.15);
\end{tikzpicture}~.
Consider (the subgraphs of) $\hat{\alpha}_{1}$ and $\hat{\alpha}_{2}$ as illustrated in Figure \ref{fig eg (Ha,x)}, labelled according to Convention \ref{convention edge labels and map for alpha hat}.
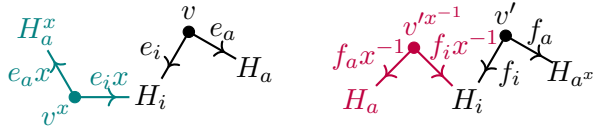
\begin{figure}[h]
\centering
\begin{tikzpicture}
\draw[thick,->-] (0,0) -- (0.866,-0.5);
\draw[thick,->-] (0,0) -- (-0.5,-0.866);
\draw[thick,->-,teal] (-1.5,-0.866) -- (-0.5,-0.866);
\draw[thick,->-,teal] (-1.5,-0.866) -- (-2,0);
\filldraw (0,0) circle [radius=0.075cm]; 
\filldraw[teal] (-1.5,-0.866) circle [radius=0.075cm];
\node at (0,0.25) {$v$};
\node at (-1.75,-1.1) {\textcolor{teal}{$v^{x}$}};
\filldraw[white] (0.866,-0.5) circle [radius=0.275cm]; 
\filldraw[white] (-0.5,-0.866) circle [radius=0.275cm]; 
\filldraw[white] (-2,0) circle [radius=0.275cm]; 
\node at (0.866,-0.5) {$H_{a}$};
\node at (-0.5,-0.866) {$H_{i}$};
\node at (-2,0) {\textcolor{teal}{$H_{a}^{x}$}};
\node at (0.45,0) {$e_{a}$};
\node at (-0.5,-0.25) {$e_{i}$};
\node at (-1.05,-0.65) {\textcolor{teal}{$e_{i}{x}$}};
\node at (-2.1,-0.6) {\textcolor{teal}{$e_{a}{x}$}};
\filldraw[white] (0,-1.2) circle [radius=0.1cm];
\end{tikzpicture}
\hspace{0.25cm}
\begin{tikzpicture}
\draw[thick,->-] (0,0) -- (0.866,-0.5);
\draw[thick,->-] (0,0) -- (-0.5,-0.866);
\draw[thick,->-,purple] (-1.207,-0.159) -- (-0.5,-0.866);
\draw[thick,->-,purple] (-1.207,-0.159) -- (-1.904,-0.866);
\filldraw (0,0) circle [radius=0.075cm]; 
\filldraw[purple] (-1.207,-0.159) circle [radius=0.075cm];
\node at (0,0.3) {$v'$};
\node at (-0.95,0.2) {\textcolor{purple}{$v'^{x^{-1}}$}};
\filldraw[white] (0.866,-0.5) circle [radius=0.275cm]; 
\filldraw[white] (-0.5,-0.866) circle [radius=0.275cm]; 
\filldraw[white] (-1.904,-0.866) circle [radius=0.275cm]; 
\node at (0.866,-0.5) {$H_{a^{x}}$};
\node at (-0.5,-0.866) {$H_{i}$};
\node at (-1.904,-0.866) {\textcolor{purple}{$H_{a}$}};
\node at (0.45,0.05) {$f_{a}$};
\node at (0.05,-0.5) {$f_{i}$};
\node at (-0.55,-0.125) {\textcolor{purple}{$f_{i}{x^{-1}}$}};
\node at (-1.8,-0.25) {\textcolor{purple}{$f_{a}{x^{-1}}$}};
\filldraw[white] (0,-1.2) circle [radius=0.1cm];
\end{tikzpicture}
\caption{Subgraphs of $\hat{\alpha}_{1}$ (left) and $\hat{\alpha}_{2}$ (right)}
\label{fig eg (Ha,x)}
\end{figure}
Observe that $(e_{a})\varphi_{(H_{a},x)}=f_{i}(\overbar{f_{i}x^{-1}})(f_{a}x^{-1})$, where
$f_{i}{x^{-1}}$ is the image in $\hat{\alpha}_{2}$ of  $f_{i}$ under the action of $x^{-1}$, and $\overbar{f_{i}{x^{-1}}}$ is its inverse edge (the same 1-cell, with opposite orientation), and for $j\ne a$ we have $(e_{j})\varphi_{(H_{a},x)}=f_{j}$.
Moreover, we see that $(e_{i}(\overbar{e_{i}x})(e_{a}x))\varphi_{(H_{a},x)}=f_{a}$.
Thus $\varphi_{(H_{a},x)}$ expands some edge paths, contracts some edge paths, and does not change the length of other edge paths.
\end{example}

We now prove some technical lemmas.
Unless otherwise stated, $\Lambda_{w}$ will always concern the word $w$ relating to $\alpha_{1}$.
The following lemmas (and proofs) are similar in structure to a lemma (and proof) of Collins and Zieschang \cite[Lemma 1.5]{Collins1984}, in that we count `subwords' to determine how an automorphism changes the height of a domain. However, since our arguments are applied to different objects, we recover quite different formulae.

\begin{lemma}\label{word length (Ha,x)}
Let $\alpha_{1}$ and $\alpha_{2}=\alpha_{1}(H_{a},x)$ be as in Example \ref{eg change in word length}, and let $w$ be an edge path in $\hat{\alpha}_{1}$.
Then
\[|(w)\varphi_{(H_{a},x)}|_{\alpha_{2}}=|w|_{\alpha_{1}}+2\Lambda_{w}(e_{a})-2\Lambda_{w}(\bar{e_{i}}e_{a})-2\Lambda_{w}((e_{i}{x^{-1}})\bar{e_{i}}e_{a})\]
\end{lemma}

\begin{proof}
Given any word $u$, let $l(u)$ be its unreduced length.
We may assume that $w$ is a reduced word, that is if $w=d_{1}d_{2}\dots d_{m}$, then for any $d_{j}$, we have that $d_{j+1}\ne \overbar{d_{j}}$.
Then $l(w)=\sum_{k=1}^{n}\Lambda_{w}(e_{k})=|w|_{\alpha_{1}}=m$.
Since $w$ is an edge path in $\hat{\alpha}_{1}$, then given a letter (edge) $d_{j}$, there is some $k\in\{1,\dots,n\}$ and some $y\in G$ so that either $d_{j}=e_{k}y$ or $d_{j}=\overbar{e_{k}y}$.
Equivariance of $\varphi_{(H_{a},x)}$ means that $(e_{k}y)\varphi_{(H_{a},x)}=(e_{k})\varphi_{(H_{a},x)}y$, and $(\overbar{e_{k}})\varphi_{(H_{a},x)}=\overbar{(e_{k})\varphi_{(H_{a},x)}}$.

Let $w'$ be the unreduced word $(w)\varphi_{(H_{a},x)}$ in $\alpha_{2}$.
That is, $w'=(d_{1})\varphi_{(H_{a},x)}\dots(d_{m})\varphi_{(H_{a},x)}$
and $l(w')=l((d_{1})\varphi_{(H_{a},x)})+\dots+l((d_{m})\varphi_{(H_{a},x)})$.
We will say $d_{j}\simeq e_{k}$ if $d_{j}=e_{k}y$ or $d_{j}=\overbar{e_{k}y}$ for some $k\in\{1,\dots,n\}$ and some $y\in G$.
If $d_{j}\simeq e_{k}$ for some $k\ne a$ then $(d_{j})\varphi_{(H_{a},x)}\simeq f_{k}$ and $l((d_{j})\varphi_{(H_{a},x)})=l(d_{j})=1$.
If on the other hand $d_{j}\simeq e_{a}$, then $(d_{j})\varphi_{(H_{a},x)}\simeq f_{i}(\overbar{f_{i}x^{-1}})(f_{a}x^{-1})$, and $l((d_{j})\varphi_{(H_{a},x)})=3l(d_{j})$.
Thus:
\[l(w')=\sum_{k\ne a}\Lambda_{w}(e_{k})+3\Lambda(e_{a})=l(w)+2\Lambda(e_{a})=|w|_{\alpha_{1}}+2\Lambda_{w}(e_{a})\]

We now consider reductions to $w'$.
Note that $(\overbar{e_{i}}\varphi_{(H_{a},x)}(e_{a})\varphi_{(H_{a},x)}=\overbar{f_{i}}f_{i}(\overbar{f_{i}x^{-1}})(f_{a}x^{-1})$, which reduces to $(\overbar{f_{i}x^{-1}})(f_{a}x^{-1})$.
Let $w''$ be the result of applying all such reductions to $w'$ (including inversions and $G$-translations of the subword $\overbar{f_{i}}f_{i}$ resulting from images $(\overbar{e_{i}}e_{a})\varphi_{(H_{a},x)}$).
Then the length of $(\overbar{e_{i}}e_{a})\varphi_{(H_{a},x)}$ (and its inversions and $G$-translations) is 2 less in $w''$ than it is in $w'$.
Hence $l(w'')=l(w')-2\Lambda_{w}(\overbar{e_{i}}e_{a})$.

We also have that $(e_{i}x^{-1})\varphi_{(H_{a},x)}(\overbar{e_{i}}e_{a})\varphi_{(H_{a},x)}=(f_{i}x^{-1})(\overbar{f_{i}x^{-1}})(f_{a}x^{-1})$, which reduces to $f_{a}x^{-1}$.
Let $w'''$ be the result of applying all such reductions to $w''$ (including inversions and $G$-translations of the subword $((e_{i}x^{-1})\overbar{e_{i}}e_{a})\varphi_{(H_{a},x)}$).
Then the length of $((e_{i}x^{-1})\overbar{e_{i}}e_{a})\varphi_{(H_{a},x)}$ (and its inversions and $G$-translations) is 2 less in $w'''$ than it is in $w''$.
So $l(w''')=l(w'')-2\Lambda_{w}((e_{i}x^{-1})\overbar{e_{i}}e_{a})$.
Since $w$ was assumed to be reduced, there are no further reductions we can apply to $w'''$, hence $l(w''')=|(w)\varphi_{(H_{a},x)}|_{\alpha_{2}}$.
We therefore have:
\begin{align*}
|(w)\varphi_{(H_{a},x)}|_{\alpha_{2}}=l(w''') 	& 	=l(w'')-2\Lambda_{w}((e_{i}x^{-1})\overbar{e_{i}}e_{a})	\\
							& 	=l(w')-2\Lambda_{w}(\overbar{e_{i}}e_{a})-2\Lambda_{w}((e_{i}x^{-1})\overbar{e_{i}}e_{a})	\\
							& 	=|w|_{\alpha_{1}}+2\Lambda_{w}(e_{a})-2\Lambda_{w}(\overbar{e_{i}}e_{a})-2\Lambda_{w}((e_{i}x^{-1})\overbar{e_{i}}e_{a})	.
\end{align*}
\end{proof}

\begin{lemma}\label{word length (A,x)}
Let $\alpha_{1}$ be the $\alpha$-graph with $\mathfrak{S}$-labelling $(H_{1},\dots,H_{n})$, let $x\in H_{i}$ for some $i$, and let $A\subseteq\{H_{1},\dots,H_{n}\}-\{H_{i}\}$.
Set $\alpha_{2}:=\alpha_{1}(A,x)$, label $\hat{\alpha}_{1}$ and $\hat{\alpha}_{2}$ according to Convention \ref{convention edge labels and map for alpha hat}, and let $\varphi_{(A,x)}:\hat{\alpha}_{1}\to\hat{\alpha}_{2}$  be the equivariant map described in Convention \ref{convention edge labels and map for alpha hat}.
Given an edge path $w$ in $\hat{\alpha}_{1}$, we have:
\[ |(w)\varphi_{(A,x)}|_{\alpha_{2}} = |w|_{\alpha_{1}} + 2\smashoperator[lr]{\sum_{a:H_{a}\in A}} \ \left( \Lambda_{w}(e_{a}) - \Lambda_{w}(\overbar{e_{i}}e_{a}) - \Lambda_{w}( (e_{i}x^{-1})\overbar{e_{i}}e_{a} ) - \smashoperator[lr]{\sum_{b:H_{b}\in A-\{H_{a}\}}} \Lambda_{w}(\overbar{e_{a}}e_{b}) \right) \]
\end{lemma}

\begin{proof}
As in Lemma \ref{word length (Ha,x)}, we let $l(u)$ be the unreduced length of a given word $u$, and we assume that $w$ is a reduced word $w=d_{1}d_{2}\dots d_{m}$ with $l(w)=m=|w|_{\alpha_{1}}$.

Let $w'$ be the unreduced word $(w)\varphi_{(A,x)}$ in $\hat{\alpha}_{2}$.
That is, $w'=(d_{1})\varphi_{(A,x)}\dots(d_{m})\varphi_{(A,x)}$ and $l(w')=l((d_{1})\varphi_{(A,x)})+\dots+l((d_{m})\varphi_{(A,x)})$.
Extended from Lemma \ref{word length (Ha,x)},
we have that $(e_{a})\varphi_{(A,x)}=f_{i}(\overbar{f_{i}x^{-1}})(f_{a}x^{-1})$ for any $a$ such that $H_{a}\in A$, and $(e_{k})\varphi_{(A,x)}=f_{k}$ for any $k$ such that $H_{k}\not\in A$.
Thus if $d_{j}\simeq e_{a}$ where $H_{a}\in A$ then $l((d_{j})\varphi_{(A,x)})=3=l(d_{j})+2$ and if $d_{j}\simeq e_{k}$ where $H_{k}\not\in A$ then $l((d_{j})\varphi_{(A,x)})=1=l(d_{j})$.
Now:
\[l(w') = \sum_{j=1}^{m}l((d_{j})\varphi_{(A,x)}) = 3\smashoperator[lr]{\sum_{a:H_{a}\in A}} \Lambda_{w}(e_{a}) + \smashoperator[lr]{\sum_{k:H_{k}\not\in A}} \Lambda_{w}(e_{k}) = l(w)+2\smashoperator[lr]{\sum_{a:H_{a}\in A}} \Lambda(e_{a})\]

We now consider reductions to the word $w'$.
As in Lemma \ref{word length (Ha,x)}, we have that 
for any $a$ where $H_{a}\in A$, $(\overbar{e_{i}})\varphi_{(A,x)}(e_{a})\varphi_{(A,x)}=\overbar{f_{i}}f_{i}(\overbar{f_{i}x^{-1}})(f_{a}x^{-1})=(\overbar{f_{i}x^{-1}})(f_{a}x^{-1})$
and $(e_{i}x^{-1})\varphi_{(A,x)}(\overbar{e_{i}}e_{a})\varphi_{(A,x)}=(f_{i}x^{-1})(\overbar{f_{i}x^{-1}})(f_{a}x^{-1})=f_{a}x^{-1}$.
Let $w''$ be the result of applying all such reductions to $w'$. Then:
\[l(w'')=l(w')-2\smashoperator[lr]{\sum_{a:H_{a}\in A}} \left( \Lambda_{w}(\overbar{e_{i}}e_{a}) + \Lambda_{w}((e_{i}x^{-1})\overbar{e_{i}}e_{a}) \right) \]

Contrary to Lemma \ref{word length (Ha,x)}, $w''$ is not yet fully reduced.
Indeed, observe that for any distinct $a$ and $b$ with $H_{a},H_{b}\in A$, we have that $(\overbar{e_{a}}e_{b})\varphi_{(A,x)}= \\ \noindent (\overbar{f_{a}x^{-1}})(f_{i}x^{-1})\overbar{f_{i}} f_{i}(\overbar{f_{i}x^{-1}})(f_{b}x^{-1})=(\overbar{f_{a}x^{-1}})(f_{b}x^{-1})$.
Let $w'''$ be the result of applying all such reductions to $w''$ (including inversions and $G$-translations).
Then for each distinct $a$ and $b$ with $H_{a},H_{b}\in A$, the length of $(\overbar{e_{a}}e_{b})\varphi_{(A,x)}$ is 4 less in $w'''$ than it is in $w''$.
Note that for any $a$ and $b$ we have $\Lambda_{w}(\overbar{e_{a}}e_{b})=\Lambda_{w}(\overbar{ \overbar{e_{a}}e_{b}})=\Lambda_{w}(\overbar{e_{b}}e_{a})$.
Thus:
\[ l(w''')=l(w'')-\frac{1}{2}\smashoperator[l]{\sum_{a:H_{a}\in A}} \smashoperator[r]{\sum_{b:H_{b}\in A-\{H_{a}\}}} 4\Lambda_{w}(\overbar{e_{a}}e_{b})=l(w'')-2\smashoperator[l]{\sum_{a:H_{a}\in A}} \smashoperator[r]{\sum_{b:H_{b}\in A-\{H_{a}\}}} \Lambda_{w}(\overbar{e_{a}}e_{b}) \]

Since $w$ was assumed to be reduced, there are now no further reductions we can apply to $w'''$, hence $l(w''')=|(w)\varphi_{(A,x)}|_{\alpha_{2}}$.
We therefore have:
\begin{align*}
& |(w)\varphi_{(A,x)}|_{\alpha_{2}}	\\
=&	l(w''') 	\\
=&	l(w'')-2\smashoperator[l]{\sum_{a:H_{a}\in A}}\smashoperator[r]{ \sum_{b:H_{b}\in A-\{H_{a}\}}} \Lambda_{w}(\overbar{e_{a}}e_{b})	\\
=& 	l(w')-2\smashoperator[lr]{\sum_{a:H_{a}\in A}}\left( \Lambda_{w}(\overbar{e_{i}}e_{a})+\Lambda_{w}((e_{i}x^{-1})\overbar{e_{i}}e_{a})\right)-2\smashoperator[l]{\sum_{a:H_{a}\in A}} \smashoperator[r]{\sum_{b:H_{b}\in A-\{H_{a}\}}} \Lambda_{w}(\overbar{e_{a}}e_{b})	\\
=& 	l(w)+2\smashoperator[lr]{\sum_{a:H_{a}\in A}}\Lambda(e_{a})-2\smashoperator[l]{\sum_{a:H_{a}\in A}}\left(\Lambda_{w}(\overbar{e_{i}}e_{a})+\Lambda_{w}((e_{i}x^{-1})\overbar{e_{i}}e_{a})+ \smashoperator[lr]{\sum_{b:H_{b}\in A-\{H_{a}\}}} \Lambda_{w}(\overbar{e_{a}}e_{b})\right)	\\
=& 	|w|_{\alpha_{1}} +2\smashoperator[l]{\sum_{a:H_{a}\in A}}\left( \Lambda(e_{a})-\Lambda_{w}(\overbar{e_{i}}e_{a})-\Lambda_{w}((e_{i}x^{-1})\overbar{e_{i}}e_{a})-\smashoperator[lr]{\sum_{b:H_{b}\in A-\{H_{a}\}}} \Lambda_{w}(\overbar{e_{a}}e_{b}) \right)	.
\end{align*}
\end{proof}

\begin{lemma}\label{domain height (A,x)}
Let $\alpha_{1}$ be the $\alpha$-graph with $\mathfrak{S}$-labelling $(H_{1},\dots,H_{n})$, and let $(\mathbf{A},\mathbf{x})$ be a relative multiple Whitehead automorphism with respect to $\alpha_{1}$, where $\mathbf{x}\subset H_{i}$ for some $i$.
For brevity, we will write $\sum_{H_{a}\in A_{j}}$ for $\sum_{a:H_{a}\in A_{j}}$ (etc.).
If $\alpha_{2}=\alpha_{1}(\mathbf{A},\mathbf{x})$, then $||\alpha_{2}||-||\alpha_{1}||$ is equal to:
\[ 2\sum_{w\in\mathcal{W}}\sum_{A_{j}\in\mathbf{A}}\sum_{H_{a}\in A_{j}}\left(
\Lambda_{w}(e_{a}) - \Lambda_{w}(\bar{e_{i}}e_{a}) - \Lambda_{w}((e_{i}{x_{j}^{-1}})\bar{e_{i}}e_{a}) - \smashoperator[lr]{\sum_{H_{b}\in A_{j}-\{H_{a}\}}}\Lambda_{w}(\overbar{e_{a}}e_{b}) - \frac{1}{2}\smashoperator[lr]{\sum_{H_{c}\in\hat{A}-A_{j}}}\Lambda_{w}(\overbar{e_{a}}e_{c})
\right)\]
\end{lemma}

\begin{proof}
We have that $(\mathbf{A},\mathbf{x})=(A_{1},x_{1})\dots(A_{K},x_{K})$ for some disjoint subsets $A_{1},\dots,A_{K}\subset\{H_{1},\dots,H_{n}\}-\{H_{i}\}$ and some distinct $x_{1},\dots,x_{K}\in H_{i}$.
We consider a word $w=d_{1}\dots d_{m}$ in $\hat{\alpha}_{1}$ and its unreduced image $w'=(d_{1})\varphi_{(\mathbf{A},\mathbf{x})}\dots(d_{m})\varphi_{(\mathbf{A},\mathbf{x})}$ in $\hat{\alpha}_{2}$.
For any word $u$, let $l(u)$ be its unreduced length.
As in Lemmas \ref{word length (Ha,x)} and \ref{word length (A,x)}, we have that for any $a$ with $H_{a}\in A_{j}\in\mathbf{A}$, $(e_{a})\varphi_{(\mathbf{A},\mathbf{x})}=f_{i}(\overbar{f_{i}x_{j}^{-1}})(f_{a}x_{j}^{-1})$, and for any $k$ with $H_{k}\not\in\hat{A}$, $(e_{k})\varphi_{(\mathbf{A},\mathbf{x})}=f_{k}$.
Thus $l(w')=l(w)+2\sum_{H_{a}\in\hat{A}} \Lambda_{w}(e_{a}) = |w|_{\alpha_{1}}+2\sum_{A_{j}\in\mathbf{A}}\sum_{H_{a}\in A_{j}} \Lambda_{w}(e_{a})$.

As in Lemma \ref{word length (A,x)}, each $(A_{j},x_{j})$ leads to reductions of the forms 
$(\overbar{e_{i}}e_{a})\varphi_{(\mathbf{A},\mathbf{x})}=\overbar{f_{i}}f_{i}(\overbar{f_{i}x_{j}^{-1}})(f_{a}x_{j}^{-1})=(\overbar{f_{i}x_{j}^{-1}})(f_{a}x_{j}^{-1})$,
$((e_{i}x_{j}^{-1})\overbar{e_{i}}e_{a})\varphi_{(\mathbf{A},\mathbf{x})}=(f_{i}x_{j}^{-1})(\overbar{f_{i}x_{j}^{-1}})(f_{a}x_{j}^{-1})=f_{a}x^{-1}$,
and $(\overbar{e_{a}}e_{b})\varphi_{(\mathbf{A},\mathbf{x})}=(\overbar{f_{a}x_{j}^{-1}})(f_{i}x_{j}^{-1})\overbar{f_{i}} f_{i}(\overbar{f_{i}x_{j}^{-1}})(f_{b}x_{j}^{-1})=(\overbar{f_{a}x_{j}^{-1}})(f_{b}x_{j}^{-1})$,
where $a$ and $b$ are such that $H_{a}$ and $H_{b}$ are distinct elements of $A_{j}$.
Let $w''$ be the result of applying all such reductions to $w'$, and observe then that:
\begin{align*}
 l(w'')	&=	l(w')-2\sum_{A_{j}\in\mathbf{A}}\sum_{H_{a}\in A_{j}}\left( \Lambda_{w}(\overbar{e_{i}}e_{a}) + \Lambda_{w}((e_{i}x_{j}^{-1})\overbar{e_{i}}e_{a}) + \smashoperator[lr]{\sum_{H_{b}\in A_{j}-\{H_{a}\}}} \Lambda_{w}(\overbar{e_{a}}e_{b}) \right) 	\\
 	&=	|w|_{\alpha_{1}} + 2\sum_{A_{j}\in\mathbf{A}}\sum_{H_{a}\in A_{j}}\left( \Lambda_{w}(e_{a}) - \Lambda_{w}(\overbar{e_{i}}e_{a}) - \Lambda_{w}((e_{i}x_{j}^{-1})\overbar{e_{i}}e_{a}) - \smashoperator[lr]{\sum_{H_{b}\in A_{j}-\{H_{a}\}}} \Lambda_{w}(\overbar{e_{a}}e_{b}) \right) 
 \end{align*}
 
 We now consider further reductions to $w''$ which come from interactions between distinct $(A_{j},x_{j})$ and $(A_{k},x_{k})$.
 Suppose that $H_{a}\in A_{j}$ and $H_{c}\in A_{k}$, and observe that
 \begin{align*}
 (\overbar{e_{a}}e_{c})\varphi_{(\mathbf{A},\mathbf{x})}	&=	(\overbar{e_{a}})\varphi_{(\mathbf{A},\mathbf{x})}(e_{c})\varphi_{(\mathbf{A},\mathbf{x})}	\\
										&=	(\overbar{f_{a}x_{j}^{-1}})(f_{i}x_{j}^{-1})\overbar{f_{i}}f_{i}(\overbar{f_{i}x_{k}^{-1}})(f_{c}x_{k}^{-1})	\\
										&=	(\overbar{f_{a}x_{j}^{-1}})(f_{i}x_{j}^{-1})(\overbar{f_{i}x_{k}^{-1}})(f_{c}x_{k}^{-1})	.
\end{align*}
Let $w'''$ be the result of applying all such reductions to $w''$, and note that the length of $(\overbar{e_{a}}e_{c})\varphi_{(\mathbf{A},\mathbf{x})}$ is 2 less in $w'''$ than it is in $w''$.
Recall that for any $a$ and $c$, $\Lambda_{w}(\overbar{e_{a}}e_{c})=\Lambda_{w}(\overbar{e_{c}}e_{a})$.
Thus $l(w''')=l(w'')-\frac{1}{2}\sum_{A_{j}\in\mathbf{A}}\sum_{H_{a}\in A_{j}}\sum_{H_{c}\in\hat{A}-A_{j}} 2\Lambda_{w}(\overbar{e_{a}}e_{c})$.

Since $w$ was assumed to be reduced, we now have that there are no further reductions to $w'''$.
Thus:
\begin{align*}
	&|(w)\varphi_{(\mathbf{A},\mathbf{x})}|_{\alpha_{2}}	\\
=	&	l(w''')	\\
=	&	l(w'')-2\sum_{A_{j}\in\mathbf{A}}\sum_{H_{a}\in A_{j}}\sum_{H_{c}\in\hat{A}-A_{j}} \frac{1}{2}\Lambda_{w}(\overbar{e_{a}}e_{c})			\\
=	&	|w|_{\alpha_{1}} + 2\sum_{A_{j}\in\mathbf{A}}\sum_{H_{a}\in A_{j}}\left( \Lambda_{w}(e_{a}) - \Lambda_{w}(\overbar{e_{i}}e_{a}) - \Lambda_{w}((e_{i}x_{j}^{-1})\overbar{e_{i}}e_{a}) - \smashoperator[lr]{\sum_{H_{b}\in A_{j}-\{H_{a}\}}} \Lambda_{w}(\overbar{e_{a}}e_{b}) \right)	\\
	&	 -2\sum_{A_{j}\in\mathbf{A}}\sum_{H_{a}\in A_{j}}\sum_{H_{c}\in\hat{A}-A_{j}} \frac{1}{2}\Lambda_{w}(\overbar{e_{a}}e_{c})	\\
=	&	|w|_{\alpha_{1}} + 2\smashoperator[l]{\sum_{A_{j}\in\mathbf{A}}}\smashoperator[r]{\sum_{H_{a}\in A_{j}}} \ \left( \Lambda_{w}(e_{a}) - \Lambda_{w}(\overbar{e_{i}}e_{a}) - \Lambda_{w}((e_{i}x_{j}^{-1})\overbar{e_{i}}e_{a}) - \smashoperator[lr]{\sum_{H_{b}\in A_{j}-\{H_{a}\}}} \Lambda_{w}(\overbar{e_{a}}e_{b}) \right.	\\
	&	\left. - \frac{1}{2}\smashoperator[lr]{\sum_{H_{c}\in\hat{A}-A_{j}}} \Lambda_{w}(\overbar{e_{a}}e_{c}) \right) 
\end{align*}

Now:
\begin{align*}
||\alpha_{2}||=	&	\smashoperator[lr]{\sum_{w\in\mathcal{W}}}\left(|w|_{\alpha_{2}}\right)-2|\mathcal{W}| 		\\
=	&	\smashoperator[lr]{\sum_{w\in\mathcal{W}}}\left(|w|_{\alpha_{1}} + 2\smashoperator[l]{\sum_{A_{j}\in\mathbf{A}}}\smashoperator[r]{\sum_{H_{a}\in A_{j}}} \ \left(\vphantom{\sum_{H_{c}\in\hat{A}-A_{j}}} \Lambda_{w}(e_{a}) - \Lambda_{w}(\overbar{e_{i}}e_{a}) - \Lambda_{w}((e_{i}x_{j}^{-1})\overbar{e_{i}}e_{a}) 	- \smashoperator[lr]{\sum_{H_{b}\in A_{j}-\{H_{a}\}}} \Lambda_{w}(\overbar{e_{a}}e_{b}) \right.\right.	\\
	&	\left.\left.  - \frac{1}{2}\smashoperator[lr]{\sum_{H_{c}\in\hat{A}-A_{j}}} \Lambda_{w}(\overbar{e_{a}}e_{c}) \right) \right)-2|\mathcal{W}| 	\\
=	&	2\smashoperator[l]{\sum_{w\in\mathcal{W}}}\sum_{A_{j}\in\mathbf{A}}\smashoperator[r]{\sum_{H_{a}\in A_{j}}} \ \left( \Lambda_{w}(e_{a}) - \Lambda_{w}(\overbar{e_{i}}e_{a}) - \Lambda_{w}((e_{i}x_{j}^{-1})\overbar{e_{i}}e_{a}) - \smashoperator[lr]{\sum_{H_{b}\in A_{j}-\{H_{a}\}}} \Lambda_{w}(\overbar{e_{a}}e_{b}) \right. 	\\	
	&	\left. - \frac{1}{2}\smashoperator[lr]{\sum_{H_{c}\in\hat{A}-A_{j}}} \Lambda_{w}(\overbar{e_{a}}e_{c}) \right) +\smashoperator[lr]{\sum_{w\in\mathcal{W}}}\left(|w|_{\alpha_{1}}\right)-2|\mathcal{W}| 	\\
=	&	2\smashoperator[l]{\sum_{w\in\mathcal{W}}}\sum_{A_{j}\in\mathbf{A}}\smashoperator[r]{\sum_{H_{a}\in A_{j}}} \ \left( \Lambda_{w}(e_{a}) - \Lambda_{w}(\overbar{e_{i}}e_{a}) - \Lambda_{w}((e_{i}x_{j}^{-1})\overbar{e_{i}}e_{a}) - \smashoperator[lr]{\sum_{H_{b}\in A_{j}-\{H_{a}\}}} \Lambda_{w}(\overbar{e_{a}}e_{b}) \right.	\\	
	&	\left. - \frac{1}{2}\smashoperator[lr]{\sum_{H_{c}\in\hat{A}-A_{j}}} \Lambda_{w}(\overbar{e_{a}}e_{c}) \right)+||\alpha_{1}||	.
\end{align*}
\end{proof}

\begin{rem}
Since an edge path $w$ in $\hat{\alpha}_{1}$ is uniquely defined by its endpoints, which are preseved by the map $\varphi_{(\mathbf{A},\mathbf{x})}$, we will often write $|w|_{\alpha_{2}}$ for $|(w)\varphi_{(\mathbf{A},\mathbf{x})}|_{\alpha_{2}}$.
\end{rem}

\subsection{Reducible Peaks}

\begin{defn}\label{defn peak}
We will say a path $\alpha_{1}\dash \alpha_{2}\dash \alpha_{3}$ in our Graph/Space of Domains is a \emph{peak} if $||\alpha_{2}||\ge||\alpha_{1}||$ and $||\alpha_{2}||\ge||\alpha_{3}||$, and either $||\alpha_{2}||>||\alpha_{1}||$ or $||\alpha_{2}||>||\alpha_{3}||$ (or both).
Equivalently, $\alpha_{1}\dash \alpha_{2}\dash \alpha_{3}$ is a peak if $||\alpha_{2}||\ge\max\left(||\alpha_{1}||,||\alpha_{3}||\right)$ and $||\alpha_{2}||>\min\left(||\alpha_{1}||,||\alpha_{3}||\right)$.
\end{defn}

In this section, we claim that given a path $\alpha_{1}\dash \alpha_{2}\dash \alpha_{3}$, there exist domains $\alpha_{4}$ and $\alpha_{5}$ such that
\begin{tikzpicture}
\node at (0,0) {$\alpha_{1}$};
\node at (1.5,1) {$\alpha_{2}$};
\node at (3,0) {$\alpha_{3}$};
\node at (1.5,-1) {$\alpha_{4}$};

\draw[thick] (0.18,0.12) -- (1.32,0.88); 
\draw[thick] (1.68,0.88) -- (2.76,0.16); 
\draw[thick] (0.18,-0.12) -- (1.26,-0.84); 
\draw[thick] (1.74,-0.84) -- (2.76,-0.16); 
\end{tikzpicture}
or
\begin{tikzpicture}
\node at (0,0) {$\alpha_{1}$};
\node at (1.5,1) {$\alpha_{2}$};
\node at (3,0) {$\alpha_{3}$};
\node at (1,-1) {$\alpha_{4}$};
\node at (2,-1) {$\alpha_{5}$};

\draw[thick] (0.18,0.12) -- (1.32,0.88); 
\draw[thick] (1.68,0.88) -- (2.76,0.16); 
\draw[thick] (0.15,-0.15) -- (0.85,-0.85); 
\draw[thick] (1.2,-1) -- (1.75,-1); 
\draw[thick] (2.15,-0.85) -- (2.85,-0.15); 
\end{tikzpicture}
forms a loop in the Graph of Domains.

Moreover, we claim that this loop is contractible in our Space of Domains.
Further, we claim that if $\alpha_{1}\dash \alpha_{2}\dash \alpha_{3}$ was a peak, then 
$\alpha_{1}\dash \alpha_{4}\dash \alpha_{3}$ or $\alpha_{1}\dash \alpha_{4}\dash \alpha_{5}\dash \alpha_{3}$
is a reduction; that is, $||\alpha_{4}||<||\alpha_{2}||$ and (if we are in the second of these cases) $||\alpha_{5}||<||\alpha_{2}||$.
In other words, we are claiming that the peak is \emph{reducible}.

Later in this section we will encounter many lemmas, divided into multiple cases. The series of lemmas in each case will roughly follow the structure outlined above.
These will often correspond to lemmas used by Gilbert \cite[Section 2]{Gilbert1987}, but as Gilbert is reducing (cyclic) words of $G$ and we are reducing domains in $\mathcal{C}_{n}$, the proofs are quite different. Recall as well that the notation we use differs subtly to that used by Gilbert.

\begin{defn}\label{defn reducible}
We say a peak
$\alpha_{1}\dash \alpha_{2}\dash \alpha_{3}$ 
is \emph{reducible} (or `can be reduced') if the path $\alpha_{1}\dash \alpha_{2}\dash \alpha_{3}$ is homotopic (in the Space of Domains) to some path $\alpha_{1}=\chi_{0}\dash \chi_{1}\dash \cdots\dash \chi_{k\dash 1}\dash \chi_{k}=\alpha_{3}$ where $||\chi_{i}||<||\alpha_{2}||$ for every $1\le i\le k- 1$.
\end{defn}

\begin{prop}\label{prop i=j peaks reducible} 
Suppose
\begin{tikzcd}[cramped]
\alpha_{1} \ar[r,-<-,dash,"(\mathbf{A}{,}\mathbf{x})"]	& \alpha_{2} \ar[r,->-,dash,"(\mathbf{B}{,}\mathbf{y})"]	& \alpha_{3}
\end{tikzcd}
is a peak in the Space of Domains (whose edges are both of Type $A$), where the $\alpha$-graph in the domain $\alpha_{2}$ has $\mathfrak{S}$-labelling $(H_{1},\dots,H_{n})$.
If there is some $i\in\{1,\dots,n\}$ so that $\mathbf{x},\mathbf{y}\subset H_{i}$, then this peak is reducible.
\end{prop}

This proposition corresponds to \cite[Lemma 2.4]{Gilbert1987}.

\begin{proof}
If $(\mathbf{A},\mathbf{x})=(\mathbf{B},\mathbf{y})$, then $\alpha_{1}=\alpha_{2}(\mathbf{A},\mathbf{x})=\alpha_{2}(\mathbf{B},\mathbf{y})=\alpha_{3}$.
Then our peak is really the loop $\alpha_{1}\dash \alpha_{2}\dash \alpha_{1}$. Since this is a forwards and backwards traversal of a single edge in the Space of Domains, then this is clearly contractible to the point $\alpha_{1}$.
By Definition \ref{defn peak}, we must have that $||\alpha_{1}||<||\alpha_{2}$.
Thus the constant `path' at $\alpha_{1}$ is a reduction of the peak $\alpha_{1}\dash \alpha_{2}\dash \alpha_{1}$.

Now suppose $(\mathbf{A},\mathbf{x})\ne(\mathbf{B},\mathbf{y})$. Since $\mathbf{x},\mathbf{y}\subset H_{i}$, then the $A$-graph in the domain $\alpha_{2}$ with central vertex group $H_{i}$ (call it $A_{i}$) belongs to both $\alpha_{2}\cap\alpha_{1}$ and $\alpha_{2}\cap\alpha_{3}$.
In particular, $A_{i}\in\alpha_{1}\cap\alpha_{3}$, so there is an edge $[\alpha_{1},\alpha_{3}]$ between $\alpha_{1}$ and $\alpha_{3}$ in the space of domains.
Moreover, $A_{i}\in\alpha_{1}\cap\alpha_{2}\cap\alpha_{3}$, so there is a 2-cell $[\alpha_{1},\alpha_{2},\alpha_{3}]$ in the Space of Domains.
Hence the path $\alpha_{1}\dash \alpha_{2}\dash \alpha_{3}$ is homotopic in the Space of Domains to the single edge path $\alpha_{1}\dash \alpha_{3}$.
The condition in Definition \ref{defn reducible} is vacuously satisfied here, thus $\alpha_{1}\dash \alpha_{3}$ is a reduction of the peak $\alpha_{1}\dash \alpha_{2}\dash \alpha_{3}$.
\end{proof}

From now on, we will be considering peaks of the form
\begin{tikzcd}[cramped,sep=small]
\alpha_{1} \ar[r,dash,"A_{i}"]	& \alpha_{2} \ar[r,dash,"A_{j}"]	& \alpha_{3}
\end{tikzcd}
where $i\ne j$, that is, paths
\begin{tikzcd}[cramped]
\alpha_{1} \ar[r,-<-,dash,"(\mathbf{A}{,}\mathbf{x})"]	& \alpha_{2} \ar[r,->-,dash,"(\mathbf{B}{,}\mathbf{y})"]	& \alpha_{3}
\end{tikzcd}
where $\mathbf{x}$ and $\mathbf{y}$ belong to different factor groups of the splitting associated to $\alpha_{2}$.
If the $\alpha$-graph contained in the domain $\alpha_{2}$ has $\mathfrak{S}$-labelling $(H_{1},\dots,H_{n})$, let $i$ and $j$ be the (distinct) elements of $\{1,\dots,n\}$ such that $\mathbf{x}\subset H_{i}$ and $\mathbf{y}\subset H_{j}$.
Recall from Definition \ref{defn rel whitehead auto} that $\mathbf{A}=(A_{1},\dots,A_{k})$, a disjoint partition of a subset of $\{H_{1},\dots,H_{n}\}$, and $\hat{A}:=A_{1}\cup\dots\cup A_{k}$.
Similarly, $\mathbf{B}=(B_{1},\dots,B_{l})$ is another disjoint partition of some subset of $\{H_{1},\dots H_{n}\}$ and $\hat{B}:=B_{1}\cup\dots\cup B_{l}$. 

\begin{obs}\label{obs cases of peak}
Note that we necessarily have $H_{i}\not\in\hat{A}$ and $H_{j}\not\in\hat{B}$.
We adopt the four cases used by Gilbert \cite[Lemma 2.12]{Gilbert1987}:

\begin{description}
\item[Case 1:] $H_{i}\not\in\hat{B}$ and $H_{j}\not\in\hat{A}$
\item[Case 2:] $H_{i}\in\hat{B}$ and $H_{j}\not\in\hat{A}$
\item[Case 3:] $H_{i}\not\in\hat{B}$ and $H_{j}\in\hat{A}$
\item[Case 4:] $H_{i}\in\hat{B}$ and $H_{j}\in\hat{A}$
\end{description}

As Cases 2 and 3 are symmetric, we will not consider Case 3 (since after renaming, this will be identical to Case 2).
Additionally, if $H_{j}\in\hat{A}$ then there exists $A_{p}\in\mathbf{A}$ with $H_{j}\in A_{p}$,
and if $H_{i}\in \hat{B}$, then there exists $B_{q}\in\mathbf{B}$ with $H_{i}\in B_{q}$.
As Gilbert does in \cite[Lemma 2.12]{Gilbert1987}, we further split the remaining cases as follows:

\begin{description}
\item[Case 1(a):] $H_{i}\not\in\hat{B}$ and $H_{j}\not\in\hat{A}$, with $\hat{A}\cap\hat{B}=\emptyset$
\item[Case 1(b):] $H_{i}\not\in\hat{B}$ and $H_{j}\not\in\hat{A}$, with $\hat{A}\cap\hat{B}\neq\emptyset$
\item[Case 2(a):] $H_{i}\in\hat{B}$ (say $H_{i}\in B_{q}$) and $H_{j}\not\in\hat{A}$, with $\hat{A}\subseteq B_{q}$
\item[Case 2(b):] $H_{i}\in\hat{B}$ (say $H_{i}\in B_{q}$) and $H_{j}\not\in\hat{A}$, with $\hat{A}\not\subseteq B_{q}$
\item[Case 4:] $H_{i}\in\hat{B}$ and $H_{j}\in\hat{A}$ (say $H_{i}\in B_{q}$ and $H_{j}\in A_{p}$)
\end{description}
\end{obs}

We now present a series of lemmas in order to prove that in each of the above cases, the peak
\begin{tikzcd}[cramped]
\alpha_{1} \ar[r,-<-,dash,"(\mathbf{A}{,}\mathbf{x})"]	& \alpha_{2} \ar[r,->-,dash,"(\mathbf{B}{,}\mathbf{y})"]	& \alpha_{3}
\end{tikzcd} is reducible.


\subsubsection*{Case 1(a): ${H_{i}\not\in\hat{B}}$ and ${H_{j}\not\in\hat{A}}$, with ${\hat{A}\cap\hat{B}=\emptyset}$}

The lemmas for this case are adapted from \cite[Lemma 2.6]{Gilbert1987}.

\begin{lemma}\label{1a loop} 
If $H_{i}\not\in\hat{B}$ and $H_{j}\not\in\hat{A}$ with $\hat{A}\cap\hat{B}=\emptyset$, then $(\mathbf{A},\mathbf{x})^{-1}(\mathbf{B},\mathbf{y})=(\mathbf{B},\mathbf{y})(\mathbf{A},\mathbf{x})^{-1}$.
\end{lemma}
That is, there exists a vertex $\alpha_{4}$ in our Graph of Domains such that
\begin{tikzpicture}
\draw[-<-] (0.2,0.2) -- (1,1);
\draw[->-] (1.5,1) -- (2.3,0.2);
\draw[->-] (0.2,-0.2) -- (1,-1);
\draw[-<-] (1.5,-1) -- (2.3,-0.2);
\node at (0,0) {$\alpha_{1}$};
\node at (1.25,1.15) {$\alpha_{2}$};
\node[scale=0.8] at (0.2,0.85) {$(\mathbf{A},\mathbf{x})$};
\node[scale=0.8] at (2.35,-0.85) {$(\mathbf{A},\mathbf{x})$};
\node at (2.5,0) {$\alpha_{3}$};
\node[scale=0.8] at (2.3,0.85) {$(\mathbf{B},\mathbf{y})$};
\node at (1.25,-1.15) {$\alpha_{4}$};
\node[scale=0.8] at (0.2,-0.85) {$(\mathbf{B},\mathbf{y})$};
\end{tikzpicture}
is a loop.

\begin{proof} 
By assumption, we have $\alpha_{1}=\alpha_{2}(\mathbf{A},\mathbf{x})$.
Since $H_{j}\not\in\hat{A}$, then the $\mathfrak{S}$-labelling for the $\alpha$-graph $\alpha_{1}$ contained in the domain $\alpha_{1}$ contains the group $H_{j}$. Thus as vertices of $\mathcal{C}_{n}$, we can collapse an edge of $\alpha_{1}$ to achieve an $A$-graph with the group $H_{j}$ at its centre. Then $(\mathbf{B},\mathbf{y})$ is in the stabiliser of this $A$-graph, meaning there is an edge in the Graph of Domains from $\alpha_{1}$ to $\alpha_{1}(\mathbf{B},\mathbf{y})=\alpha_{2}(\mathbf{A},\mathbf{x})(\mathbf{B},\mathbf{y})$, which we will call $\alpha_{4}$.

Now by Definition \ref{defn rel whitehead auto} we can write $(\mathbf{A},\mathbf{x})(\mathbf{B},\mathbf{y})$ as $(A_{1},x_{1})\dots(A_{k},x_{k})(B_{1},y_{1})\dots(B_{l},y_{l})$ for some $k$ and $l$ in $\mathbb{N}$.
So by repeated applications of Lemma \ref{disjoint whitehead autos}, we have $(\mathbf{A},\mathbf{x})(\mathbf{B},\mathbf{y})=(\mathbf{B},\mathbf{y})(\mathbf{A},\mathbf{x})$, noting that each $A_{a}$ and $B_{b}$ are pairwise disjoint.

Hence $\alpha_{3}\cdot(\mathbf{A},\mathbf{x})=\alpha_{2}\cdot(\mathbf{B},\mathbf{y})(\mathbf{A},\mathbf{x})=\alpha_{2}\cdot(\mathbf{A},\mathbf{x})(\mathbf{B},\mathbf{y})=\alpha_{4}$.
\end{proof}

\begin{lemma}\label{1a contractible} 
The loop 
\begin{tikzpicture}
\draw[-<-] (0.2,0.2) -- (1,1);
\draw[->-] (1.5,1) -- (2.3,0.2);
\draw[->-] (0.2,-0.2) -- (1,-1);
\draw[-<-] (1.5,-1) -- (2.3,-0.2);
\node at (0,0) {$\alpha_{1}$};
\node at (1.25,1.15) {$\alpha_{2}$};
\node[scale=0.8] at (0.2,0.85) {$(\mathbf{A},\mathbf{x})$};
\node[scale=0.8] at (2.35,-0.85) {$(\mathbf{A},\mathbf{x})$};
\node at (2.5,0) {$\alpha_{3}$};
\node[scale=0.8] at (2.3,0.85) {$(\mathbf{B},\mathbf{y})$};
\node at (1.25,-1.15) {$\alpha_{4}$};
\node[scale=0.8] at (0.2,-0.85) {$(\mathbf{B},\mathbf{y})$};
\end{tikzpicture}
described in Lemma \ref{1a loop} (with $\mathbf{x}\subset H_{i}$, $\mathbf{y}\subset H_{j}$, $H_{i}\not\in\hat{B}$, $H_{j}\not\in\hat{A}$, and $\hat{A}\cap\hat{B}=\emptyset$) is contractible in our Space of Domains.
\end{lemma}

\begin{proof}
Let $\hat{A}=\{H_{A_{1}},\dots,H_{A_{K}}\}$ and $\hat{B}=\{H_{B_{1}},\dots,H_{B_{L}}\}$. For $a\in\{1,\dots,K\}$, write $(\mathbf{A},\mathbf{x})|_{a}:=(\mathbf{A}\cap\{H_{A_{a}}\},\mathbf{x})$.
Set $\hat{\alpha}_{0,0}:=\alpha_{2}$, and recursively define $\hat{\alpha}_{a+1,b}:=\hat{\alpha}_{a,b}(\mathbf{A},\mathbf{x})|_{a+1}$ and $\hat{\alpha}_{a,b+1}:=\hat{\alpha}_{a,b}(\mathbf{B},\mathbf{y})|_{b+1}$.
Note then that $\alpha_{1}=\hat{\alpha}_{K,0}$, $\alpha_{3}=\hat{\alpha}_{0,L}$, and $\alpha_{4}=\hat{\alpha}_{K,L}$.
We can now build the lattice depicted in Figure \ref{fig lattice 1a}.

\begin{figure}
\centering
\adjustbox{scale=0.8}{
\begin{tikzcd}
\hat{\alpha}_{0,0}
\ar[r,red]
\ar[d,blue]
\ar[rrrrr,bend left=20,"(\mathbf{B}{,}\mathbf{y})"]
\ar[ddddd,bend right=30,"(\mathbf{A}{,}\mathbf{x})" ']
				& \cdots
				\ar[r,red]
						& \hat{\alpha}_{0,b}
						\ar[r,red,"(\mathbf{B}{,}\mathbf{y})|_{b+1}" ']
						\ar[d,blue]
										& \hat{\alpha}_{0,b+1}
										\ar[r,red]
										\ar[d,blue]				
															& \cdots
															\ar[r,red]	
																	& \hat{\alpha}_{0,L}
																	\ar[d,blue]	
																	\ar[ddddd,bend left=30,"(\mathbf{A}{,}\mathbf{x})"]		\\
\vdots	
\ar[d,blue]
				& \ddots	& \vdots
						\ar[d,blue]
										& \vdots
										\ar[d,blue]
															& \ddots	& \vdots
																	\ar[d,blue]			\\
\hat{\alpha}_{a,o}
\ar[r,red]
\ar[d,blue,"(\mathbf{A}{,}\mathbf{x})|_{a+1}"]			
				& \cdots
				\ar[r,red]
						& \hat{\alpha}_{a,b}
						\ar[r,red]
						\ar[d,blue]			
										& \hat{\alpha}_{a,b+1}
										\ar[r,red]
										\ar[d,blue]				
															& \cdots
															\ar[r,red]	
																	& \hat{\alpha}_{a,L}
																	\ar[d,blue]			\\
\hat{\alpha}_{a+1,0}
\ar[r,red]
\ar[d,blue]
				& \cdots
				\ar[r,red]
						& \hat{\alpha}_{a+1,b}
						\ar[r,red]
						\ar[d,blue]
										& \hat{\alpha}_{a+1,b+1}
										\ar[r,red]
										\ar[d,blue]
															& \cdots
															\ar[r,red]	
																	& \hat{\alpha}_{a+1,L}
																	\ar[d,blue]			\\
\vdots	
\ar[d,blue]
				& \ddots	& \vdots
						\ar[d,blue]
										& \vdots
										\ar[d,blue]
															& \ddots	& \vdots
																	\ar[d,blue]			\\
\hat{\alpha}_{K,0}
\ar[r,red]
\ar[rrrrr,bend right=20,"(\mathbf{B}{,}\mathbf{y})" ']			& \cdots
				\ar[r,red]	& \hat{\alpha}_{K,b}
						\ar[r,red]			& \hat{\alpha}_{K,b+1}
										\ar[r,red]				& \cdots
															\ar[r,red]
																	& \hat{\alpha}_{K,L}	\\
\end{tikzcd}
}
\caption{Lattice Describing $(A,x)(B,y)=(B,y)(A,x)$ in Case 1a}
\label{fig lattice 1a}
\end{figure}
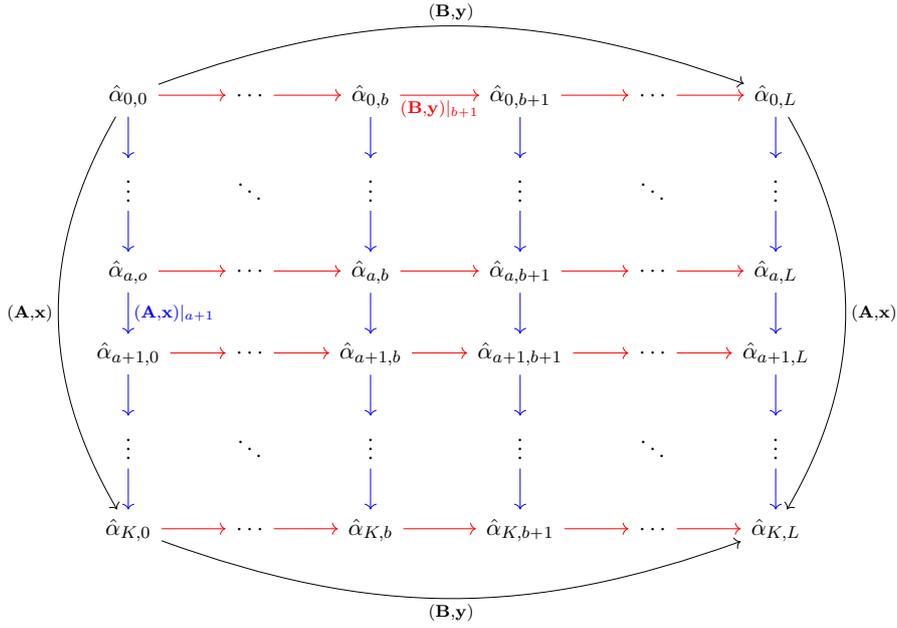

Note that for each $a\in\{0,\dots,K-1\}$ and $b\in\{0,\dots,L-1\}$, the square

\begin{tikzcd}
\hat{\alpha}_{a,b}
\ar[r,red,"(\mathbf{B}{,}\mathbf{y})|_{b+1}"]
\ar[d,blue,"(\mathbf{A}{,}\mathbf{x})|_{a+1}" ']
			& \hat{\alpha}_{a,b+1}
			\ar[d,blue,"(\mathbf{A}{,}\mathbf{x})|_{a+1}"]				\\
\hat{\alpha}_{a+1,b}
\ar[r,red,"(\mathbf{B}{,}\mathbf{y})|_{b+1}" ']
			& \hat{\alpha}_{a+1,b+1}	\\
\end{tikzcd}
is such that the $B$-graph $B_{i,j,B_{b+1}}$ in the domain $\hat{\alpha}_{a,b}$
(and similarly the graph $B_{j,i,A_{a+1}}$) lives in the intersection $\hat{\alpha}_{a,b}\cap\hat{\alpha}_{a+1,b}\cap\hat{\alpha}_{a,b+1}\cap\hat{\alpha}_{a+1,b+1}$, hence this intersection is non-empty. By Definition \ref{Space of Domains defn}, this means that the square is contractible in the Space of Domains.
In the same way, loops $\hat{\alpha}_{0,b}$ -- $\cdots$ -- $\hat{\alpha}_{K,b}$ -- $\hat{\alpha}_{0,b}$ are contractible via the $A$-graph $A_{j}$ in $\hat{\alpha}_{0,b}$, 
and similarly $\hat{\alpha}_{a,0}$ -- $\cdots$ -- $\hat{\alpha}_{a,L}$ -- $\hat{\alpha}_{a,0}$ via $A_{i}$.
Since every `cell' in our lattice is contractible in our Space of Domains, then so too is our initial loop (for which the lattice is akin to a tiling).
\end{proof}

Note that we only actually needed the first and last columns (or rows) of this lattice, since the graph $B_{j,i,A_{a}}$ (or $B_{i,j,B_{b}}$) lies in the intersection of all the domains in row $a$ (or column $b$).

\begin{lemma}\label{1a word length} 
Suppose we have a loop
\begin{tikzpicture}
\draw[-<-] (0.2,0.2) -- (1,1);
\draw[->-] (1.5,1) -- (2.3,0.2);
\draw[->-] (0.2,-0.2) -- (1,-1);
\draw[-<-] (1.5,-1) -- (2.3,-0.2);
\node at (0,0) {$\alpha_{1}$};
\node at (1.25,1.15) {$\alpha_{2}$};
\node[scale=0.8] at (0.2,0.85) {$(H_{a},{x})$};
\node[scale=0.8] at (2.35,-0.85) {$(H_{a},{x})$};
\node at (2.5,0) {$\alpha_{3}$};
\node[scale=0.8] at (2.3,0.85) {$(H_{b},{y})$};
\node at (1.25,-1.15) {$\alpha_{4}$};
\node[scale=0.8] at (0.2,-0.85) {$(H_{b},{y})$};
\end{tikzpicture}
where $x\in H_{i}$, $y\in H_{j}$, $H_{a}\ne H_{b}$, and $\{H_{a},H_{b}\}\cap\{H_{i},H_{j}\}=\emptyset$.
Then for any edge path $w$ in $\alpha_{2}$, we have
$|w|_{\alpha_{4}}-|w|_{\alpha_{1}}=|w|_{\alpha_{3}}-|w|_{\alpha_{2}}$.
\end{lemma}

\begin{proof}
Let $w$ be a reduced edge path in $\hat{\alpha}_{2}$, the $G$-tree associated to the $\alpha$-graph $\alpha_{2}$ contained in the domain $\alpha_{2}$.
By assumption, $\alpha_{2}$ is the (domain containing the) $\alpha$-graph with $\mathfrak{S}$-labelling $(H_{1},\dots,H_{n})=(H_{a},H_{b},H_{i},H_{j},H_{v_{1}},\dots,H_{v_{n-4}})$.
It follows that $\alpha_{4}$ is the (domain containing the) $\alpha$-graph with $\mathfrak{S}$-labelling $(H_{a}^{x},H_{b}^{y},H_{i},H_{j},H_{v_{1}},\dots,H_{v_{n-4}})$, where $x\in H_{i}$ and $y\in H_{j}$. Suppose the edges in $\hat{\alpha}_{2}$ are labelled with $e$'s, and the edges in $\hat{\alpha}_{4}$ are labelled with $h$'s.
Let $\varphi_{24}:\hat{\alpha}_{2}\to\hat{\alpha}_{4}$ be the equivariant map described in Convention \ref{convention edge labels and map for alpha hat}.
Then $(e_{a})\varphi_{24}=h_{i}(\overbar{h_{i}x^{-1}})(h_{a}x^{-1})$, $(e_{b})\varphi_{24}=h_{j}(\overbar{h_{j}y^{-1}})(h_{j}y^{-1})$, and $(e_{k})\varphi_{24}=h_{k}$ for any $k\ne a,b$.

Utilising the ideas from the proof of Lemma \ref{word length (Ha,x)}, we will let $l(u)$ be the unredced length of a word $u$, and let $w'$ be the unreduced word $(w)\varphi_{24}$ in $\hat{\alpha}_{4}$.
Then $l(w')=l(w)+2\Lambda_{w}(e_{a})+2\Lambda_{w}(e_{b})=|w|_{\alpha_{2}}+2\left( \Lambda_{w}(e_{a})+\Lambda_{w}(e_{b}) \right)$.

Let $w''$ be the result of applying all reductions of the forms
\begin{align*}
&	(\overbar{e_{i}}e_{a})\varphi			=	\overbar{h_{i}}h_{i}(\overbar{h_{i}x^{-1}})(h_{a}x^{-1})	
								=	(\overbar{h_{i}x^{-1}})(h_{a}x^{-1})					\\
&	((e_{i}x^{-1})\overbar{e_{i}}e_{a})\varphi	=	(h_{i}x^{-1})(\overbar{h_{i}x^{-1}})(h_{a}x^{-1})			
								=	h_{a}x^{-1}									\\
&	(\overbar{e_{j}}e_{b})\varphi			=	\overbar{h_{j}}h_{j}(\overbar{h_{j}y^{-1}})(h_{j}y^{-1})	
								=	(\overbar{h_{j}y^{-1}})(h_{j}y^{-1})					\\
&	((e_{j}y^{-1})\overbar{e_{j}}e_{b})\varphi	=	(h_{j}y^{-1})(\overbar{h_{j}y^{-1}})(h_{j}y^{-1})			
								=	h_{j}y^{-1}									,
\end{align*}
as in the proof of Lemma \ref{word length (Ha,x)}.
Then:
\begin{align*}
l(w'')=	&	l(w')-2\Lambda_{w}(\overbar{e_{i}}e_{a})-2\Lambda_{w}((e_{i}x^{-1})\overbar{e_{i}}e_{a})-2\Lambda_{w}(\overbar{e_{j}}e_{b})-2\Lambda_{w}((e_{j}y^{-1})\overbar{e_{j}}e_{b})	\\
=	&	|w|_{\alpha_{2}}+2\left( \Lambda_{w}(e_{a})+\Lambda_{w}(e_{b})-\Lambda_{w}(\overbar{e_{i}}e_{a})-\Lambda_{w}(\overbar{e_{j}}e_{b})-\Lambda_{w}((e_{i}x^{-1})\overbar{e_{i}}e_{a}) \right.	\\
	&	\left. -\Lambda_{w}((e_{j}y^{-1})\overbar{e_{j}}e_{b}) \right)	.
\end{align*}

Since $H_{i}$, $H_{j}$, $H_{a}$, and $H_{b}$ are all distinct, then there are no possible `cross-reductions' to $w''$ as found in the proofs of Lemmas \ref{word length (A,x)} and \ref{domain height (A,x)}.
Since $w$ was assumed to be reduced in $\hat{\alpha}_{2}$, this now implies that $w''$ is reduced in $\hat{\alpha}_{4}$, hence $|w|_{\alpha_{4}}=l(w'')$.

Thus by Lemma \ref{word length (Ha,x)}, we have:
\begin{align*}
	&	|w|_{\alpha_{1}}	+	|w|_{\alpha_{3}}	-	|w|_{\alpha_{2}}											\\
=	&	|w|_{\alpha_{2}}	+	2\Lambda_{w}(e_{a})-2\Lambda_{w}(\bar{e_{i}}e_{a})-2\Lambda_{w}((e_{i}x^{-1})\bar{e_{i}}e_{a})	\\
	&+	|w|_{\alpha_{2}}	+	2\Lambda_{w}(e_{b})-2\Lambda_{w}(\bar{e_{j}}e_{b})-2\Lambda_{w}((e_{j}y^{-1})\bar{e_{j}}e_{b})	- 	|w|_{\alpha_{2}}			\\	
=	&	|w|_{\alpha_{2}}	+	2\left(	\Lambda_{w}(e_{a})-\Lambda_{w}(\bar{e_{i}}e_{a})-\Lambda_{w}((e_{i}x^{-1})\bar{e_{i}}e_{a})	+	\Lambda_{w}(e_{b})-\Lambda_{w}(\bar{e_{j}}e_{b})-\Lambda_{w}((e_{j}y^{-1})\bar{e_{j}}e_{b})	 \right)	\\
=	&	|w|_{\alpha_{4}}		.
\end{align*}
\end{proof}

\begin{obs}\label{obs 1a word length}
Recall from the proof of Lemma \ref{1a contractible} that we can write the loop 
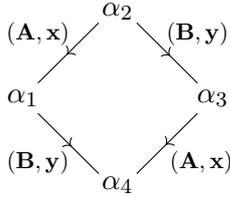
\begin{tikzpicture}
\draw[-<-] (0.2,0.2) -- (1,1);
\draw[->-] (1.5,1) -- (2.3,0.2);
\draw[->-] (0.2,-0.2) -- (1,-1);
\draw[-<-] (1.5,-1) -- (2.3,-0.2);
\node at (0,0) {$\alpha_{1}$};
\node at (1.25,1.15) {$\alpha_{2}$};
\node[scale=0.8] at (0.2,0.85) {$(\mathbf{A},\mathbf{x})$};
\node[scale=0.8] at (2.35,-0.85) {$(\mathbf{A},\mathbf{x})$};
\node at (2.5,0) {$\alpha_{3}$};
\node[scale=0.8] at (2.3,0.85) {$(\mathbf{B},\mathbf{y})$};
\node at (1.25,-1.15) {$\alpha_{4}$};
\node[scale=0.8] at (0.2,-0.85) {$(\mathbf{B},\mathbf{y})$};
\end{tikzpicture}
as a lattice, each `cell' of which has the form of the loop in Lemma \ref{1a word length}. Using the notation from Lemma \ref{1a contractible}, for any edge path $w$ in $\alpha_{2}=\hat{\alpha}_{0,0}$, we have that:
\begin{align*}
	&|w|_{\alpha_{4}}-|w|_{\alpha_{1}}	\\
=	&|w|_{\hat{\alpha}_{K,L}} - |w|_{\hat{\alpha}_{K,0}}	\\
=	&|w|_{\hat{\alpha}_{K,L}} - |w|_{\hat{\alpha}_{K,L-1}} + |w|_{\hat{\alpha}_{K,L-1}} - \cdots - |w|_{\hat{\alpha}_{K,1}} + |w|_{\hat{\alpha}_{K,1}} - |w|_{\hat{\alpha}_{K,0}}	\\
=	&|w|_{\hat{\alpha}_{K-1,L}} - |w|_{\hat{\alpha}_{K-1,L-1}} + |w|_{\hat{\alpha}_{K-1,L-1}} - \cdots - |w|_{\hat{\alpha}_{K-1,1}} + |w|_{\hat{\alpha}_{K-1,1}} - |w|_{\hat{\alpha}_{K-1,0}}	\\
\vdots 	&	\\
=	&|w|_{\hat{\alpha}_{0,L}} - |w|_{\hat{\alpha}_{0,L-1}} + |w|_{\hat{\alpha}_{0,L-1}} - \cdots - |w|_{\hat{\alpha}_{0,1}} + |w|_{\hat{\alpha}_{0,1}} - |w|_{\hat{\alpha}_{0,0}}	\\
=	&|w|_{\hat{\alpha}_{0,L}} - |w|_{\hat{\alpha}_{0,0}}	\\
=	&|w|_{\alpha_{3}}-|w|_{\alpha_{2}} 	.	\\
\end{align*}
\end{obs}

\begin{lemma}\label{1a domain height} 
If $H_{i}\not\in\hat{B}$ and $H_{j}\not\in\hat{A}$ and $\hat{A}\cap\hat{B}=\emptyset$ then
$||\alpha_{4}||-||\alpha_{1}||=||\alpha_{3}||-||\alpha_{2}||$, where $\alpha_{4}=\alpha_{1}(\mathbf{A},\mathbf{x})(\mathbf{B},\mathbf{y})$.
\end{lemma}

\begin{proof}
By Lemma \ref{1a word length} and the Observation \ref{obs 1a word length},
\begin{align*}
||\alpha_{4}||-||\alpha_{1}|| 	&=	\sum_{w\in\mathcal{W}}\Bigl(|w|_{\alpha_{4}}-2\Bigr) - \sum_{w\in\mathcal{W}}\Bigl(|w|_{\alpha_{1}}-2\Bigr) 	\\
					&=	\sum_{w\in\mathcal{W}}\Bigl(|w|_{\alpha_{4}}-2-|w|_{\alpha_{1}}+2\Bigr)						\\
					&=	\sum_{w\in\mathcal{W}}\Bigl(|w|_{\alpha_{3}}-|w|_{\alpha_{2}}\Bigr)							\\
					&=	\sum_{w\in\mathcal{W}}\Bigl(|w|_{\alpha_{3}}-2\Bigr) - \sum_{w\in\mathcal{W}}\Bigl(|w|_{\alpha_{2}}-2\Bigr) 	\\
					&=	||\alpha_{3}||-||\alpha_{2}||	.
\end{align*}
\end{proof}

\begin{lemma}\label{1a peak} 
If
\begin{tikzcd}[cramped]
\alpha_{1} \ar[r,-<-,dash,"(\mathbf{A}{,}\mathbf{x})"]	& \alpha_{2} \ar[r,->-,dash,"(\mathbf{B}{,}\mathbf{y})"]	& \alpha_{3}
\end{tikzcd}
is a peak in the Graph of Domains, where $H_{i}\not\in\hat{B}$ and $H_{j}\not\in\hat{A}$ with $\hat{A}\cap\hat{B}=\emptyset$, then $||\alpha_{2}\cdot(\mathbf{A},\mathbf{x})(\mathbf{B},\mathbf{y})||<||\alpha_{2}||$.
\end{lemma}

\begin{proof}
By Lemma \ref{1a domain height}, $||\alpha_{2}(\mathbf{A},\mathbf{x})(\mathbf{B},\mathbf{y})||=||\alpha_{4}||=||\alpha_{3}||+||\alpha_{1}||-||\alpha_{2}||$.
Since
\begin{tikzcd}[cramped]
\alpha_{1} \ar[r,-<-,dash,"(\mathbf{A}{,}\mathbf{x})"]	& \alpha_{2} \ar[r,->-,dash,"(\mathbf{B}{,}\mathbf{y})"]	& \alpha_{3}
\end{tikzcd}
is a peak, then $||\alpha_{2}||\ge\max(||\alpha_{1}||,||\alpha_{3}||)$ and $||\alpha_{2}||>\min(||\alpha_{1}||,||\alpha_{3}||)$.
Now:
\begin{align*}
||\alpha_{4}||	&=\max(||\alpha_{1}||,||\alpha_{3}||)+\min(||\alpha_{1}||,||\alpha_{3}||)-||\alpha_{2}||	&\\
			&<||\alpha_{2}||+||\alpha_{2}||-||\alpha_{2}||								&\\
			&=||\alpha_{2}||		.											&
\end{align*}
\end{proof}

\begin{prop}\label{prop 1a summary}
Let $H_{1}\ast\dots\ast H_{n}$ be an $\mathfrak{S}$ free factor splitting for $G$.
Let $(\mathbf{A},\mathbf{x})$ and $(\mathbf{B},\mathbf{y})$ be relative multiple Whitehead automorphisms with $\mathbf{x}\subset H_{i}$ and $\mathbf{y}\subset H_{j}$ and $\hat{A},\hat{B}\subset\{H_{1},\dots,H_{n}\}-\{H_{i},H_{j}\}$ such that $\hat{A}\cap\hat{B}=\emptyset$.
Let $\alpha_{2}$ be the domain whose $\alpha$-graph has $\mathfrak{S}$-labelling $(H_{1},\dots,H_{n})$, and let $\alpha_{1}=\alpha_{2}(\mathbf{A},\mathbf{x})$ and $\alpha_{3}=\alpha_{2}(\mathbf{B},\mathbf{y})$.
If
\begin{tikzcd}[cramped]
\alpha_{1} \ar[r,-<-,dash,"(\mathbf{A}{,}\mathbf{x})"]	& \alpha_{2} \ar[r,->-,dash,"(\mathbf{B}{,}\mathbf{y})"]	& \alpha_{3}
\end{tikzcd}
is a peak, then it is reducible.
\end{prop}

\begin{proof}
By Lemmas \ref{1a loop} and \ref{1a contractible}, the path 
\begin{tikzcd}[cramped]
\alpha_{1} \ar[r,-<-,dash,"(\mathbf{A}{,}\mathbf{x})"]	& \alpha_{2} \ar[r,->-,dash,"(\mathbf{B}{,}\mathbf{y})"]	& \alpha_{3}
\end{tikzcd}
is homotopic in the Space of Domains to the path
\begin{tikzcd}[cramped]
\alpha_{1} \ar[r,->-,dash,"(\mathbf{B}{,}\mathbf{y})"]	& \alpha_{2}(\mathbf{A},\mathbf{x})(\mathbf{B},\mathbf{y}) \ar[r,-<-,dash,"(\mathbf{A}{,}\mathbf{x})"]	& \alpha_{3}
\end{tikzcd}
, and by Lemma \ref{1a peak}, $||\alpha_{2}(\mathbf{A},\mathbf{x})(\mathbf{B},\mathbf{y})||<||\alpha_{2}||$.
Thus by Definition \ref{defn reducible}, the peak
\begin{tikzcd}[cramped]
\alpha_{1} \ar[r,-<-,dash,"(\mathbf{A}{,}\mathbf{x})"]	& \alpha_{2} \ar[r,->-,dash,"(\mathbf{B}{,}\mathbf{y})"]	& \alpha_{3}
\end{tikzcd}
is reducible.
\end{proof}


\subsubsection*{Case 1(b): $H_{i}\not\in\hat{B}$ and $H_{j}\not\in\hat{A}$, with $\hat{A}\cap\hat{B}\neq\emptyset$}

\begin{lemma}\label{1b loop} 
If $H_{i}\not\in\hat{B}$ and $H_{j}\not\in\hat{A}$ with $\hat{A}\cap\hat{B}\neq\emptyset$, then 
$(\mathbf{A},\mathbf{x})^{-1}(\mathbf{B},\mathbf{y})
=(\mathbf{A}\cap\hat{B},\mathbf{x})^{-1}(\mathbf{B},\mathbf{y})(\mathbf{A}-\hat{B},\mathbf{x})^{-1}$ and 
$(\mathbf{A},\mathbf{x})^{-1}(\mathbf{B},\mathbf{y})
=(\mathbf{B}-\hat{A},\mathbf{y})(\mathbf{A},\mathbf{x})^{-1}(\mathbf{B}\cap\hat{A},\mathbf{y})$.
\end{lemma}
That is, there exist vertices $\alpha_{4}$, $\alpha_{5}$, $\alpha_{4}'$, $\alpha_{5}'$ in our Graph of Domains such that 
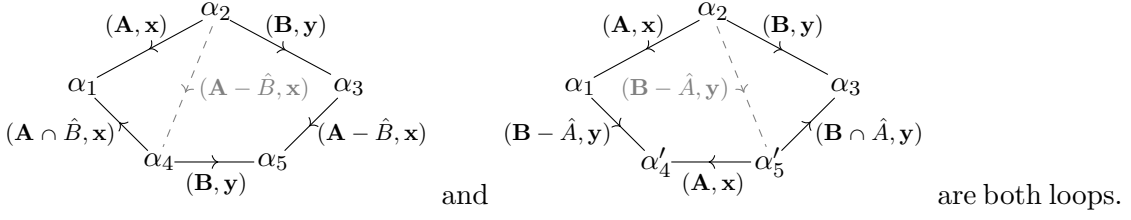
\begin{tikzpicture}
\node at (0,0) {$\alpha_{1}$};
\node at (1.75,1) {$\alpha_{2}$};
\node at (3.5,0) {$\alpha_{3}$};
\node at (1,-1) {$\alpha_{4}$};
\node at (2.5,-1) {$\alpha_{5}$};
\draw[-<-] (0.18,0.12) -- (1.57,0.88); 
\draw[->-] (1.93,0.88) -- (3.26,0.16); 
\draw[-<-] (0.15,-0.15) -- (0.85,-0.85); 
\draw[->-] (1.2,-1) -- (2.25,-1); 
\draw[-<-] (2.65,-0.85) -- (3.35,-0.15); 
\draw[->-,gray,dashed] (1.7,0.8) -- (1.05,-0.8); 
\node[scale=0.8] at (0.7,0.8) {$(\mathbf{A},\mathbf{x})$};
\node[scale=0.8] at (2.8,0.8) {$(\mathbf{B},\mathbf{y})$};
\node[scale=0.8] at (1.75,-1.3) {$(\mathbf{B},\mathbf{y})$};
\node[scale=0.8] at (-0.3,-0.6) {$(\mathbf{A}\cap\hat{B},\mathbf{x})$};
\node[scale=0.8] at (3.8,-0.6) {$(\mathbf{A}-\hat{B},\mathbf{x})$};
\node[scale=0.8] at (2.25,0) {\textcolor{gray}{$(\mathbf{A}-\hat{B},\mathbf{x})$}};
\end{tikzpicture}
and
\begin{tikzpicture}
\node at (0,0) {$\alpha_{1}$};
\node at (1.75,1) {$\alpha_{2}$};
\node at (3.5,0) {$\alpha_{3}$};
\node at (1,-1) {$\alpha_{4}'$};
\node at (2.5,-1) {$\alpha_{5}'$};
\draw[-<-] (0.18,0.12) -- (1.57,0.88); 
\draw[->-] (1.93,0.88) -- (3.26,0.16); 
\draw[->-] (0.15,-0.15) -- (0.85,-0.85); 
\draw[-<-] (1.2,-1) -- (2.25,-1); 
\draw[->-] (2.65,-0.85) -- (3.35,-0.15); 
\draw[->-,gray,dashed] (1.8,0.8) -- (2.45,-0.8); 
\node[scale=0.8] at (0.7,0.8) {$(\mathbf{A},\mathbf{x})$};
\node[scale=0.8] at (2.8,0.8) {$(\mathbf{B},\mathbf{y})$};
\node[scale=0.8] at (1.75,-1.3) {$(\mathbf{A},\mathbf{x})$};
\node[scale=0.8] at (-0.3,-0.6) {$(\mathbf{B}-\hat{A},\mathbf{y})$};
\node[scale=0.8] at (3.8,-0.6) {$(\mathbf{B}\cap\hat{A},\mathbf{y})$};
\node[scale=0.8] at (1.25,0) {\textcolor{gray}{$(\mathbf{B}-\hat{A},\mathbf{y})$}};
\end{tikzpicture}
are both loops.

\begin{proof} 
By Proposition \ref{whitehead notational properties} ($1$), $(\mathbf{A},\mathbf{x})(\mathbf{A}\cap\hat{B},\mathbf{x})^{-1}=(\mathbf{A}-\hat{B},\mathbf{x})$.
Note that $\widehat{\mathbf{A}-\hat{B}}$ and $\hat{B}$ are disjoint, and we have $H_{i}\not\in\hat{B}$ and $H_{j}\not\in\hat{A}-\hat{B}$.
Thus by Lemma \ref{1a loop}, $(\mathbf{A}-\hat{B},\mathbf{x})(\mathbf{B},\mathbf{y})=(\mathbf{B},\mathbf{y})(\mathbf{A}-\hat{B},\mathbf{x})$.
The second statement follows similarly, by appropriately switching $A$ and $B$ and $x$ and $y$.
\end{proof}

\begin{lemma}\label{1b contractible} 
The loops

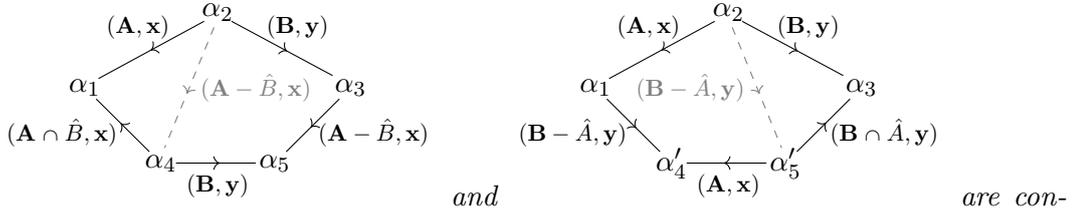
\begin{tikzpicture}
\node at (0,0) {$\alpha_{1}$};
\node at (1.75,1) {$\alpha_{2}$};
\node at (3.5,0) {$\alpha_{3}$};
\node at (1,-1) {$\alpha_{4}$};
\node at (2.5,-1) {$\alpha_{5}$};
\draw[-<-] (0.18,0.12) -- (1.57,0.88); 
\draw[->-] (1.93,0.88) -- (3.26,0.16); 
\draw[-<-] (0.15,-0.15) -- (0.85,-0.85); 
\draw[->-] (1.2,-1) -- (2.25,-1); 
\draw[-<-] (2.65,-0.85) -- (3.35,-0.15); 
\draw[->-,gray,dashed] (1.7,0.8) -- (1.05,-0.8); 
\node[scale=0.8] at (0.7,0.8) {$(\mathbf{A},\mathbf{x})$};
\node[scale=0.8] at (2.8,0.8) {$(\mathbf{B},\mathbf{y})$};
\node[scale=0.8] at (1.75,-1.3) {$(\mathbf{B},\mathbf{y})$};
\node[scale=0.8] at (-0.3,-0.6) {$(\mathbf{A}\cap\hat{B},\mathbf{x})$};
\node[scale=0.8] at (3.8,-0.6) {$(\mathbf{A}-\hat{B},\mathbf{x})$};
\node[scale=0.8] at (2.25,0) {\textcolor{gray}{$(\mathbf{A}-\hat{B},\mathbf{x})$}};
\end{tikzpicture}
and
\begin{tikzpicture}
\node at (0,0) {$\alpha_{1}$};
\node at (1.75,1) {$\alpha_{2}$};
\node at (3.5,0) {$\alpha_{3}$};
\node at (1,-1) {$\alpha_{4}'$};
\node at (2.5,-1) {$\alpha_{5}'$};
\draw[-<-] (0.18,0.12) -- (1.57,0.88); 
\draw[->-] (1.93,0.88) -- (3.26,0.16); 
\draw[->-] (0.15,-0.15) -- (0.85,-0.85); 
\draw[-<-] (1.2,-1) -- (2.25,-1); 
\draw[->-] (2.65,-0.85) -- (3.35,-0.15); 
\draw[->-,gray,dashed] (1.8,0.8) -- (2.45,-0.8); 
\node[scale=0.8] at (0.7,0.8) {$(\mathbf{A},\mathbf{x})$};
\node[scale=0.8] at (2.8,0.8) {$(\mathbf{B},\mathbf{y})$};
\node[scale=0.8] at (1.75,-1.3) {$(\mathbf{A},\mathbf{x})$};
\node[scale=0.8] at (-0.3,-0.6) {$(\mathbf{B}-\hat{A},\mathbf{y})$};
\node[scale=0.8] at (3.8,-0.6) {$(\mathbf{B}\cap\hat{A},\mathbf{y})$};
\node[scale=0.8] at (1.25,0) {\textcolor{gray}{$(\mathbf{B}-\hat{A},\mathbf{y})$}};
\end{tikzpicture}
are contractible in our Space of Domains.
\end{lemma}

\begin{proof}
Note that the $\alpha_{1}$--$\alpha_{2}$--$\alpha_{4}$ triangle can be `filled' with an $A_{i}$ graph (that is to say, since $(\mathbf{A},\mathbf{x})$, $(\mathbf{A}\cap\hat{B},\mathbf{x})$ and $(\mathbf{A}-\hat{B},\mathbf{x})$ all live in the stabiliser of the $A_{i}$-graph with $H_{i}$ at its centre in the domain $\alpha_{2}$, then by Definition \ref{Space of Domains defn}, we have a 2-cell $[\alpha_{1},\alpha_{2},\alpha_{3}]$).
Similarly, the $\alpha_{2}$--$\alpha_{3}$--$\alpha_{5}'$ triangle can be `filled' with an $A_{j}$ graph.
Now $(\widehat{\mathbf{A}-\hat{B}})\cap\hat{B}=\emptyset=\hat{A}\cap(\widehat{B-\hat{A}})$, so by Lemma \ref{1a contractible} (Case 1a) the squares $\alpha_{4}\dash \alpha_{2}\dash \alpha_{3}\dash \alpha_{5}\dash \alpha_{4}$ and $\alpha_{1}\dash \alpha_{2}\dash \alpha_{5}'\dash \alpha_{4}'\dash \alpha_{1}$ are both contractible.
\end{proof}

\begin{lemma}\label{lemma acting on common C}
Let $H'=(H_{1}',\dots,H_{n}')$, and suppose $\mathbf{u}\subset H_{i}'$, $\mathbf{v}\subset H_{j}'$, $\mathbf{C},\mathbf{D}\subset\hat{H'}$ with $\hat{C}=\hat{D}$ and $H_{i}',H_{j}'\not\in\hat{C}$.
If $\|\alpha(\mathbf{C},\mathbf{u})\|-\|\alpha\|\le0$ and $\|\alpha(\mathbf{D},\mathbf{v})\|-\|\alpha\|\le0$ then $\|\alpha(\mathbf{C},\mathbf{u})\|-\|\alpha\|=\|\alpha(\mathbf{D},\mathbf{v})\|-\|\alpha\|=0$.
\end{lemma}

\begin{proof}
Suppose $\|\alpha(\mathbf{C},\mathbf{u})\|-\|\alpha\|\le0$ and $\|\alpha(\mathbf{D},\mathbf{v})\|-\|\alpha\|\le0$.
By Lemma \ref{domain height (A,x)}:
\begin{align*}
\|\alpha(\mathbf{C},\mathbf{u})\|-\|\alpha\| 
=	&	2\smashoperator[l]{\sum_{w\in\mathcal{W}}}\sum_{C_{k}\in\mathbf{C}}\smashoperator[r]{\sum_{H_{c}\in C_{k}}} \ \left( \vphantom{\sum_{H_{a}\in\hat{C}-C_{k}}}	\Lambda_{w}(e_{c}) - \Lambda_{w}(\overbar{e_{i}}e_{c}) - \Lambda_{w}((e_{i}{u_{k}^{-1}})\overbar{e_{i}}e_{c}) 	\right.	\\ 
	&	\left. - \smashoperator[lr]{\sum_{H_{a}\in C_{k}-\{H_{c}\}}}\Lambda_{w}(\overbar{e_{a}}e_{c}) - \frac{1}{2}\smashoperator[lr]{\sum_{H_{a}\in\hat{C}-C_{k}}}\Lambda_{w}(\overbar{e_{a}}e_{c}) \right)	.
\end{align*}
Note that we can write $\hat{C}-\{H_{c}\}=(\hat{C}-C_{k})\sqcup(C_{k}-\{H_{c}\})$.
Thus $\sum_{H_{a}\in\hat{C}-\{H_{c}\}}\Lambda_{w}(\overbar{e_{a}}e_{c})=\sum_{H_{a}\in\hat{C}-C_{k}}\Lambda_{w}(\overbar{e_{a}}e_{c})+\sum_{H_{a}\in C_{k}-\{H_{c}\}}\Lambda_{w}(\overbar{e_{a}}e_{c})$ for all $c$ such that $H_{c}\in\hat{C}$.

Since $\Lambda_{w}$ counts occurrences of subwords of $w$, we must have that $\Lambda_{w}((e_{i}u_{k}^{-1})\overbar{e_{i}}e_{c}) \le \Lambda_{w}(\overbar{e_{i}}e_{c}) \le \Lambda_{w}(e_{c})$ for every $k$ such that $C_{k}\in\mathbf{C}$ and every $c$ such that $H_{c}\in C_{k}$.
Since $e_{i}$, $e_{j}$, and $e_{a}$ (where $H_{a}\in\hat{C}$) are distinct, we must also have that $\Lambda_{w}(\overbar{e_{i}}e_{c}) + \Lambda_{w}(\overbar{e_{j}}e_{c}) + \sum_{H_{a}\in\hat{C}-\{H_{c}\}}\Lambda_{w}(\overbar{e_{a}}e_{c}) \le \Lambda_{w}(e_{c})$ for every $c$ such that $H_{c}\in\hat{C}$.
By assumption, $\|\alpha(\mathbf{C},\mathbf{u})\|-\|\alpha\|\le0$, that is:
\begin{align*}
	&	\sum_{w\in \mathcal{W}}\sum_{C_{k}\in\mathbf{C}}\smashoperator[r]{\sum_{H_{c}\in C_{k}}} \ \left( \Lambda_{w}(e_{c}) - \smashoperator[lr]{\sum_{H_{a}\in\hat{C}-C_{k}}}\Lambda_{w}(\overbar{e_{a}}e_{c}) - \frac{1}{2}\smashoperator[lr]{\sum_{H_{a}\in C_{k}-\{H_{c}\}}}\Lambda_{w}(\overbar{e_{a}}e_{c}) \right)	\\
\le 	&	 \sum_{w\in \mathcal{W}}\sum_{C_{k}\in\mathbf{C}}\sum_{H_{c}\in C_{k}}\Bigl( \Lambda_{w}(\overbar{e_{i}}e_{c}) + \Lambda_{w}((e_{i}u_{k}^{-1})\overbar{e_{i}}e_{c}) \Bigr)	.
\end{align*}
We now deduce the following system of inequalities:
\begin{align}
	&	\sum_{w\in \mathcal{W}}\sum_{H_{c}\in \hat{C}}\left( \Lambda_{w}(e_{c}) - \smashoperator[lr]{\sum_{H_{a}\in\hat{C}-\{H_{c}\}}}\Lambda_{w}(\overbar{e_{a}}e_{c}) \right)	\\
=	&	\sum_{w\in \mathcal{W}}\sum_{C_{k}\in\mathbf{C}}\smashoperator[r]{\sum_{H_{c}\in C_{k}}} \ \left( \Lambda_{w}(e_{c}) - \smashoperator[lr]{\sum_{H_{a}\in\hat{C}-\{H_{c}\}}}\Lambda_{w}(\overbar{e_{a}}e_{c}) \right)	\\
=	&	\sum_{w\in \mathcal{W}}\sum_{C_{k}\in\mathbf{C}}\smashoperator[r]{\sum_{H_{c}\in C_{k}}} \ \left( \Lambda_{w}(e_{c}) - \smashoperator[lr]{\sum_{H_{a}\in\hat{C}-C_{k}}}\Lambda_{w}(\overbar{e_{a}}e_{c}) - \smashoperator[lr]{\sum_{H_{a}\in C_{k}-\{H_{c}\}}}\Lambda_{w}(\overbar{e_{a}}e_{c}) \right)	\\
\le 	&	\sum_{w\in \mathcal{W}}\sum_{C_{k}\in\mathbf{C}}\smashoperator[r]{\sum_{H_{c}\in C_{k}}} \ \left( \Lambda_{w}(e_{c}) - \smashoperator[lr]{\sum_{H_{a}\in\hat{C}-C_{k}}}\Lambda_{w}(\overbar{e_{a}}e_{c}) - \frac{1}{2}\smashoperator[lr]{\sum_{H_{a}\in C_{k}-\{H_{c}\}}}\Lambda_{w}(\overbar{e_{a}}e_{c}) \right)	\\
\le 	&	\sum_{w\in \mathcal{W}}\sum_{C_{k}\in\mathbf{C}}\smashoperator[r]{\sum_{H_{c}\in C_{k}}} \ \left( \Lambda_{w}(\overbar{e_{i}}e_{c}) + \Lambda_{w}((e_{i}u_{k}^{-1})\overbar{e_{i}}e_{c}) \right)	\\
\le 	&	\sum_{w\in \mathcal{W}}\sum_{C_{k}\in\mathbf{C}}\smashoperator[r]{\sum_{H_{c}\in C_{k}}} \ \left( 2\Lambda_{w}(\overbar{e_{i}}e_{c}) \right)	\\
=	& 	2\smashoperator[l]{\sum_{w\in \mathcal{W}}}\sum_{H_{c}\in\hat{C}} \Lambda_{w}(\overbar{e_{i}}e_{c})		.
\end{align}

The same argument yields that $\sum_{w\in \mathcal{W}}\sum_{H_{d}\in \hat{D}}\left( \Lambda_{w}(e_{d}) - \sum_{b\in\hat{D}-\{H_{d}\}}\Lambda_{w}(\overbar{e_{b}}e_{d}) \right)
\le 2\sum_{w\in \mathcal{W}}\sum_{H_{d}\in\hat{D}} \Lambda_{w}(\overbar{e_{j}}e_{d})$.
Since it is assumed that $\hat{D}=\hat{C}$, we can rewrite this to give
$\sum_{w\in \mathcal{W}}\sum_{H_{c}\in \hat{C}}\left( \Lambda_{w}(e_{c}) - \sum_{a\in\hat{C}-\{H_{c}\}}\Lambda_{w}(\overbar{e_{a}}e_{c}) \right)
\le 2\sum_{w\in \mathcal{W}}\sum_{H_{c}\in\hat{C}} \Lambda_{w}(\overbar{e_{j}}e_{c})$.

Now $\sum_{w\in \mathcal{W}}\sum_{H_{c}\in\hat{C}} \Lambda_{w}(\overbar{e_{i}}e_{c}) \ge \frac{1}{2}\sum_{w\in \mathcal{W}}\sum_{H_{c}\in \hat{C}}\left( \Lambda_{w}(e_{c}) - \sum_{H_{a}\in\hat{C}-\{H_{c}\}}\Lambda_{w}(\overbar{e_{a}}e_{c}) \right)$
and $\sum_{w\in \mathcal{W}}\sum_{H_{c}\in\hat{C}} \Lambda_{w}(\overbar{e_{j}}e_{c}) \ge \frac{1}{2}\sum_{w\in \mathcal{W}}\sum_{H_{c}\in \hat{C}}\left( \Lambda_{w}(e_{c}) - \sum_{H_{a}\in\hat{C}-\{H_{c}\}}\Lambda_{w}(\overbar{e_{a}}e_{c}) \right)$.
 Since \\ \noindent $\sum_{w\in \mathcal{W}}\sum_{H_{c}\in\hat{C}} \Lambda_{w}(\overbar{e_{i}}e_{c}) + \sum_{w\in \mathcal{W}}\sum_{H_{c}\in\hat{C}} \Lambda_{w}(\overbar{e_{j}}e_{c}) = \sum_{w\in \mathcal{W}}\sum_{H_{c}\in \hat{C}}
 \Bigl( \Lambda_{w}(e_{c}) -  \textcolor{white}{\Bigr)}$
\\ \noindent $\textcolor{white}{\Bigl(} \sum_{H_{a}\in\hat{C}-\{H_{c}\}}\Lambda_{w}(\overbar{e_{a}}e_{c}) \Bigr)$,
then we must in fact have $\sum_{w\in \mathcal{W}}\sum_{H_{c}\in\hat{C}} \Lambda_{w}(\overbar{e_{i}}e_{c}) 
\\ \noindent = \sum_{w\in \mathcal{W}}\sum_{H_{c}\in\hat{C}} \Lambda_{w}(\overbar{e_{j}}e_{c}) = \frac{1}{2}\sum_{w\in \mathcal{W}}\sum_{H_{c}\in \hat{C}}\left( \Lambda_{w}(e_{c}) - \sum_{H_{a}\in\hat{C}-\{H_{c}\}}\Lambda_{w}(\overbar{e_{a}}e_{c}) \right)$.
This in turn forces each line of (1)--(7) to be an equality.
In particular, \\ \noindent $\sum_{w\in \mathcal{W}}\sum_{C_{k}\in\mathbf{C}}\sum_{H_{c}\in C_{k}}\left( \Lambda_{w}(e_{c}) - \sum_{H_{a}\in\hat{C}-C_{k}}\Lambda_{w}(\overbar{e_{a}}e_{c}) - \frac{1}{2}\sum_{H_{a}\in C_{k}-\{H_{c}\}}\Lambda_{w}(\overbar{e_{a}}e_{c}) \right) 
\\ \noindent = \sum_{w\in \mathcal{W}}\sum_{C_{k}\in\mathbf{C}}\sum_{H_{c}\in C_{k}}\left( \Lambda_{w}(\overbar{e_{i}}e_{c}) + \Lambda_{w}((e_{i}u_{k}^{-1})\overbar{e_{i}}e_{c}) \right)$.
That is, $\|\alpha(\mathbf{C},\mathbf{u})\|-\|\alpha\|=0$.
The same argument applies to see that we must also have $||\alpha(\mathbf{D},\mathbf{v})||-||\alpha||=0$.
\end{proof}

\begin{lemma}\label{1b peak} 
Suppose
\begin{tikzcd}[cramped]
\alpha_{1}	&	\alpha_{2}
			\ar[l, dash, ->-,  "(\mathbf{A}{,}\mathbf{x})" ']
			\ar[r, dash, ->-,  "(\mathbf{B}{,}\mathbf{y})"]
									&	\alpha_{3}
\end{tikzcd} is a peak.
If $H_{i}\not\in\hat{B}$ and $H_{j}\not\in\hat{A}$ with $\hat{A}\cap\hat{B}\neq\emptyset$, then 
either
$||\alpha_{2}(\mathbf{A},\mathbf{x})(\mathbf{A}\cap\hat{B},\mathbf{x})^{-1}||<||\alpha_{2}||$
and
$||\alpha_{2}(\mathbf{B},\mathbf{y})(\mathbf{A}-\hat{B},\mathbf{x})||<||\alpha_{2}||$
, or
$||\alpha_{2}(\mathbf{A},\mathbf{x})(\mathbf{B}-\hat{A},\mathbf{y})||<||\alpha_{2}||$
and
$||\alpha_{2}(\mathbf{B},\mathbf{y})(\mathbf{B}\cap\hat{A},\mathbf{y})^{-1}||<||\alpha_{2}||$
.
\end{lemma}
That is, $||\alpha_{2}||>\max(||\alpha_{4}||,||\alpha_{5}||)$ or $||\alpha_{2}||>\max(||\alpha_{4}'||,||\alpha_{5}'||)$.

\begin{proof}
Let $\alpha_{6}=\alpha_{2}(\mathbf{A}-\hat{B},\mathbf{x})(\mathbf{B}-\hat{A},\mathbf{y})$.
For brevity, set $\mathbf{U}=\mathbf{A}-\hat{B}$, $\mathbf{V}=\mathbf{B}-\hat{A}$, $\mathbf{W}=\mathbf{A}\cap\hat{B}$, and $\mathbf{W^{\backprime}}=\mathbf{B}\cap\hat{A}$.
Note that $\hat{W}=\hat{W^{\backprime}}$ and $\hat{U}\cap\hat{B}=\hat{U}\cap\hat{V}=\hat{U}\cap\hat{W}=\hat{V}\cap\hat{W}=\hat{V}\cap\hat{A}=\emptyset$.
Then by Proposition \ref{whitehead notational properties} (1), the diagram in Figure \ref{fig commuting diagram 1b} commutes.

\begin{figure}
\centering
\begin{tikzcd}[cramped]	
			&			&	\alpha_{2}
							\ar[dll, "(\mathbf{A}{,}\mathbf{x})" ',purple]
							\ar[dl, "(\mathbf{U}{,}\mathbf{x})",pos=0.2,Green]
							\ar[drr, "(\mathbf{B}{,}\mathbf{y})",violet]
							\ar[dddr, bend left=90, looseness=2.5, "(\mathbf{V}{,}\mathbf{y})",blue]
							\ar[dddl, bend right=90, looseness=2.5, "(\mathbf{V}{,}\mathbf{y})" ',white]
									&			&
															\\
\alpha_{1}
\ar[ddr, "(\mathbf{V}{,}\mathbf{y})" ',blue]	
			&	\alpha_{4}
				\ar[l, "(\mathbf{W}{,}\mathbf{x})",red]
				\ar[dr, "(\mathbf{V}{,}\mathbf{y})" ',blue]
				\ar[rr, "(\mathbf{B}{,}\mathbf{y})",violet]
						&			&	\alpha_{5}	&	\alpha_{3}
													\ar[l, "(\mathbf{U}{,}\mathbf{x})",Green]	
															\\
			&			&	\alpha_{6}
							\ar[ur, "(\mathbf{W^{\backprime}}{,}\mathbf{y})",pos=0.3,orange]
							\ar[dl, "(\mathbf{W}{,}\mathbf{x})" ',red]
									&			&	
															\\
			&	\alpha_{4}'	&			&	\alpha_{5}'
										\ar[ll, "(\mathbf{A}{,}\mathbf{x})",purple]
										\ar[ul, "(\mathbf{U}{,}\mathbf{x})",Green]
										\ar[uur, "(\mathbf{W^{\backprime}}{,}\mathbf{y})",orange]
												&			\\
\end{tikzcd}
\caption{Commuting Diagram for Case 1b}
\label{fig commuting diagram 1b}
\end{figure}
By applying Lemma \ref{1a domain height} to each of the squares $\alpha_{1}\dash\alpha_{4}\dash\alpha_{6}\dash\alpha_{4}'\dash\alpha_{1}$ and \\ \noindent $\alpha_{3}\dash\alpha_{5}'\dash\alpha_{6}\dash\alpha_{5}\dash\alpha_{1}$ (which both fall under Case 1a), we recover that
$\|\alpha_{1}\|-\|\alpha_{4}\|=\|\alpha_{4}'\|-\|\alpha_{6}\|$ and $\|\alpha_{3}\|-\|\alpha_{5}'\|=\|\alpha_{5}\|-\|\alpha_{6}\|$.
By considering the triangles $\alpha_{1}\dash\alpha_{2}\dash\alpha_{4}\dash\alpha_{1}$ and $\alpha_{3}\dash\alpha_{2}\dash\alpha_{5}'\dash\alpha_{3}$, we see that
$\|\alpha_{1}\|-\|\alpha_{2}\|=(\|\alpha_{1}\|-\|\alpha_{4}\|)+(\|\alpha_{4}\|-\|\alpha_{2}\|)$ and $\|\alpha_{3}\|-\|\alpha_{2}\|=(\|\alpha_{3}\|-\|\alpha_{5}'\|)+(\|\alpha_{5}'\|-\|\alpha_{2}\|)$.
By Lemma \ref{lemma acting on common C}, either $\max(\|\alpha_{4}'\|-\|\alpha_{6}\|,\|\alpha_{5}\|-\|\alpha_{6}\|)>0$ or $\|\alpha_{4}'\|-\|\alpha_{6}\|=\|\alpha_{5}\|-\|\alpha_{6}\|=0$.
Since 
\begin{tikzcd}[cramped]
\alpha_{1}	&	\alpha_{2}
			\ar[l, dash,->-, "(\mathbf{A}{,}\mathbf{x})" ']
			\ar[r,, dash, ->-, "(\mathbf{B}{,}\mathbf{y})"]
									&	\alpha_{3}
\end{tikzcd}
is a peak, then $\max(\|\alpha_{1}\|-\|\alpha_{2}\|,\|\alpha_{3}\|-\|\alpha_{2}\|)\le0$ and $\min(\|\alpha_{1}\|-\|\alpha_{2}\|,\|\alpha_{3}\|-\|\alpha_{2}\|)<0$.

We claim that $\min(\|\alpha_{4}\|-\|\alpha_{2}\|,\|\alpha_{5}'\|-\|\alpha_{2}\|)<0$.
If $\|\alpha_{4}'\|-\|\alpha_{6}\|=\|\alpha_{5}\|-\|\alpha_{6}\|=0$, then:
\begin{align*}
	&	\min(\|\alpha_{4}\|-\|\alpha_{2}\|,\|\alpha_{5}'\|-\|\alpha_{2}\|)	\\
=	&	\min\left( (\|\alpha_{1}\|-\|\alpha_{2}\|)-(\|\alpha_{1}\|-\|\alpha_{4}\|) , (\|\alpha_{3}\|-\|\alpha_{2}\|)-(\|\alpha_{3}\|-\|\alpha_{5}'\|) \right)	\\
=	&	\min\left( (\|\alpha_{1}\|-\|\alpha_{2}\|)-(\|\alpha_{4}'\|-\|\alpha_{6}\|) , (\|\alpha_{3}\|-\|\alpha_{2}\|)-(\|\alpha_{5}\|-\|\alpha_{6}\|) \right)	\\
=	&	\min(\|\alpha_{1}\|-\|\alpha_{2}\|,\|\alpha_{3}\|-\|\alpha_{2}\|)	\\
<	&	0	.
\end{align*}
On the other hand, if $\max(\|\alpha_{4}'\|-\|\alpha_{6}\|,\|\alpha_{5}\|-\|\alpha_{6}\|)>0$ (without loss of generality, say $\|\alpha_{4}'\|-\|\alpha_{6}\|>0$ --- a symmetrically identical argument holds if $\|\alpha_{5}\|-\|\alpha_{6}\|>0$), then:
\begin{align*}
\|\alpha_{4}\|-\|\alpha_{2}\|	&=	 (\|\alpha_{1}\|-\|\alpha_{2}\|)-(\|\alpha_{1}\|-\|\alpha_{4}\|)	\\
					&=	(\|\alpha_{1}\|-\|\alpha_{2}\|)-(\|\alpha_{4}'\|-\|\alpha_{6}\|)	\\
					&<	\|\alpha_{1}\|-\|\alpha_{2}\|	\\
					&\le 	0	.
\end{align*}
In either case, we have that $\min(\|\alpha_{4}\|-\|\alpha_{2}\|,\|\alpha_{5}'\|-\|\alpha_{2}\|)<0$.

Without loss of generality, assume $\|\alpha_{4}\|-\|\alpha_{2}\|<0$.
Then $\alpha_{4}\dash\alpha_{2}\dash\alpha_{3}$ is a peak falling under Case 1a, and by Lemma \ref{1a peak}, $\|\alpha_{5}\|<\|\alpha_{2}\|$.
An identical (symmetric) argument holds if instead $\|\alpha_{5}'\|-\|\alpha_{2}\|<0$.
Thus $\min\left( \max(\|\alpha_{4}\|,\|\alpha_{5}\|), \max(\|\alpha_{4}'\|,\|\alpha_{5}'\|) \right) < \|\alpha_{2}\|$, as required.
\end{proof}

\begin{prop}\label{prop 1b summary} 
Let $H_{1}\ast\dots H_{n}$ be an $\mathfrak{S}$ free factor splitting for $G$.
Let $(\mathbf{A},\mathbf{x})$ and $(\mathbf{B},\mathbf{y})$ be relative multiple Whitehead automorphisms with $\mathbf{x}\subset H_{i}$ and $\mathbf{y}\subset H_{j}$ and $\hat{A},\hat{B}\subset\{H_{1},\dots,H_{n}\}-\{H_{i},H_{j}\}$ such that $\hat{A}\cap\hat{B}\ne\emptyset$.
Let $\alpha_{2}$ be the domain whose $\alpha$-graph has $\mathfrak{S}$-labelling $(H_{1},\dots,H_{n})$, and let $\alpha_{1}=\alpha_{2}(\mathbf{A},\mathbf{x})$ and $\alpha_{3}=\alpha_{2}(\mathbf{B},\mathbf{y})$.
If
\begin{tikzcd}[cramped]
\alpha_{1} \ar[r,-<-,dash,"(\mathbf{A}{,}\mathbf{x})"]	& \alpha_{2} \ar[r,->-,dash,"(\mathbf{B}{,}\mathbf{y})"]	& \alpha_{3}
\end{tikzcd}
is a peak, then it is reducible.
\end{prop}

\begin{proof}
By Lemmas \ref{1b loop} and \ref{1b contractible}, the path 
\begin{tikzcd}[cramped]
\alpha_{1} \ar[r,-<-,dash,"(\mathbf{A}{,}\mathbf{x})"]	& \alpha_{2} \ar[r,->-,dash,"(\mathbf{B}{,}\mathbf{y})"]	& \alpha_{3}
\end{tikzcd}
is homotopic in the Space of Domains to each of the paths
\\ \noindent
\begin{tikzcd}
\alpha_{1} \ar[r,-<-,dash,"(\mathbf{A}\cap\hat{B}{,}\mathbf{x})"]	& \alpha_{2}(\mathbf{A},\mathbf{x})(\mathbf{A}\cap\hat{B},\mathbf{x})^{-1} \ar[r,->-,dash,"(\mathbf{B}{,}\mathbf{y})"] 	& \alpha_{2}(\mathbf{B},\mathbf{y})(\mathbf{A}-\hat{B},\mathbf{x}) \ar[r,-<-,dash,"(\mathbf{A}-\hat{B}{,}\mathbf{x})"]	& \alpha_{3}
\end{tikzcd}
\ and \\ \noindent
\begin{tikzcd}
\alpha_{1} \ar[r,->-,dash,"(\mathbf{B}-\hat{A}{,}\mathbf{y})"]	& \alpha_{2}(\mathbf{A},\mathbf{x})(\mathbf{B}-\hat{A},\mathbf{y}) \ar[r,-<-,dash,"(\mathbf{A}{,}\mathbf{x})"] 	& \alpha_{2}(\mathbf{B},\mathbf{y})(\mathbf{B}\cap\hat{A},\mathbf{y})^{-1} \ar[r,->-,dash,"(\mathbf{B}\cap\hat{A}{,}\mathbf{y})"]	& \alpha_{3}
\end{tikzcd}
.

By Lemma \ref{1b peak}, either
$||\alpha_{2}(\mathbf{A},\mathbf{x})(\mathbf{A}\cap\hat{B},\mathbf{x})^{-1}||<||\alpha_{2}||$
and
$||\alpha_{2}(\mathbf{B},\mathbf{y})(\mathbf{A}-\hat{B},\mathbf{x})||<||\alpha_{2}||$
, or
$||\alpha_{2}(\mathbf{A},\mathbf{x})(\mathbf{B}-\hat{A},\mathbf{y})||<||\alpha_{2}||$
and
$||\alpha_{2}(\mathbf{B},\mathbf{y})(\mathbf{B}\cap\hat{A},\mathbf{y})^{-1}||<||\alpha_{2}||$
Thus by Definition \ref{defn reducible}, one of the above paths is a reduction for the peak
\begin{tikzcd}[cramped]
\alpha_{1} \ar[r,-<-,dash,"(\mathbf{A}{,}\mathbf{x})"]	& \alpha_{2} \ar[r,->-,dash,"(\mathbf{B}{,}\mathbf{y})"]	& \alpha_{3}
\end{tikzcd}
.
\end{proof}


\subsubsection*{Case 2(a): $H_{i}\in\hat{B}$ (say $H_{i}\in B_{q}$) and $H_{j}\not\in\hat{A}$, with $\hat{A}\subseteq B_{q}$}

The lemmas for this case are adapted from \cite[Lemma 2.7]{Gilbert1987}.

\begin{lemma}\label{2a loop} 
If $H_{i}\in B_{q}$ and $H_{j}\not\in\hat{A}$ with $\hat{A}\subseteq B_{q}$, then 
$(\mathbf{A},\mathbf{x})^{-1}(\mathbf{B},\mathbf{y})
=(\mathbf{B'},\mathbf{y})(\mathbf{A}^{y_{q}},\mathbf{x}^{y_{q}})^{-1}$
where $[\mathbf{B'}]_{b}:=B_{b}$ for $b\in\{1,\dots,l\}-\{q\}$ and $[\mathbf{B'}]_{q}:=(B_{q}-\hat{A})\cup\widehat{(B_{q}\cap\mathbf{A})^{\mathbf{x}}}=(B_{q}-\hat{A})\cup\bigcup_{a=1}^{k}A_{a}^{x_{a}}$.
\end{lemma}
That is, there exists a vertex $\alpha_{4}$ in our Graph of Domains such that
\begin{tikzpicture}
\draw[-<-] (0.2,0.2) -- (1,1);
\draw[->-] (1.5,1) -- (2.3,0.2);
\draw[->-] (0.2,-0.2) -- (1,-1);
\draw[-<-] (1.5,-1) -- (2.3,-0.2);
\node at (0,0) {$\alpha_{1}$};
\node at (1.25,1.15) {$\alpha_{2}$};
\node[scale=0.8] at (0.2,0.85) {$(\mathbf{A},\mathbf{x})$};
\node[scale=0.8] at (2.45,-0.85) {$(\mathbf{A}^{y_{q}},\mathbf{x}^{y_{q}})$};
\node at (2.5,0) {$\alpha_{3}$};
\node[scale=0.8] at (2.3,0.85) {$(\mathbf{B},\mathbf{y})$};
\node at (1.25,-1.15) {$\alpha_{4}$};
\node[scale=0.8] at (0.175,-0.85) {$(\mathbf{B'},\mathbf{y})$};
\end{tikzpicture}
is a loop.

\begin{proof} 
We have that $B_{q}=(B_{q}-\hat{A})\sqcup \hat{A}$.
In particular, $\hat{A}\cap B_{b}=\emptyset$ for any $b\ne q$.
By Lemma \ref{lemma whitehead properties} (3) we have that $(B_{q},y_{q})=(B_{q}-\hat{A},y_{q})(\hat{A},y_{q})$.
By (1) and (2) of Lemma \ref{lemma whitehead properties}, we have that $(\mathbf{A}^{y_{q}},\mathbf{x}^{y_{q}})=(\hat{A}^{y_{q}},y_{q}^{-1})(\mathbf{A}^{y_{q}y_{q}^{-1}},\mathbf{x}y_{q})=(\hat{A},y_{q})^{-1}(\mathbf{A},\mathbf{x}y_{q})$.
Now
	$\begin{aligned}[t]
	(\mathbf{A},\mathbf{x})(\mathbf{B'},\mathbf{y})	&=	(\mathbf{A},\mathbf{x})(B_{1},y_{1})\dots( (B_{q}-\hat{A})\cup \widehat{\mathbf{A}^{\mathbf{x}}},y_{q})\dots(B_{l},y_{l})	\\
	&=	(\mathbf{A},\mathbf{x})(B_{1},y_{1})\dots(B_{q}-\hat{A},y_{q})(\widehat{\mathbf{A}^{\mathbf{x}}},y_{q})\dots(B_{l},y_{l})	\\
	&=	(B_{1},y_{1})\dots(B_{q}-\hat{A},y_{q})\dots(B_{l},y_{l})(\mathbf{A},\mathbf{x})(\widehat{\mathbf{A}^{\mathbf{x}}},y_{q})	\\
	&=	(\mathbf{B}-\hat{A},\mathbf{y})(\mathbf{A},\mathbf{x}y_{q})	\\
	&=	(\mathbf{B}-\hat{A},\mathbf{y})(\hat{A},y_{q})(\mathbf{A}^{y_{q}},\mathbf{x}^{y_{q}})	\\
	&=	(\mathbf{B},\mathbf{y})(\mathbf{A}^{y_{q}},\mathbf{x}^{y_{q}})	.
	\end{aligned}$
\\
\end{proof}

\begin{lemma}\label{2a contractible} 
The loop 
\begin{tikzpicture}
\draw[-<-] (0.2,0.2) -- (1,1);
\draw[->-] (1.5,1) -- (2.3,0.2);
\draw[->-] (0.2,-0.2) -- (1,-1);
\draw[-<-] (1.5,-1) -- (2.3,-0.2);
\node at (0,0) {$\alpha_{1}$};
\node at (1.25,1.15) {$\alpha_{2}$};
\node[scale=0.8] at (0.2,0.85) {$(\mathbf{A},\mathbf{x})$};
\node[scale=0.8] at (2.45,-0.85) {$(\mathbf{A}^{y_{q}},\mathbf{x}^{y_{q}})$};
\node at (2.5,0) {$\alpha_{3}$};
\node[scale=0.8] at (2.3,0.85) {$(\mathbf{B},\mathbf{y})$};
\node at (1.25,-1.15) {$\alpha_{4}$};
\node[scale=0.8] at (0.175,-0.85) {$(\mathbf{B'},\mathbf{y})$};
\end{tikzpicture}
(where $H_{i}\in B_{q}$ and $H_{j}\not\in\hat{A}$ with $\hat{A}\subseteq B_{q}$, and $\mathbf{B'}$ is given by $[\mathbf{B'}]_{b}=B_{b}$ for $b\ne q$ and $[\mathbf{B'}]_{q}=(B_{q}-\hat{A})\cup\widehat{\mathbf{A}^{\mathbf{x}}}$) is contractible in our Space of Domains.
\end{lemma}

\begin{proof}
Let $\hat{A}=\{H_{A_{1}},\dots,H_{A_{K}}\}$.
Recall from Lemma \ref{1a contractible} that we denote $(\mathbf{A},\mathbf{x})|_{a}:=(\mathbf{A}\cap\{H_{a}\},\mathbf{x})$. Then there is some $x_{A_{a}}\in\mathbf{x}$ so that $(\mathbf{A},\mathbf{x})|_{a}=(H_{a},x_{A_{a}})$.
Let $\tilde{\alpha}_{0,0}:=\alpha_{2}$
and for $a=1,\dots,K$, recursively define $\tilde{\alpha}_{a,0}:=\tilde{\alpha}_{a-1,0}(\mathbf{A},\mathbf{x})|_{a}$.
Then for each $a$, $\tilde{\alpha}_{a,0}$ is the (domain whose) $\alpha$-graph has $\mathfrak{S}$-labelling comprising the groups $H_{A_{1}}^{x_{A_{1}}},\dots,H_{A_{a}}^{x_{A_{a}}},H_{A_{a+1}},\\\noindent\dots,H_{A_{K}},H_{v_{1}},\dots,H_{v_{n-K}}$.
For each $a$, define $\mathbf{B^{a\prime}}$ by $[\mathbf{B^{a\prime}}]_{b}:=B_{b}$ and $[\mathbf{B^{a\prime}}]_{q}:=(B_{q}-\{H_{A_{1}},\dots,H_{A_{a}}\})\cup\{H_{A_{1}}^{x_{A_{1}}},\dots,H_{A_{a}}^{x_{A_{a}}}\}$.
Note that $\mathbf{B^{K\prime}}=\mathbf{B'}$.
Again for each $a=0,\dots,K$, we define $\tilde{\alpha}_{a,L}:=\tilde{\alpha}_{a,0}(\mathbf{B^{a\prime}},\mathbf{y})$.
Then $\tilde{\alpha}_{a+1,L}=\tilde{\alpha}_{a,L}(H_{A_{a}}^{y_{q}},x_{A_{a}}^{y_{q}})$.
Observe that $\tilde{\alpha}_{K,0}=\alpha_{1}$, $\tilde{\alpha}_{0,L}=\alpha_{3}$, and $\tilde{\alpha}_{K,L}=\alpha_{4}$.
Thus we have constructed a lattice as depicted in Figure \ref{fig lattice 2a}.
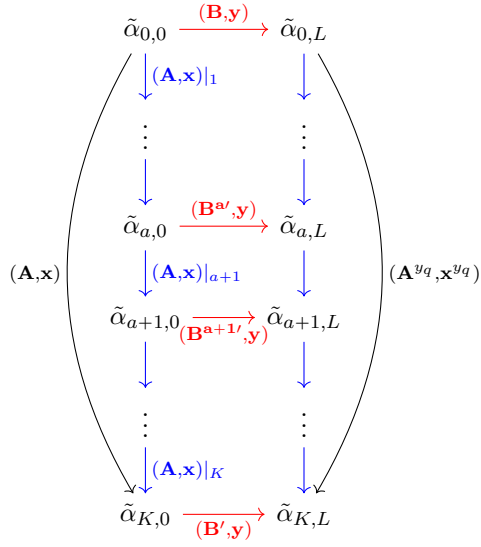
\begin{figure}
\centering
\adjustbox{scale=0.9}{
\begin{tikzcd}
\tilde{\alpha}_{0,0}
\ar[r,red,"(\mathbf{B}{,}\mathbf{y})"]
\ar[d,blue,"(\mathbf{A}{,}\mathbf{x})|_{1}"]
\ar[ddddd,bend right=30,"(\mathbf{A}{,}\mathbf{x})" ']
				&	\tilde{\alpha}_{0,L}
					\ar[d,blue]
					\ar[ddddd,bend left=30,"(\mathbf{A}^{y_{q}}{,}\mathbf{x}^{y_{q}})"]
									\\
\vdots
\ar[d,blue]
				&	\vdots
					\ar[d,blue]
									\\
\tilde{\alpha}_{a,0}
\ar[r,red,"(\mathbf{B^{a\prime}}{,}\mathbf{y})"]
\ar[d,blue,"(\mathbf{A}{,}\mathbf{x})|_{a+1}"]
				&	\tilde{\alpha}_{a,L}
					\ar[d,blue]
									\\
\tilde{\alpha}_{a+1,0}
\ar[r,red,"(\mathbf{B^{a+1\prime}}{,}\mathbf{y})" ']
\ar[d,blue]
				&	\tilde{\alpha}_{a+1,L}
					\ar[d,blue]
									\\
\vdots
\ar[d,blue,"(\mathbf{A}{,}\mathbf{x})|_{K}"]
				&	\vdots
					\ar[d,blue]
									\\
\tilde{\alpha}_{K,0}
\ar[r,red,"(\mathbf{B'}{,}\mathbf{y})" ']
				&	\tilde{\alpha}_{K,L}
									\\
\end{tikzcd}
}
\caption{Lattice Describing $(A,x)(B',y)=(B,y)(A^{y_{q}},x^{y_{q}})$ in Case 2a}
\label{fig lattice 2a}
\end{figure}
As in the proof of Lemma \ref{1a contractible}, the loops $\tilde{\alpha}_{0,0}\dash\dots\dash \tilde{\alpha}_{K,0}\dash \tilde{\alpha}_{0,0}$
and $\tilde{\alpha}_{0,L}\dash\dots\dash \tilde{\alpha}_{K,L}\dash \tilde{\alpha}_{0,L}$ are contractible via $A_{i}$-graphs living in domains $\tilde{\alpha}_{0,0}$ and $\tilde{\alpha}_{0,L}$, respectively.
Also, for each $a\in\{0,\dots,K-1\}$ the square

\noindent
\begin{tikzcd}
\tilde{\alpha}_{a,0}
\ar[r,red,"(\mathbf{B^{a\prime}}{,}\mathbf{y})"]
\ar[d,blue,"(\mathbf{A}{,}\mathbf{x})|_{a+1}" ']
				&	\tilde{\alpha}_{a,L}
					\ar[d,blue,"(H_{A_{a}}^{y_{q}}{,}x_{A_{a}}^{y_{q}})"]
									\\
\tilde{\alpha}_{a+1,0}
\ar[r,red,"(\mathbf{B^{a+1\prime}}{,}\mathbf{y})" ']
				&	\tilde{\alpha}_{a+1,L}
									\\
\end{tikzcd}
is contractible via the graph $B_{j,i,a}$ in the domain $\tilde{\alpha}_{a,0}$ (that is, the graph \TB{$n-3$}{$H_{j}$}{$H_{i}$}{$H_{a}$} lives in the intersection $\tilde{\alpha}_{a,0}\cap\tilde{\alpha}_{a,L}\cap\tilde{\alpha}_{a+1,0}\cap\tilde{\alpha}_{a+1,L}$).
\end{proof}

\begin{lemma}\label{2a word length} 
Suppose we have a loop
\begin{tikzpicture}
\draw[-<-] (0.2,0.2) -- (1,1);
\draw[->-] (1.5,1) -- (2.3,0.2);
\draw[->-] (0.2,-0.2) -- (1,-1);
\draw[-<-] (1.5,-1) -- (2.3,-0.2);
\node at (0,0) {$\alpha_{1}$};
\node at (1.25,1.15) {$\alpha_{2}$};
\node[scale=0.8] at (0.2,0.85) {$(H_{a},{x})$};
\node[scale=0.8] at (2.45,-0.85) {$(H_{a}^{y_{q}},{x}^{y_{q}})$};
\node at (2.5,0) {$\alpha_{3}$};
\node[scale=0.8] at (2.3,0.85) {$(B_{q},y_{q})$};
\node at (1.25,-1.15) {$\alpha_{4}$};
\node[scale=0.8] at (0.15,-0.85) {$(B'_{q},y_{q})$};
\end{tikzpicture}
where $B_{q}'=\{H_{a}^{x}\}\cup(B_{q}-\{H_{a}\})$, and $y_{q}\in H_{j}\ne H_{a}$, $x\in H_{i}\in B_{q}$, and $H_{a}\in B_{q}$.
Then for any edge path $w$ in $\hat{\alpha}_{2}$, we have
$|w|_{\alpha_{4}}-|w|_{\alpha_{1}}=|w|_{\alpha_{3}}-|w|_{\alpha_{2}}$.
\end{lemma}

\begin{proof}
Set $\tilde{B}_{q}:=B_{q}-\{H_{i},H_{a}\}$.
By Lemma \ref{word length (Ha,x)}, we have that 
\[|w|_{\alpha_{1}}=|w|_{\alpha_{2}}+2\Lambda_{w}(e_{a})-2\Lambda_{w}(\overbar{e_{i}}e_{a})-2\Lambda_{w}((e_{i}x^{-1})\overbar{e_{i}}e_{a}) . \]
By Lemma \ref{word length (A,x)}, we have that
\begin{align*}
|w|_{\alpha_{3}}=	&	|w|_{\alpha_{2}}+2\smashoperator[lr]{\sum_{H_{b}\in B_{q}}} \ \left( \Lambda_{w}(e_{b})-\Lambda_{w}(\overbar{e_{j}}e_{b})-\Lambda_{w}((e_{j}y_{q}^{-1})\overbar{e_{j}}e_{b}) - \smashoperator[lr]{\sum_{H_{c}\in B_{q}-\{H_{b}\}}}\Lambda_{w}(\overbar{e_{c}}e_{b}) \right)	\\
=			&	|w|_{\alpha_{2}}+2\smashoperator[lr]{\sum_{H_{b}\in \tilde{B}_{q}}} \ \left( \Lambda_{w}(e_{b})-\Lambda_{w}(\overbar{e_{j}}e_{b})-\Lambda_{w}((e_{j}y_{q}^{-1})\overbar{e_{j}}e_{b}) - \smashoperator[lr]{\sum_{H_{c}\in \tilde{B}_{q}-\{H_{b}\}}}\Lambda_{w}(\overbar{e_{c}}e_{b})  \right) 	\\
			&	+ 2\Lambda_{w}(e_{i})-2\Lambda_{w}(\overbar{e_{j}}e_{i})-2\Lambda_{w}((e_{j}y_{q}^{-1})\overbar{e_{j}}e_{i}) - 4\smashoperator[lr]{\sum_{H_{c}\in\tilde{B}_{q}}}\Lambda_{w}(\overbar{e_{c}}e_{i})	\\
			&	+ 2\Lambda_{w}(e_{a})-2\Lambda_{w}(\overbar{e_{j}}e_{a})-2\Lambda_{w}((e_{j}y_{q}^{-1})\overbar{e_{j}}e_{a}) - 4\smashoperator[lr]{\sum_{H_{c}\in\tilde{B}_{q}}}\Lambda_{w}(\overbar{e_{c}}e_{a})	\\
			&	-4\Lambda_{w}(\overbar{e_{i}}e_{a})	.
\end{align*}
We will follow the methods used in Section \ref{height} to compute $|w|_{\alpha_{4}}$.

Let $\varphi_{24}:=\varphi_{(H_{a},x)(B_{q}',y_{q})}:\hat{\alpha}_{2}\to\hat{\alpha}_{4}$ be the equivariant map described in Convention \ref{convention edge labels and map for alpha hat}.
If the edges of $\hat{\alpha}_{2}$ are labelled by $e$'s, and the edges of $\hat{\alpha}_{4}$ are labeled by $h$'s, then we have:
\begin{align*}
(e_{a})\varphi_{24}	&=	h_{j}(\overbar{h_{j}y_{q}^{-1}})(h_{i}y_{q}^{-1})(\overbar{h_{i}y_{q}^{-1}x^{-1}})(h_{a}y_{q}^{-1}x^{-1})	\\
(e_{i})\varphi_{24}		&=	h_{j}(\overbar{h_{j}y_{q}^{-1}})(h_{i}y_{q}^{-1})		\\
(e_{b})\varphi_{24}	&=	h_{j}(\overbar{h_{j}y_{q}^{-1}})(h_{b}y_{q}^{-1})	\\
(e_{k})\varphi_{24} 	&=	h_{k}	
\end{align*}
for all $b$ with $H_{b}\in\tilde{B}_{q}$, and for all $k$ with $H_{k}\not\in B_{q}$.
In particular, since $H_{j}\not\in B_{q}$, we have 
$(e_{j})\varphi_{24} =h_{j}$.	

Given a word (edge path) $w$ in $\hat{\alpha}_{2}$, set $w'$ to be the unreduced word $(w)\varphi_{24}$ in $\hat{\alpha}_{4}$.
If $l(u)$ is the unreduced length of a given word $u$, then 
$l(w')=l(w) + 4\Lambda_{w}(e_{a}) +2\Lambda_{w}(e_{i}) + 2\sum_{H_{b}\in\tilde{B}_{q}}\Lambda_{w}(e_{b})$.

Observe that for any $b$ with $H_{b}\in \tilde{B}_{q}$ we have
$(\overbar{e_{j}}e_{b})\varphi_{24}=\overbar{h_{j}}h_{j}(\overbar{h_{j}y_{q}^{-1}})(h_{b}y_{q}^{-1})=(\overbar{h_{j}y_{q}^{-1}})(h_{b}y_{q}^{-1})$.
We also have that
$(\overbar{e_{j}}e_{a})\varphi_{24}=\overbar{h_{j}}h_{j}(\overbar{h_{j}y_{q}^{-1}})(h_{i}y_{q}^{-1})(\overbar{h_{i}y_{q}^{-1}x^{-1}})(h_{a}y_{q}^{-1}x^{-1})
\\ \noindent =(\overbar{h_{j}y_{q}^{-1}})(h_{i}y_{q}^{-1})(\overbar{h_{i}y_{q}^{-1}x^{-1}})(h_{a}y_{q}^{-1}x^{-1})$
and $(\overbar{e_{j}}e_{i})\varphi_{24}=\overbar{h_{j}}h_{j}(\overbar{h_{j}y_{q}^{-1}})(h_{i}y_{q}^{-1})=
\\ \noindent (\overbar{h_{j}y_{q}^{-1}})(h_{i}y_{q}^{-1})$.
Let $w''$ be the result of performing all such reductions (i.e. of the form $(\overbar{e_{j}}e_{b})\varphi_{24}$ where $H_{b}\in B_{q}$) to $w'$.
Then for each $b$ with $H_{b}\in B_{q}$, we have that the length of $(\overbar{e_{j}}e_{b})\varphi_{24}$  is 2 less in $w''$ than it is in $w'$.
Thus $l(w'')=l(w')-2\sum_{H_{b}\in B_{q}}\Lambda_{w}(\overbar{e_{j}}e_{b})=l(w')-2\Lambda_{w}(\overbar{e_{j}}e_{i})-2\Lambda_{w}(\overbar{e_{j}}e_{a})-2\sum_{H_{b}\in\tilde{B}_{q}}\Lambda_{w}(\overbar{e_{j}}e_{b})$.

We now observe that for any $b$ with $H_{b}\in\tilde{B}_{q}$, we have
\\ \noindent
$((e_{j}y_{q}^{-1})\overbar{e_{j}}e_{b})\varphi_{24}=(h_{j}y_{q}^{-1})(\overbar{h_{j}y_{q}^{-1}})(h_{b}y_{q}^{-1})=(h_{b}y_{q}^{-1})$,
and similarly,
\\ \noindent
$((e_{j}y_{q}^{-1})\overbar{e_{j}}e_{i})\varphi_{24}=(h_{j}y_{q}^{-1})(\overbar{h_{j}y_{q}^{-1}})(h_{i}y_{q}^{-1})=(h_{i}y_{q}^{-1})$.
Also note that
$((e_{j}y_{q}^{-1})\overbar{e_{j}}e_{a})\varphi_{24}=(h_{j}y_{q}^{-1})(\overbar{h_{j}y_{q}^{-1}})(h_{i}y_{q}^{-1})(\overbar{h_{i}y_{q}^{-1}x^{-1}})(h_{a}y_{q}^{-1}x^{-1})=
(h_{i}y_{q}^{-1})(\overbar{h_{i}y_{q}^{-1}x^{-1}})(h_{a}y_{q}^{-1}x^{-1})$.
Let $w'''$ be the result of applying all such reductions to $w''$, and note that for each $b$ with $H_{b}\in B_{q}$, the length of $((e_{j}y_{q}^{-1})\overbar{e_{j}}e_{b})\varphi_{24}$ is 2 less in $w'''$ than it is in $w''$.
Now
$l(w''')=l(w'')-2\sum_{H_{b}\in B_{q}}\Lambda_{w}((e_{j}y_{q}^{-1})\overbar{e_{j}}e_{b})=l(w'')-2\Lambda_{w}((e_{j}y_{q}^{-1})\overbar{e_{j}}e_{i})-2\Lambda_{w}((e_{j}y_{q}^{-1})\overbar{e_{j}}e_{a})-2\sum_{H_{b}\in\tilde{B}_{q}}\Lambda_{w}((e_{j}y_{q}^{-1})\overbar{e_{j}}e_{b})$.

We now consider `cross-reductions' between elements of $B_{q}$.
Let $b$ and $c$ be such that $H_{b}\in\tilde{B}_{q}$ and $H_{c}\in\tilde{B}_{q}-\{H_{b}\}$.
Then 
$(\overbar{e_{c}}e_{b})\varphi_{24}=(\overbar{h_{c}y_{q}{-1}})(h_{j}y_{q}^{-1})\overbar{h_{j}}h_{j}(\overbar{h_{j}y_{q}^{-1}})(h_{b}y_{q}^{-1})=(\overbar{h_{c}y_{q}^{-1}})(h_{b}y_{q}^{-1})$.
Additionally,
$(\overbar{e_{i}}e_{b})\varphi_{24}=(\overbar{h_{i}y_{q}{-1}})(h_{j}y_{q}^{-1})\overbar{h_{j}}h_{j}(\overbar{h_{j}y_{q}^{-1}})(h_{b}y_{q}^{-1})=
\\ \noindent (\overbar{h_{i}y_{q}^{-1}})(h_{b}y_{q}^{-1})$,
and
$(\overbar{e_{a}}e_{b})\varphi_{24}=(\overbar{h_{a}y_{q}^{-1}x^{-1}})(h_{i}y_{q}^{-1}x^{-1})(\overbar{h_{i}y_{q}{-1}})(h_{j}y_{q}^{-1})\overbar{h_{j}}h_{j}(\overbar{h_{j}y_{q}^{-1}})(h_{b}y_{q}^{-1})
\\ \noindent =(\overbar{h_{a}y_{q}^{-1}x^{-1}})(h_{i}y_{q}^{-1}x^{-1})(\overbar{h_{i}y_{q}^{-1}})(h_{b}y_{q}^{-1})$.
Finally,
$(\overbar{e_{i}}e_{a})\varphi_{24}=
\\ \noindent (\overbar{h_{i}y_{q}^{-1}})(h_{j}y_{q}^{-1})\overbar{h_{j}}h_{j}(\overbar{h_{j}y_{q}^{-1}})(h_{i}y_{q}^{-1})(\overbar{h_{i}y_{q}^{-1}x^{-1}})(h_{a}y_{q}^{-1}x^{-1})=(\overbar{h_{i}y_{q}^{-1}x^{-1}})(h_{a}y_{q}^{-1}x^{-1})$.
Letting $w''''$ be the result of applying all such reductions to $w'''$, we see that for $b$ and $c$ with $H_{b},H_{c}\in\tilde{B}_{q}$ distinct, the length of $(\overbar{e_{c}}e_{b})\varphi_{24}$ is 4 less in $w''''$ than it is in $w'''$, the lengths of $(\overbar{e_{i}}e_{b})\varphi_{24}$ and $(\overbar{e_{a}}e_{b})\varphi_{24}$ are each 4 less in $w''''$ than in $w'''$, and the length of $(\overbar{e_{i}}e_{a})\varphi_{24}$ is 6 less in $w''''$ than it is in $w'''$.
Observe that $\Lambda_{w}(\overbar{e_{c}}e_{b})=\Lambda_{w}(\overbar{\overbar{e_{c}}e_{b}})=\Lambda_{w}(\overbar{e_{b}}e_{c})$.
Thus $l(w'''')=l(w''')-6\Lambda_{w}(\overbar{e_{i}}e_{a})-4\sum_{H_{b}\in\tilde{B}_{q}}\left( \Lambda_{w}(\overbar{e_{a}}e_{b}) + \Lambda_{w}(\overbar{e_{i}}e_{b}) + \frac{1}{2}\sum_{H_{c}\in\tilde{B}_{q}-\{H_{b}\}} \Lambda_{w}(\overbar{e_{c}}e_{b}) \right) $.

Assuming $w$ was reduced to begin with, then there is only one final type of reduction we can apply to $w''''$.
We have that 
$((e_{i}x^{-1})\overbar{e_{i}}e_{a})\varphi=
\\ \noindent (h_{j}x^{-1})(\overbar{h_{j}y_{q}^{-1}x^{-1}})(h_{i}y_{q}^{-1}x^{-1})(\overbar{h_{i}y_{q}^{-1}x^{-1}})(h_{a}y_{q}^{-1}x^{-1})=(h_{j}x^{-1})(\overbar{h_{j}y_{q}^{-1}x^{-1}})(h_{a}y_{q}^{-1}x^{-1})$.
Letting $w'''''$ be the result of applying all such reductions to $w''''$, we see that the length of $((e_{i}x^{-1})\overbar{e_{i}}e_{a})\varphi$ is 2 less in $w'''''$ than it is in $w''''$.
Then $l(w''''')=l(w'''')-2\Lambda_{w}((e_{i}x^{-1})\overbar{e_{i}}e_{a})$.

Since $w'''''$ and $w$ are both fully reduced, we have:
\begin{align*}
	&	|w|_{\alpha_{4}}	=	|(w)\varphi_{24}|_{\alpha_{4}}	=	l(w''''')	\\
=	&	l(w'''')-2\Lambda_{w}((e_{i}x^{-1})\overbar{e_{i}}e_{a})	\\
=	&	l(w''')-6\Lambda_{w}(\overbar{e_{i}}e_{a})-4\smashoperator[lr]{\sum_{H_{b}\in\tilde{B}_{q}}} \ \left( \Lambda_{w}(\overbar{e_{a}}e_{b}) + \Lambda_{w}(\overbar{e_{i}}e_{b}) + \frac{1}{2}\smashoperator[lr]{\sum_{H_{c}\in\tilde{B}_{q}-\{H_{b}\}}} \Lambda_{w}(\overbar{e_{c}}e_{b}) \right) - 2\Lambda_{w}((e_{i}x^{-1})\overbar{e_{i}}e_{a})	\\
=	&	l(w'')-2\Lambda_{w}((e_{j}y_{q}^{-1})\overbar{e_{j}}e_{i})-2\Lambda_{w}((e_{j}y_{q}^{-1})\overbar{e_{j}}e_{a})-2\smashoperator[lr]{\sum_{H_{b}\in\tilde{B}_{q}}}\Lambda_{w}((e_{j}y_{q}^{-1})\overbar{e_{j}}e_{b}) 	-6\Lambda_{w}(\overbar{e_{i}}e_{a})	\\
	&	-4\smashoperator[lr]{\sum_{H_{b}\in\tilde{B}_{q}}} \ \left( \Lambda_{w}(\overbar{e_{a}}e_{b}) + \Lambda_{w}(\overbar{e_{i}}e_{b}) + \frac{1}{2}\smashoperator[lr]{\sum_{H_{c}\in\tilde{B}_{q}-\{H_{b}\}}} \Lambda_{w}(\overbar{e_{c}}e_{b}) \right) -2\Lambda_{w}((e_{i}x^{-1})\overbar{e_{i}}e_{a})	\\
=	&	l(w')-2\Lambda_{w}(\overbar{e_{j}}e_{i})-2\Lambda_{w}(\overbar{e_{j}}e_{a})-2\smashoperator[lr]{\sum_{H_{b}\in\tilde{B}_{q}}}\Lambda_{w}(\overbar{e_{j}}e_{b})	-2\Lambda_{w}((e_{j}y_{q}^{-1})\overbar{e_{j}}e_{i})-2\Lambda_{w}((e_{j}y_{q}^{-1})\overbar{e_{j}}e_{a})	\\
	&	-2\smashoperator[lr]{\sum_{H_{b}\in\tilde{B}_{q}}}\Lambda_{w}((e_{j}y_{q}^{-1})\overbar{e_{j}}e_{b}) 	-6\Lambda_{w}(\overbar{e_{i}}e_{a})-4\smashoperator[lr]{\sum_{H_{b}\in\tilde{B}_{q}}} \ \left( \Lambda_{w}(\overbar{e_{a}}e_{b}) + \Lambda_{w}(\overbar{e_{i}}e_{b}) + \frac{1}{2}\smashoperator[lr]{\sum_{H_{c}\in\tilde{B}_{q}-\{H_{b}\}}} \Lambda_{w}(\overbar{e_{c}}e_{b}) \right)	\\
	&	-2\Lambda_{w}((e_{i}x^{-1})\overbar{e_{i}}e_{a})	\\
=	&	l(w) + 4\Lambda_{w}(e_{a}) +2\Lambda_{w}(e_{i}) + 2\smashoperator[lr]{\sum_{H_{b}\in\tilde{B}_{q}}}\Lambda_{w}(e_{b})	-2\Lambda_{w}(\overbar{e_{j}}e_{i})-2\Lambda_{w}(\overbar{e_{j}}e_{a})	-2\smashoperator[lr]{\sum_{H_{b}\in\tilde{B}_{q}}}\Lambda_{w}(\overbar{e_{j}}e_{b})	\\
	&	-2\Lambda_{w}((e_{j}y_{q}^{-1})\overbar{e_{j}}e_{i})-2\Lambda_{w}((e_{j}y_{q}^{-1})\overbar{e_{j}}e_{a})-2\smashoperator[lr]{\sum_{H_{b}\in\tilde{B}_{q}}}\Lambda_{w}((e_{j}y_{q}^{-1})\overbar{e_{j}}e_{b}) -6\Lambda_{w}(\overbar{e_{i}}e_{a})	\\
		&	-4\smashoperator[lr]{\sum_{H_{b}\in\tilde{B}_{q}}} \ \left( \Lambda_{w}(\overbar{e_{a}}e_{b}) + \Lambda_{w}(\overbar{e_{i}}e_{b}) + \frac{1}{2}\smashoperator[lr]{\sum_{H_{c}\in\tilde{B}_{q}-\{H_{b}\}}} \Lambda_{w}(\overbar{e_{c}}e_{b}) \right) -2\Lambda_{w}((e_{i}x^{-1})\overbar{e_{i}}e_{a})	\\
=	&	|w|_{\alpha_{2}} + 2\Lambda_{w}(e_{a})-2\Lambda_{w}(\overbar{e_{i}}e_{a})-2\Lambda_{w}((e_{i}x^{-1})\overbar{e_{i}}e_{a})		\\
	&	+2\smashoperator[lr]{\sum_{H_{b}\in\tilde{B}_{q}}} \ \left( \Lambda_{w}(e_{b})-\Lambda_{w}(\overbar{e_{j}}e_{b})-\Lambda_{w}((e_{j}y_{q}^{-1})\overbar{e_{j}}e_{b})-\smashoperator[lr]{\sum_{H_{c}\in\tilde{B}_{q}-\{H_{b}\}}}\Lambda_{w}(\overbar{e_{c}}e_{b}) \right)	\\
	&	+ 2\Lambda_{w}(e_{i})-2\Lambda_{w}(\overbar{e_{j}}e_{i})-2\Lambda_{w}((e_{j}y_{q}^{-1})\overbar{e_{j}}e_{i}) - 4\smashoperator[lr]{\sum_{H_{c}\in\tilde{B}_{q}}}\Lambda_{w}(\overbar{e_{c}}e_{i})	\\
	&	+ 2\Lambda_{w}(e_{a})-2\Lambda_{w}(\overbar{e_{j}}e_{a})-2\Lambda_{w}((e_{j}y_{q}^{-1})\overbar{e_{j}}e_{a}) - 4\smashoperator[lr]{\sum_{H_{c}\in\tilde{B}_{q}}}\Lambda_{w}(\overbar{e_{c}}e_{a})	\\
	&	-4\Lambda_{w}(\overbar{e_{i}}e_{a})	\\
=	&	|w|_{\alpha_{1}}+|w|_{\alpha_{3}}-|w|_{\alpha_{2}}	.
\end{align*}
\end{proof}

\begin{lemma}\label{2a domain height} 
If $\mathbf{x}\subset H_{i}\in\ B_{q}$ and $\mathbf{y}\subset H_{j}\not\in\hat{A}$ and $\hat{A}\subseteq B_{q}$ then
$||\alpha_{4}||-||\alpha_{1}||=||\alpha_{3}||-||\alpha_{2}||$,
where $\alpha_{1}=\alpha_{2}(\mathbf{A},\mathbf{x})$, $\alpha_{3}=\alpha_{2}(\mathbf{B},\mathbf{y})$, and $\alpha_{4}=\alpha_{2}(\mathbf{B},\mathbf{y})(\mathbf{A}^{y_{q}},\mathbf{x}^{y_{q}})$.
\end{lemma}

\begin{proof}
By Lemma \ref{2a word length},
\begin{align*}
||\alpha_{4}||-||\alpha_{1}|| 	&=	\sum_{w\in\mathcal{W}}\Bigl(|w|_{\alpha_{4}}-2\Bigr) - \sum_{w\in\mathcal{W}}\Bigl(|w|_{\alpha_{1}}-2\Bigr) 	\\
					&=	\sum_{w\in\mathcal{W}}\Bigl(|w|_{\alpha_{4}}-2-|w|_{\alpha_{1}}+2\Bigr)						\\
					&=	\sum_{w\in\mathcal{W}}\Bigl(|w|_{\alpha_{3}}-|w|_{\alpha_{2}}\Bigr)							\\
					&=	\sum_{w\in\mathcal{W}}\Bigl(|w|_{\alpha_{3}}-2\Bigr) - \sum_{w\in\mathcal{W}}\Bigl(|w|_{\alpha_{2}}-2\Bigr) 	\\
					&=	||\alpha_{3}||-||\alpha_{2}||	.
\end{align*}
\end{proof}

\begin{lemma}\label{2a peak} 
If
\begin{tikzcd}[cramped]
\alpha_{1} \ar[r,-<-,dash,"(\mathbf{A}{,}\mathbf{x})"]	& \alpha_{2} \ar[r,->-,dash,"(\mathbf{B}{,}\mathbf{y})"]	& \alpha_{3}
\end{tikzcd}
is a peak
with $H_{i}\in\hat{B}$, $H_{j}\not\in\hat{A}$, and $\hat{A}\subseteq B_{q}$, then 
 $||\alpha_{2}(\mathbf{B},\mathbf{y})(\mathbf{A}^{y_{q}},\mathbf{x}^{y_{q}})||<||\alpha_{2}||$.
\end{lemma}

\begin{proof}
By Lemma \ref{2a domain height}, $||\alpha_{2}(\mathbf{B},\mathbf{y})(\mathbf{A}^{y_{q}},\mathbf{x}^{y_{q}})||=||\alpha_{4}||=||\alpha_{3}||+||\alpha_{1}||-||\alpha_{2}||$.
Since
\begin{tikzcd}[cramped]
\alpha_{1} \ar[r,-<-,dash,"(\mathbf{A}{,}\mathbf{x})"]	& \alpha_{2} \ar[r,->-,dash,"(\mathbf{B}{,}\mathbf{y})"]	& \alpha_{3}
\end{tikzcd}
is a peak, then $||\alpha_{2}||\ge\max(||\alpha_{1}||,||\alpha_{3}||)$ and $||\alpha_{2}||>\min(||\alpha_{1}||,||\alpha_{3}||)$.
Now:
\begin{align*}
||\alpha_{4}||	&=\max(||\alpha_{1}||,||\alpha_{3}||)+\min(||\alpha_{1}||,||\alpha_{3}||)-||\alpha_{2}||	&\\
			&<||\alpha_{2}||+||\alpha_{2}||-||\alpha_{2}||								&\\
			&=||\alpha_{2}||		.											&
\end{align*}
\end{proof}

\begin{prop}\label{prop 2a summary} 
Let $H_{1}\ast\dots\ast H_{n}$ be an $\mathfrak{S}$ free factor splitting for $G$.
Let $(\mathbf{A},\mathbf{x})$ and $(\mathbf{B},\mathbf{y})$ be relative multiple Whitehead automorphisms with $\mathbf{x}\subset H_{i}$ and $\mathbf{y}\subset H_{j}$, $\hat{A}\subset\{H_{1},\dots,H_{n}\}-\{H_{i},H_{j}\}$, $\hat{B}\subset\{H_{1},\dots,H_{n}\}-\{H_{j}\}$ such that $\{H_{i}\}\cup\hat{A}\subset B_{q}$ for some $q$.
Let $\alpha_{2}$ be the domain whose $\alpha$-graph has $\mathfrak{S}$-labelling $(H_{1},\dots,H_{n})$, and let $\alpha_{1}=\alpha_{2}(\mathbf{A},\mathbf{x})$ and $\alpha_{3}=\alpha_{2}(\mathbf{B},\mathbf{y})$.
If
\begin{tikzcd}[cramped]
\alpha_{1} \ar[r,-<-,dash,"(\mathbf{A}{,}\mathbf{x})"]	& \alpha_{2} \ar[r,->-,dash,"(\mathbf{B}{,}\mathbf{y})"]	& \alpha_{3}
\end{tikzcd}
is a peak, then it is reducible.
\end{prop}

\begin{proof}
By Lemmas \ref{2a loop} and \ref{2a contractible}, the path 
\begin{tikzcd}[cramped]
\alpha_{1} \ar[r,-<-,dash,"(\mathbf{A}{,}\mathbf{x})"]	& \alpha_{2} \ar[r,->-,dash,"(\mathbf{B}{,}\mathbf{y})"]	& \alpha_{3}
\end{tikzcd}
is homotopic in the Space of Domains to the path
\begin{tikzcd}[cramped]
\alpha_{1} \ar[r,->-,dash,"(\mathbf{B'}{,}\mathbf{y})"]	& \alpha_{2}(\mathbf{B},\mathbf{y})(\mathbf{A}^{y_{q}},\mathbf{x}^{y_{q}}) \ar[r,-<-,dash,"(\mathbf{A}^{y_{q}}{,}\mathbf{x}^{y_{q}})"]	&[0.4cm] \alpha_{3}
\end{tikzcd}
where $\mathbf{B'}=(B_{1},\dots,(B_{q}-\hat{A})\cup\widehat{\mathbf{A}^{\mathbf{x}}},\dots,B_{k})$.
By Lemma \ref{2a peak}, $||\alpha_{2}(\mathbf{B},\mathbf{y})(\mathbf{A}^{y_{q}},\mathbf{x}^{y_{q}})||<||\alpha_{2}||$.
Thus by Definition \ref{defn reducible}, the peak
\begin{tikzcd}[cramped]
\alpha_{1} \ar[r,-<-,dash,"(\mathbf{A}{,}\mathbf{x})"]	& \alpha_{2} \ar[r,->-,dash,"(\mathbf{B}{,}\mathbf{y})"]	& \alpha_{3}
\end{tikzcd}
is reducible.
\end{proof}


\subsubsection*{Case 2(b): $H_{i}\in\hat{B}$ (say $H_{i}\in B_{q}$) and $H_{j}\not\in\hat{A}$, with $\hat{A}\not\subseteq B_{q}$}

\begin{lemma}\label{2b loop} 
If $H_{i}\in\hat{B}$ and $H_{j}\not\in\hat{A}$ with $\hat{A}\not\subseteq B_{q}$, then 
$(\mathbf{A},\mathbf{x})^{-1}(\mathbf{B},\mathbf{y})=(\mathbf{A}-B_{q},\mathbf{x})^{-1}(\mathbf{B},\mathbf{y})(\mathbf{A}\cap B_{q},\mathbf{x}^{y_{q}})^{-1}$ and
$(\mathbf{A},\mathbf{x})^{-1}(\mathbf{B},\mathbf{y})=(\mathbf{B}+_{q}\hat{A},\mathbf{y})(\mathbf{A},\mathbf{x}^{y_{q}})^{-1}(\mathbf{\bar{B}}_{q}\cap\hat{A},y_{q}^{-1}\mathbf{\tilde{y}}_{q})$,
where $[\mathbf{B'}]_{q}:=(B_{q}-\hat{A})\cup\bigcup(\mathbf{A}^{\mathbf{x}})$ and $[\mathbf{B'}]_{k}:=B_{k}$ for $k\ne q$,
and $[(\mathbf{B}+_{q}\hat{A})']_{q}:=([(\mathbf{B}+_{q}\hat{A})]_{q}-\hat{A})\cup\bigcup([(\mathbf{B}+_{q}\hat{A})]_{q}\cap\mathbf{A})^{\mathbf{x}}=(B_{q}-\hat{A})\cup\bigcup\mathbf{A}^{\mathbf{x}}$ and $[(\mathbf{B}+_{q}\hat{A})']_{k}:=[(\mathbf{B}+_{q}\hat{A})]_{k}=B_{k}-\hat{A}$ for $k\ne q$.
\end{lemma}

That is, there exist vertices $\alpha_{4}$, $\alpha_{5}$, $\alpha_{4}'$, $\alpha_{5}'$ in our Graph of Domains such that 
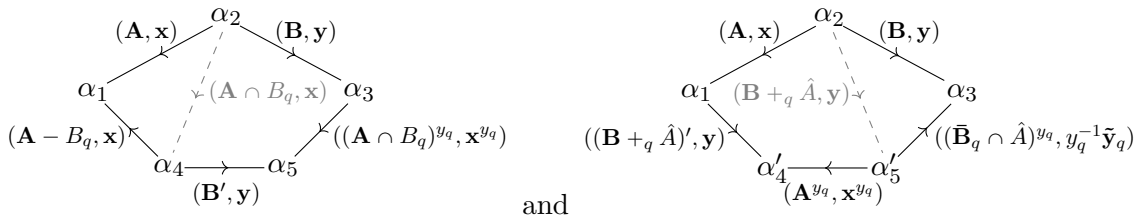
\begin{tikzpicture}
\node at (0,0) {$\alpha_{1}$};
\node at (1.75,1) {$\alpha_{2}$};
\node at (3.5,0) {$\alpha_{3}$};
\node at (1,-1) {$\alpha_{4}$};
\node at (2.5,-1) {$\alpha_{5}$};
\draw[-<-] (0.18,0.12) -- (1.57,0.88); 
\draw[->-] (1.93,0.88) -- (3.26,0.16); 
\draw[-<-] (0.15,-0.15) -- (0.85,-0.85); 
\draw[->-] (1.2,-1) -- (2.25,-1); 
\draw[-<-] (2.65,-0.85) -- (3.35,-0.15); 
\draw[->-,gray,dashed] (1.7,0.8) -- (1.05,-0.8); 
\node[scale=0.8] at (0.7,0.8) {$(\mathbf{A},\mathbf{x})$};
\node[scale=0.8] at (2.8,0.8) {$(\mathbf{B},\mathbf{y})$};
\node[scale=0.8] at (1.75,-1.35) {$(\mathbf{B'},\mathbf{y})$};
\node[scale=0.8] at (-0.35,-0.6) {$(\mathbf{A}-B_{q},\mathbf{x})$};
\node[scale=0.8] at (4.3,-0.6) {$( (\mathbf{A}\cap B_{q})^{y_{q}},\mathbf{x}^{y_{q}})$};
\node[scale=0.8] at (2.3,0) {\textcolor{gray}{$(\mathbf{A}\cap B_{q},\mathbf{x})$}};
\end{tikzpicture}
and
\begin{tikzpicture}
\node at (0,0) {$\alpha_{1}$};
\node at (1.75,1) {$\alpha_{2}$};
\node at (3.5,0) {$\alpha_{3}$};
\node at (1,-1) {$\alpha_{4}'$};
\node at (2.5,-1) {$\alpha_{5}'$};
\draw[-<-] (0.18,0.12) -- (1.57,0.88); 
\draw[->-] (1.93,0.88) -- (3.26,0.16); 
\draw[->-] (0.15,-0.15) -- (0.85,-0.85); 
\draw[-<-] (1.2,-1) -- (2.25,-1); 
\draw[->-] (2.65,-0.85) -- (3.35,-0.15); 
\draw[->-,gray,dashed] (1.8,0.8) -- (2.45,-0.8); 
\node[scale=0.8] at (0.7,0.8) {$(\mathbf{A},\mathbf{x})$};
\node[scale=0.8] at (2.8,0.8) {$(\mathbf{B},\mathbf{y})$};
\node[scale=0.8] at (1.8,-1.35) {$(\mathbf{A}^{y_{q}},\mathbf{x}^{y_{q}})$};
\node[scale=0.8] at (-0.55,-0.6) {$( (\mathbf{B}+_{q}\hat{A})\mathbf{'},\mathbf{y})$};
\node[scale=0.8] at (4.45,-0.6) {$( (\mathbf{\bar{B}}_{q}\cap\hat{A})^{y_{q}},y_{q}^{-1}\mathbf{\tilde{y}}_{q})$};
\node[scale=0.8] at (1.25,0) {\textcolor{gray}{$(\mathbf{B}+_{q}\hat{A},\mathbf{y})$}};
\end{tikzpicture}
are both loops.

\begin{proof} 
By Proposition \ref{whitehead notational properties} ($1$), $(\mathbf{A},\mathbf{x})(\mathbf{A}-B_{q},\mathbf{x})^{-1}=(\mathbf{A}\cap B_{q},\mathbf{x})$.
Writing $\bigcup{\mathbf{A}}=\hat{A}$, we have that $\bigcup{(\mathbf{A}\cap B_{q})}=\hat{A}\cap B_{q}\subseteq B_{q}$.
Similarly, by Proposition \ref{whitehead notational properties} ($4$), $(\mathbf{B},\mathbf{y})( (\mathbf{\bar{B}}_{q}\cap\hat{A})^{y_{q}},y_{q}^{-1}\mathbf{\tilde{y}}_{q})^{-1}=(\mathbf{B}+_{q}\hat{A},\mathbf{y})$.
Note that $[\mathbf{B}+_{q}\hat{A}]_{q}=B_{q}\cup\hat{A}\supseteq\hat{A}$.
We have now reduced both problems to the form required by Case 2a, so the result follows from Lemma \ref{2a loop}.
\end{proof}

\begin{lemma}\label{2b contractible} 
The loops 
\\ \noindent
\begin{tikzpicture}
\node at (0,0) {$\alpha_{1}$};
\node at (1.75,1) {$\alpha_{2}$};
\node at (3.5,0) {$\alpha_{3}$};
\node at (1,-1) {$\alpha_{4}$};
\node at (2.5,-1) {$\alpha_{5}$};
\draw[-<-] (0.18,0.12) -- (1.57,0.88); 
\draw[->-] (1.93,0.88) -- (3.26,0.16); 
\draw[-<-] (0.15,-0.15) -- (0.85,-0.85); 
\draw[->-] (1.2,-1) -- (2.25,-1); 
\draw[-<-] (2.65,-0.85) -- (3.35,-0.15); 
\draw[->-,gray,dashed] (1.7,0.8) -- (1.05,-0.8); 
\node[scale=0.8] at (0.7,0.8) {$(\mathbf{A},\mathbf{x})$};
\node[scale=0.8] at (2.8,0.8) {$(\mathbf{B},\mathbf{y})$};
\node[scale=0.8] at (1.75,-1.35) {$(\mathbf{B'},\mathbf{y})$};
\node[scale=0.8] at (-0.35,-0.6) {$(\mathbf{A}-B_{q},\mathbf{x})$};
\node[scale=0.8] at (4.3,-0.6) {$( (\mathbf{A}\cap B_{q})^{y_{q}},\mathbf{x}^{y_{q}})$};
\node[scale=0.8] at (2.3,0) {\textcolor{gray}{$(\mathbf{A}\cap B_{q},\mathbf{x})$}};
\end{tikzpicture}
and
\begin{tikzpicture}
\node at (0,0) {$\alpha_{1}$};
\node at (1.75,1) {$\alpha_{2}$};
\node at (3.5,0) {$\alpha_{3}$};
\node at (1,-1) {$\alpha_{4}'$};
\node at (2.5,-1) {$\alpha_{5}'$};
\draw[-<-] (0.18,0.12) -- (1.57,0.88); 
\draw[->-] (1.93,0.88) -- (3.26,0.16); 
\draw[->-] (0.15,-0.15) -- (0.85,-0.85); 
\draw[-<-] (1.2,-1) -- (2.25,-1); 
\draw[->-] (2.65,-0.85) -- (3.35,-0.15); 
\draw[->-,gray,dashed] (1.8,0.8) -- (2.45,-0.8); 
\node[scale=0.8] at (0.7,0.8) {$(\mathbf{A},\mathbf{x})$};
\node[scale=0.8] at (2.8,0.8) {$(\mathbf{B},\mathbf{y})$};
\node[scale=0.8] at (1.8,-1.35) {$(\mathbf{A}^{y_{q}},\mathbf{x}^{y_{q}})$};
\node[scale=0.8] at (-0.55,-0.6) {$( (\mathbf{B}+_{q}\hat{A})\mathbf{'},\mathbf{y})$};
\node[scale=0.8] at (4.45,-0.6) {$( (\mathbf{\bar{B}}_{q}\cap\hat{A})^{y_{q}},y_{q}^{-1}\mathbf{\tilde{y}}_{q})$};
\node[scale=0.8] at (1.25,0) {\textcolor{gray}{$(\mathbf{B}+_{q}\hat{A},\mathbf{y})$}};
\end{tikzpicture}
are both contractible in our Space of Domains.
\end{lemma}

\begin{proof}
As in Lemma \ref{1b contractible}, the $\alpha_{1}$--$\alpha_{2}$--$\alpha_{4}$ triangle can be `filled' with an $A_{i}$ graph, and the $\alpha_{2}$--$\alpha_{3}$--$\alpha_{5}'$ triangle can be `filled' with an $A_{j}$ graph.
Now $\hat{A}\cap B_{q}\subseteq B_{q}$ and $\hat{A}\subseteq[\mathbf{B}+_{q}\hat{A}]_{q}=B_{q}\cup\hat{A}$, so by Lemma \ref{2a contractible} (Case 2a), the squares $\alpha_{4}\dash \alpha_{2}\dash \alpha_{3}\dash \alpha_{5}\dash \alpha_{4}$ and $\alpha_{1}\dash \alpha_{2}\dash \alpha_{5}'\dash \alpha_{4}'\dash \alpha_{1}$ are both contractible.
\end{proof}
  
\begin{lemma}\label{2b peak} 
If
\begin{tikzcd}[cramped]
\alpha_{1} \ar[r,-<-,dash,"(\mathbf{A}{,}\mathbf{x})"]	& \alpha_{2} \ar[r,->-,dash,"(\mathbf{B}{,}\mathbf{y})"]	& \alpha_{3}
\end{tikzcd}
is a peak
with $H_{i}\in\hat{B}$, $H_{j}\not\in\hat{A}$, and $\hat{A}\not\subseteq B_{q}$, then either
$||\alpha_{2}(\mathbf{A},\mathbf{x})(\mathbf{A}-B_{q},\mathbf{x})^{-1}||<||\alpha_{2}||$
and
$||\alpha_{2}(\mathbf{A}\cap B_{q},\mathbf{x})(\mathbf{B'},\mathbf{y})||<||\alpha_{2}||$
, or
$||\alpha_{2}(\mathbf{A},\mathbf{x})((\mathbf{B}+_{q}\hat{A})\mathbf{'},\mathbf{y})||<||\alpha_{2}||$
and
$||\alpha_{2}(\mathbf{B}+_{q}\hat{A},\mathbf{y})||<||\alpha_{2}||$
.
\end{lemma}
That is, $||\alpha_{2}||>\max(||\alpha_{4}||,||\alpha_{5}||)$ or $||\alpha_{2}||>\max(||\alpha_{4}'||,||\alpha_{5}'||)$.

\begin{proof}
It will suffice to show that either $||\alpha_{4}||<||\alpha_{2}||$ or $||\alpha_{5}'||<||\alpha_{2}||$. The problem then reduces to Case 2a (Lemma \ref{2a peak}).

Let $\alpha_{6}=\alpha_{2}(\mathbf{A}\cap B_{q},\mathbf{x})(\mathbf{B}+_{q}\hat{A},\mathbf{y})$.
For brevity, let $\mathbf{P}=\mathbf{A}-B_{q}$, $\mathbf{Q}=\mathbf{A}\cap B_{q}$, $\mathbf{R}=\mathbf{B}+_{q}\hat{A}$, and $\mathbf{S}=\mathbf{\bar{B}}_{q}\cap\hat{A}$.
For $b\ne q$ set $[\mathbf{R''}]_{b}=[\mathbf{R'}]_{b}:=B_{b}-\hat{A}=[\mathbf{R}]_{b}$,
$[\mathbf{R''}]_{q}:=(B_{q}-\hat{A})\sqcup \widehat{(B_{q}\cap \mathbf{A})^{\mathbf{x}}}\sqcup (\hat{A}-B_{q})$, and
$[\mathbf{R'}]_{q}:=(B_{q}-\hat{A})\sqcup \widehat{(\mathbf{A})^{\mathbf{x}}}$.
Then the diagram in Figure \ref{fig commuting diagram 2b} commutes.
\begin{figure}
\centering
\begin{tikzcd}[cramped]
			&			&	\alpha_{2} 
							\ar[dll, "(\mathbf{A}{,}\mathbf{x})" ' ,purple]
							\arrow[Green]{dl}[pos=0.2,Green,xshift=-0.5em]{(\mathbf{Q}{,}\mathbf{x})}
							\ar[drr, "(\mathbf{B}{,}\mathbf{y})",violet]
							\ar[dddr, bend left=90, looseness=2.5, "(\mathbf{R}{,}\mathbf{y})",blue]
							\ar[dddl, bend right=90, looseness=2.5, "(\mathbf{R}{,}\mathbf{y})" ',white]
									&			&
															\\
	\alpha_{1} 
	\ar[ddr, "(\mathbf{R'}{,}\mathbf{y})" ',blue]	
			&	\alpha_{4} 
				\ar[l, pos=0.45, "(\mathbf{P}{,}\mathbf{x})",red]
				\ar[dr, "(\mathbf{R''}{,}\mathbf{y})" ',blue]
				\ar[rr, "(\mathbf{B'}{,}\mathbf{y})",violet]
						&			&	\alpha_{5} 
												&	\alpha_{3} 
													\arrow[Green]{l}[Green,pos=0.7,yshift=-0.2em]{(\mathbf{Q}^{y_{q}}{,}\mathbf{x}^{y_{q}})}	
															\\	
			&			&	\alpha_{6} 
							\arrow[orange]{ur}[xshift=0.5em]{(\mathbf{S}^{y_{q}}{,}y_{q}^{-1}\mathbf{\tilde{y}}_{q})}
							\ar[dl, pos=0.3, "(\mathbf{P}^{y_{q}}{,}\mathbf{x}^{y_{q}})" ',red]
									&			&	
															\\
			&	\alpha_{4}' 
						&			&	\alpha_{5}' 
										\ar[ll, "(\mathbf{A}^{y_{q}}{,}\mathbf{x}^{y_{q}})",purple]
										\ar[Green]{ul}[Green,xshift=0.5em]{(\mathbf{Q}^{y_{q}}{,}\mathbf{x}^{y_{q}})}
										\ar[orange]{uur}[pos=0.3,xshift=0.5em]{(\mathbf{S}^{y_{q}}{,}y_{q}^{-1}\mathbf{\tilde{y}}_{q})}
												&			\\
\end{tikzcd}
\caption{Commuting Diagram for Case 2b}
\label{fig commuting diagram 2b}
\end{figure}

Since
\begin{tikzcd}[cramped]
\alpha_{1}	&	\alpha_{2}
			\ar[l, dash,->-, "(\mathbf{A}{,}\mathbf{x})" ']
			\ar[r,, dash, ->-, "(\mathbf{B}{,}\mathbf{y})"]
									&	\alpha_{3}
\end{tikzcd}
is a peak, we have that $\min(\|\alpha_{1}\|-\|\alpha_{2}\|,\|\alpha_{3}\|-\|\alpha_{2}\|)<0$ and $\max(\|\alpha_{1}\|-\|\alpha_{2}\|,\|\alpha_{3}\|-\|\alpha_{2}\|)\le0$.

Observe that $\widehat{\mathbf{Q}^{y_{q}}} \cap \widehat{\mathbf{S}^{y_{q}}} = (\hat{A}\cap B_{q})^{y_{q}} \cap (\hat{A}-\hat{B})^{y_{q}} = \emptyset$.
Additionally, $y_{q}^{-1}\mathbf{\tilde{y}}_{q}\in H_{j}\not\in\hat{A}$ (so $y_{q}^{-1}\mathbf{\tilde{y}}_{q}\not\in\widehat{\mathbf{Q}^{y_{q}}}$), and
$\mathbf{x}^{y_{q}}\in H_{i}^{y_{q}}\in B_{q}^{y_{q}}$ (so $\mathbf{x}^{y_{q}}\not\in\widehat{\mathbf{S}^{y_{q}}}$).
Thus $\alpha_{5}'\dash\alpha_{6}\dash\alpha_{5}\dash\alpha_{3}\dash\alpha_{5}'$ is a square falling under Case 1a, so by Lemma \ref{1a domain height}, we have that
$\|\alpha_{3}\|-\|\alpha_{5}'\|=\|\alpha_{5}\|-\|\alpha_{6}\|$.

Also observe that $\hat{P}=\hat{A}-B_{q}\subseteq[\mathbf{R''}]_{q}$, $\mathbf{x}\in H_{i}\in B_{q}-\hat{A}\subset[\mathbf{R''}]_{q}$, and $\mathbf{y}\in H_{j}\not\in\hat{A}-B_{q}=\hat{P}$.
So $\alpha_{4}\dash\alpha_{6}\dash\alpha_{4}'\dash\alpha_{1}\dash\alpha_{4}$ is a square falling under Case 2a, and by Lemma \ref{2a domain height}, we have that
$\|\alpha_{1}\|-\|\alpha_{4}\|=\|\alpha_{4}'\|-\|\alpha_{6}\|$.

Since $\widehat{\mathbf{P}^{y_{q}}}=(\hat{A}-B_{q})^{y_{q}}=\widehat{\mathbf{S}^{y_{q}}}$ (and $H_{i}^{y_{q}},H_{j}^{y_{q}}\not\in\widehat{\mathbf{P}^{y_{q}}}$) then by Lemma \ref{lemma acting on common C}, we have that either 
$\max(\|\alpha_{4}'\|-\|\alpha_{6}\|,\|\alpha_{5}\|-\|\alpha_{6}\|)>0$ or $\|\alpha_{4}'\|-\|\alpha_{6}\|=\|\alpha_{5}\|-\|\alpha_{6}\|=0$.

Finally, we note that $\|\alpha_{4}\|-\|\alpha_{2}\|=(\|\alpha_{1}\|-\|\alpha_{2}\|)-(\|\alpha_{1}\|-\|\alpha_{4}\|)$, and similarly, $\|\alpha_{5}'\|-\|\alpha_{2}\|=(\|\alpha_{3}\|-\|\alpha_{2}\|)-(\|\alpha_{3}\|-\|\alpha_{5}'\|)$.

As in the proof of Lemma \ref{1b peak}, we now deduce from this information that $\min(\|\alpha_{4}\|-\|\alpha_{2}\|,\|\alpha_{5}'\|-\|\alpha_{2}\|)<0$, as required.
\end{proof}

\begin{prop}\label{prop 2b summary} 
Let $H_{1}\ast\dots\ast H_{n}$ be an $\mathfrak{S}$ free factor splitting for $G$.
Let $(\mathbf{A},\mathbf{x})$ and $(\mathbf{B},\mathbf{y})$ be relative multiple Whitehead automorphisms with $\mathbf{x}\subset H_{i}$, $\mathbf{y}\subset H_{j}$, $\hat{A}\subset\{H_{1},\dots,H_{n}\}-\{H_{i},H_{j}\}$, $\hat{B}\subset\{H_{1},\dots,H_{n}\}-\{H_{j}\}$, $H_{i}\in B_{q}$ for some $Q$, and $\hat{A}\not\subset B_{q}$.
Let $\alpha_{2}$ be the domain whose $\alpha$-graph has $\mathfrak{S}$-labelling $(H_{1},\dots,H_{n})$, and let $\alpha_{1}=\alpha_{2}(\mathbf{A},\mathbf{x})$ and $\alpha_{3}=\alpha_{2}(\mathbf{B},\mathbf{y})$.
If
\begin{tikzcd}[cramped]
\alpha_{1} \ar[r,-<-,dash,"(\mathbf{A}{,}\mathbf{x})"]	& \alpha_{2} \ar[r,->-,dash,"(\mathbf{B}{,}\mathbf{y})"]	& \alpha_{3}
\end{tikzcd}
is a peak, then it is reducible.
\end{prop}

\begin{proof}
By Lemmas \ref{2b loop} and \ref{2b contractible}, the path 
\begin{tikzcd}[cramped]
\alpha_{1} \ar[r,-<-,dash,"(\mathbf{A}{,}\mathbf{x})"]	& \alpha_{2} \ar[r,->-,dash,"(\mathbf{B}{,}\mathbf{y})"]	& \alpha_{3}
\end{tikzcd}
is homotopic in the Space of Domains to each of the paths
\\ \noindent
\begin{tikzcd}[sep=large]
\alpha_{1} \ar[r,-<-,dash,"(\mathbf{A}-B_{q}{,}\mathbf{x})"]	& \alpha_{2}(\mathbf{A},\mathbf{x})(\mathbf{A}-B_{q},\mathbf{x})^{-1} \ar[r,->-,dash,"(\mathbf{B'}{,}\mathbf{y})"] 	&[-0.4cm] \alpha_{2}(\mathbf{A}\cap B_{q},\mathbf{x})(\mathbf{B'},\mathbf{y}) \ar[r,-<-,dash,"((\mathbf{A}\cap B_{q})^{y_{q}}{,}\mathbf{x}^{y_{q}})"]	&[0.3cm] \alpha_{3}
\end{tikzcd}
\\ \noindent and
\begin{tikzcd}[sep=large]
\alpha_{1} \ar[r,->-,dash,"((\mathbf{B}+_{q}\hat{A})\mathbf{'}{,}\mathbf{y})"]	& \alpha_{2}(\mathbf{A},\mathbf{x})((\mathbf{B}+_{q}\hat{A})\mathbf{'},\mathbf{y}) \ar[r,-<-,dash,"(\mathbf{A}^{y_{q}}{,}\mathbf{x}^{y_{q}})"] 	&[-0.3cm] \alpha_{2}(\mathbf{B}+_{q}\hat{A},\mathbf{y}) \ar[r,->-,dash,"((\mathbf{\bar{B}}_{q}\cap\hat{A})^{y_{q}}{,}y_{q}^{-1}\mathbf{\tilde{y}}_{q})"]	&[0.5cm] \alpha_{3}
\end{tikzcd}
.

By Lemma \ref{2b peak}, either
$||\alpha_{2}(\mathbf{A},\mathbf{x})(\mathbf{A}-B_{q},\mathbf{x})^{-1}||<||\alpha_{2}||$
and
$||\alpha_{2}(\mathbf{A}\cap B_{q},\mathbf{x})(\mathbf{B'},\mathbf{y})|| \\ \noindent <||\alpha_{2}||$
, or
$||\alpha_{2}(\mathbf{A},\mathbf{x})((\mathbf{B}+_{q}\hat{A})\mathbf{'},\mathbf{y})||<||\alpha_{2}||$
and
$||\alpha_{2}(\mathbf{B}+_{q}\hat{A},\mathbf{y})||<||\alpha_{2}||$.
Thus by Definition \ref{defn reducible}, one of the above paths is a reduction for the peak
\begin{tikzcd}[cramped]
\alpha_{1} \ar[r,-<-,dash,"(\mathbf{A}{,}\mathbf{x})"]	&[-0.1cm] \alpha_{2} \ar[r,->-,dash,"(\mathbf{B}{,}\mathbf{y})"]	&[-0.1cm] \alpha_{3}
\end{tikzcd}~.
\end{proof}


\subsubsection*{Case 4: $H_{i}\in\hat{B}$ and $H_{j}\in\hat{A}$ (say $H_{i}\in B_{q}$ and $H_{j}\in A_{p}$)}

\begin{lemma}\label{4 loop} 
If $H_{i}\in B_{q}$ and $H_{j}\in A_{p}$, then \\
$(\mathbf{A},\mathbf{x})^{-1}(\mathbf{B},\mathbf{y})=((\mathbf{A}+_{p}B_{q})',\mathbf{x})^{-1}(\mathbf{B'},\mathbf{y})((\mathbf{\bar{A}}_{p}\cap B_{q})^{y_{q}},(\mathbf{\tilde{x}}_{p}x_{p}^{-1})^{y_{q}})^{-1}$
and \\ $(\mathbf{A},\mathbf{x})^{-1}(\mathbf{B},\mathbf{y})=((\mathbf{\bar{B}}_{q}\cap A_{p})^{x_{p}},(\mathbf{\tilde{y}}_{q}y_{q}^{-1})^{x_{p}})(\mathbf{A'},\mathbf{x})^{-1}((\mathbf{B}+_{q}A_{p})',\mathbf{y})$,
where $(\mathbf{A}+_{p}B_{q})'$ (and $(\mathbf{B}+_{q}A_{p})'$) are as defined in Proposition \ref{whitehead notational properties} (4), and $\mathbf{A'}$ and $\mathbf{B'}$ are defined similarly.
\end{lemma}

That is, there exist vertices $\alpha_{4}$, $\alpha_{5}$, $\alpha_{4}'$, and $\alpha_{5}'$ in our Graph of Domains such that 
\begin{tikzpicture}
\node at (0,0) {$\alpha_{1}$};
\node at (1.75,1) {$\alpha_{2}$};
\node at (3.5,0) {$\alpha_{3}$};
\node at (1,-1) {$\alpha_{4}$};
\node at (2.5,-1) {$\alpha_{5}$};
\draw[-<-] (0.18,0.12) -- (1.57,0.88); 
\draw[->-] (1.93,0.88) -- (3.26,0.16); 
\draw[-<-] (0.15,-0.15) -- (0.85,-0.85); 
\draw[->-] (1.2,-1) -- (2.25,-1); 
\draw[-<-] (2.65,-0.85) -- (3.35,-0.15); 
\draw[->-,gray,dashed] (1.7,0.8) -- (1.05,-0.8); 
\node[scale=0.8] at (0.7,0.8) {$(\mathbf{A},\mathbf{x})$};
\node[scale=0.8] at (2.8,0.8) {$(\mathbf{B},\mathbf{y})$};
\node[scale=0.8] at (1.75,-1.3) {$(\mathbf{B'},\mathbf{y})$};
\node[scale=0.8] at (-0.55,-0.6) {$((\mathbf{A}+_{p}B_{q})',\mathbf{x})$};
\node[scale=0.8] at (4.8,-0.6) {$((\mathbf{\bar{A}}_{p}\cap B_{q})^{y_{q}},(\mathbf{\tilde{x}}_{p}x_{p}^{-1})^{y_{q}})$};
\node[scale=0.8] at (1.75,-0.2) {\textcolor{gray}{$(\mathbf{\bar{A}}_{p}\cap B_{q},\mathbf{\tilde{x}}_{p}x_{p}^{-1})$}};
\end{tikzpicture}

and
\begin{tikzpicture}
\node at (0,0) {$\alpha_{1}$};
\node at (1.75,1) {$\alpha_{2}$};
\node at (3.5,0) {$\alpha_{3}$};
\node at (1,-1) {$\alpha_{4}'$};
\node at (2.5,-1) {$\alpha_{5}'$};
\draw[-<-] (0.18,0.12) -- (1.57,0.88); 
\draw[->-] (1.93,0.88) -- (3.26,0.16); 
\draw[->-] (0.15,-0.15) -- (0.85,-0.85); 
\draw[-<-] (1.2,-1) -- (2.25,-1); 
\draw[->-] (2.65,-0.85) -- (3.35,-0.15); 
\draw[->-,gray,dashed] (1.8,0.8) -- (2.45,-0.8); 
\node[scale=0.8] at (0.7,0.8) {$(\mathbf{A},\mathbf{x})$};
\node[scale=0.8] at (2.8,0.8) {$(\mathbf{B},\mathbf{y})$};
\node[scale=0.8] at (1.75,-1.3) {$(\mathbf{A'},\mathbf{x})$};
\node[scale=0.8] at (-1.3,-0.6) {$((\mathbf{\bar{B}}_{q}\cap A_{p})^{x_{p}},(\mathbf{\tilde{y}}_{q}y_{q}^{-1})^{x_{p}})$};
\node[scale=0.8] at (4.1,-0.6) {$((\mathbf{B}+_{q}A_{p})',\mathbf{y})$};
\node[scale=0.8] at (1.75,-0.25) {\textcolor{gray}{$(\mathbf{\bar{B}}_{q}\cap A_{p},\mathbf{\tilde{y}}_{q}y_{q}^{-1})$}};
\end{tikzpicture}
are both loops.

\begin{proof} 
By Proposition \ref{whitehead notational properties} (4), $(\mathbf{A},\mathbf{x})=(\mathbf{\bar{A}}_{p}\cap B_{q},\mathbf{\tilde{x}}_{p}x_{p}^{-1})((\mathbf{A}+_{p}B_{q})',\mathbf{x})$.
Now $\mathbf{\tilde{x}}_{p}x_{p}^{-1}\in H_{i}\in B_{q}$ still, and as $H_{j}\not\in\mathbf{B}$ then $H_{j}\not\in\mathbf{\bar{A}}_{p}\cap B_{q}$.
Also $\mathbf{\bar{A}}_{p}\cap B_{q}\subseteq B_{q}$, so by Case 2a (Lemma \ref{2a loop}), $(\mathbf{\bar{A}}_{p}\cap B_{q},\mathbf{\tilde{x}}_{p}x_{p}^{-1})(\mathbf{B'},\mathbf{y})=(\mathbf{B},\mathbf{y})((\mathbf{\bar{A}}_{p}\cap B_{q})^{y_{q}},(\mathbf{\tilde{x}}_{p}x_{p}^{-1})^{y_{q}})$.
The second loop is achieved by renaming $A$ to $B$ and $x$ to $y$, and vice versa.
\end{proof}

\begin{lemma}\label{4 contractible} 
The loops
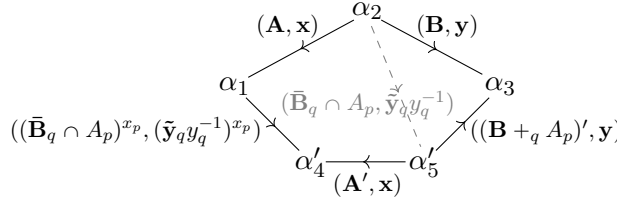
\begin{tikzpicture}
\node at (0,0) {$\alpha_{1}$};
\node at (1.75,1) {$\alpha_{2}$};
\node at (3.5,0) {$\alpha_{3}$};
\node at (1,-1) {$\alpha_{4}$};
\node at (2.5,-1) {$\alpha_{5}$};
\draw[-<-] (0.18,0.12) -- (1.57,0.88); 
\draw[->-] (1.93,0.88) -- (3.26,0.16); 
\draw[-<-] (0.15,-0.15) -- (0.85,-0.85); 
\draw[->-] (1.2,-1) -- (2.25,-1); 
\draw[-<-] (2.65,-0.85) -- (3.35,-0.15); 
\draw[->-,gray,dashed] (1.7,0.8) -- (1.05,-0.8); 
\node[scale=0.8] at (0.7,0.8) {$(\mathbf{A},\mathbf{x})$};
\node[scale=0.8] at (2.8,0.8) {$(\mathbf{B},\mathbf{y})$};
\node[scale=0.8] at (1.75,-1.3) {$(\mathbf{B'},\mathbf{y})$};
\node[scale=0.8] at (-0.55,-0.6) {$((\mathbf{A}+_{p}B_{q})',\mathbf{x})$};
\node[scale=0.8] at (4.8,-0.6) {$((\mathbf{\bar{A}}_{p}\cap B_{q})^{y_{q}},(\mathbf{\tilde{x}}_{p}x_{p}^{-1})^{y_{q}})$};
\node[scale=0.8] at (1.75,-0.2) {\textcolor{gray}{$(\mathbf{\bar{A}}_{p}\cap B_{q},\mathbf{\tilde{x}}_{p}x_{p}^{-1})$}};
\end{tikzpicture}

and
\begin{tikzpicture}
\node at (0,0) {$\alpha_{1}$};
\node at (1.75,1) {$\alpha_{2}$};
\node at (3.5,0) {$\alpha_{3}$};
\node at (1,-1) {$\alpha_{4}'$};
\node at (2.5,-1) {$\alpha_{5}'$};
\draw[-<-] (0.18,0.12) -- (1.57,0.88); 
\draw[->-] (1.93,0.88) -- (3.26,0.16); 
\draw[->-] (0.15,-0.15) -- (0.85,-0.85); 
\draw[-<-] (1.2,-1) -- (2.25,-1); 
\draw[->-] (2.65,-0.85) -- (3.35,-0.15); 
\draw[->-,gray,dashed] (1.8,0.8) -- (2.45,-0.8); 
\node[scale=0.8] at (0.7,0.8) {$(\mathbf{A},\mathbf{x})$};
\node[scale=0.8] at (2.8,0.8) {$(\mathbf{B},\mathbf{y})$};
\node[scale=0.8] at (1.75,-1.3) {$(\mathbf{A'},\mathbf{x})$};
\node[scale=0.8] at (-1.3,-0.6) {$((\mathbf{\bar{B}}_{q}\cap A_{p})^{x_{p}},(\mathbf{\tilde{y}}_{q}y_{q}^{-1})^{x_{p}})$};
\node[scale=0.8] at (4.1,-0.6) {$((\mathbf{B}+_{q}A_{p})',\mathbf{y})$};
\node[scale=0.8] at (1.75,-0.25) {\textcolor{gray}{$(\mathbf{\bar{B}}_{q}\cap A_{p},\mathbf{\tilde{y}}_{q}y_{q}^{-1})$}};
\end{tikzpicture}
are both contractible in our Space of Domains.
\end{lemma}

\begin{proof}
As in Lemma \ref{1b contractible}, the $\alpha_{1}$--$\alpha_{2}$--$\alpha_{4}$ triangle can be `filled' with an $A_{i}$ graph, and the $\alpha_{2}$--$\alpha_{3}$--$\alpha_{5}'$ triangle can be `filled' with an $A_{j}$ graph.
Now $\bigcup(\mathbf{\bar{A}}_{p}\cap B_{q})\subseteq B_{q}$ and $\bigcup(\mathbf{\bar{B}}\cap A_{p})\subseteq A_{p}$, so (after relabelling) by Lemma \ref{2a contractible} (Case 2a), the squares $\alpha_{4}\dash \alpha_{2}\dash \alpha_{3}\dash \alpha_{5}\dash \alpha_{4}$ and $\alpha_{1}\dash \alpha_{2}\dash \alpha_{5}'\dash \alpha_{4}'\dash \alpha_{1}$ are both contractible.
\end{proof}

\begin{lemma}\label{4 peak} 
If
\begin{tikzcd}[cramped]
\alpha_{1} \ar[r,-<-,dash,"(\mathbf{A}{,}\mathbf{x})"]	& \alpha_{2} \ar[r,->-,dash,"(\mathbf{B}{,}\mathbf{y})"]	& \alpha_{3}
\end{tikzcd}
is a peak
with $H_{i}\in\hat{B}$ and $H_{j}\in\hat{A}$, then (up to relabelling)
$||\alpha_{2}(\mathbf{A},\mathbf{x})((\mathbf{A}+_{p}B_{q})',\mathbf{x})^{-1}||<||\alpha_{2}||$
and
$||\alpha_{2}(\mathbf{\bar{A}}_{p}\cap B_{q},\mathbf{\tilde{x}}_{p}x_{p}^{-1})(\mathbf{B'},\mathbf{y})||<||\alpha_{2}||$.
\end{lemma}

\begin{proof}
First, note that inner automorphisms stabilise each point of $\mathcal{C}_{n}$, and hence each domain in the Space of Domains. Writing $\gamma(z)$ for the inner automorphism which conjugates everything by $z$, we then see that
$\alpha_{1}=\alpha_{2}(\mathbf{A},\mathbf{x})=\alpha_{2}(\mathbf{A},\mathbf{x})\gamma(x_{p}^{-1})=\alpha_{2}(\mathbf{\bar{A}}_{p},\mathbf{\tilde{x}}_{p}x_{p}^{-1})$.
Similarly, $\alpha_{3}=\alpha_{2}(\mathbf{\bar{B}}_{q},\mathbf{\tilde{y}}_{q}y_{q}^{-1})$.
Note that $\mathbf{\hat{\bar{A}}}_{p}=\bar{A_{p}}$ and $\mathbf{\hat{\bar{B}}}_{q}=\bar{B_{q}}$.
Since $\mathbf{x}\subset H_{i}\in B_{q}$ and $\mathbf{y}\subset H_{j}\in A_{p}$, then $\mathbf{\tilde{x}}_{p}x_{p}^{-1}\subset H_{i}\not\in\mathbf{\hat{\bar{A}}}_{p}$ and  $\mathbf{\tilde{y}}_{q}y_{q}^{-1}\subset H_{j}\not\in\mathbf{\hat{\bar{B}}}_{q}$. \\
If $A_{p}\cup B_{q}\not=\hat{H}$ then by Lemma \ref{1b peak}, either $\|\alpha_{2}(\mathbf{\bar{A}}_{p}-\bar{B_{q}},\mathbf{\tilde{x}}_{p}x_{p}^{-1})\|<\|\alpha_{2}\|$ or $\|\alpha_{2}(\mathbf{\bar{B}}_{q}-\bar{A_{p}},\mathbf{\tilde{y}}_{q}y_{q}^{-1})\|<\|\alpha_{2}\|$.
But given arbitrary sets $C$ and $D$, $C-\bar{D}=C\cap D$.
Hence either $\|\alpha_{4}\|<\|\alpha_{2}\|$ or $\|\alpha_{5}'\|<\|\alpha_{2}\|$. \\
If $A_{p}\cup B_{q}=\hat{H}$ then $\alpha_{4}=\alpha_{1}(\mathbf{A}+_{p}B_{q},\mathbf{x})^{-1}=\alpha_{1}\gamma(x_{p}^{-1})=\alpha_{1}$.
Similarly, $\alpha_{5}'=\alpha_{3}$, and since $\alpha_{1}\dash \alpha_{2}\dash \alpha_{3}$ is a peak, then $\|\alpha_{2}\|>\min(\|\alpha_{4}\|,\|\alpha_{5}'\|)$. \\
Now one of $\alpha_{4}\dash \alpha_{2}\dash \alpha_{3}$ or $\alpha_{1}\dash \alpha_{2}\dash \alpha_{5}'$ is a peak satisfying Case 2a, and the result follows from Lemma \ref{2a peak}.
\end{proof}

\begin{prop}\label{prop 4 summary} 
Let $H_{1}\ast\dots\ast H_{n}$ be an $\mathfrak{S}$ free factor splitting for $G$.
Let $(\mathbf{A},\mathbf{x})$ and $(\mathbf{B},\mathbf{y})$ be relative multiple Whitehead automorphisms with $\mathbf{x}\subset H_{i}$ and $\mathbf{y}\subset H_{j}$ and $\hat{A}\subset\{H_{1},\dots,H_{n}\}-\{H_{i}\}$, $\hat{B}\subset\{H_{1},\dots,H_{n}\}-\{H_{j}\}$ such that for some $p$ and $q$ we have $H_{i}\in B_{q}$ and $H_{j}\in A_{p}$.
Let $\alpha_{2}$ be the domain whose $\alpha$-graph has $\mathfrak{S}$-labelling $(H_{1},\dots,H_{n})$, and let $\alpha_{1}=\alpha_{2}(\mathbf{A},\mathbf{x})$ and $\alpha_{3}=\alpha_{2}(\mathbf{B},\mathbf{y})$.
If
\begin{tikzcd}[cramped]
\alpha_{1} \ar[r,-<-,dash,"(\mathbf{A}{,}\mathbf{x})"]	& \alpha_{2} \ar[r,->-,dash,"(\mathbf{B}{,}\mathbf{y})"]	& \alpha_{3}
\end{tikzcd}
is a peak, then it is reducible.
\end{prop}

\begin{proof}
By Lemmas \ref{4 loop} and \ref{4 contractible}, the path 
\begin{tikzcd}[cramped]
\alpha_{1} \ar[r,-<-,dash,"(\mathbf{A}{,}\mathbf{x})"]	& \alpha_{2} \ar[r,->-,dash,"(\mathbf{B}{,}\mathbf{y})"]	& \alpha_{3}
\end{tikzcd}
is homotopic in the Space of Domains to each of the paths
\\ \noindent
\begin{tikzcd}[sep=huge]
\alpha_{1} \ar[r,-<-,dash,"((\mathbf{A}+_{p}B_{q})'{,}\mathbf{x})"]
&[-0.5cm] \alpha_{2}(\mathbf{\bar{A}}_{p}\cap B_{q},\mathbf{\tilde{x}}_{p}x_{p}^{-1}) \ar[r,->-,dash,"(\mathbf{B'}{,}\mathbf{y})"]
&[-1.2cm] \alpha_{2}(\mathbf{\bar{A}}_{p}\cap B_{q},\mathbf{\tilde{x}}_{p}x_{p}^{-1})(\mathbf{B'},\mathbf{y}) \ar[r,-<-,dash,"((\mathbf{\bar{A}}_{p}\cap B_{q})^{y_{q}}{,}(\mathbf{\tilde{x}}_{p}x_{p}^{-1})^{y_{q}})"]
&[0.7cm] \alpha_{3}
\end{tikzcd}
\\ \noindent and \\ \noindent
\begin{tikzcd}[sep=huge]
\alpha_{1} \ar[r,->-,dash,"((\mathbf{\bar{B}}_{q}\cap A_{p})^{x_{p}}{,}(\mathbf{\tilde{y}}_{q}y_{q}^{-1})^{x_{p}})"]
&[0.7cm] \alpha_{2}(\mathbf{\bar{B}}_{q}\cap A_{p}{,}\mathbf{\tilde{y}}_{q}y_{q}^{-1})(\mathbf{A'},\mathbf{x}) \ar[r,-<-,dash,"(\mathbf{A'}{,}\mathbf{x})"] 
&[-1.2cm] \alpha_{2}(\mathbf{\bar{B}}_{q}\cap A_{p}{,}\mathbf{\tilde{y}}_{q}y_{q}^{-1})
\ar[r,->-,dash,"((\mathbf{B}+_{q}A_{p})'{,}\mathbf{y})"]
&[-0.5cm] \alpha_{3}	\ .
\end{tikzcd}

By Lemma \ref{4 peak}, either
$||\alpha_{2}(\mathbf{A},\mathbf{x})((\mathbf{A}+_{p}B_{q})',\mathbf{x})^{-1}||<||\alpha_{2}||$
and \\ \noindent
$||\alpha_{2}(\mathbf{\bar{A}}_{p}\cap B_{q},\mathbf{\tilde{x}}_{p}x_{p}^{-1})(\mathbf{B'},\mathbf{y})||<||\alpha_{2}||$
, or
$||\alpha_{2}(\mathbf{\bar{B}}_{q}\cap A_{p},\mathbf{\tilde{y}}_{q}y_{q}^{-1})(\mathbf{A'},\mathbf{x})||<||\alpha_{2}||$
and
$||\alpha_{2}(\mathbf{\bar{B}}_{q}\cap A_{p}{,}\mathbf{\tilde{y}}_{q}y_{q}^{-1})||<||\alpha_{2}||$.
Thus by Definition \ref{defn reducible}, one of the above paths is a reduction for the peak
\begin{tikzcd}[cramped]
\alpha_{1} \ar[r,-<-,dash,"(\mathbf{A}{,}\mathbf{x})"]	& \alpha_{2} \ar[r,->-,dash,"(\mathbf{B}{,}\mathbf{y})"]	& \alpha_{3}
\end{tikzcd}~.
\end{proof}


\subsection{Simple Connectivity}

We have now done all the required work to conclude:

\begin{prop}\label{peak reduction thm} 
Every peak
\begin{tikzcd}[cramped]
\alpha_{1} \ar[r,-<-,dash,"(\mathbf{A}{,}\mathbf{x})"]	& \alpha_{2} \ar[r,->-,dash,"(\mathbf{B}{,}\mathbf{y})"]	& \alpha_{3}
\end{tikzcd}
(whose edges are both of Type $A$)
in the Space of Domains is reducible (to a path of length $2$ or $3$ whose edges are all of Type $A$).
\end{prop}

\begin{proof}
Let
\begin{tikzcd}[cramped]
\alpha_{1} \ar[r,-<-,dash,"(\mathbf{A}{,}\mathbf{x})"]	& \alpha_{2} \ar[r,->-,dash,"(\mathbf{B}{,}\mathbf{y})"]	& \alpha_{3}
\end{tikzcd}
be a peak in the Space of Domains whose edges are both of Type $A$.
Suppose $\alpha_{2}$ has $\mathfrak{S}$-labelling $(H_{1},\dots,H_{n})$. Then for some $i$ and $j$ we have $\mathbf{x}\subset H_{i}$ and $\mathbf{y}\subset H_{j}$.
By assumption, $H_{i}\not\in\hat{A}$ and $H_{j}\not\in\hat{B}$.
If $i=j$ then by Proposition \ref{prop i=j peaks reducible}, the peak is reducible.
Otherwise, our peak falls into one of the Cases 1--4 as described in Observation \ref{obs cases of peak}, and by Propositions \ref{prop 1a summary}, \ref{prop 1b summary}, \ref{prop 2a summary}, \ref{prop 2b summary}, and \ref{prop 4 summary}, we are done (after renaming, if we fell under Case 3).
\end{proof}

We can now use this Peak Reduction Proposition to argue that the Space of Domains, and hence the complex $\mathcal{C}_{n}$, is simply connected.

\begin{thm}\label{Space of Domains is simply connected}
Our Space of Domains is simply connected.
\end{thm}

\begin{proof}
Let $\lambda$ be a loop in the Space of Domains. 
Note that any loop is homotopic to a based loop, so without loss of generality, we assume $\lambda$ contains the basepoint $\alpha_{0}$.
By Corollary \ref{type A edges} and Proposition \ref{peak reduction thm}, $\lambda$ is homotopic (in the Space of Domains) to a peak reduced loop $\lambda'$.
But any peak reduced loop must have constant height (else it would contain some `heighest' point, i.e. a peak).
Since the basepoint has height $0$ (Lemma \ref{zero height is basepoint}) then every point in $\lambda'$ must have height $0$.
But again by Lemma \ref{zero height is basepoint}, the only point with height $0$ is the basepoint.
Hence $\lambda'$ is actually the constant `loop' at the basepoint, $\alpha_{0}$.
Thus any loop $\lambda$ is homotopic to a constant loop, hence the Space of Domains is simply connected.
\end{proof}

\simplyconnected

\begin{proof}
This follows directly from Theorem \ref{Space of Domains is simply connected} and Proposition \ref{prop Cn is sc if SoD is}.
\end{proof}




\end{document}